\documentclass[12pt,english]{article}
\usepackage[T1]{fontenc}
\usepackage[latin9]{inputenc}
\usepackage{geometry}
\usepackage{babel}
\usepackage{mathtools}
\usepackage{amsmath}
\usepackage{amsthm}
\usepackage{amssymb}
\usepackage[unicode=true,pdfusetitle,
 bookmarks=true,bookmarksnumbered=false,bookmarksopen=false,
 breaklinks=false,pdfborder={0 0 0},pdfborderstyle={},backref=false,colorlinks=false]
 {hyperref}

\makeatletter
\theoremstyle{plain}
\newtheorem*{thm*}{\protect\theoremname}
\theoremstyle{plain}
\newtheorem{thm}{\protect\theoremname}[section]
\theoremstyle{remark}
\newtheorem{notation}[thm]{\protect\notationname}
\theoremstyle{remark}
\newtheorem*{acknowledgement*}{\protect\acknowledgementname}
\theoremstyle{definition}
\newtheorem{defn}[thm]{\protect\definitionname}
\theoremstyle{definition}
\newtheorem{example}[thm]{\protect\examplename}
\theoremstyle{remark}
\newtheorem{rem}[thm]{\protect\remarkname}
\theoremstyle{plain}
\newtheorem{prop}[thm]{\protect\propositionname}
\theoremstyle{plain}
\newtheorem{lem}[thm]{\protect\lemmaname}
\theoremstyle{plain}
\newtheorem{cor}[thm]{\protect\corollaryname}

\@ifundefined{date}{}{\date{}}
\newcommand{\bjarrow}{
  \hookrightarrow\mathrel{\mspace{-15mu}}\rightarrow
}

\newcommand{\End}{\operatorname{End}}
\newcommand{\Id}{\operatorname{Id}}
\newcommand{\Hom}{\operatorname{Hom}}
\newcommand{\Aut}{\operatorname{Aut}}
\newcommand{\tr}{\operatorname{tr}}
\renewcommand{\r}{\operatorname{r}}
\newcommand{\Tr}{\operatorname{Tr}}

\newcommand{\lui}[2]{\prescript{#1}{}{#2}}

\newcommand{\R}{\mathcal{R}}
\newcommand{\Z}{\mathbb{Z}}
\newcommand{\Zp}[1][p]{\Z_{#1}}

\newcommand{\F}{\mathcal{F}}
\newcommand{\Fc}{\F^c}
\renewcommand{\H}{\mathcal{H}}
\newcommand{\G}{\mathcal{G}}

\newcommand{\FAB}[2]{\F_{#1}\left(#2\right)}
\newcommand{\FSG}[1][G]{\FAB{S}{#1}}
\newcommand{\FHH}[1][H]{\F_{#1}} %\newcommand{\FHH}[1][H]{\FAB{#1}{#1}}

\newcommand{\Orbitize}[1]{\mathcal{O}\left(#1\right)}
\newcommand{\OF}{\Orbitize{\F}}
\newcommand{\ONF}{\Orbitize{\NF}}

\newcommand{\OFc}{\Orbitize{\Fc}}

\newcommand{\additiveCompletion}[1]{#1_{\sqcup}}
\newcommand{\aOFc}{\additiveCompletion{\OFc}}
\newcommand{\aOF}{\additiveCompletion{\OF}}

\newcommand{\Norm}[2]{N_{#1}\left(#2\right)}
\newcommand{\NF}{N_{\F}}
\newcommand{\NFH}[1][H]{N_{\F}\left(#1\right)}
\newcommand{\NSH}[1][H]{\Norm{S}{#1}}
\newcommand{\NS}{N_S}
\renewcommand{\NG}[1][H]{\Norm{G}{#1}}

\newcommand{\FmuFR}[2][\R]{\mu_{#1}\left(#2\right)}
\newcommand{\muFR}[1][\R]{\FmuFR[#1]{\F}}
\newcommand{\muFcR}[1][\R]{\FmuFR[#1]{\Fc}}
\newcommand{\FmuFRmod}[2][\R]{\FmuFR[#1]{#2}\operatorname{-mod}}
\newcommand{\muFRmod}[1][\R]{\muFR[#1]\operatorname{-mod}}

\newcommand{\I}{\mathcal{I}}
\newcommand{\J}{\mathcal{J}}

\newcommand{\MackHR}[2][\R]{\text{Mack}_{#1}\left(#2\right)}
\newcommand{\MackFHR}[3][\R]{\text{Mack}_{#1}^{{#2}^c}\left(#3\right)}

\newcommand{\MackFR}[1][\R]{\MackHR[#1]{\F}}
\newcommand{\MackFcR}[1][\R]{\MackHR[#1]{\F^c}}

\newcommand{\BGR}[2][\R]{B_{#1}^{#2}}
\newcommand{\BFcR}[1][\R]{\BGR[#1]{\Fc}}

\makeatother

\providecommand{\acknowledgementname}{Acknowledgement}
\providecommand{\corollaryname}{Corollary}
\providecommand{\definitionname}{Definition}
\providecommand{\examplename}{Example}
\providecommand{\lemmaname}{Lemma}
\providecommand{\notationname}{Notation}
\providecommand{\propositionname}{Proposition}
\providecommand{\remarkname}{Remark}
\providecommand{\theoremname}{Theorem}

\begin{document}
\title{Green correspondence on centric Mackey functors over fusion systems.}
\author{Marco Praderio Bova}
\maketitle
\begin{abstract}
In this paper we give a definition of (centric) Mackey functor over
a fusion system (Definitions \ref{def:Mackey-functor.} and \ref{def:F-centric-Mackey-functor.})
which generalizes the notion of Mackey functor over a group. In this
context we prove that, given some conditions on a related ring, the
centric Burnside ring over a fusion system (as defined in \cite{BurnsideRingFusionSystemsDiazLibman})
acts on any centric Mackey functor (Proposition \ref{prop:Action-of-centric-burnside-ring.}).
We also prove that the Green correspondence holds for centric Mackey
functors over fusion systems (Theorem \ref{thm:Green-correspondence.}).
As a mean to prove this we introduce a notion of relative projectivity
for centric Mackey functors over fusion systems (Definition \ref{def:Relative-projectivity.})
and provide a decomposition of a particular product in $\aOFc$ (Definition
\ref{def:Additive-extension}) in terms of the product in $\additiveCompletion{\Orbitize{\left(\NF\right)^{c}}}$
(Theorem \ref{thm:decomposition-direct-product.}).
\end{abstract}
\tableofcontents{}

\section{Introduction.}

A Mackey functor is an algebraic structure possessing operations which
behave like the induction, restriction and conjugation maps in group
representation theory. The concept of Mackey functor has been generalized
to algebraic structures other than groups (see for example \cite{TwoClassificationWebb}).
We are particularly interested on their generalization to fusion systems.

Fusion systems, as defined by Puig in \cite{FrobeniusCategoriesPuig}
(where he calls them Frobenius Categories), are categories intended
to convey the $p$-local structure of a finite group $G$. Oddly enough
not all fusion systems can be derived from finite groups. This gives
them an interest of their own. 

When generalizing to fusion systems, Mackey functors inevitably lose
some properties. One of these is the existence of a Green correspondence.

The Green correspondence first appeared in \cite{GREENTransferTheoremForModularRepresentations}
under the following form.
\begin{thm*}
(\cite[Theorem 2]{GREENTransferTheoremForModularRepresentations})
Let $p$ be a prime, let $\R$ be a complete local $PID$ with residue
field of characteristic $p$, let $G$ be a finite group and let $H$
be a $p$-subgroup of $G$. There exists a one to one correspondence
between finitely generated indecomposable $\R G$-modules with vertex
$H$ and finitely generated $\R\NG$-modules with vertex $H$.
\end{thm*}
This result was later generalized in \cite{GreenAxiomaticRepresentation}
and \cite{SASAKI198298} to Green functors and Mackey functors over
groups respectively.

In this paper we will prove that, even thought there is in general
no Green correspondence for Mackey functors over fusion systems, a
similar result can be found for centric Mackey functors over a fusion
system (see Definition \ref{def:F-centric-Mackey-functor.} and Theorem
\ref{thm:Green-correspondence.}). This can be used in order to study
centric Mackey functors over a fusion system $\F$ in terms of Mackey
functors over fusion systems of the form $N_{\F}\left(H\right)$ (see
Example \ref{exa:definition-NF.}) with $H\in\Fc$ (see Definition
\ref{def:F-centric.}) fully $\F$-normalized (see Definition \ref{def:Fully-F-normalized.}).
It is known (see \cite[Section 4]{SubgroupFamiliesControllingpLocalFiniteGroups})
that fusion systems of this form derive from finite groups. This makes
them easier to work with than other fusion systems and, therefore,
motivates the interest in proving that the Green correspondence holds
for centric Mackey functors over fusion systems.

The paper is organized as follows.

In Section \ref{sec:Background.} we will briefly recall the definitions
of (saturated) fusion system (Definitions \ref{def:Fusion-system.}
and \ref{def:Saturated-fusion-system.}), of (centric) Mackey functor
over a fusion system (Definitions \ref{def:Mackey-functor.} and \ref{def:F-centric-Mackey-functor.})
and of centric Burnside ring of a fusion system (Definition \ref{def:centric-Burnside-ring.}).
In this section we will also recall some well known properties regarding
these concepts and prove 3 further results. The first one (Proposition
\ref{prop:Decomposition-centric-Mackey-functor.}) will describe a
decomposition of certain induced Mackey functors (see Definition \ref{def:Restriction-induction-and-conjugation-functors.}).
The second result (Lemma \ref{lem:Mackey-formula-induction-restriction.})
will allow us to rewrite the composition of certain induction and
restriction functors (see Definition \ref{def:Restriction-induction-and-conjugation-functors.}).
The third one (Proposition \ref{prop:Action-of-centric-burnside-ring.})
will, under certain conditions concerning a related ring, describe
an action of the centric Burnside ring over a fusion system on any
centric Mackey functor over that fusion system.

In Section \ref{sec:Relative-projectivity-and-the-Higman's-criterion.}
we will introduce the concept of relative projectivity of a Mackey
functor over a fusion system (Definition \ref{def:Relative-projectivity.})
and prove that Higman's criterion holds for Mackey functors over fusion
systems (Theorem \ref{thm:Higman's-criterion.}). To do this we will
need to define the transfer and restriction maps (Definition \ref{def:Transfer-and-restriction.})
and list some of the properties they satisfy (Proposition \ref{prop:Properties-of-transfer-restriction-and-conjugation-maps.}).
These properties will later be needed in Subsections \ref{subsec:Restriction-map-from-trHF-to-trHNF+trYNF.}
and \ref{subsec:decomposition-of-product-in-aOFc.}.

We will conclude with Section \ref{sec:Green-correspondence.} where
we will prove our two main results (Theorems \ref{thm:decomposition-direct-product.}
and \ref{thm:Green-correspondence.}). In Subsection \ref{subsec:Categorical-version.}
we will state and prove Proposition \ref{prop:Green-correspondence-for-endomorphisms.}
which generalizes \cite[Proposition 4.34]{GreenAxiomaticRepresentation}
(see Example \ref{exa:Green.correspondence-Green-functors.}) and
is key too proving that the Green correspondence holds for centric
Mackey functors over fusion systems. Subsections \ref{subsec:N-is-a-direct-summand-of-Ninduction-restriction.}-\ref{subsec:Transfer-map-from-trHNF-to-trHF.}
will be dedicated to developing the tools necessary to prove that
Proposition \ref{prop:Green-correspondence-for-endomorphisms.} can
be applied in the context of centric Mackey functors over fusion systems.
More precisely, during these subsections, we will study different
compositions of the induction and restriction functors (see Definition
\ref{def:Restriction-induction-and-conjugation-functors.}) and of
transfer and restriction maps (see Definition \ref{def:Transfer-and-restriction.})
and prove Theorem \ref{thm:decomposition-direct-product.} which will
allow us to write certain products in $\aOFc$ (Definition \ref{def:Additive-extension})
in terms of products in $\additiveCompletion{\mathcal{O}\left(N_{\F}\left(H\right)\right)}$
for some fully $\F$-normalized (see Definition \ref{def:Fully-F-normalized.}),
$\F$-centric (see Definition \ref{def:F-centric.}) group $H$. Finally,
we will conclude in Subsection \ref{subsec:Proof-of-Green-correspondence.}
where we will use the developed tools in order to apply Proposition
\ref{prop:Green-correspondence-for-endomorphisms.} in the context
of centric Mackey functors over fusion systems and deduce from it
Theorem \ref{thm:Green-correspondence.} which shows that the Green
correspondence holds in the context of centric Mackey functors over
fusion systems.

We conclude this introduction with a brief summary of some common
notation that we will be using throughout the paper
\begin{notation}
\label{nota:Initial-notation.}$\phantom{.}$
\begin{itemize}
\item Given a unital ring $\R$ we will denote by $1_{\R}$ its multiplicative
identity element.
\item Given a group $G$ we will denote by $1_{G}$ the neutral element
of $G$.
\item Given groups $H,K$ such that $H\le K$ we will denote by $\iota_{H}^{K}$
(or simply by $\iota$ if $H$ and $K$ are clear) the natural inclusion
map from $H$ to $K$.
\item All modules over rings will be understood to be left modules unless
otherwise specified.
\item Given rings $\R,\mathcal{S}$ and $\mathcal{S}'$ such that $\R\subseteq\mathcal{S},\mathcal{S}'$
and modules $M$ and $N$ over $\mathcal{S}$ and $\mathcal{S}'$
respectively we will write $M\cong_{\R}N$ to denote that $M$ and
$N$ are equivalent as $\R$-modules.
\item Given finite groups $H$ and $K$ and an $\left(H,K\right)$-biset
$X$ (e.g. we can take $X:=J$ for some finite group $J$ satisfying
$H,K\le J$) we will denote by $\left[H\backslash X/K\right]$ any
choice of representatives of $\left(H,K\right)$-orbits of $X$.
\item Let $\mathcal{D}\subseteq\mathcal{C}$ be categories, unless otherwise
specified, we will write $X\in\mathcal{C}$ to denote that $X$ is
an object of $\mathcal{C}$ and $X\in\mathcal{C}\backslash\mathcal{D}$
to denote that $X\in\mathcal{C}$ and $X\not\in\mathcal{D}$.
\item Given a fusion system $\F$, objects $A,B\in\F$ and a morphism $\varphi\in\Hom_{\F}\left(A,B\right)$,
we will denote by $\overline{\varphi}\in\Hom_{\OF}\left(A,B\right)$
the morphism in $\OF$ with representative $\varphi$ (see Definition
\ref{def:Orbit-category.}).
\item Given a category $\mathcal{C}$, objects $X,Y,X'$ and $Y'$ in $\mathcal{C}$
and a morphism $\varphi\in\Hom_{\mathcal{C}}\left(X,Y\right)$ we
will denote by $\varphi_{*}$ and $\varphi^{*}$ the following induced
maps between hom sets 
\begin{align*}
\varphi_{*}:=\Hom\left(X',-\right)\left(\varphi\right): & \underset{\psi}{\Hom\left(X',X\right)}\underset{\to}{\to}\underset{\varphi\psi}{\Hom\left(X',Y\right)},\\
\varphi^{*}:=\Hom\left(-,Y'\right)\left(\varphi\right): & \underset{\theta}{\Hom\left(Y,Y'\right)}\underset{\to}{\to}\underset{\theta\varphi}{\Hom\left(X,Y'\right)}.
\end{align*}
\end{itemize}
\end{notation}

\begin{acknowledgement*}
The author would like to thank his PhD supervisor Nadia Mazza for
her guidance on his research and the seemingly limitless amount of
resources she is able to provide. He would also like to thank Lancaster
University for the funding provided to conduct his PhD and the reviewers
for their big help on improving the quality of the paper.
\end{acknowledgement*}

\section{\label{sec:Background.}Background and first results.}

In this section we will review the concepts of fusion systems, of
Mackey functors over a fusion system and of centric Burnside ring
of a fusion system. The main results shown in this section are:
\begin{itemize}
\item Proposition \ref{prop:properties-HXK.}: which provide us with ways
of rewriting the sets $\left[H\times K\right]$ (see Definition \ref{def:HX_FK}).
These will become very useful in future calculations.
\item Propositions \ref{prop:Mackey-algebra-basis.} and \ref{prop:Decomposition-centric-Mackey-functor.}
and Lemma \ref{lem:Mackey-formula-induction-restriction.}: which
will translate \cite[Propositions 3.2 and 5.3]{StructureMackeyFunctors}
to the context of Mackey functors over fusion systems thus providing
us with some insight concerning the Mackey algebra (see Definition
\ref{def:Mackey-algebra.}) and the induction and restriction functors
(see Definition \ref{def:Restriction-induction-and-conjugation-functors.}).
\item Proposition \ref{prop:Action-of-centric-burnside-ring.}: which will
translate \cite[Proposition 9.2]{StructureMackeyFunctors} to the
context of centric Mackey functors over fusion systems by first describing
an action of the centric Burnside ring of a fusion system (see Definition
\ref{def:centric-Burnside-ring.}) on any centric Mackey functor over
a fusion system and then rewriting it in terms of the morphisms $\theta^{H}$
and $\theta_{H}$ (see Definition \ref{def:theta_S-and-theta^S.}). 
\end{itemize}
The reader already familiar with these concepts may safely skip this
section keeping in mind the results mentioned above.

\subsection{\label{subsec:Fusion-systems.}Fusion systems.}

What follows is a brief introduction to fusion systems which mostly
aims to establish some notation. For a more thorough introduction
please refer to \cite{IntroductionToFusionSystemsLinckelmann}. In
this subsection we will also report the main results of \cite[Section 4]{FrobeniusCategoriesPuig}
which, given a saturated fusion system $\F$, proves constructively
the existence of products and pullbacks in the category $\aOFc$ (see
Definition \ref{def:Additive-extension} and Propositions \ref{prop:Pullback-OFc.}
and \ref{prop:Product-OFc.}). We will conclude this subsection with
Proposition \ref{prop:properties-HXK.} which will allow us to write
products in $\aOFc$ in terms of other products in the same category.
\begin{defn}
\label{def:Fusion-system.}Let $p$ be a prime and let $S$ be a finite
$p$-group. A\textbf{ fusion system} \textbf{over }$S$ is a category
$\F$ having as objects subgroups of $S$ and satisfying the following
properties for every $H,K\le S$:
\begin{enumerate}
\item Every morphism $\varphi\in\Hom_{\F}\left(H,K\right)$ is an injective
group homomorphism and the composition of morphisms in $\F$ is the
same as the composition of morphisms in the category of groups.
\item $\Hom_{S}\left(H,K\right)\subseteq\Hom_{\F}\left(H,K\right)$. That
is, every group homomorphism from $H$ to $K$ that can be described
as conjugation by an element of $S$ followed by inclusion is a morphism
in $\F$.
\item For every $\varphi\in\Hom_{\F}\left(H,K\right)$ let $\tilde{\varphi}:H\to\varphi\left(H\right)$
be the isomorphism obtained by looking at $\varphi$ as an isomorphism
onto its image. Both $\tilde{\varphi}$ and $\tilde{\varphi}^{-1}$
are isomorphisms in $\F$.
\end{enumerate}
\end{defn}

\begin{example}
\label{exa:fusion-system.}The most common example of fusion system
is obtained by taking a finite group $G$ containing a $p$-group
$S$ and defining $\FSG$ as the fusion system over $S$ whose morphisms
are given by conjugation with elements of $G$ followed by inclusion.
When $S=G$ we will often write $\FHH[S]$ instead of $\F_{S}\left(S\right)$
although the later is the more common notation in the literature.
\end{example}

Definition \ref{def:Fusion-system.} and Example \ref{exa:fusion-system.}
motivate the introduction of the following notation.
\begin{notation}
\label{nota:p,S,F}From now on, unless otherwise specified, all introduced
groups will be understood to be finite, \textbf{$p$ will denote a
prime integer}, \textbf{$S$ will denote a finite $p$-group} and
\textbf{$\F$ will denote a fusion system over $S$}. Moreover, given
subgroups $H,K\le S$ we will write $H=_{\F}K$ if $H$ and $K$ are
isomorphic in $\F$, $H\le_{\F}K$ if there exists $J\le K$ such
that $H=_{\F}J$ and either $H\lneq_{\F}K$ or $H<_{\F}K$ if $H\le_{\F}K$
but $H\not=_{\F}K$. 
\end{notation}

When the term fusion system appears in the literature it is usually
in reference to a particular type of fusion system called saturated
fusion system. These are fusion systems that are built to generalize
Example \ref{exa:fusion-system.} in the case where $S$ is a Sylow
$p$-subgroup of $G$.
\begin{defn}
\label{def:Fully-F-normalized.}Let $H\le S$. We say that $H$ is
\textbf{fully $\F$-normalized} if for every $K=_{\F}H$ we have that
$\left|\NSH[K]\right|\le\left|\NSH\right|$.
\end{defn}

\begin{defn}
\label{def:varphi-normalizer.}Let $H,K\le S$ and let $\varphi\colon H\to K$
be a morphism in $\F$. We define the \textbf{$\varphi$-normalizer}
as the following subgroup of $\NSH$
\[
N_{\varphi}:=\left\{ x\in\NSH\,:\exists z\in\NSH[\varphi\left(H\right)]\text{ such that }\varphi\left(\lui{x}{h}\right)=\lui{z}{\varphi\left(h\right)}\;\forall h\in H\right\} .
\]
\end{defn}

\begin{defn}
\label{def:Saturated-fusion-system.}A fusion system $\F$ is said
to be \textbf{saturated} if the following 2 conditions are satisfied:
\begin{enumerate}
\item $\Aut_{S}\left(S\right)$ is a Sylow $p$-subgroup of $\Aut_{\F}\left(S\right)$.
\item For every $H\le S$ and every $\varphi\in\Hom_{\F}\left(H,S\right)$
such that $\varphi\left(H\right)$ is fully $\F$-normalized there
exists $\hat{\varphi}\in\Hom_{\F}\left(N_{\varphi},S\right)$ such
that $\hat{\varphi}\iota_{H}^{N_{\varphi}}=\varphi$.
\end{enumerate}
\end{defn}

\begin{example}
The fusion system $\FSG$ of Example \ref{exa:fusion-system.} is
saturated if $S$ is a Sylow $p$ subgroup of $G$.
\end{example}

\begin{example}
\label{exa:definition-NF.}Given a saturated fusion system $\F$ and
a fully $\F$-normalized subgroup $H\le S$, we can define the saturated
fusion system $\NFH$ over $\NSH$ by setting for every $A,B\le\NSH$
\[
\Hom_{\NFH}\left(A,B\right):=\left\{ \varphi\in\Hom_{\F}\left(A,B\right)\,|\,\exists\hat{\varphi}\in\Hom_{\F}\left(AH,BH\right)\text{ s.t. }\iota_{B}^{BH}\varphi=\hat{\varphi}\iota_{A}^{AH}\right\} .
\]
\end{example}

Definition \ref{def:Saturated-fusion-system.} motivates the introduction
of the following notation.
\begin{notation}
From now on, unless otherwise specified, \textbf{all introduced fusion
systems will be understood to be saturated}. In particular $\F$ will
denote a saturated fusion system over a finite $p$-group $S$.
\end{notation}

When dealing with Mackey functors over fusion systems (as we will
be doing throughout this paper) it is convenient not to work with
the fusion system directly but rather with its orbit category.
\begin{defn}
\label{def:Orbit-category.}We define the \textbf{orbit category of
a fusion system $\F$ }as the category \textbf{$\OF$} having as objects
the same objects as $\F$ and as morphisms
\[
\Hom_{\OF}\left(H,K\right):=\Aut_{K}\left(K\right)\backslash\Hom_{\F}\left(H,K\right),
\]
for every $H,K\le S$. Here $\Aut_{K}\left(K\right)$ is acting on
$\Hom_{\F}\left(H,K\right)$ by composing on the left.
\end{defn}

An important subcategory of $\OF$ which we will be dealing with often
because of its nice properties is its centric subcategory.
\begin{defn}
\label{def:F-centric.}Let $H\le S$. We say that $H$ is \textbf{$\F$-centric}
if $C_{S}\left(K\right)\le K$ for every $K=_{\F}H$. The \textbf{centric
subcategory of $\F$} (denoted by $\Fc$) is defined as the full subcategory
of $\F$ having as objects $\F$-centric subgroups of $S$. Likewise,
the \textbf{centric subcategory of $\OFc$} (denoted by $\OFc$) is
the full subcategory of $\OF$ having as objects the $\F$-centric
subgroups of $S$.
\end{defn}

We are in fact particularly interested in the additive extension of
$\OFc$.
\begin{defn}
\label{def:Additive-extension}(see \cite[Section 4]{JackowskiMcClureHomotopyDecompositionViaAbelianSubgroups})
Let $\F$ be a fusion system. We denote by $\aOFc$ the \textbf{additive
extension }of $\OFc$. That is $\aOFc$ is the category having as
objects formal finite (possibly empty) coproducts of the form $\bigsqcup_{i=1}^{n}H_{i}$,
where each $H_{i}$ is an object in $\OFc$, and as morphisms $\boldsymbol{f}:\bigsqcup_{i=1}^{n}H_{i}\to\bigsqcup_{j=1}^{m}K_{j}$
tuples of the form $\boldsymbol{f}:=\left(\sigma,\left\{ f_{i}\right\} _{i=1,\dots,n}\right)$
where $\sigma\colon\left\{ 1,\dots,n\right\} \to\left\{ 1,\dots,m\right\} $
is any map and $f_{i}\in\Hom_{\OFc}\left(H_{i},K_{\sigma\left(i\right)}\right)$.
Composition is given by
\[
\left(\tau,\left\{ g_{j}\right\} _{i=j,\dots,m}\right)\left(\sigma,\left\{ f_{i}\right\} _{i=1,\dots,n}\right)=\left(\tau\sigma,\left\{ g_{\sigma\left(i\right)}f_{i}\right\} _{i=1,\dots,n}\right).
\]
Whenever $\sigma$ is clear (for example when $m=1$), we will simply
write
\[
\bigsqcup_{i=1}^{n}f_{i}:=\left(\sigma,\left\{ f_{i}\right\} _{i=1,\dots,n}\right).
\]
We will often abuse notation and consider objects in $\OFc$ as objects
in $\aOFc$ via the natural inclusion of categories $\OFc\hookrightarrow\aOFc$.
\end{defn}

\begin{rem}
It may help to think of the additive extension of a category $\mathcal{C}$
as the full subcategory of the category of diagrams $\text{Set}^{\mathcal{C}^{\text{op}}}$
having as objects finite coproducts of contravariant functors of the
form $\Hom_{\mathcal{C}}\left(-,X\right)$ with $X$ an object in
$\mathcal{C}$. Yoneda's Lemma assures us that the category obtained
in this way is equivalent to the one described in Definition \ref{def:Additive-extension}.
\end{rem}

In \cite{FrobeniusCategoriesPuig} Puig proves constructively that
the category $\aOFc$ admits both products and pullbacks.
\begin{prop}
\label{prop:Pullback-OFc.}(\cite[4.8]{FrobeniusCategoriesPuig})
The category $\aOFc$ admits pullbacks which are distributive with
respect to its coproducts. Moreover, given $H,K,J\in\Fc$ such that
$H,K\le J$ the pullback of the diagram $H\stackrel{\overline{\iota}}{\to}J\stackrel{\overline{\iota}}{\leftarrow}K$
is given by
\begin{align*}
H\times_{J}K & :=\bigsqcup_{\begin{array}{c}
{\scriptstyle x\in\left[H\backslash J/K\right]}\\
{\scriptstyle H^{x}\cap K\in\Fc}
\end{array}}H^{x}\cap K, & \pi_{H}^{H\times_{J}K} & :=\bigsqcup_{\begin{array}{c}
{\scriptstyle x\in\left[H\backslash J/K\right]}\\
{\scriptstyle H^{x}\cap K\in\Fc}
\end{array}}\overline{\iota c_{x}}, & \pi_{K}^{H\times_{J}K} & :=\bigsqcup_{\begin{array}{c}
{\scriptstyle x\in\left[H\backslash J/K\right]}\\
{\scriptstyle H^{x}\cap K\in\Fc}
\end{array}}\overline{\iota}.
\end{align*}
\end{prop}

\begin{prop}
\label{prop:Product-OFc.}(\cite[Proposition 4.7]{FrobeniusCategoriesPuig})
The category $\aOFc$ admits products which are distributive with
respect to its coproducts.
\end{prop}

In \cite{FrobeniusCategoriesPuig}, Puig explicitly describes the
product of Proposition \ref{prop:Product-OFc.}. Since products are
distributive with respect to coproducts then, in order to define the
products in $\aOFc$, it suffices to describe the product between
any two objects $H,K\in\OFc$.

This product, denoted by $H\times_{\F}F$, can be built as follows;

First take all pairs $\left(A,\overline{\varphi}\right)$ with $A\in\OFc$
satisfying $A\le H$ and $\overline{\varphi}\in\Hom_{\OFc}\left(A,K\right)$.

Then define the preorder $\precsim_{H}$ on the set of all such pairs
by setting $\left(A,\overline{\varphi}\right)\precsim_{H}\left(B,\overline{\psi}\right)$
if and only if there exists $h\in H$ such that $A^{h}\le B$ and
$\overline{\varphi}\overline{c_{h}}=\overline{\psi}\overline{\iota_{A^{h}}^{B}}$.

Then take all pairs that are maximal under such preorder and define
among them the equivalence relation 
\begin{equation}
\left(A,\overline{\varphi}\right)\sim\left(B,\overline{\psi}\right)\stackrel{\text{def}}{\Longleftrightarrow}\left(A,\overline{\varphi}\right)\precsim_{H}\left(B,\overline{\psi}\right)\text{ and }\left(B,\overline{\psi}\right)\precsim_{H}\left(A,\overline{\varphi}\right)\label{eq:Equivalence-relation.}
\end{equation}
Finally fix any set $\left[H\times_{\F}K\right]$ containing exactly
one representative for each equivalence class of maximal elements
under this relation and define
\begin{align}
H\times_{\F}K & :=\bigsqcup_{\left(A,\overline{\varphi}\right)}A, & \pi_{H}^{H\times_{\F}K} & :=\bigsqcup_{\left(A,\overline{\varphi}\right)}\overline{\iota_{A}^{H}}, & \pi_{K}^{H\times_{\F}K} & :=\bigsqcup_{\left(A,\overline{\varphi}\right)}\overline{\varphi},\label{eq:Definition-product.}
\end{align}
where the tuples $\left(A,\overline{\varphi}\right)$ run over the
set $\left[H\times_{\F}K\right]$ and $\pi_{H}^{H\times_{\F}K}:H\times_{\F}K\to H$
and $\pi_{K}^{H\times_{\F}K}:H\times_{\F}K\to K$ denote the natural
projections associated to the product. The definition of the equivalence
$\sim$ ensures us that any choice of the set $\left[H\times_{\F}K\right]$
will lead to isomorphic constructions of $H\times_{\F}K$. Whenever
the fusion system $\F$ is clear we will simply write $H\times K$
and $\left[H\times K\right]$.

In order to reference the previous construction, it is worth introducing
the following.
\begin{defn}
\label{def:HX_FK}For every $H,K\in\Fc$ we denote by $\left[H\times_{\F}K\right]$
(or simply $\left[H\times K\right]$ if $\F$ is clear) any choice
of the set of representatives built as above. In other words $\left[H\times_{\F}K\right]$
is any set of tuples $\left(A,\overline{\varphi}\right)$ such that
$A\in\Fc$, $\overline{\varphi}\in\Hom_{\OFc}\left(A,K\right)$ and
Equation \eqref{eq:Definition-product.} is satisfied.
\end{defn}

We conclude this subsection with a series of identities that will
allow us to write $H\times K$ in terms of other products in $\aOFc$.
\begin{prop}
\label{prop:properties-HXK.}For every $H,K\in\Fc$
\begin{enumerate}
\item \label{enu:HX_FK-and-KX_FH}We can take
\[
\left[K\times H\right]=\left\{ \left(\varphi\left(A\right),\overline{\iota\varphi^{-1}}\right)\,:\,\left(A,\overline{\varphi}\right)\in\left[H\times K\right]\right\} .
\]
Where we are viewing the representative $\varphi$ of $\,\overline{\varphi}$
as an isomorphism onto its image.
\item \label{enu:prod-if-F-is-F_S(S).}If $\F=\FHH[S]$ we can take
\[
\left[H\times_{\FHH[S]}K\right]=\bigsqcup_{\begin{array}{c}
{\scriptstyle x\in\left[K\backslash S/H\right]}\\
{\scriptstyle K^{x}\cap H\in\FHH[S]^{c}}
\end{array}}\left\{ \left(K^{x}\cap H,\overline{\iota c_{x}}\right)\right\} .
\]
\item \label{enu:prod-iso-right.} For every isomorphism $\psi\colon K\to\psi\left(K\right)$
we can take
\[
\left[H\times\psi\left(K\right)\right]=\left\{ \left(A,\overline{\psi\varphi}\right)\,:\,\left(A,\overline{\varphi}\right)\in\left[H\times K\right]\right\} .
\]
\item \label{enu:prod-iso-left.}For every isomorphism $\psi\colon H\to\psi\left(H\right)$
we can take
\[
\left[\psi\left(H\right)\times K\right]=\left\{ \left(\psi\left(A\right),\overline{\varphi\psi^{-1}}\right)\,:\,\left(A,\overline{\varphi}\right)\in\left[H\times K\right]\right\} .
\]
Where we are viewing $\psi$ as an isomorphism between the appropriate
restrictions.
\item \label{enu:prod-pullback-right.}For every $J\in\Fc$ such that $J\le K$
we can take
\[
\left[H\times J\right]=\bigsqcup_{\left(A,\overline{\varphi}\right)\in\left[H\times K\right]}\bigsqcup_{\begin{array}{c}
{\scriptstyle x\in\left[J\backslash K/\varphi\left(A\right)\right]}\\
{\scriptstyle J^{x}\cap\varphi\left(A\right)\in\Fc}
\end{array}}\left\{ \left(\varphi^{-1}\left(J^{x}\cap\varphi\left(A\right)\right),\overline{\iota c_{x}\varphi}\right)\right\} .
\]
Where we are fixing a representative $\varphi$ of $\overline{\varphi}$
and viewing it as an isomorphism between the appropriate restrictions.
\item \label{enu:prod-pullback-left.}For every $J\in\Fc$ such that $J\le H$
we can take
\[
\left[J\times K\right]=\bigsqcup_{\left(A,\overline{\varphi}\right)\in\left[H\times K\right]}\bigsqcup_{\begin{array}{c}
{\scriptstyle x\in\left[A\backslash H/J\right]}\\
{\scriptstyle A^{x}\cap J\in\Fc}
\end{array}}\left\{ \left(A^{x}\cap J,\overline{\varphi\iota c_{x}}\right)\right\} .
\]
\item \label{enu:triple-prod-into-pullback-prod.}For every $J\in\Fc$ we
can take
\[
\bigsqcup_{\left(A,\overline{\varphi}\right)\in\left[H\times K\right]}\bigsqcup_{\left(B,\overline{\psi}\right)\in\left[J\times A\right]}\left\{ \left(B,\overline{\iota\psi}\right)\right\} =\bigsqcup_{\begin{array}{c}
{\scriptstyle \left(C,\overline{\theta}\right)\in\left[J\times H\right]}\\
{\scriptstyle \left(D,\overline{\gamma}\right)\in\left[J\times K\right]}
\end{array}}\bigsqcup_{\begin{array}{c}
{\scriptstyle x\in\left[D\backslash J/C\right]}\\
{\scriptstyle D^{x}\cap C\in\Fc}
\end{array}}\left\{ \left(D^{x}\cap C,\overline{\theta\iota}\right)\right\} .
\]
\end{enumerate}
\end{prop}

\begin{proof}
$\phantom{.}$
\begin{enumerate}
\item With the notation of Item \eqref{enu:HX_FK-and-KX_FH} we have that
$\varphi\left(A\right)\le K$ and, since $A\le H$, we can conclude
that $\overline{\iota\varphi^{-1}}\in\Hom_{\OF}\left(\varphi\left(A\right),H\right)$.
With notation as in the statement we can now define 
\[
f:=\bigsqcup_{\left(A,\overline{\varphi}\right)\in\left[H\times K\right]}\overline{\varphi^{-1}}:\bigsqcup_{\left(A,\overline{\varphi}\right)\in\left[H\times K\right]}\varphi\left(A\right)\to H\times K:=\bigsqcup_{\left(A,\overline{\varphi}\right)\in\left[H\times K\right]}A.
\]
With this setup we have that
\begin{align*}
\pi_{H}^{H\times K}f=\bigsqcup_{\left(A,\overline{\varphi}\right)\in\left[H\times K\right]}\overline{\iota_{A}^{H}}f & =\bigsqcup_{\left(A,\overline{\varphi}\right)\in\left[H\times K\right]}\overline{\iota\varphi^{-1}},\\
\pi_{K}^{H\times K}f=\bigsqcup_{\left(A,\overline{\varphi}\right)\in\left[H\times K\right]}\overline{\varphi}f & =\bigsqcup_{\left(A,\overline{\varphi}\right)\in\left[H\times K\right]}\overline{\iota_{\varphi\left(A\right)}^{K}}.
\end{align*}
Since $K\times H\cong H\times K$ and $f$ is an isomorphism then
the above identities prove Item \eqref{enu:HX_FK-and-KX_FH}. 
\item All morphisms in $\FHH[S]$ are, by definition, of the form $c_{x}$
for some $x\in S$. Thus, for any element $\left(A,\overline{\varphi}\right)\in\left[H\times_{\FHH[S]}K\right]$,
there exists $x\in S$ such that $\lui{x}{A}\le K$ and $\overline{\varphi}=\overline{\iota c_{x}}$.
In particular we have that $A\le K^{x}$. Since $A\le H$ by construction
we can conclude that $A\le K^{x}\cap H$. Therefore we can take $\overline{\iota_{A}^{K^{x}\cap H}}\in\Hom_{\Orbitize{\FHH[S]^{c}}}\left(A,K^{x}\cap H\right)$
and, viewing $c_{x}$ as an isomorphism from $K^{x}\cap H$ to $K^{x}\cap\lui{x}{H}$,
we have that $\overline{\varphi}=\overline{\iota_{K^{x}\cap H}^{K}c_{x}}\,\overline{\iota_{A}^{K^{x}\cap H}}$.
From maximality of the pair $\left(A,\overline{\varphi}\right)$ we
can conclude that $A=K^{x}\cap H$. In other words, all elements in
$\left[H\times_{\FHH[S]}K\right]$ are of the form $\left(K^{x}\cap H,\overline{\iota c_{x}}\right)$
for some $x\in S$. Notice now that, for every $k\in K$, we have
$K^{kx}\cap H=K^{x}\cap H$ and $\overline{\iota c_{kx}}=\overline{\iota c_{x}}$.
Moreover, we know that $\left[H\times_{\FHH[S]}K\right]$ contains
exactly one representative for each of the equivalence classes given
by the relation of Equation \eqref{eq:Equivalence-relation.}. It
is therefore possible to choose $\left[H\times_{\FHH[S]}K\right]$
and $\left[K\backslash S/H\right]$ so that
\begin{equation}
\left[H\times_{\FHH[S]}K\right]=\bigcup_{\begin{array}{c}
{\scriptstyle x\in\left[K\backslash S/H\right]}\\
{\scriptstyle K^{x}\cap H\in\FHH[S]^{c}}
\end{array}}\left\{ \left(K^{xh_{x}}\cap H,\overline{\iota c_{xh_{x}}}\right)\right\} .\label{eq:Almost-prod-in-F_S(S).}
\end{equation}
For some appropriate $h_{x}\in H$. Assume now that there exist $x,y\in S$
and $h\in H$ such that $K^{x}\cap H=K^{yh}\cap H\in\FHH[S]^{c}$
and that $\overline{\iota c_{x}}=\overline{\iota c_{yh}}$. From this
last identity we can deduce that there exist $k\in K$ and $z\in C_{S}\left(K^{x}\cap H\right)$
such that $x=kyhz$. Since $K^{x}\cap H\in\FHH[S]^{c}$ then we have
that $C_{S}\left(K^{x}\cap H\right)\le K^{x}\cap H$ and, in particular,
$z\in H$. We can therefore conclude that $y\in KxH$ and, therefore,
that the union in Equation \eqref{eq:Almost-prod-in-F_S(S).} is disjoint.
Item \eqref{enu:prod-if-F-is-F_S(S).} then follows by making a more
appropriate choice of the representatives $\left[K\backslash S/H\right]$
(i.e. taking $xh_{x}$ instead of $x$).
\item Let $\mathcal{C}$ be a category, let $X,Y,Z\in\mathcal{C}$ be objects
and let $\alpha\colon Y\to Z$ be an isomorphism in $\mathcal{C}$.
We know from category theory that, if the product $X\times Y$ exists
in $\mathcal{C}$, then the product $X\times Z$ also exists in $\mathcal{C}$
and satisfies
\begin{align*}
X\times Z & =X\times Y, & \pi_{Z}^{X\times Z} & =\alpha\pi_{Y}^{X\times Y}, & \pi_{X}^{X\times Z} & =\pi_{X}^{X\times Y}.
\end{align*}
where $\pi_{A}^{A\times B}$ denote the natural projections. With
the notation of Item \eqref{enu:prod-iso-right.} we have that $\overline{\psi}\in\Hom_{\OFc}\left(K,\psi\left(K\right)\right)$
is an isomorphism in $\aOFc$ and for every $\left(A,\overline{\varphi}\right)\in\left[H\times K\right]$
we have $A\le H$ and $\overline{\psi\varphi}\in\Hom_{\OF}\left(A,\psi\left(K\right)\right)$.
We can therefore apply the previous result taking $\mathcal{C}:=\aOFc$,
$X:=H$, $Y:=K$, $Z:=\psi\left(K\right)$ and $\alpha=\overline{\psi}$
and Item \eqref{enu:prod-iso-right.} follows.
\item The same arguments used to prove Item \eqref{enu:prod-iso-right.}
can be used to prove Item \eqref{enu:prod-iso-left.}.
\item Let $\mathcal{C}$ be a category admitting products and pullbacks,
let $X,Y,Z\in\mathcal{C}$ be objects, let $\alpha\colon Z\to Y$
be a morphism in $\mathcal{C}$ and let $\left(X\times Y\right)\times_{Y}Z$
be the pullback of the diagram $X\times Y\stackrel{\pi_{Y}^{X\times Y}}{\longrightarrow}Y\stackrel{\alpha}{\longleftarrow}Z$.
We know from category theory that
\begin{align*}
\left(X\times Y\right)\times_{Y}Z & =X\times Z, & \pi_{Z}^{X\times Z} & =\pi_{Z}^{\left(X\times Y\right)\times_{Y}Z}, & \pi_{X}^{X\times Z} & =\pi_{X}^{X\times Y}\pi_{X\times Y}^{\left(X\times Y\right)\times_{Y}Z}.
\end{align*}
Where $\pi_{A}^{A\times B}$ and $\pi_{A}^{A\times_{C}B}$ denote
the natural projections. With the notation of Item \eqref{enu:prod-pullback-right.}
we have for every $\left(A,\overline{\varphi}\right)\in\left[H\times K\right]$
and every $x\in\left[J\backslash K/\varphi\left(A\right)\right]$
that $\varphi^{-1}\left(J^{x}\cap\varphi\left(A\right)\right)\le A\le H$
and $\overline{\iota c_{x}\varphi}\in\Hom_{\OF}\left(\varphi^{-1}\left(J^{x}\cap\varphi\left(A\right)\right),J\right)$.
Using Proposition \ref{prop:Pullback-OFc.} we can now apply the previous
result taking $Y:=K$, $X:=H$, $Z:=J$ and $\alpha:=\overline{\iota_{J}^{K}}$
and Item \eqref{enu:prod-pullback-right.} follows.
\item The same arguments used to prove Item \eqref{enu:prod-pullback-right.}
can be used to prove Item \eqref{enu:prod-pullback-left.}.
\item Let $\mathcal{C}$ be a category admitting products and pullbacks,
let $X,Y,Z\in\mathcal{C}$ be objects and let $\left(X\times Y\right)\times_{Y}\left(Y\times Z\right)$
be the pullback of the diagram $Y\times X\stackrel{\pi_{Y}^{Y\times X}}{\longrightarrow}Y\stackrel{\pi_{Y}^{Y\times Z}}{\longleftarrow}Y\times Z$.
We know from category theory that
\begin{align*}
\left(Y\times X\right)\times_{Y}\left(Y\times Z\right) & =Y\times\left(X\times Z\right), & \pi_{X}^{Y\times X}\pi_{Y\times X}^{\left(Y\times X\right)\times_{Y}\left(Y\times Z\right)} & =\pi_{X}^{X\times Z}\pi_{X\times Z}^{Y\times\left(X\times Z\right)},\\
\pi_{Y}^{Y\times X}\pi_{Y\times X}^{\left(Y\times X\right)\times_{Y}\left(Y\times Z\right)} & =\pi_{Y}^{Y\times\left(X\times Z\right)}.
\end{align*}
Where $\pi_{A}^{A\times B}$ and $\pi_{A}^{A\times_{C}B}$ denote
the natural projections. For every $\left(A,\overline{\varphi}\right)\in\left[H\times K\right]$,
every $\left(B,\overline{\psi}\right)\in\left[J\times A\right]$,
every $\left(C,\overline{\theta}\right)\in\left[J\times H\right]$,
every $\left(D,\overline{\gamma}\right)\in\left[J\times K\right]$
and every $x\in\left[D\backslash J/C\right]$ we have that $B,D^{x}\cap C\le J$,
that $\overline{\iota\psi}\in\Hom_{\OFc}\left(B,H\right)$ and that
$\overline{\theta\iota}\in\Hom_{\OFc}\left(D^{x}\cap C,H\right)$.
Using Propositions \ref{prop:Pullback-OFc.} and \ref{prop:Product-OFc.}
we can now apply the previous result taking $\mathcal{C}:=\aOFc$,
$X:=H$, $Z:=K$ and $Y:=J$ and Item \eqref{enu:triple-prod-into-pullback-prod.}
follows.
\end{enumerate}
\end{proof}

\subsection{\label{subsec:Mackey-functors-over-fusion-systems.}Mackey functors
over fusion systems.}

In this subsection we will define (centric) Mackey functor over a
fusion system (Definitions \ref{def:Mackey-functor.} and \ref{def:F-centric-Mackey-functor.})
and the (centric) induction, restriction and conjugation functors
(Definition \ref{def:Restriction-induction-and-conjugation-functors.}
and Proposition \ref{prop:centric-Induction-restriction.}). Moreover
we will provide some tools for studying certain induced Mackey functors
(Proposition \ref{prop:Decomposition-centric-Mackey-functor.}) and
certain compositions of the induction and restriction functors (Lemma
\ref{lem:Mackey-formula-induction-restriction.}).

Let us start by defining the Mackey algebra of a fusion system. In
order to do that we will use methods similar to those used in \cite{fuserdFusionBouc,MackeyFunctorsAndBisets}.
\begin{defn}
Let $G$ be a group and let $f\colon G\to f\left(G\right)$ be a group
isomorphism. We define the \textbf{$f$ twisted diagonal of $G$ }as
the subgroup of $f\left(G\right)\times G$ given by
\[
\Delta\left(G,f\right):=\left\{ \left(f\left(x\right),x\right)\in f\left(G\right)\times G\,:\,x\in G\right\} .
\]
\end{defn}

\begin{defn}
\label{def:Restriction-Induction-and-conjugation-bisets.}Throughout
this definition we will denote with an overline $\overline{X}$ the
isomorphism class of a biset $X$. Let $H,K$ be subgroups of $S$
such that $H\le K$ and view $\Delta\left(H,\Id_{H}\right)$ as a
subgroup of $H\times K$. We define the \textbf{restriction from $K$
to $H$ }as the isomorphism class of $\left(H,K\right)$-bisets given
by
\[
R_{H}^{K}:=\overline{\left(H\times K\right)/\Delta\left(H,\Id_{H}\right)}.
\]
Here the $\left(H,K\right)$-biset structure on $\left(H\times K\right)/\Delta\left(H,\Id_{H}\right)$
is given by setting
\[
h\cdot\left(\left(x\times y\right)\Delta\left(H,\Id_{H}\right)\right)\cdot k=\left(hx\right)\times\left(k^{-1}y\right)\Delta\left(H,\Id_{H}\right)
\]
for every $h,x\in H$ and $k,y\in K$.

Likewise, viewing $\Delta\left(H,\Id_{H}\right)$ as a subgroup of
$K\times H$, we define the \textbf{induction from $H$ to $K$ }as
the isomorphism class of $\left(K,H\right)$-bisets given by
\[
I_{H}^{K}:=\overline{\left(K\times H\right)/\Delta\left(H,\Id_{H}\right)}.
\]
Where the $\left(K,H\right)$-biset structure on $\left(K\times H\right)/\Delta\left(H,\Id_{H}\right)$
is defined as before.

Finally, given an isomorphism $\varphi\colon H\to\varphi\left(H\right)$,
we define the \textbf{conjugation by $\varphi$} as the isomorphism
class of $\left(\varphi\left(H\right),H\right)$-bisets given by
\[
c_{\varphi,H}:=\overline{\left(\varphi\left(H\right)\times H\right)/\Delta\left(H,\varphi\right)}.
\]
Again the $\left(\varphi\left(H\right),H\right)$-biset structure
on $\left(\varphi\left(H\right)\times H\right)/\Delta\left(H,\varphi\right)$
is given as before. If $H$ is clear we will simply write $c_{\varphi}$
instead of $c_{\varphi,H}$. 
\end{defn}

We want the Mackey algebra to be an algebra over a commutative ring
$\R$ and generated by elements of the form $R_{H}^{K},I_{H}^{K}$
and $c_{\varphi}$ with $H\le K\le S$ and $\varphi$ an isomorphism
in $\F$. To do this, we will start by giving a standard definition
of the product (also called composition) of two bisets (see \cite{fuserdFusionBouc,MackeyFunctorsAndBisets,bouc-serge-biset-functors}).
Let $H,J,J'$ and $K$ be finite groups, let $X$ be an $\left(H,J\right)$-biset
and let $Y$ be a $\left(J',K\right)$-biset. If $J=J'$ we define
$X\times_{J}Y$ as the $\left(H,K\right)$-biset obtained as a quotient
of the $\left(H,K\right)$-biset $X\times Y$ modulo the equivalence
relation
\[
\left(x\cdot j\right)\times y\sim x\times\left(j\cdot y\right),
\]
where $x\in X,\,y\in Y$ and $j\in J$. With this notation, we define
the product of $X$ and $Y$ as the $\left(H,K\right)$-biset given
by
\[
XY:=X\cdot Y:=\begin{cases}
X\times_{J}Y & \text{if }J=J'\\
\emptyset & \text{else}
\end{cases}.
\]

Notice that, given an $\left(H,J\right)$-biset $X'$ isomorphic to
$X$ and a $\left(J',K\right)$-biset $Y'$ isomorphic to $Y$ then
the $\left(H,K\right)$-biset $X'Y'$ is isomorphic to $XY$. This
allows us to define the product of two isomorphism classes of bisets
as the isomorphism class of the product of any two of their representatives
(see \cite[Lemma 2.3.20 and Theorem 2.4.5]{bouc-serge-biset-functors}).
That is, using the notation of Definition \ref{def:Restriction-Induction-and-conjugation-bisets.},
we define the product of two isomorphism classes of bisets $\overline{X}$
and $\overline{Y}$ as
\[
\overline{X}\,\overline{Y}:=\overline{X}\cdot\overline{Y}:=\overline{XY}.
\]
It is straight forward to prove that this product is associative (see
\cite[Proposition 2.3.14]{bouc-serge-biset-functors}).

We can now define $\mathcal{A}$ as the abelian semigroup generated
by an artificial zero element ($0$) and all the isomorphism classes
of non-empty bisets over any pair of finite groups and with relations
\begin{align*}
\overline{X}+\overline{Y} & =\overline{X\sqcup Y}, & \overline{X}+\overline{Z} & =\overline{Z}+\overline{X},\\
0+0 & =0, & 0+\overline{Z} & =\overline{Z}+0=\overline{Z},
\end{align*}
for all bisets $X,Y$ and $Z$ such that $X$ and $Y$ are bisets
over the same pair of groups. By sending $\overline{\emptyset}$ (seen
as an isomorphism class of biset over any 2 finite groups) to the
element $0$ in $\mathcal{A}$, the previously defined product between
isomorphism classes of bisets can be uniquely extended to $\mathcal{A}$
in a way that is distributive with respect to $+$. With this setup
we have that $\left(\mathcal{A},+,\cdot\right)$ is a semiring. We
can now take the subsemiring of $\mathcal{A}$ generated by isomorphism
classes of bisets of the form $I_{H}^{K},R_{H}^{K}$ and $c_{\varphi}$
with $H\le K\le S$ and $\varphi$ an isomorphism in $\F$. This subsemiring
can be used in order to define the Mackey algebra.
\begin{defn}
\label{def:Mackey-algebra.}The \textbf{Mackey algebra of $\F$ on
the ring $\Z$} (denoted as $\muFR[\Z]$) is the Grothendieck group
of the previously described subsemiring. For every commutative ring
with unit $\R$, we define the \textbf{Mackey algebra of $\F$ on
the ring $\R$ }(or simply \textbf{Mackey algebra} if $\F$ and $\R$
are clear) as the $\R$-algebra
\[
\muFR:=\R\otimes_{\Z}\muFR[\Z].
\]
\end{defn}

The previous definitions motivate the introduction of the following
notation.
\begin{notation}
\label{nota:R-commutative.}From now, unless otherwise specified,
\textbf{$\R$ will denote a commutative ring with unit}.
\end{notation}

The following relations on the elements of the Mackey algebra will
be useful in what follows.
\begin{lem}
\label{lem:Many-properties-definition.}The elements $I_{H}^{K},R_{H}^{K}$
and $c_{\varphi}$ of the Mackey algebra $\muFR$ satisfy relations
analogous to the similarly denoted elements in the Mackey algebra
of a group (see \cite[Section 3]{StructureMackeyFunctors}). More
precisely, the following are satisfied:
\begin{enumerate}
\item \label{enu:property-I_H^H-is-identity.}Let $H$ be a subgroup of
$S$, and let $h\in H$. We have that $I_{H}^{H}=R_{H}^{H}=c_{c_{h},H}$.
Moreover $I_{H}^{H}$ is an idempotent in $\muFR$.
\item \label{enu:property-composition-is-nice.}Let $H,K$ and $J$ be subgroups
of $S$ such that $H\le K\le J$ and let $\varphi\colon H\to\varphi\left(H\right)$
and $\psi\colon\varphi\left(H\right)\to\psi\left(\varphi\left(H\right)\right)$
be isomorphisms in $\F$. We have that
\begin{align*}
R_{H}^{K}R_{K}^{J} & =R_{H}^{J}, & I_{K}^{J}I_{H}^{K} & =I_{H}^{J}, & c_{\psi,\varphi\left(H\right)}c_{\varphi,H} & =c_{\psi\varphi,H}.
\end{align*}
\item \label{enu:Property-conjugation-commutes.}Let $H$ and $K$ be subgroups
of $S$ such that $H\le K$ and let $\theta\colon K\to\theta\left(K\right)$
be an isomorphism in $\F$. We have that
\begin{align*}
c_{\theta,K}I_{H}^{K} & =I_{\theta\left(H\right)}^{\theta\left(K\right)}c_{\theta_{|H},H}, & c_{\theta_{|H},H}R_{H}^{K} & =R_{\theta\left(H\right)}^{\theta\left(K\right)}c_{\theta,K}.
\end{align*}
Where $\theta_{|H}:H\to\theta\left(H\right)$ is the restriction of
$\theta$ to $H$.
\item \label{enu:property-Mackey-formula.}Let $H,K$ and $J$ be subgroups
of $S$ such that $H,K\le J$. We have that
\begin{align*}
R_{K}^{J}I_{H}^{J} & =\sum_{x\in\left[K\backslash J/H\right]}I_{\left(K\cap\lui{x}{H}\right)}^{K}c_{c_{x},\left(K^{x}\cap H\right)}R_{\left(K^{x}\cap H\right)}^{H}.
\end{align*}
\item \label{enu:property-all-0.}All other combinations of induction restriction
and conjugation are $0$.
\end{enumerate}
\end{lem}

\begin{proof}
See \cite[Section 2.3]{bouc-serge-biset-functors}. Alternatively
notice that the elements $I_{H}^{K},R_{H}^{K}$ and $c_{\varphi}$
of the Mackey algebra are, by definition, isomorphism classes of the
bisets $\text{Ind}_{x}$, $\text{Res}_{x}$ and $\mathfrak{L}_{x}$
of \cite[Definition 6.8]{MackeyFunctorsAndBisets}. With this in mind
the above relations follow from \cite[Proposition 6.9 and Theorem 5.3]{MackeyFunctorsAndBisets}.
\end{proof}
As an immediate consequence of Lemma \ref{lem:Many-properties-definition.}
we have the following.
\begin{cor}
\label{cor:conjugation-in-F-and-in-orbit-F.}Let $H$ and $K$ be
subgroups of $S$, let $k\in K$ and let $\varphi\colon H\to\varphi\left(H\right)$
be an isomorphism in $\F$ such that $\varphi\left(H\right)\le K$.
We have that $I_{\lui{k}{\varphi\left(H\right)}}^{K}c_{c_{k}\varphi}=I_{\varphi\left(H\right)}^{K}c_{\varphi}$
and that $c_{\varphi^{-1}c_{k^{-1}}}R_{\lui{k}{\varphi\left(H\right)}}^{K}=c_{\varphi^{-1}}R_{\varphi\left(H\right)}^{K}$.
In particular, given $\overline{\varphi}\in\Hom_{\OF}\left(H,K\right)$
and representatives $\varphi_{1},\varphi_{2}\in\overline{\varphi}$
then, seeing $\varphi_{1}$ and $\varphi_{2}$ as isomorphisms onto
their images we can define
\begin{align*}
I_{\overline{\varphi}\left(H\right)}^{K}c_{\overline{\varphi}} & :=I_{\varphi_{1}\left(H\right)}^{K}c_{\varphi_{1}}=I_{\varphi_{2}\left(H\right)}^{K}c_{\varphi_{2}}, & c_{\overline{\varphi^{-1}}}R_{\overline{\varphi}\left(H\right)}^{K} & :=c_{\varphi_{1}^{-1}}R_{\varphi_{1}\left(H\right)}^{K}=c_{\varphi_{2}^{-1}}R_{\varphi_{2}\left(H\right)}^{K}.
\end{align*}
Moreover, given $J\le S$ and $\overline{\psi}\in\Hom_{\OF}\left(K,J\right)$,
we have that
\begin{align*}
I_{\overline{\psi\varphi}\left(H\right)}^{J}c_{\overline{\psi\varphi}} & =I_{\overline{\psi}\left(K\right)}^{J}c_{\overline{\psi}}I_{\overline{\varphi}\left(H\right)}^{K}c_{\overline{\varphi}}, & c_{\overline{\left(\psi\varphi\right)^{-1}}}R_{\overline{\psi\varphi}\left(H\right)}^{J} & =c_{\overline{\varphi^{-1}}}I_{\overline{\varphi}\left(H\right)}^{K}c_{\overline{\psi^{-1}}}R_{\overline{\psi}\left(K\right)}^{J}.
\end{align*}
\end{cor}

\begin{proof}
We will only prove the statement for the case involving induction,
the proof for the case involving restriction is analogous. The first
part of the statement follows from Lemma \ref{lem:Many-properties-definition.}
\eqref{enu:property-I_H^H-is-identity.}-\eqref{enu:Property-conjugation-commutes.}
via the identities below
\[
I_{\lui{k}{\varphi\left(H\right)}}^{K}c_{c_{k}\varphi}=I_{\lui{k}{\varphi\left(H\right)}}^{K}c_{c_{k}}c_{\varphi}=c_{c_{k}}I_{\varphi\left(H\right)}^{K}c_{\varphi}=I_{K}^{K}I_{\varphi\left(H\right)}^{K}c_{\varphi}=I_{\varphi\left(H\right)}^{K}c_{\varphi}.
\]
The second part of the statement follows from Items \eqref{enu:property-composition-is-nice.}
and \eqref{enu:Property-conjugation-commutes.} of Lemma \ref{lem:Many-properties-definition.}
via the identities below
\[
I_{\overline{\psi\varphi}\left(H\right)}^{J}c_{\overline{\psi\varphi}}=I_{\psi\left(K\right)}^{J}I_{\psi\left(\varphi\left(H\right)\right)}^{\psi\left(K\right)}c_{\psi_{|\varphi\left(H\right)}}c_{\varphi}=I_{\psi\left(K\right)}^{J}c_{\psi}I_{\varphi\left(H\right)}^{K}c_{\varphi}=I_{\overline{\psi}\left(K\right)}^{J}c_{\overline{\psi}}I_{\overline{\varphi}\left(H\right)}^{K}c_{\overline{\varphi}}.
\]
\end{proof}
Another important consequence of Lemma \ref{lem:Many-properties-definition.}
is the following result which translates \cite[Proposition 3.2]{StructureMackeyFunctors}
to the context of Mackey functors over fusion systems.
\begin{prop}
\label{prop:Mackey-algebra-basis.} The Mackey algebra $\muFR$ admits
an $\R$-basis of the form $\mathcal{B}:=\bigsqcup_{A,B\le S}\mathcal{B}_{\left(A,B\right)}$,
where
\[
\mathcal{B}_{\left(A,B\right)}:=\bigsqcup_{\begin{array}{c}
{\scriptstyle C\le A}\\
{\scriptstyle \text{up to }A\text{-conj}}
\end{array}}\bigsqcup_{\varphi\in\left[\Aut_{B}\left(B\right)\backslash\Hom_{\F}\left(C,B\right)/\Aut_{A}\left(C\right)\right]}\left\{ I_{\overline{\varphi}\left(C\right)}^{B}c_{\overline{\varphi}}R_{C}^{A}\right\} .
\]
In particular, $\muFR$ is finitely generated as an $\R$-module.
\end{prop}

\begin{proof}
From Items \eqref{enu:property-I_H^H-is-identity.}, \eqref{enu:property-composition-is-nice.}
and \eqref{enu:property-all-0.} of Lemma \ref{lem:Many-properties-definition.}
we know that $1_{\muFR}=\sum_{H\le S}I_{H}^{H}$ and that the $I_{H}^{H}$
are mutually orthogonal idempotents. With this in mind we can obtain
the following $\R$-module decomposition of $\muFR$
\[
\muFR\cong_{\R}\bigoplus_{A,B\le S}I_{A}^{A}\muFR I_{B}^{B}.
\]

Fix now $A,B\le S$. From the above it suffices to prove that $\mathcal{B}_{\left(A,B\right)}$
is an $\R$-basis of $I_{A}^{A}\muFR I_{B}^{B}$. Using Lemma \ref{lem:Many-properties-definition.}
we can write any element in $I_{A}^{A}\muFR I_{B}^{B}$ as an $\R$-linear
combination of elements of the form $I_{\varphi\left(C\right)}^{B}c_{\varphi,C}R_{C}^{A}$
with $C\le A$ and $\varphi\colon C\to\varphi\left(C\right)$ an isomorphism
in $\F$ satisfying $\varphi\left(C\right)\le B$. For $i=1,2$ let
$C_{i}$ be a subgroup of $A$, let $\varphi_{i}\in\text{Hom}_{\F}\left(C_{i},\varphi_{i}\left(C_{i}\right)\right)$
be an isomorphism in $\F$ such that $\varphi_{i}\left(C_{i}\right)\le B$,
view $\Delta\left(C_{i},\varphi_{i}\right)$ as a subgroup of $B\times A$
and define the representative $X_{i}:=\left(B\times A\right)/\Delta\left(C_{i},\varphi_{i}\right)$
of $I_{\varphi_{i}\left(C_{i}\right)}^{B}c_{\varphi_{i}}R_{C_{i}}^{A}$.
We know (see \cite[Lemma 2.3.4 (1)]{bouc-serge-biset-functors}) that
each $X_{i}$ is a transitive biset. Therefore we can use \cite[Lemma 2.1.9 and Definition 2.3.1]{bouc-serge-biset-functors}
in order to deduce that $I_{\varphi_{1}\left(C_{1}\right)}^{B}c_{\varphi_{1},C_{1}}R_{C_{1}}^{A}=I_{\varphi_{2}\left(C_{2}\right)}^{B}c_{\varphi_{2},C_{2}}R_{C_{2}}^{A}$
if and only if there exist $a\in A$ and $b\in B$ such that $C_{2}=C_{1}^{a}$
and $\varphi_{2}=c_{b}\varphi_{1}c_{a}$. We also know (see \cite[Lemmas 2.1.9 and 2.2.2]{bouc-serge-biset-functors})
that any finite $\left(A,B\right)$-biset can be written in a unique
way (up to isomorphism) as a disjoint union of finite transitive $\left(A,B\right)$-bisets.
Let now $M_{\left(A,B\right)}$ be the commutative monoid generated
by isomorphism classes of $\left(A,B\right)$-bisets of the form $I_{\varphi\left(C\right)}^{B}c_{\varphi,C}R_{C}^{A}$
with addition given by $\overline{X}+\overline{Y}=\overline{X\sqcup Y}$
(see notation of Definition \ref{def:Restriction-Induction-and-conjugation-bisets.}).
We can deduce from the above that $\mathcal{B}_{\left(A,B\right)}$
(viewed as a subset of $M_{\left(A,B\right)}$) is an $\mathbb{N}$-basis
of $M_{\left(A,B\right)}$. Recall now that $I_{A}^{A}\muFR[\Z]I_{B}^{B}$
is, by definition, the Grothendieck group of $M_{\left(A,B\right)}$.
Thus, we can deduce that $\mathcal{B}_{\left(A,B\right)}$ (viewed
as a subset of $I_{A}^{A}\muFR[\Z]I_{B}^{B}$) is a $\Z$-basis of
$I_{A}^{A}\muFR[\Z]I_{B}^{B}$. Since tensor product preserves direct
sum decomposition and $\R\otimes_{\Z}\Z\cong\R$ for any commutative
ring $\R$, then we can deduce that $\mathcal{B}_{\left(A,B\right)}$
(viewed as a subset of $I_{A}^{A}\muFR I_{B}^{B}$) is an $\R$-basis
of $I_{A}^{A}\muFR I_{B}^{B}$ thus concluding the proof.
\end{proof}
\begin{cor}
\label{cor:Mackey-algebra-inclusion.}Let $H\le S$ and let $\H$
be a fusion system over $H$ satisfying $\H\subseteq\F$. There exists
a natural inclusion of Mackey algebras $\FmuFR{\H}\subseteq\muFR$
and this inclusion preserves unit if and only if $H=S$.
\end{cor}

\begin{proof}
From Proposition \ref{prop:Mackey-algebra-basis.} we know that $\FmuFR{\H}$
is generated as an $\R$-module by elements of the form $I_{\varphi\left(C\right)}^{B}c_{\varphi}R_{C}^{A}$
such that $A,B,C\le H$ and $\varphi$ is an isomorphism in $\H$.
Since $\H\subseteq\F$ we know that any isomorphism in $\H$ is also
an isomorphism in $\F$ and since $H\le S$ we know that every subgroup
of $H$ is also a subgroup of $S$. Therefore, with $A,B,C$ and $\varphi$
as before, we have that $I_{\varphi\left(C\right)}^{A}c_{\varphi}R_{C}^{B}\in\muFR$.
This gives us the natural inclusion $\FmuFR{\H}\subseteq\muFR$. For
this inclusion to preserve the unit we need to have $1_{\muFR}=\sum_{K\le S}I_{K}^{K}=\sum_{K'\le H}I_{K'}^{K'}=1_{\FmuFR{\H}}$
which happens if and only if $H=S$. 
\end{proof}
We are now ready to define what a Mackey functor over a fusion system
is.
\begin{defn}
\label{def:Mackey-functor.}A \textbf{Mackey functor over $\F$ on
$\R$ }(or simply \textbf{Mackey functor} if $\F$ and $\R$ are clear)
is a finitely generated left $\muFR$-module. The \textbf{category
of Mackey functors over $\F$ on $\R$ }(denoted by \textbf{$\MackFR$})
is the category $\muFRmod$.
\end{defn}

\begin{example}
Any globally defined Mackey functor (see \cite[Section 1]{TwoClassificationWebb})
inherits a structure of Mackey functor over any fusion system. Any
conjugation invariant Mackey functor over a finite group $G$ with
Sylow $p$ subgroup $S$ leads naturally to a Mackey functor over
$\FSG$ (see \cite{fuserdFusionBouc}). The Mackey algebra $\muFR$
is itself a Mackey functor over $\F$.
\end{example}

This definition of Mackey functor over a fusion system allows us to
use some well known results of ring theory in order to define the
induction, restriction and conjugation functors.
\begin{defn}
\label{def:Restriction-induction-and-conjugation-functors.}Let $H\le S$
and let $\H\subseteq\F$ be a fusion system on $H$. From Corollary
\ref{cor:Mackey-algebra-inclusion.} we have that $\FmuFR{\H}\subseteq\muFR$.
This allows us to define the \textbf{restriction from $\F$ to $\H$
functor} as the functor $\downarrow_{\H}^{\F}:\muFRmod\to\FmuFRmod{\H}$,
that sends any $\muFR$-module $M$ to the $\FmuFR{\H}$-module 
\[
M\downarrow_{\H}^{\F}:=\downarrow_{\H}^{\F}\left(M\right):=1_{\FmuFR{\H}}M.
\]
Here $1_{\FmuFR{\H}}=\sum_{K\le H}I_{K}^{K}$ denotes the identity
of $\FmuFR{\H}$ seen as an element of $\muFR$ via the natural inclusion
of Corollary \ref{cor:Mackey-algebra-inclusion.}. 

Analogously, we can define the \textbf{induction from $\H$ to $\F$
functor }as the functor\textbf{ $\uparrow_{\H}^{\F}:\FmuFRmod{\H}\to\muFRmod$},
that sends any $\FmuFR{\H}$-module $N$ to the $\muFR$-module 
\[
N\uparrow_{\H}^{\F}:=\uparrow_{\H}^{\F}\left(N\right):=\muFR1_{\FmuFR{\H}}\otimes_{\FmuFR{\H}}N.
\]

Finally, let $K\le S$ and let $\varphi\colon H\bjarrow K$ be an
isomorphism of groups (not necessarily in $\F$). This isomorphism
induces an isomorphism of $\R$-algebras $\hat{\varphi}:\FmuFR{\FHH}\bjarrow\FmuFR{\FHH[K]}$
obtained by setting 
\[
\hat{\varphi}\left(I_{\lui{h}{C}}^{B}c_{c_{h}}R_{C}^{A}\right):=I_{\lui{\varphi\left(h\right)}{\varphi\left(C\right)}}^{\varphi\left(B\right)}c_{c_{\varphi\left(h\right)}}R_{\varphi\left(C\right)}^{\varphi\left(A\right)},
\]
for every $A,B,C\le H$ and $h\in H$ such that $C\le A$ and $\lui{h}{C}\le B$.
This allows us to define the \textbf{conjugation by $\varphi^{-1}$
functor} as the invertible functor
\[
\lui{\varphi^{-1}}{\cdot}:=\hat{\varphi}^{*}:\FmuFRmod{\FHH[K]}\to\FmuFRmod{\FHH},
\]
that sends any Mackey functor $L$ over $\FHH[K]$ to the Mackey functor
$\lui{\varphi^{-1}}{L}$ over $\FHH$ which equals $L$ as an $\R$-module
and such that for every $I_{\lui{h}{C}}^{B}c_{c_{h}}R_{C}^{A}\in\FmuFR{\FHH}$
as before and every $x\in\lui{\varphi^{-1}}{L}$
\[
I_{\lui{h}{C}}^{B}c_{c_{h}}R_{C}^{A}\cdot x:=I_{\lui{\varphi\left(h\right)}{\varphi\left(C\right)}}^{\varphi\left(B\right)}c_{c_{\varphi\left(h\right)}}R_{\varphi\left(C\right)}^{\varphi\left(A\right)}x.
\]
Where, on the right hand side, we are viewing $x$ as an element of
$L$ in order to apply the action of $\FmuFR{\FHH[K]}$ on it but
we are viewing the result of this action as an element of $\lui{\varphi^{-1}}{L}$.
\end{defn}

Let's now take a moment to notice a key difference between Mackey
functors over groups and Mackey functors over fusion systems. Let
$G$ be a finite group, let $H\le G$ and let $M$ be a Mackey functor
over $G$ on $\R$. It is a well known result (see \cite[Section 3]{GuideToMackeyFunctorsWebb})
that
\begin{equation}
M\downarrow_{H}^{G}\uparrow_{H}^{G}\left(K\right)\cong_{\R}\bigoplus_{x\in\left[K\backslash G/H\right]}M\left(K^{x}\cap H\right).\label{eq:restriction-induction-mackey-groups.}
\end{equation}
It is also well known (see \cite[Proposition 5.3]{StructureMackeyFunctors})
that for every Mackey functor $N$ over $H$ and every $K\le G$ the
following equivalence of Mackey functors over $K$ holds
\begin{equation}
N\uparrow_{H}^{G}\downarrow_{K}^{G}\cong\bigoplus_{x\in\left[K\backslash G/H\right]}\left(\lui{x}{\left(N\downarrow_{K^{x}\cap H}^{H}\right)}\right)\uparrow_{K\cap\lui{x}{H}}^{H}.\label{eq:Induction-restriction-groups.}
\end{equation}
Equations \eqref{eq:restriction-induction-mackey-groups.} and \eqref{eq:Induction-restriction-groups.}
play a key role in the arguments used in \cite{SASAKI198298} in order
to obtain a Green correspondence for Mackey functors over groups.
However, when trying to obtain similar results in the context of Mackey
functors over fusion systems, the author was met with many complications.
All of them can be traced back to the fact that the category $\aOF$
does not in general admit products. In order to avoid such complications,
Proposition \ref{prop:Product-OFc.} suggests that we should introduce
the following.
\begin{defn}
\label{def:F-centric-Mackey-functor.}Let $H\le S$, let $\H\subseteq\F$
be a fusion system over $H$ and let $M$ be a Mackey functor over
$\H$ on $\R$. We say that $M$ is \textbf{$\F$-centric }if $I_{K}^{K}\cdot M=0$
for every $K\in\H\backslash\left(\Fc\cap\H\right)$. The \textbf{category
of $\F$-centric Mackey functors over $\H$ on $\R$ }(denoted by
$\MackFHR{\F}{\H}$) is the full subcategory of $\FmuFRmod{\H}$ whose
objects are $\F$-centric Mackey functors over $\H$.

If $H=S$ and $\H=\F$ we simply say that $M$ is \textbf{centric}
and denote by $\MackFcR$ the \textbf{category of centric Mackey functors
over $\F$ on $\R$}.
\end{defn}

Let $\H$ be a fusion subsystem of $\F$ and let $M$ be an $\H$-centric
Mackey functor over $\H$. The induced Mackey functor $M\uparrow_{\H}^{\F}$
over $\F$ might not be $\F$-centric since we don't necessarily have
$\H^{c}\subseteq\Fc$. However, we have the following result.
\begin{prop}
\label{prop:centric-Induction-restriction.}Let $H$ and $K$ be subgroups
of $S$ such that $H\le K$ and let $\H$ and $\mathcal{K}$ be fusion
systems over $H$ and $K$ respectively such that $\H\subseteq\mathcal{K}\subseteq\F$.
Then we have that:
\begin{enumerate}
\item The functor $\downarrow_{\H}^{\mathcal{K}}$ maps $\MackFHR{\F}{\mathcal{K}}$
to $\MackFHR{\F}{\H}$. In particular $\downarrow_{\H}^{\F}$ maps
$\MackFcR$ to $\MackFHR{\F}{\H}$.
\item The functor $\uparrow_{\H}^{\mathcal{K}}$ maps $\MackFHR{\F}{\H}$
to $\MackFHR{\F}{\mathcal{K}}$. In particular $\uparrow_{\H}^{\F}$
maps $\MackFHR{\F}{\H}$ to $\MackFcR$.
\item For every isomorphism $\varphi\colon H\to\varphi\left(H\right)$ in
$\F$ the functor $\lui{\varphi}{\cdot}$ maps $\MackFHR{\F}{\FHH}$
to $\MackFHR{\F}{\FHH[\varphi\left(H\right)]}$.
\end{enumerate}
\end{prop}

\begin{proof}
$\phantom{.}$
\begin{enumerate}
\item Let $M\in\MackFHR{\F}{\mathcal{K}}$. For every $J\in\H\backslash\left(\H\cap\Fc\right)$
we have that $J\not\in\Fc$ and, therefore, $I_{J}^{J}\left(M\downarrow_{\H}^{\mathcal{K}}\right)=I_{J}^{J}M=0$.
This proves that $M\downarrow_{\H}^{\mathcal{K}}\in\MackFHR{\F}{\H}$.
\item Let $M\in\MackFHR{\F}{\H}$ and let $J\in\mathcal{K}\backslash\left(\mathcal{K}\cap\Fc\right)$.
From Proposition \ref{prop:Mackey-algebra-basis.} and Definition
\ref{def:Restriction-induction-and-conjugation-functors.} we know
that any element in $M\uparrow_{\H}^{\mathcal{K}}$ can be written
as an $\R$-linear combination of elements of the form 
\[
y:=I_{\varphi\left(C\right)}^{B}c_{\varphi}R_{C}^{A}1_{\FmuFR{\H}}\otimes x.
\]
Where $x\in M$ and $I_{\varphi\left(C\right)}^{B}c_{\varphi}R_{C}^{A}\in\FmuFR{\mathcal{K}}$.
Thus, it suffices to prove that $I_{J}^{J}y=0$ for every such $y\in M\uparrow_{\H}^{\mathcal{K}}$.
From Lemma \ref{lem:Many-properties-definition.} \eqref{enu:property-all-0.}
we can assume without loss of generality that $A\le H$ and $B=J$.
With this setup we have that
\begin{align*}
I_{J}^{J}y & =I_{J}^{J}I_{\varphi\left(C\right)}^{J}c_{\varphi}R_{C}^{A}1_{\FmuFR{\H}}\otimes x,\\
 & =I_{\varphi\left(C\right)}^{J}c_{\varphi}R_{C}^{A}\otimes x,\\
 & =I_{\varphi\left(C\right)}^{J}c_{\varphi}\otimes R_{C}^{A}x.
\end{align*}
Where, in the last identity, we are using the fact that the tensor
product is over $\FmuFR{\FHH}$ and $R_{C}^{A}\in\FmuFR{\FHH}$ (see
Corollary \ref{cor:Mackey-algebra-inclusion.}). Since $J\not\in\Fc$
and $\varphi\left(C\right)\le J$ then we can deduce from \cite[Proposition 4.4]{IntroductionToFusionSystemsLinckelmann}
that $\varphi\left(C\right)\not\in\Fc$. From definition of $\F$-centric
subgroup we can deduce that $C\not\in\Fc$. Since $C\le H$ then we
can conclude that $C\in\H\backslash\left(\H\cap\Fc\right)$. Since
$M\in\MackFHR{\F}{\H}$ and $x\in M$ this implies that $R_{C}^{A}x\in I_{C}^{C}M=0$.
Therefore we can conclude once again that $I_{J}^{J}y=0$ thus proving
that $M\uparrow_{\H}^{\mathcal{K}}\in\MackFHR{\F}{\mathcal{K}}$.
\item Let $M\in\MackFHR{\F}{\FHH}$, let $\varphi$ be as in the statement
and let $J\in\FHH[\varphi\left(H\right)]\backslash\left(\FHH[\varphi\left(H\right)]\cap\Fc\right)$.
By definition of $\F$-centric subgroup we know that $\varphi^{-1}\left(J\right)\in\FHH\backslash\left(\FHH\cap\Fc\right)$.
Then, by definition of the functor $\lui{\varphi}{\cdot}$ (see Definition
\ref{def:Restriction-induction-and-conjugation-functors.}), we have
that $I_{J}^{J}\lui{\varphi}{M}=I_{\varphi^{-1}\left(J\right)}^{\varphi^{-1}\left(J\right)}M=0$
thus proving that $\lui{\varphi}{M}\in\MackFHR{\F}{\FHH[\varphi\left(H\right)]}$.
\end{enumerate}
\end{proof}
Proposition \ref{prop:centric-Induction-restriction.} motivates the
introduction of the following.
\begin{notation}
\label{nota:Induction-and-restriction-functors.}Let $\H,\mathcal{K},\F$
and $\varphi\colon H\to\varphi\left(H\right)$ be as in Proposition
\ref{prop:centric-Induction-restriction.}. We will use the same notation
to refer to the functors $\uparrow_{\H}^{\mathcal{K}}$, $\downarrow_{\H}^{\mathcal{K}}$
and $\lui{\varphi}{\cdot}$ of Definition \ref{def:Restriction-induction-and-conjugation-functors.}
and their restrictions given by Proposition \ref{prop:centric-Induction-restriction.}.
\end{notation}

With this setup we are now just one Lemma away from providing a result
analogue to Equation \eqref{eq:restriction-induction-mackey-groups.}
in the context of centric Mackey functors over fusion systems.
\begin{lem}
\label{lem:Equivalence-R-mod-I_A^AM.}Let $H,K\in\Fc$, let $M\in\MackFHR{\F}{\FHH}$,
let $\left(A,\overline{\varphi}\right)\in\left[H\times_{\F}K\right]$,
let $y\in I_{A}^{A}\FmuFR{\FHH}I_{A}^{A}$, let $\I$ be the two sided
ideal of $\muFR$ generated by elements of the form $I_{J}^{J}$ with
$J\in\FHH\backslash\left(\FHH\cap\Fc\right)$ and let $\pi\colon\muFR\to\muFR/\I$
be the natural projection. If $\pi\left(I_{\overline{\varphi}\left(A\right)}^{K}c_{\overline{\varphi}}y\right)=\pi\left(I_{\overline{\varphi}\left(A\right)}^{K}c_{\overline{\varphi}}\right)$
(see Corollary \ref{cor:Mackey-algebra-inclusion.}) then $\pi\left(y\right)=\pi\left(I_{A}^{A}\right)$.
In particular, viewing the subset $I_{\overline{\varphi}\left(A\right)}^{K}c_{\overline{\varphi}}\FmuFR{\FHH}$
of $\muFR$ as a right $\FmuFR{\FHH}$\textup{-module}, and defining
\[
I_{\overline{\varphi}\left(A\right)}^{K}c_{\overline{\varphi}}\otimes_{\FmuFR{\FHH}}M:=I_{\overline{\varphi}\left(A\right)}^{K}c_{\overline{\varphi}}\FmuFR{\FHH}\otimes_{\FmuFR{\FHH}}M,
\]
we have an isomorphism of $\R$-modules from $I_{A}^{A}M$ to $I_{\overline{\varphi}\left(A\right)}^{K}c_{\overline{\varphi}}\otimes_{\FmuFR{\FHH}}M$
that sends $x$\textup{ to $I_{\overline{\varphi}\left(A\right)}^{K}c_{\overline{\varphi}}\otimes x$.}
\end{lem}

\begin{proof}
From Lemma \ref{lem:Many-properties-definition.}, Proposition \ref{prop:Mackey-algebra-basis.}
and \cite[Proposition 4.4]{IntroductionToFusionSystemsLinckelmann}
we know that the ideal $\I$ is spanned as an $\R$-module by elements
of the form $I_{\psi\left(J\right)}^{C}c_{\psi}R_{J}^{B}$ such that
exists $J'\in\FHH\backslash\left(\FHH\cap\Fc\right)$ satisfying $J=_{\F}J'$.
Define now $\J:=\I\cap\FmuFR{\FHH}$. From the above we can conclude
that $\J$ is spanned as an $\R$-module by elements of the form $I_{\lui{h}{J}}^{C}c_{c_{h}}R_{J}^{B}$
with $J\in\FHH\backslash\left(\FHH\cap\Fc\right)$ and $h\in H$.
Since $M$ is $\F$-centric then, by definition, we have that $\J M=0$.
On the other hand, from the above description of $\I$ and $\J$,
we know that $\pi\left(I_{\overline{\varphi}\left(A\right)}^{J}c_{\overline{\varphi}}\FmuFR{\FHH}\right)$
is equivalent, as a right $\FmuFR{\FHH}$ module, to $\left(I_{\overline{\varphi}\left(A\right)}^{J}c_{\overline{\varphi}}\FmuFR{\FHH}\right)/\left(I_{\overline{\varphi}\left(A\right)}^{J}c_{\overline{\varphi}}\J\right)$.
We can therefore conclude that
\[
\pi\left(I_{\overline{\varphi}\left(A\right)}^{J}c_{\overline{\varphi}}\right)\otimes_{\FmuFR{\FHH}}M:=\pi\left(I_{\overline{\varphi}\left(A\right)}^{J}c_{\overline{\varphi}}\FmuFR{\FHH}\right)\otimes_{\FmuFR{\FHH}}M\cong I_{\overline{\varphi}\left(A\right)}^{J}c_{\overline{\varphi}}\otimes_{\FmuFR{\FHH}}M.
\]
With this setup we have a surjective morphism of $\R$-modules $\Gamma\colon I_{A}^{A}M\to\pi\left(I_{\overline{\varphi}\left(A\right)}^{J}c_{\overline{\varphi}}\right)\otimes_{\FmuFR{\FHH}}M$
that sends any $x\in I_{A}^{A}M$ to $\pi\left(I_{\overline{\varphi}\left(A\right)}^{J}c_{\overline{\varphi}}\right)\otimes x$.
Assume that $\pi\left(I_{\overline{\varphi}\left(A\right)}^{J}c_{\overline{\varphi}}\right)\otimes x=0$.
Then there exists $y\in I_{A}^{A}\FmuFR{\FHH}I_{A}^{A}$ such that
$\pi\left(I_{\overline{\varphi}\left(A\right)}^{J}c_{\overline{\varphi}}y\right)=\pi\left(I_{\overline{\varphi}\left(A\right)}^{J}c_{\overline{\varphi}}\right)$
and $yx=0$. If the first part of the statement is true then, since
$\J M=0$ we would have that $yx=I_{A}^{A}x=x$. This would prove
that $x=0$ and, therefore, that $\Gamma$ is an isomorphism of $\R$-modules.
In other words we have proven that the second part of the statement
follows from the first.

Let's now prove the first part of the statement. For $i=1,2$ let
$h_{i}\in H$ and $B_{i}\le A^{h_{i}}\cap A$ such that $B_{i}\in\FHH\cap\Fc$
and that
\[
\pi\left(I_{\overline{\varphi\iota c_{h_{1}}}\left(B_{1}\right)}^{K}c_{\overline{\varphi\iota c_{h_{1}}}}R_{B_{1}}^{A}\right)=\pi\left(I_{\overline{\varphi\iota c_{h_{2}}}\left(B_{2}\right)}^{K}c_{\overline{\varphi\iota c_{h_{2}}}}R_{B_{2}}^{A}\right).
\]
Since $B_{i}\not\in\Fc$ then we can deduce from the description of
$\I$ given at the start of the proof, the above identity and Proposition
\ref{prop:Mackey-algebra-basis.} that
\[
I_{\overline{\varphi\iota c_{h_{1}}}\left(B_{1}\right)}^{K}c_{\overline{\varphi\iota c_{h_{1}}}}R_{B_{1}}^{A}=I_{\overline{\varphi\iota c_{h_{2}}}\left(B_{2}\right)}^{K}c_{\overline{\varphi\iota c_{h_{2}}}}R_{B_{2}}^{A}.
\]
From Items \eqref{enu:property-I_H^H-is-identity.} and \eqref{enu:Property-conjugation-commutes.}
of Lemma \ref{lem:Many-properties-definition.} and Proposition \ref{prop:Mackey-algebra-basis.}
we can conclude that there exists $a\in A$ such that $B:=B_{1}=B_{2}^{a}$
and $\overline{\varphi}\,\overline{\iota c_{h_{1}a}}=\overline{\varphi}\,\overline{\iota c_{h_{2}}}$.
Since $h_{1}a,h_{2}\in H$ we also have that $\overline{\iota_{A}^{H}}\,\overline{\iota c_{h_{1}a}}=\overline{\iota_{A}^{H}}=\overline{\iota_{A}^{H}}\,\overline{\iota c_{h_{2}}}$.
From the universal properties of product we can therefore conclude
that $\overline{\iota_{\lui{h_{1}}{B}}^{A}c_{h_{1}a}}=\overline{\iota_{\lui{h_{2}}{B}}^{A}c_{h_{2}}}$.
From definition of $\OFc$ this implies that there exists $b\in A$
such that $c_{bh_{1}a}=c_{h_{2}}$ as an isomorphism from $B$ to
$\lui{h_{2}}{B}$. Therefore, there exists $z\in C_{H}\left(B\right)$
such that $bh_{1}az=h_{2}$. Since $B\in\Fc$ we can conclude that
$z\in B\le A$ and, therefore, $h_{2}\in Ah_{1}A$. Now let $y$ be
as in the statement. From Proposition \ref{prop:Mackey-algebra-basis.}
we can write
\[
y:=\sum_{h\in\left[A\backslash H/A\right]}\sum_{\begin{array}{c}
{\scriptstyle B\le A^{h}\cap A}\\
{\scriptstyle \text{up to }A\text{-conj.}}
\end{array}}\lambda_{h,B}I_{\lui{h}{B}}^{A}c_{c_{h}}R_{B}^{A},
\]
for some $\lambda_{h,B}\in\R$. Since we are only interested in the
projection $\pi\left(y\right)$ we can assume without loss of generality
that $\lambda_{h,B}=0$ whenever $B\in\FHH\backslash\left(\FHH\cap\Fc\right)$.
From the above and Proposition \ref{prop:Mackey-algebra-basis.} we
can conclude that if $y$ satisfies $\pi\left(I_{\overline{\varphi}\left(A\right)}^{J}c_{\overline{\varphi}}y\right)=\pi\left(I_{\overline{\varphi}\left(A\right)}^{J}c_{\overline{\varphi}}\right)$
then $\lambda_{h,B}=0$ unless $B=A$ and $h\in A$ in which case
it equals $1$. In other words we have that $\pi\left(y\right)=\pi\left(I_{A}^{A}\right)$
just as we wanted to prove.
\end{proof}
We can now prove an analogue to Equation \eqref{eq:restriction-induction-mackey-groups.}
in the context of centric Mackey functors over fusion systems.
\begin{prop}
\label{prop:Decomposition-centric-Mackey-functor.}Let $H\in\Fc$,
let $\I$ be the two sided ideal of $\muFR$ generated by elements
of the form $I_{K}^{K}$ with $K\in\F\backslash\Fc$ and let $\pi\colon\muFR\to\muFR/\I$
be the natural projection map. Then the set $\pi\left(\muFR1_{\FmuFR{\FHH}}\right)$
inherits from $\muFR1_{\FmuFR{\FHH}}$ a right $\FmuFR{\FHH}$-module
structure and the following is a $\FmuFR{\FHH}$ basis of $\pi\left(\muFR1_{\FmuFR{\FHH}}\right)$
\[
\mathcal{B}:=\bigsqcup_{K\in\Fc}\bigsqcup_{\left(A,\overline{\varphi}\right)\in\left[H\times_{\F}K\right]}\left\{ \pi\left(I_{\overline{\varphi}\left(A\right)}^{K}c_{\overline{\varphi}}\right)\right\} .
\]
In particular, for any $M\in\MackFHR{\F}{\FHH}$, we have the following
equivalence of $\R$-modules
\[
M\uparrow_{\FHH}^{\F}\cong_{\R}\bigoplus_{K\in\Fc}\bigoplus_{\left(A,\overline{\varphi}\right)\in\left[H\times_{\F}K\right]}I_{\overline{\varphi}\left(A\right)}^{K}c_{\overline{\varphi}}\otimes M\cong_{\R}\bigoplus_{K\in\Fc}\bigoplus_{\left(A,\overline{\varphi}\right)\in\left[H\times_{\F}K\right]}I_{A}^{A}M.
\]
Where each $I_{\overline{\varphi}\left(A\right)}^{K}c_{\overline{\varphi}}\otimes M$
is seen as an $\R$-submodule of $M\uparrow_{\FHH}^{\F}$. 
\end{prop}

\begin{proof}
Throughout this proof we will denote the right $\FmuFR{\FHH}$-module
$\pi\left(\muFR1_{\FmuFR{\FHH}}\right)$ simply by $\overline{\muFR_{H}}$.

From Lemma \ref{lem:Many-properties-definition.}, Proposition \ref{prop:Mackey-algebra-basis.}
and \cite[Proposition 4.4]{IntroductionToFusionSystemsLinckelmann}
we know that the ideal $\I$ is spanned as an $\R$-module by elements
of the form $I_{\varphi\left(C\right)}^{B}c_{\varphi}R_{C}^{A}$ with
$C\in\F\backslash\Fc$. If $A\not\le H$ we have that $R_{C}^{A}1_{\FmuFR{\FHH}}\otimes_{\FmuFR{\FHH}}M=0$.
On the other hand, if $A\le H$, we have that $C\le H$ and, therefore,
$C\in\FHH\backslash\left(\FHH\cap\Fc\right)$. Since $M\in\MackFHR{\F}{\FHH}$
this implies that $R_{C}^{A}1_{\FmuFR{\FHH}}\otimes_{\FmuFR{\FHH}}M=I_{C}^{C}\otimes_{\FmuFR{\FHH}}R_{C}^{A}M=0$.
In either case we have that $\I1_{\FmuFR{\FHH}}\otimes_{\FmuFR{\FHH}}M=0$.
Using right exactness of the tensor product functor we can conclude
from the above and definition of $\uparrow_{\FHH}^{\F}$ (see Definition
\ref{def:Restriction-induction-and-conjugation-functors.}) that 
\begin{align*}
M\uparrow_{\FHH}^{\F} & \cong\overline{\muFR_{H}}\otimes_{\FmuFR{\FHH}}M.
\end{align*}
Assume now that $\mathcal{B}$ is a $\FmuFR{\FHH}$ basis of the right
$\FmuFR{\FHH}$-module $\overline{\muFR_{H}}$. Since tensor product
preserves direct sums, we obtain from the previous equivalence that
\begin{align*}
M\uparrow_{\FHH}^{\F} & \cong\bigoplus_{K\in\Fc}\bigoplus_{\left(A,\overline{\varphi}\right)\in\left[H\times_{\F}K\right]}\pi\left(I_{\overline{\varphi}\left(A\right)}^{K}c_{\overline{\varphi}}\FmuFR{\FHH}\right)\otimes_{\FmuFR{\FHH}}M,\\
 & \cong\bigoplus_{K\in\Fc}\bigoplus_{\left(A,\overline{\varphi}\right)\in\left[H\times_{\F}K\right]}I_{\overline{\varphi}\left(A\right)}^{K}c_{\overline{\varphi}}\FmuFR{\FHH}\otimes_{\FmuFR{\FHH}}M.
\end{align*}
Where, for the second identity, we are using that $\I1_{\FmuFR{\FHH}}\otimes_{\FmuFR{\FHH}}M=0$
and right exactness of tensor product. The second part of the statement
follows from Lemma \ref{lem:Equivalence-R-mod-I_A^AM.} and the above
by viewing each $I_{\overline{\varphi}\left(A\right)}^{K}c_{\overline{\varphi}}\FmuFR{\FHH}\otimes_{\FmuFR{\FHH}}M$
as the $\R$-submodule $I_{\overline{\varphi}\left(A\right)}^{K}c_{\overline{\varphi}}\otimes_{\FmuFR{\FHH}}M$
of $M\uparrow_{\FHH}^{\F}$. This proves that the second part of the
statement follows from the first.

Let's now prove the first part of the statement. From Proposition
\ref{prop:Mackey-algebra-basis.} and the previous description of
$\I$ we obtain the following equivalence of right $\FmuFR{\FHH}$-modules
\[
\overline{\muFR_{H}}\cong\bigoplus_{K\in\Fc}\pi\left(I_{K}^{K}\right)\overline{\muFR_{H}}.
\]
For every $K\in\Fc$ we can now define
\[
\mathcal{B}^{K}:=\bigsqcup_{\left(A,\overline{\varphi}\right)\in\left[H\times_{\F}K\right]}\left\{ \pi\left(I_{\overline{\varphi}\left(A\right)}^{K}c_{\overline{\varphi}}\right)\right\} .
\]
In order to prove the statement it will suffice to prove that $\mathcal{B}^{K}$
is a right $\FmuFR{\FHH}$-basis of $\pi\left(I_{K}^{K}\right)\overline{\muFR_{H}}$.
In other words we need to prove that for every $K\in\Fc$ there exists
a direct sum decomposition of right $\FmuFR{\FHH}$-modules of the
form
\begin{equation}
\pi\left(I_{K}^{K}\right)\overline{\muFR_{H}}=\bigoplus_{\left(A,\overline{\varphi}\right)\in\left[H\times_{\F}K\right]}\pi\left(I_{\overline{\varphi}\left(A\right)}^{K}c_{\overline{\varphi}}\right)\overline{\muFR_{H}},\label{eq:aux-decomposition-to-prove.}
\end{equation}
Where the summands on the right hand side are seen as right $\FmuFR{\FHH}$-submodules
of $\pi\left(I_{K}^{K}\right)\overline{\muFR_{H}}$.

Fix $K\in\Fc$. From Proposition \ref{prop:Mackey-algebra-basis.}
and the above description of $\I$ we know that $\pi\left(I_{K}^{K}\right)\overline{\muFR_{H}}$
has an $\R$-basis of the form
\[
\mathcal{B}_{\R}^{K}:=\bigsqcup_{J\in\FHH\cap\Fc}\bigsqcup_{\begin{array}{c}
{\scriptstyle B\in\FHH[J]\cap\Fc}\\
{\scriptstyle \text{up to }J\text{-conj.}}
\end{array}}\bigsqcup_{\psi\in\left[\Aut_{K}\left(K\right)\backslash\Hom_{\FHH}\left(B,K\right)/\Aut_{J}\left(B\right)\right]}\left\{ \pi\left(I_{\overline{\psi}\left(B\right)}^{K}c_{\overline{\psi}}R_{B}^{J}\right)\right\} ,
\]
For each $\pi\left(I_{\overline{\psi}\left(B\right)}^{K}c_{\overline{\psi}}R_{B}^{J}\right)\in\mathcal{B}_{\R}^{K}$
we get a map $\overline{\psi}:B\to K$ and a map $\overline{\iota_{B}^{H}}:B\to H$.
From the universal properties of product we can then conclude that
there exists a unique $\left(B^{H,K},\overline{\psi^{H,K}}\right)\in\left[H\times_{\F}K\right]$
and a unique $\overline{\gamma_{\left(B,\overline{\psi}\right)}^{H,K}}\in\Hom_{\OFc}\left(B,B^{H,K}\right)$
such that $\overline{\iota_{B^{H,K}}^{H}}\overline{\gamma_{\left(B,\overline{\psi}\right)}^{H,K}}=\overline{\iota_{B}^{H}}$
and that $\overline{\psi^{H,K}}\overline{\gamma_{\left(B,\overline{\psi}\right)}^{H,K}}=\overline{\psi}$.
From the first identity and definition of $\OF$ we can conclude that
$\overline{\gamma_{\left(B,\overline{\psi}\right)}^{H,K}}\in\Orbitize{\FHH}$.
From the second identity and Corollary \ref{cor:conjugation-in-F-and-in-orbit-F.}
we can deduce that 
\begin{equation}
I_{\overline{\psi}\left(B\right)}^{K}c_{\overline{\psi}}=I_{\overline{\psi^{H,K}}\left(B^{H,K}\right)}^{K}c_{\overline{\psi^{H,K}}}I_{\overline{\gamma_{\left(B,\overline{\psi}\right)}^{H,K}}\left(B\right)}^{B^{H,K}}c_{\overline{\gamma_{\left(B,\overline{\psi}\right)}^{H,K}}}.\label{eq:aux-B^K,(A,f)-generates.}
\end{equation}
This allows us to write $\mathcal{B}_{\R}^{K}=\bigsqcup_{\left(A,\overline{\varphi}\right)\in\left[H\times_{\F}K\right]}\mathcal{B}_{\R}^{K,\left(A,\overline{\varphi}\right)}$
where
\[
\mathcal{B}_{\R}^{K,\left(A,\overline{\varphi}\right)}:=\bigsqcup_{J\in\FHH\cap\Fc}\bigsqcup_{\begin{array}{c}
{\scriptstyle B\in\FHH[J]\cap\Fc}\\
{\scriptstyle \text{up to }J\text{-conj.}}
\end{array}}\bigsqcup_{\begin{array}{c}
{\scriptstyle \psi\in\left[\Aut_{K}\left(K\right)\backslash\Hom_{\FHH}\left(B,K\right)/\Aut_{J}\left(B\right)\right]}\\
{\scriptstyle \left(B^{H,K},\overline{\psi^{H,K}}\right)=\left(A,\overline{\varphi}\right)}
\end{array}}\left\{ \pi\left(I_{\overline{\psi}\left(B\right)}^{K}c_{\overline{\psi}}R_{B}^{J}\right)\right\} .
\]
Fix $\left(A,\overline{\varphi}\right)\in\left[H\times_{\F}K\right]$.
From Equation \eqref{eq:aux-B^K,(A,f)-generates.} we know that $\mathcal{B}_{\R}^{K,\left(A,\overline{\varphi}\right)}\subseteq\pi\left(I_{\overline{\varphi}\left(A\right)}^{K}c_{\overline{\varphi}}\right)\overline{\muFR_{H}}$.
If we now prove that $\mathcal{B}_{\R}^{K,\left(A,\overline{\varphi}\right)}$
is in fact a generating set of $\pi\left(I_{\overline{\varphi}\left(A\right)}^{K}c_{\overline{\varphi}}\right)\overline{\muFR_{H}}$
(as an $\R$-module) then, since $\pi\left(I_{\overline{\varphi}\left(A\right)}^{K}c_{\overline{\varphi}}\right)\overline{\muFR_{H}}$
is a right $\FmuFR{\FHH}$-submodule of $\pi\left(I_{K}^{K}\right)\overline{\muFR_{H}}$
and $\mathcal{B}_{\R}^{K}$ is an $\R$-basis of $\pi\left(I_{K}^{K}\right)\overline{\muFR_{H}}$,
we will obtain Equation \eqref{eq:aux-decomposition-to-prove.} and
the result will follow. From Proposition \ref{prop:Mackey-algebra-basis.}
and the above description of $\I$ it suffices to prove that for every
$J\in\FHH\cap\Fc$, every $C\in\FHH[J]\cap\Fc$ and every $\overline{\theta}\in\Hom_{\Orbitize{\FHH}}\left(C,A\right)$
there exists $\pi\left(I_{\overline{\psi}\left(B\right)}^{K}c_{\overline{\psi}}R_{B}^{J}\right)\in\mathcal{B}_{\R}^{K,\left(A,\overline{\varphi}\right)}$
such that $\pi\left(I_{\overline{\psi}\left(B\right)}^{K}c_{\overline{\psi}}R_{B}^{J}\right)=\pi\left(I_{\overline{\varphi}\left(A\right)}^{K}c_{\overline{\varphi}}I_{\overline{\theta}\left(C\right)}^{A}c_{\overline{\theta}}R_{C}^{J}\right)$.
From the description of $\mathcal{B}_{\R}^{K}$ there exist $j\in J$
and $\pi\left(I_{\overline{\psi}\left(B\right)}^{K}c_{\overline{\psi}}R_{B}^{J}\right)\in\mathcal{B}_{\R}^{K}$
such that $\lui{j}{B}=C$ and $\overline{\psi}=\overline{\varphi}\overline{\theta}\overline{c_{j}}$.
Here we are viewing $c_{j}$ as an isomorphism from $B$ to $C$.
Therefore, by definition, we have that $\overline{\gamma_{\left(B,\overline{\psi}\right)}^{H,K}}=\overline{\theta}\overline{c_{j}}$
and $\left(B^{H,K},\overline{\psi^{H,K}}\right)=\left(A,\overline{\varphi}\right)$.
In other words $\pi\left(I_{\overline{\psi}\left(B\right)}^{K}c_{\overline{\psi}}R_{B}^{J}\right)\in\mathcal{B}_{\R}^{K,\left(A,\overline{\varphi}\right)}$.
From Lemma \ref{lem:Many-properties-definition.} \eqref{enu:Property-conjugation-commutes.}
we know that $c_{j}R_{B}^{J}=R_{C}^{J}$ and, therefore, we can conclude
from the identities above and Corollary \ref{cor:conjugation-in-F-and-in-orbit-F.}
that $I_{\overline{\psi}\left(B\right)}^{K}c_{\overline{\psi}}R_{B}^{J}=I_{\overline{\varphi}\left(A\right)}^{K}c_{\overline{\varphi}}I_{\overline{\theta}\left(C\right)}^{A}c_{\overline{\theta}}R_{C}^{J}$
thus concluding the proof.
\end{proof}
Before proceeding it is worth introducing the following result motivated
by the notation of Proposition \ref{prop:Decomposition-centric-Mackey-functor.}. 
\begin{lem}
\label{lem:proving-muFR-morphism.}Let $H,K\in\Fc$. Then we have
that:
\begin{enumerate}
\item \label{enu:Def-A^theta.}For every $\left(A,\overline{\varphi}\right)\in\left[H\times_{\F}K\right]$,
every $J\in\Fc$ and every $\theta\in\Hom_{\Fc}\left(K,J\right)$
there exist a unique $\left(A^{\theta},\overline{\varphi^{\theta}}\right)\in\left[H\times_{\F}J\right]$
and a unique $\overline{\gamma_{\left(A,\overline{\varphi}\right)}^{\theta}}\in\Hom_{\OFc}\left(A,A^{\theta}\right)$
such that $\overline{\varphi^{\theta}}\overline{\gamma_{\left(A,\overline{\varphi}\right)}^{\theta}}=\overline{\theta}\overline{\varphi}$
and $\overline{\iota_{A^{\theta}}^{H}}\overline{\gamma_{\left(A,\overline{\varphi}\right)}^{\theta}}=\overline{\iota_{A}^{H}}$.
Moreover $\overline{\gamma_{\left(A,\overline{\varphi}\right)}^{\theta}}\in\Orbitize{\FHH}$
and, given $J'\in\Fc$ and $\delta\in\Hom_{\Fc}\left(J,J'\right)$
we have that $A^{\delta\theta}=\left(A^{\theta}\right)^{\delta}$,
that $\overline{\varphi^{\delta\theta}}=\overline{\left(\varphi^{\theta}\right)^{\delta}}$
and that $\overline{\gamma_{\left(A,\overline{\varphi}\right)}^{\delta\theta}}=\overline{\gamma_{\left(A^{\theta},\overline{\varphi^{\theta}}\right)}^{\delta}}\overline{\gamma_{\left(A,\overline{\varphi}\right)}^{\theta}}$.
If $\theta=\iota_{K}^{J}$ we simply write $\left(A^{J},\overline{\varphi^{J}}\right)$
and $\overline{\gamma_{\left(A,\overline{\varphi}\right)}^{J}}$.
\item \label{enu:Composing.}Let $J\in\Fc$ such that $J\ge K$ and let
$\left(A,\overline{\varphi}\right)\in\left[H\times_{\F}K\right]$.
The following identities are satisfied
\begin{align*}
I_{K}^{J}I_{\overline{\varphi}\left(A\right)}^{K}c_{\overline{\varphi}} & =I_{\overline{\varphi^{J}}\left(A^{J}\right)}^{J}c_{\overline{\varphi^{J}}}I_{\overline{\gamma_{\left(A,\overline{\varphi}\right)}^{J}}\left(A\right)}^{A^{J}}c_{\overline{\gamma_{\left(A,\overline{\varphi}\right)}^{J}}},\\
c_{\overline{\varphi^{-1}}}R_{\overline{\varphi}\left(A\right)}^{K}R_{K}^{J} & =c_{\overline{\left(\gamma_{\left(A,\overline{\varphi}\right)}^{J}\right)^{-1}}}R_{\overline{\gamma_{\left(A,\overline{\varphi}\right)}^{J}}\left(A\right)}^{A^{J}}c_{\overline{\left(\varphi^{J}\right)^{-1}}}R_{\overline{\varphi^{J}}\left(A^{J}\right)}^{J}.
\end{align*}
\item \label{enu:using-mackey.}Let $J\in\Fc$ such that $J\ge K$ and let
$\I$ be the two sided ideal of $\muFR$ generated by elements of
the form $I_{C}^{C}$ such that $C\in\F\backslash\Fc$. The following
equivalences are satisfied
\begin{align*}
\sum_{\left(B,\overline{\psi}\right)\in\left[H\times_{\F}J\right]}R_{K}^{J}I_{\overline{\psi}\left(B\right)}^{J}c_{\overline{\psi}} & \equiv\sum_{\left(A,\overline{\varphi}\right)\in\left[H\times_{\F}K\right]}I_{\overline{\varphi}\left(A\right)}^{K}c_{\overline{\varphi}}c_{\overline{\left(\gamma_{\left(A,\overline{\varphi}\right)}^{J}\right)^{-1}}}R_{\overline{\gamma_{\left(A,\overline{\varphi}\right)}^{J}}\left(A\right)}^{A^{J}}, & \text{\ensuremath{\mod}}\I\\
\sum_{\left(B,\overline{\psi}\right)\in\left[H\times_{\F}J\right]}c_{\overline{\psi^{-1}}}R_{\overline{\psi}\left(B\right)}^{J}I_{K}^{J} & \equiv\sum_{\left(A,\overline{\varphi}\right)\in\left[H\times_{\F}K\right]}I_{\overline{\gamma_{\left(A,\overline{\varphi}\right)}^{J}}\left(A\right)}^{A^{J}}c_{\overline{\gamma_{\left(A,\overline{\varphi}\right)}^{J}}}c_{\overline{\varphi^{-1}}}R_{\overline{\varphi}\left(A\right)}^{K}. & \text{\ensuremath{\mod}}\I
\end{align*}
More precisely, for every $\left(B,\overline{\psi}\right)\in\left[H\times_{\F}J\right]$
we have that
\begin{align*}
R_{K}^{J}I_{\overline{\psi}\left(B\right)}^{J}c_{\overline{\psi}} & \equiv\sum_{\begin{array}{c}
{\scriptstyle \left(A,\overline{\varphi}\right)\in\left[H\times_{\F}K\right]}\\
{\scriptstyle \left(A^{J},\overline{\varphi^{J}}\right)=\left(B,\overline{\psi}\right)}
\end{array}}I_{\overline{\varphi}\left(A\right)}^{K}c_{\overline{\varphi}}c_{\overline{\left(\gamma_{\left(A,\overline{\varphi}\right)}^{J}\right)^{-1}}}R_{\overline{\gamma_{\left(A,\overline{\varphi}\right)}^{J}}\left(A\right)}^{A^{J}}, & \text{\ensuremath{\mod}}\I\\
c_{\overline{\psi^{-1}}}R_{\overline{\psi}\left(B\right)}^{J}I_{K}^{J} & \equiv\sum_{\begin{array}{c}
{\scriptstyle \left(A,\overline{\varphi}\right)\in\left[H\times_{\F}K\right]}\\
{\scriptstyle \left(A^{J},\overline{\varphi^{J}}\right)=\left(B,\overline{\psi}\right)}
\end{array}}I_{\overline{\gamma_{\left(A,\overline{\varphi}\right)}^{J}}\left(A\right)}^{A^{J}}c_{\overline{\gamma_{\left(A,\overline{\varphi}\right)}^{J}}}c_{\overline{\varphi^{-1}}}R_{\overline{\varphi}\left(A\right)}^{K}. & \text{\ensuremath{\mod}}\I
\end{align*}
\item \label{enu:Conjugating.}Let $\rho\colon K\to\rho\left(K\right)$
be an isomorphism in $\F$ then, for every $\left(A,\overline{\varphi}\right)\in\left[H\times_{\F}K\right]$,
the morphism $\overline{\gamma_{\left(A,\overline{\varphi}\right)}^{\rho}}$
is an isomorphism and we have
\begin{align*}
c_{\rho}I_{\overline{\varphi}\left(A\right)}^{K}c_{\overline{\varphi}} & =I_{\overline{\varphi^{\rho}}\left(A^{\rho}\right)}^{\rho\left(K\right)}c_{\overline{\varphi^{\rho}}}c_{\gamma_{\left(A,\overline{\varphi}\right)}^{\rho}}, & c_{\overline{\varphi^{-1}}}R_{\overline{\varphi}\left(A\right)}^{K}c_{\rho^{-1}} & =c_{\left(\gamma_{\left(A,\overline{\varphi}\right)}^{\rho}\right)^{-1}}c_{\overline{\left(\varphi^{\rho}\right)^{-1}}}R_{\overline{\varphi^{\rho}}\left(A^{\rho}\right)}^{\rho\left(K\right)}.
\end{align*}
For any representative $\gamma_{\left(A,\overline{\varphi}\right)}^{\rho}\in\overline{\gamma_{\left(A,\overline{\varphi}\right)}^{\rho}}$.
In particular, from Proposition \ref{prop:properties-HXK.} \eqref{enu:prod-iso-right.}
\begin{align*}
\sum_{\left(A,\overline{\varphi}\right)\in\left[H\times_{\F}K\right]}c_{\rho}I_{\overline{\varphi}\left(A\right)}^{K}c_{\overline{\varphi}} & =\sum_{\left(B,\overline{\psi}\right)\in\left[H\times_{\F}\rho\left(K\right)\right]}I_{\overline{\psi}\left(B\right)}^{\rho\left(K\right)}c_{\overline{\psi}}c_{\gamma_{\left(B,\overline{\psi}\right)}^{\rho^{-1}}},\\
\sum_{\left(A,\overline{\varphi}\right)\in\left[H\times_{\F}K\right]}c_{\overline{\varphi^{-1}}}R_{\overline{\varphi}\left(A\right)}^{K}c_{\rho^{-1}} & =\sum_{\left(B,\overline{\psi}\right)\in\left[H\times_{\F}\rho\left(K\right)\right]}c_{\gamma_{\left(B,\overline{\psi}\right)}^{\rho}}c_{\overline{\psi^{-1}}}R_{\overline{\psi}\left(B\right)}^{\rho\left(K\right)}.
\end{align*}
\end{enumerate}
\end{lem}

\begin{proof}
We will only prove the first equation of each item since the proof
of the second ones are analogous.
\begin{enumerate}
\item Item \eqref{enu:Def-A^theta.} is an immediate consequence of the
universal properties of products. The fact that $\overline{\gamma_{\left(A,\overline{\varphi}\right)}^{\theta}}\in\Orbitize{\FHH}$
follows from definition of $\OF$ and the identity $\overline{\iota_{A^{\theta}}^{H}}\overline{\gamma_{\left(A,\overline{\varphi}\right)}^{\theta}}=\overline{\iota_{A}^{H}}$.
\item Item \eqref{enu:Composing.} follows from the identity $\overline{\varphi^{J}}\overline{\gamma_{\left(A,\overline{\varphi}\right)}^{J}}=\overline{\iota_{K}^{J}}\overline{\varphi}$
and Corollary \ref{cor:conjugation-in-F-and-in-orbit-F.}.
\item Let $\left(B,\overline{\psi}\right)\in\left[H\times_{\F}J\right]$,
fix a representative $\psi\in\overline{\psi}$ and view it as an isomorphism
between the appropriate restrictions. Item \eqref{enu:using-mackey.}
now follows from the identities below
\begin{align*}
R_{K}^{J}I_{\overline{\psi}\left(B\right)}^{J}c_{\overline{\psi}} & =\sum_{x\in\left[K\backslash J/\psi\left(B\right)\right]}I_{K\cap\lui{x}{\left(\psi\left(B\right)\right)}}^{K}c_{c_{x}\psi}R_{\psi^{-1}\left(K^{x}\cap\psi\left(B\right)\right)}^{B}, & \text{Lemma \ref{lem:Many-properties-definition.} \eqref{enu:Property-conjugation-commutes.} and \eqref{enu:property-Mackey-formula.}}\\
 & \equiv\sum_{\begin{array}{c}
{\scriptstyle x\in\left[K\backslash J/\psi\left(B\right)\right]}\\
{\scriptstyle K^{x}\cap\psi\left(B\right)\in\Fc}
\end{array}}I_{K\cap\lui{x}{\left(\psi\left(B\right)\right)}}^{K}c_{c_{x}\psi}R_{\psi^{-1}\left(K^{x}\cap\psi\left(B\right)\right)}^{B}, & \text{\ensuremath{\mod}}\I\\
 & =\sum_{\begin{array}{c}
{\scriptstyle \left(A,\overline{\varphi}\right)\in\left[H\times_{\F}K\right]}\\
{\scriptstyle \left(A^{J},\overline{\varphi^{J}}\right)=\left(B,\overline{\psi}\right)}
\end{array}}I_{\overline{\varphi}\left(A\right)}^{K}c_{\overline{\varphi}}c_{\overline{\left(\gamma_{\left(A,\overline{\varphi}\right)}^{J}\right)^{-1}}}R_{\overline{\gamma_{\left(A,\overline{\varphi}\right)}^{J}}\left(A\right)}^{A^{J}}. & \text{Proposition \ref{prop:properties-HXK.} \eqref{enu:prod-pullback-right.}}
\end{align*}
Where, in the last identity, we are using the fact that the bijection
of Proposition \ref{prop:properties-HXK.} \eqref{enu:prod-pullback-right.},
which sends every $\left(B,\overline{\psi}\right)\in\left[H\times_{\F}J\right]$
and every $x\in\left[K\backslash J/\psi\left(B\right)\right]$ to
$\left(A,\overline{\varphi}\right)=\left(\left(\psi^{-1}\left(K^{x}\cap\psi\left(B\right)\right)\right)^{h},\overline{\iota c_{x}\psi c_{h}}\right)$
for some $h\in H$ (which in Proposition \ref{prop:properties-HXK.}
\eqref{enu:prod-pullback-right.} we can assume to be $1_{S}$), satisfies
$\left(A^{J},\overline{\varphi^{J}}\right)=\left(B,\overline{\psi}\right)$
and $\overline{\gamma_{\left(A,\overline{\varphi}\right)}^{J}}=\overline{\iota_{\lui{h}{A}}^{B}c_{h}}$.
\item From uniqueness of the map $\overline{\gamma_{\left(A,\overline{\varphi}\right)}^{\text{Id}_{K}}}$
we know that $\overline{\gamma_{\left(A,\overline{\varphi}\right)}^{\text{Id}_{K}}}=\overline{\text{Id}_{A}}$,
therefore, from Item \eqref{enu:Def-A^theta.}, we can deduce that
$\overline{\gamma_{\left(A,\overline{\varphi}\right)}^{\rho}}$ is
an isomorphism with inverse $\overline{\gamma_{\left(A^{\rho},\overline{\varphi^{\rho}}\right)}^{\rho^{-1}}}$.
Item \eqref{enu:Conjugating.} now follows from identity $\overline{\varphi^{\rho}}\overline{\gamma_{\left(A,\overline{\varphi}\right)}^{\rho}}=\overline{\rho}\overline{\varphi}$
and Corollary \ref{cor:conjugation-in-F-and-in-orbit-F.}.
\end{enumerate}
\end{proof}
As a consequence of Proposition \ref{prop:Decomposition-centric-Mackey-functor.}
we can recover a result that appears in Mackey functors over groups
and which is, in general, not true for Mackey functors over fusion
systems. We will not prove it in detail since it falls outside the
scope of this paper but it's worth sketching a proof.
\begin{rem}
Let $H\in\Fc$ and view the functors $\uparrow_{\FHH}^{\F}$ and $\downarrow_{\FHH}^{\F}$
as functors between the categories $\MackFHR{\F}{\FHH}$ and $\MackFcR$
(see Proposition \ref{prop:centric-Induction-restriction.}). Then
$\uparrow_{\FHH}^{\F}$ is both right and left adjoint to $\downarrow_{\FHH}^{\F}$.
Define the coinduction Mackey functor $\Uparrow_{\FHH}^{\F}$ as the
functor that sends any $M\in\MackHR{\FHH}$ to
\[
M\Uparrow_{\FHH}^{\F}:=\Hom_{\FmuFR{\FHH}}\left(\muFR\downarrow_{\FHH}^{\F},M\right)\in\muFRmod.
\]
Here we are viewing $M\Uparrow_{\FHH}^{\F}$ as a $\muFR$-module
by setting for every $f\in M\Uparrow_{\FHH}^{\F}$, every $y\in\muFR$
and every $x\in\muFR\downarrow_{\FHH}^{\F}$ the image $\left(y\cdot f\right)\left(x\right)=f\left(xy\right)$.
It is well known that $\Uparrow_{\FHH}^{\F}$ is the right adjoint
of the restriction functor $\downarrow_{\FHH}^{\F}$ while $\uparrow_{\FHH}^{\F}$
is its left adjoint. Therefore, proving that $\uparrow_{\FHH}^{\F}$
and $\Uparrow_{\FHH}^{\F}$ coincide on $\MackFHR{\F}{\FHH}$ would
prove the statement. The broad steps to prove this are as follows.
First use the fact that $M$ is $\F$-centric in order to obtain the
isomorphism
\[
\Hom_{\FmuFR{\FHH}}\left(\muFR\downarrow_{\FHH}^{\F},M\right)\cong\Hom_{\FmuFR{\FHH}}\left(\muFR\downarrow_{\FHH}^{\F}/\I\downarrow_{\FHH}^{\F},M\right).
\]
With $\I$ as in Proposition \ref{prop:Decomposition-centric-Mackey-functor.}.
Using again Proposition \ref{prop:Decomposition-centric-Mackey-functor.}
and the anti involution $\cdot^{*}$ of $\muFR$ which sends every
$I_{\varphi\left(C\right)}^{B}c_{\varphi}R_{C}^{A}$ in $\muFR$ to
$\left(I_{\varphi\left(C\right)}^{B}c_{\varphi}R_{C}^{A}\right)^{*}:=I_{C}^{A}c_{\varphi^{-1}}R_{\varphi\left(C\right)}^{B}$
it can now be proven using arguments dual to those of Proposition
\ref{prop:Decomposition-centric-Mackey-functor.} that the following
is a $\FmuFR{\FHH}$ basis of $\muFR\downarrow_{\FHH}^{\F}/\I\downarrow_{\FHH}^{\F}$
\[
\mathcal{B}:=\bigsqcup_{K\in\Fc}\bigsqcup_{\left(A,\overline{\varphi}\right)\in\left[H\times_{\F}K\right]}\left\{ \pi\left(c_{\overline{\varphi^{-1}}}R_{\overline{\varphi}\left(A\right)}^{K}\right)\right\} .
\]
Where $\pi\colon\muFR\downarrow_{\FHH}^{\F}\to\muFR\downarrow_{\FHH}^{\F}/\I\downarrow_{\FHH}^{\F}$
denotes the natural projection. Using this we can now define for every
$K\in\Fc$, every $\left(A,\overline{\varphi}\right)\in\left[H\times_{\F}K\right]$
and every $x\in I_{A}^{A}M$ the $\FmuFR{\FHH}$-module morphism $f_{\left(A,\overline{\varphi}\right)}^{x}\in M\Uparrow_{\FHH}^{\F}$
that sends every element in $\mathcal{B}$ to $0$ except for $\pi\left(c_{\overline{\varphi^{-1}}}R_{\overline{\varphi}\left(A\right)}^{K}\right)$
which is sent to $x$. With this notation it can be proven that every
$f\in M\Uparrow_{\FHH}^{\F}$ can be written in a unique way as an
$\R$-linear combination of $\R$-module morphisms of the form $f_{\left(A,\overline{\varphi}\right)}^{x}$.
Finally an isomorphism from $M\Uparrow_{\FHH}^{\F}$ to $M\uparrow_{\FHH}^{\F}$
can be obtained from Proposition \ref{prop:Decomposition-centric-Mackey-functor.}
by sending any morphism of the form $f_{\left(A,\overline{\varphi}\right)}^{x}$
to $I_{\overline{\varphi}\left(A\right)}^{K}c_{\overline{\varphi}}\otimes x\in M\uparrow_{\FHH}^{\F}$.
Some care is needed in this last step to prove that this morphism
is in fact a morphism of $\muFR$-modules but Proposition \ref{prop:properties-HXK.}
and Lemma \ref{lem:proving-muFR-morphism.} can be used to this end.
\end{rem}

As we will see in subsection \ref{subsec:N-is-a-direct-summand-of-Ninduction-restriction.}
there are at least 2 ways of translating Equation \eqref{eq:Induction-restriction-groups.}
to the context of Mackey functors over fusion systems. We are now
ready to give the first one.
\begin{lem}
\label{lem:Mackey-formula-induction-restriction.}Let $H,K\in\Fc$,
let $\G$ be a fusion system containing $\F$ and let $M\in\MackFHR{\G}{\FHH}$,
for every $\left(A,\overline{\varphi}\right)\in\left[H\times_{\F}K\right]$
fix a representative $\varphi$ of $\overline{\varphi}$ viewed as
an isomorphism onto its image and define $M_{\left(A,\overline{\varphi}\right)}:=\left(\lui{\varphi}{\left(M\downarrow_{\FHH[A]}^{\FHH}\right)}\right)$.
Each $M_{\left(A,\overline{\varphi}\right)}$ is $\G$-centric and
there exists an isomorphism 
\begin{equation}
\begin{array}{ccc}
\underset{\left(A,\overline{\varphi}\right)\in\left[H\times_{\F}K\right]}{\bigoplus}M_{\left(A,\overline{\varphi}\right)}\text{\ensuremath{\uparrow}}_{\FHH[\varphi\left(A\right)]}^{\FHH[K]} & \stackrel{\Gamma}{\bjarrow} & M\uparrow_{\FHH}^{\F}\downarrow_{\FHH[K]}^{\F},\\
{\scriptstyle I_{\overline{\theta}\left(C\right)}^{J}c_{\overline{\theta}}\otimes_{\FmuFR{\FHH[\varphi\left(A\right)]}}x} & {\scriptstyle \bjarrow} & {\scriptstyle I_{\overline{\theta}\left(C\right)}^{J}c_{\overline{\theta\varphi}}\otimes_{\FmuFR{\FHH}}x}
\end{array}\label{eq:Definition-Gamma.}
\end{equation}
where we are viewing $\varphi$ as an isomorphism between the appropriate
restrictions and we are using Proposition \ref{prop:Decomposition-centric-Mackey-functor.}
and the fact that $M_{\left(A,\overline{\varphi}\right)}\in\MackFHR{\G}{\FHH[\varphi\left(A\right)]}$
to define $\Gamma$ via $\R$ linearity by setting its image on elements
of the form $I_{\overline{\theta}\left(C\right)}^{J}c_{\overline{\theta}}\otimes_{\FmuFR{\FHH[\varphi\left(A\right)]}}x\in M_{\left(A,\overline{\varphi}\right)}\text{\ensuremath{\uparrow}}_{\FHH[\varphi\left(A\right)]}^{\FHH[K]}$
with $J\in\FHH[K]\cap\Fc$, $\left(C,\overline{\theta}\right)\in\left[\varphi\left(A\right)\times_{\FHH[K]}J\right]$
such that $C\in\FHH[K]\cap\Fc$ and $x\in I_{\varphi^{-1}\left(C\right)}^{\varphi^{-1}\left(C\right)}M=I_{C}^{C}M_{\left(A,\overline{\varphi}\right)}$.
\end{lem}

\begin{proof}
The fact that each $M_{\left(A,\overline{\varphi}\right)}$ is $\G$-centric
follows from their definition and Proposition \ref{prop:centric-Induction-restriction.}.

From Propositions \ref{prop:properties-HXK.} and \ref{prop:Decomposition-centric-Mackey-functor.}
we have the following isomorphism of $\R$-modules
\begin{align*}
M\uparrow_{\FHH}^{\F}\downarrow_{\FHH[K]}^{\F} & \cong_{\R}\hspace{-10bp}\bigoplus_{J\in\FHH[K]\cap\Fc}\bigoplus_{\left(B,\overline{\psi}\right)\in\left[H\times_{\F}J\right]}I_{\overline{\psi}\left(B\right)}^{J}c_{\overline{\psi}}\otimes_{\FmuFR{\FHH}}M, & \text{Proposition \ref{prop:Decomposition-centric-Mackey-functor.}}\\
 & \cong_{\R}\hspace{-10bp}\hspace{-10bp}\bigoplus_{\begin{array}{c}
{\scriptstyle J\in\FHH[K]\cap\Fc}\\
{\scriptstyle \left(A,\overline{\varphi}\right)\in\left[H\times_{\F}K\right]}
\end{array}}\hspace{-10bp}\bigoplus_{\begin{array}{c}
{\scriptstyle x\in\left[J\backslash K/\varphi\left(A\right)\right]}\\
{\scriptstyle J^{x}\cap\varphi\left(A\right)\in\Fc}
\end{array}}\hspace{-10bp}I_{J\cap\lui{x}{\left(\varphi\left(A\right)\right)}}^{J}c_{c_{x}\varphi}\otimes_{\FmuFR{\FHH}}M, & \text{Proposition \ref{prop:properties-HXK.} \eqref{enu:prod-pullback-right.}}\\
 & \cong_{\R}\hspace{-10bp}\hspace{-10bp}\bigoplus_{\begin{array}{c}
{\scriptstyle J\in\FHH[K]\cap\Fc}\\
{\scriptstyle \left(A,\overline{\varphi}\right)\in\left[H\times_{\F}K\right]}
\end{array}}\hspace{-10bp}\bigoplus_{\begin{array}{c}
{\scriptstyle \left(C,\overline{\theta}\right)\in\left[\varphi\left(A\right)\times_{\FHH[K]}J\right]}\\
{\scriptstyle C\in\FHH[K]\cap\Fc}
\end{array}}\hspace{-10bp}\hspace{-10bp}I_{\overline{\theta}\left(C\right)}^{J}c_{\overline{\theta\varphi}}\otimes_{\FmuFR{\FHH}}M. & \text{Proposition \ref{prop:properties-HXK.} \eqref{enu:prod-if-F-is-F_S(S).}}
\end{align*}
Where each $\varphi$ is as in the statement. From Proposition \ref{prop:centric-Induction-restriction.}
we know that $M_{\left(A,\overline{\varphi}\right)}\uparrow_{\FHH[\varphi\left(A\right)]}^{\F}$
is $\F$-centric and, therefore, $I_{J}^{J}M_{\left(A,\overline{\varphi}\right)}\uparrow_{\FHH[\varphi\left(A\right)]}^{\FHH[K]}=0$
for every $J\in\FHH[K]\backslash\left(\FHH[K]\cap\Fc\right)$. The
same argument also tells us that $I_{C}^{C}M_{\left(A,\overline{\varphi}\right)}=0$
for every $C\in\FHH[\varphi\left(A\right)]\backslash\left(\FHH[\varphi\left(A\right)]\cap\Fc\right)$.
We can therefore use Proposition \ref{prop:Decomposition-centric-Mackey-functor.}
in order to conclude that
\begin{align*}
M_{\left(A,\overline{\varphi}\right)}\text{\ensuremath{\uparrow}}_{\FHH[\varphi\left(A\right)]}^{\FHH[K]} & \cong_{\R}\bigoplus_{J\in\FHH[K]\cap\Fc}\bigoplus_{\begin{array}{c}
{\scriptstyle \left(C,\overline{\theta}\right)\in\left[\varphi\left(A\right)\times_{\FHH[K]}J\right]}\\
{\scriptstyle C\in\FHH[K]\cap\Fc}
\end{array}}I_{\overline{\theta}\left(C\right)}^{J}c_{\overline{\theta}}\otimes_{\FmuFR{\FHH[\varphi\left(A\right)]}}M_{\left(A,\overline{\varphi}\right)}.
\end{align*}
By definition of the functor $\lui{\varphi}{\cdot}$ (see Definition
\ref{def:Restriction-induction-and-conjugation-functors.}) we now
have that for every $J\in\FHH[K]\cap\Fc$ and every $\left(C,\overline{\theta}\right)\in\left[\varphi\left(A\right)\times_{\FHH[K]}J\right]$
such that $C\in\FHH[K]\cap\Fc$ there is an equivalence of $\R$-modules
$I_{C}^{C}M_{\left(A,\overline{\varphi}\right)}\cong_{\R}I_{\varphi^{-1}\left(C\right)}^{\varphi^{-1}\left(C\right)}M$
realized by sending every $x\in I_{C}^{C}M_{\left(A,\overline{\varphi}\right)}$
to $x$ seen as an element in $I_{\varphi^{-1}\left(C\right)}^{\varphi^{-1}\left(C\right)}M$.
This leads in turn to an equivalence of $\R$-modules $I_{\overline{\theta}\left(C\right)}^{J}c_{\overline{\theta}}\otimes_{\FmuFR{\FHH[\varphi\left(A\right)]}}M_{\left(A,\overline{\varphi}\right)}\cong_{\R}I_{\overline{\theta}\left(C\right)}^{J}c_{\overline{\theta\varphi}}\otimes_{\FmuFR{\FHH}}M$
realized by sending every $I_{\overline{\theta}\left(C\right)}^{J}c_{\overline{\theta}}\otimes_{\FmuFR{\FHH[\varphi\left(A\right)]}}x\in I_{\overline{\theta}\left(C\right)}^{J}c_{\overline{\theta}}\otimes_{\FmuFR{\FHH[\varphi\left(A\right)]}}M_{\left(A,\overline{\varphi}\right)}$
to $I_{\overline{\theta}\left(C\right)}^{J}c_{\overline{\theta\varphi}}\otimes_{\FmuFR{\FHH}}x\in I_{\overline{\theta}\left(C\right)}^{J}c_{\overline{\theta\varphi}}\otimes_{\FmuFR{\FHH}}M$.
Therefore, viewing each $I_{\overline{\theta}\left(C\right)}^{J}c_{\overline{\theta}}\otimes_{\FmuFR{\FHH[\varphi\left(A\right)]}}M_{\left(A,\overline{\varphi}\right)}$
as an $\R$-submodule of $M_{\left(A,\overline{\varphi}\right)}\text{\ensuremath{\uparrow}}_{\FHH[\varphi\left(A\right)]}^{\FHH[K]}$
and each $I_{\overline{\theta}\left(C\right)}^{J}c_{\overline{\theta\varphi}}\otimes_{\FmuFR{\FHH}}M$
as an $\R$-submodule of $M\uparrow_{\FHH}^{\F}\downarrow_{\FHH[K]}^{\F}$,
we can conclude that the morphism $\Gamma$ of the statement is a
bijective $\R$-module morphism. We are now just left with proving
that $\Gamma$ is a morphism of $\FmuFR{\FHH[K]}$-modules. Take $\left(A,\overline{\varphi}\right),J,\left(C,\overline{\theta}\right)$
and $x$ as in the statement and let $J'\in\FHH[K]\cap\Fc$ such that
$J'\le J$. Then we have that
\begin{align*}
 & R_{J'}^{J}\Gamma\left(I_{\overline{\theta}\left(C\right)}^{J}c_{\overline{\theta}}\otimes_{\FmuFR{\FHH[\varphi\left(A\right)]}}x\right)\\
 & \phantom{spa}=\sum_{\left(B,\overline{\psi}\right)}I_{\overline{\psi}\left(B\right)}^{J}c_{\overline{\psi}}c_{\overline{\left(\gamma_{\left(B,\overline{\psi}\right)}^{J}\right)^{-1}}}R_{\overline{\gamma_{\left(B,\overline{\psi}\right)}^{J}}\left(B\right)}^{C}c_{\varphi}\otimes_{\FmuFR{\FHH}}x,\\
 & \phantom{spa}=\sum_{\left(B,\overline{\psi}\right)}I_{\overline{\psi}\left(B\right)}^{J}c_{\overline{\psi}}c_{\varphi}\otimes_{\FmuFR{\FHH}}c_{\varphi^{-1}}c_{\overline{\left(\gamma_{\left(B,\overline{\psi}\right)}^{J}\right)^{-1}}}R_{\overline{\gamma_{\left(B,\overline{\psi}\right)}^{J}}\left(B\right)}^{C}c_{\varphi}\cdot x,\\
 & \phantom{spa}=\sum_{\left(B,\overline{\psi}\right)}\Gamma\left(I_{\overline{\psi}\left(B\right)}^{J}c_{\overline{\psi}}\otimes_{\FmuFR{\FHH[\varphi\left(A\right)]}}c_{\overline{\left(\gamma_{\left(B,\overline{\psi}\right)}^{J}\right)^{-1}}}R_{\overline{\gamma_{\left(B,\overline{\psi}\right)}^{J}}\left(B\right)}^{C}\cdot x\right),\\
 & \phantom{spa}=\Gamma\left(R_{J'}^{J}I_{\overline{\theta}\left(C\right)}^{J}c_{\overline{\theta}}\otimes_{\FmuFR{\FHH[\varphi\left(A\right)]}}x\right).
\end{align*}
Where the $\left(B,\overline{\psi}\right)$ are iterating over the
elements in $\left[H\times_{\FHH[K]}J\right]$ such that $\left(B^{J},\overline{\psi^{J}}\right)=\left(C,\overline{\theta}\right)$,
we are taking $\varphi$ as in the statement and in the first and
second identities we are using Items \eqref{enu:using-mackey.} and
\eqref{enu:Def-A^theta.} of Lemma \ref{lem:proving-muFR-morphism.}
respectively, in the third identity we are using the definition of
$M_{\left(A,\overline{\varphi}\right)}$ and in the last identity
we are repeating the same operations backwards.

Let now $J'\in\FHH[K]\cap\Fc$ such that $J'\ge J$ and let $\rho\colon J\to\rho\left(J\right)$
be an isomorphism in $\FHH[K]$. The same arguments used above but
now replacing Item \eqref{enu:using-mackey.} of Lemma \ref{lem:proving-muFR-morphism.}
with Items \eqref{enu:Composing.} and \eqref{enu:Conjugating.} respectively
(which remove the sum thus making the operations simpler to carry)
we obtain the identities below
\begin{align*}
I_{J}^{J'}\Gamma\left(I_{\overline{\theta}\left(C\right)}^{J}c_{\overline{\theta}}\otimes_{\FmuFR{\FHH[\varphi\left(A\right)]}}x\right) & =\Gamma\left(I_{J}^{J'}I_{\overline{\theta}\left(C\right)}^{J}c_{\overline{\theta}}\otimes_{\FmuFR{\FHH[\varphi\left(A\right)]}}x\right),\\
c_{\rho}\Gamma\left(I_{\overline{\theta}\left(C\right)}^{J}c_{\overline{\theta}}\otimes_{\FmuFR{\FHH[\varphi\left(A\right)]}}x\right) & =\Gamma\left(c_{\rho}I_{\overline{\theta}\left(C\right)}^{J}c_{\overline{\theta}}\otimes_{\FmuFR{\FHH[\varphi\left(A\right)]}}x\right).
\end{align*}

This proves that $\Gamma$ is indeed an $\FmuFR{\FHH[K]}$-module
morphism thus concluding the proof.
\end{proof}
Using Proposition \ref{prop:Decomposition-centric-Mackey-functor.}
we can now define a morphism $\theta^{H}$ from a centric Mackey functor
$M$ over $\F$ to the centric Mackey functor $M\downarrow_{\FHH}^{\F}\uparrow_{\FHH}^{\F}$
by setting for every $K\in\Fc$ and every $x\in I_{K}^{K}M$
\[
\theta_{M}^{H}\left(x\right):=\sum_{\left(A,\overline{\varphi}\right)\in\left[H\times_{\F}K\right]}I_{\overline{\varphi}\left(A\right)}^{K}c_{\overline{\varphi}}\otimes c_{\overline{\varphi^{-1}}}R_{\overline{\varphi}\left(A\right)}^{K}x.
\]
Since the tensor product is over $\FmuFR{\FHH}$ we know that $\theta^{H}$
does not depend on the choice of $\left[H\times_{\F}K\right]$. Thus
we can conclude that it is well defined and an $\R$-module morphism.
Let $K\in\Fc$, let $x\in I_{K}^{K}M$ and let $\rho\colon K\to\rho\left(K\right)$
be an isomorphism in $\F$. Applying Items \eqref{enu:Def-A^theta.}
and \eqref{enu:Conjugating.} of Lemma \ref{lem:proving-muFR-morphism.}
we have that 
\begin{align*}
c_{\rho}\theta_{M}^{H}\left(x\right) & =\sum_{\left(B,\overline{\psi}\right)\in\left[H\times_{\F}K\right]}I_{\overline{\psi^{\rho}}\left(B\right)}^{K}c_{\overline{\psi^{\rho}}}c_{\gamma_{\left(B,\overline{\psi}\right)}^{\rho}}\otimes c_{\overline{\psi^{-1}}}R_{\overline{\psi}\left(B\right)}^{K}x,\\
 & =\sum_{\left(B,\overline{\psi}\right)\in\left[H\times_{\F}K\right]}I_{\overline{\psi^{\rho}}\left(B\right)}^{K}c_{\overline{\psi^{\rho}}}\otimes c_{\left(\gamma_{\left(B,\overline{\psi}\right)}^{\rho^{-1}}\right)^{-1}}c_{\overline{\psi^{-1}}}R_{\overline{\psi}\left(B\right)}^{K}x,\\
 & =\sum_{\left(A,\overline{\varphi}\right)\in\left[H\times_{\F}\rho\left(K\right)\right]}I_{\overline{\varphi}\left(A\right)}^{\rho\left(K\right)}c_{\overline{\varphi}}\otimes c_{\overline{\varphi^{-1}}}R_{\overline{\varphi}\left(A\right)}^{K}c_{\rho}x=\theta_{M}^{H}\left(c_{\rho}x\right).
\end{align*}
With the same notation as above let $J\in\Fc$ such that $J\ge K$
then we have that
\begin{align*}
\theta_{M}^{H}\left(I_{K}^{J}x\right) & =\sum_{\left(A,\overline{\varphi}\right)\in\left[H\times_{\F}K\right]}I_{\overline{\varphi^{J}}\left(A^{J}\right)}^{J}c_{\overline{\varphi^{J}}}\otimes I_{\overline{\gamma_{\left(A,\overline{\varphi}\right)}^{J}}\left(A\right)}^{A^{J}}c_{\overline{\gamma_{\left(A,\overline{\varphi}\right)}^{J}}}c_{\overline{\varphi^{-1}}}R_{\overline{\varphi}\left(A\right)}^{K}x,\\
 & =\sum_{\left(A,\overline{\varphi}\right)\in\left[H\times_{\F}K\right]}I_{K}^{J}I_{\overline{\varphi}\left(A\right)}^{J}c_{\overline{\varphi}}\otimes c_{\overline{\varphi^{-1}}}R_{\overline{\varphi}\left(A\right)}^{K}x=I_{K}^{J}\theta_{M}^{H}\left(x\right).
\end{align*}
Where, in the first identity, we are using Lemma \ref{lem:proving-muFR-morphism.}
\eqref{enu:using-mackey.} together with the fact that $M$ is $\G$-centric
and, therefore, annihilated by the ideal $\I$ of Lemma \ref{lem:proving-muFR-morphism.}
and, in the second identity, we are using Lemma \ref{lem:proving-muFR-morphism.}
\eqref{enu:Def-A^theta.} to move things from one side of the tensor
product to the other and Lemma \ref{lem:proving-muFR-morphism.} \eqref{enu:Composing.}
to simplify the equation. If $J\in\Fc$ is such that $J\le K$ then
the exact same arguments (but starting with $R_{J}^{K}\theta_{M}^{H}\left(x\right)$
instead of $\theta_{M}^{H}\left(R_{J}^{K}x\right)$) prove that $\theta_{M}^{H}$
also commutes with restriction. We can therefore conclude that $\theta_{M}^{H}$
is a morphism of $\muFR$-modules for every $M\in\MackFcR$. This
allows us to give the following definition with which we conclude
this subsection.
\begin{defn}
\label{def:theta_S-and-theta^S.}Let $\G$ be a fusion system containing
$\F$, let $M\in\MackFHR{\G}{\F}$ and let $H\in\Fc$. From Proposition
\ref{prop:centric-Induction-restriction.} we know that the following
is a $\G$-centric Mackey functor over $\F$
\begin{align*}
M_{H}:= & M\downarrow_{\FHH}^{\F}\uparrow_{\FHH}^{\F}.
\end{align*}

Thus the above discussion allows us to define the Mackey functor morphisms
\begin{align*}
\theta_{H}^{M}: & M_{H}\to M, & \theta_{M}^{H}: & M\to M_{H},
\end{align*}
by setting for every $y\otimes x\in M_{H}$, every $K\in\Fc$ and
every $z\in I_{K}^{K}M$
\begin{align*}
\theta_{H}^{M}\left(y\otimes x\right) & :=y\cdot x, & \theta_{M}^{H}\left(z\right) & :=\sum_{\left(A,\overline{\varphi}\right)\in\left[H\times_{\F}K\right]}I_{\overline{\varphi}\left(A\right)}^{K}c_{\overline{\varphi}}\otimes c_{\overline{\varphi^{-1}}}R_{\overline{\varphi}\left(A\right)}^{K}z.
\end{align*}
If there is no possible confusion regarding $M$ we will write $\theta_{H}:=\theta_{H}^{M}$
and $\theta^{H}:=\theta_{M}^{H}$.
\end{defn}

\subsection{\label{subsec:The-F-centric-Burnside-ring.}The centric Burnside
ring over a fusion system.}

Let $G$ be a finite group. It is known (see \cite[Proposition 9.2]{StructureMackeyFunctors})
that the Burnside ring of $G$ can be embedded in the center of the
Mackey algebra of $G$. In this subsection we will prove that there
exists a similar embedding of the centric Burnside ring of $\F$ into
the center of a certain quotient of $\muFR$ (see Proposition \ref{prop:Action-of-centric-burnside-ring.}).

Let us start by recalling the definition of centric Burnside ring
of a fusion system.
\begin{defn}
\label{def:centric-Burnside-ring.}(\cite[Definition 2.11]{BurnsideRingFusionSystemsDiazLibman})
The \textbf{centric Burnside ring of $\F$ }(denoted by \textbf{$\BFcR[]$})
is the Grothendieck group of the semigroup whose elements are isomorphism
classes of $\aOFc$ and addition is given by taking the isomorphism
class of the coproduct of two representatives. This is doted with
a ring structure by taking multiplication of two isomorphism classes
to be the isomorphism class of the product of two of their representatives
and extending by linearity. Given a commutative ring $\R$ we also
define the \textbf{centric Burnside ring of $\F$ on $\R$} as
\[
\BFcR:=\R\otimes_{\Z}\BFcR[].
\]
\end{defn}

An important distinction between the ring $\BFcR$ and the Burnside
ring of a group is that, in general, the isomorphism class $\overline{S}$
of $S$ is not the identity in $\BFcR$. However, we have the following
result due to Sune Reeh.
\begin{prop}
\label{prop:Inverse-of-S.}If every integer prime other than $p$
is invertible in $\R$ then the isomorphism class $\overline{S}$
of $S$ is invertible in $\BFcR$.
\end{prop}

\begin{proof}
See \cite[Proposition 4.13]{ReehTransferCharacteristicIdempotentFS}.
\end{proof}
This result motivates the following definition.
\begin{defn}
\label{def:p-local}We say that a ring $\R$ is $p$\textbf{-local
}if all integer primes other than $p$ are invertible in $\R$.
\end{defn}

\begin{rem}
The definition of $p$-local ring does not specify if $p$ is invertible
or not. This distinction will not be relevant towards the results
shown in this paper. It is however worth noting that, if $\R$ is
a field of characteristic $0$, then arguments analogous to those
of \cite[Theorem 9.1]{ThevenazWebb1989SimpleMackeyFunctors} can be
used in order to prove that $\muFR$ is a semisimple $\R$-algebra.
The exact condition is in fact for $\R$ to be a field where $\left|\Aut_{\F}\left(H\right)\right|$
is invertible for every $H\le S$.
\end{rem}

Before proceeding let us recall exactly how the Burnside ring of a
finite group $G$ embeds into the center of the Mackey algebra. Let
$G$ be a finite group and let $\R$ be a commutative ring, \cite[Proposition 9.2]{StructureMackeyFunctors}
describes the above mentioned embedding as the map that, for every
$H\le G$, sends the isomorphism class $\overline{G/H}$ of the transitive
$G$-set $G/H$ to
\[
\overline{G/H}\to\sum_{K\le G}\sum_{x\in\left[K\backslash G/H\right]}I_{K\cap\lui{x}{H}}^{K}R_{K\cap\lui{x}{H}}^{K}\in Z\left(\mu_{\R}\left(G\right)\right).
\]
This embedding leads to an action of the Burnside ring of $G$ on
any Mackey functor over $G$. When trying to obtain a similar result
for the case of Mackey functors over fusion systems many difficulties
arise. These can, once again, be traced back to the fact that the
category $\aOF$ does not in general admit products. However, we have
the following results with which we conclude this section.
\begin{lem}
\label{lem:Image-into-center.}Let $\I$ be the two sided ideal of
$\muFR$ defined in Proposition \ref{prop:Decomposition-centric-Mackey-functor.},
define $\muFcR:=\muFR/\I$ and denote by $\pi\colon\muFR\twoheadrightarrow\muFcR$
the natural projection. The $\R$-algebra $\muFcR$ is naturally equipped
with a $\muFR$-module structure given by setting $y\cdot\pi\left(x\right)=\pi\left(y\right)\pi\left(x\right)$
for every $x,y\in\muFR$. Moreover, for every $H\in\F\backslash\Fc$
we have that $I_{H}^{H}\in\I$ and, therefore, $I_{H}^{H}\muFcR=0$.
Thus we can view $\muFcR$ as a centric Mackey functor over $\F$
and, using Definition \ref{def:theta_S-and-theta^S.}, we can define
\[
\Gamma\left(H\right):=\theta_{H}^{\muFcR}\left(\theta_{\muFcR}^{H}\left(1_{\muFcR}\right)\right)\in\muFcR
\]
for every $H\in\Fc$. With this setup we have that:
\begin{enumerate}
\item \label{enu:reinterpret-gamma(H).}$\Gamma\left(H\right)$ is in the
center of the $\R$-algebra $\muFcR$ and
\[
\Gamma\left(H\right)=\sum_{J\in\Fc}\sum_{\left(A,\overline{\varphi}\right)\in\left[J\times H\right]}\pi\left(I_{A}^{J}R_{A}^{J}\right).
\]
\item \label{enu:Gamma(H)=00003DGamma(H').}For every $H'=_{\F}H$ (see
Notation \ref{nota:p,S,F}) we have that $\Gamma\left(H'\right)=\Gamma\left(H\right)$.
\item \label{enu:Product-Gamma-H-Gamma-H'.}For every $K\in\Fc$ we have
that 
\[
\Gamma\left(K\right)\Gamma\left(H\right)=\sum_{\left(A,\overline{\varphi}\right)\in\left[K\times H\right]}\Gamma\left(A\right).
\]
\end{enumerate}
\end{lem}

\begin{proof}
$\phantom{.}$
\begin{enumerate}
\item For every $H\in\Fc$, we have that
\[
\Gamma\left(H\right)=\sum_{K\in\Fc}\sum_{\left(A,\overline{\varphi}\right)\in\left[H\times K\right]}\pi\left(I_{\varphi\left(A\right)}^{K}c_{\varphi}c_{\varphi^{-1}}R_{\varphi\left(A\right)}^{K}\right)=\sum_{K\in\Fc}\sum_{\left(B,\overline{\psi}\right)\in\left[K\times H\right]}\pi\left(I_{B}^{K}R_{B}^{K}\right).
\]
Where, for the first identity, we are using the fact that $1_{\muFcR}=\sum_{K\in\Fc}\pi\left(I_{K}^{K}\right)$
and Corollary \ref{cor:conjugation-in-F-and-in-orbit-F.} in order
to take $\varphi$ to be any representative of $\overline{\varphi}$
and view it as an isomorphism onto its image. For the second identity
we are using Items \eqref{enu:property-I_H^H-is-identity.} and \eqref{enu:property-composition-is-nice.}
of Lemma \ref{lem:Many-properties-definition.} in order to remove
$c_{\varphi}c_{\varphi^{-1}}$ and Proposition \ref{prop:properties-HXK.}
\eqref{enu:HX_FK-and-KX_FH} in order to rewrite the sum. This proves
the second part of Item \eqref{enu:reinterpret-gamma(H).}. For the
first part recall from Definition \ref{def:theta_S-and-theta^S.}
that both $\theta_{\muFcR}^{H}$ and $\theta_{H}^{\muFcR}$ are morphisms
of Mackey functors and, since the projection $\pi$ is an $\R$-algebra
morphism, we also know that $\pi\left(1_{\muFR}\right)=1_{\muFcR}$.
We can therefore conclude that, for every $x\in\muFR$, we have
\[
\pi\left(x\right)\Gamma\left(H\right)=x\cdot\theta_{H}^{\muFcR}\left(\theta_{\muFcR}^{H}\left(\pi\left(1_{\muFR}\right)\right)\right)=\theta_{H}^{\muFcR}\left(\theta_{\muFcR}^{H}\left(\pi\left(x\right)\right)\right).
\]
The fact that $\Gamma\left(H\right)$ is in the center of $\muFcR$
now follows from definition of $\theta_{\muFcR}^{H}$ and $\theta_{H}^{\muFcR}$
via the identities below
\begin{align*}
\Gamma\left(H\right)\pi\left(x\right) & =\left(\sum_{K\in\Fc}\sum_{\left(A,\overline{\varphi}\right)\in\left[H\times K\right]}I_{\varphi\left(A\right)}^{K}c_{\varphi}c_{\varphi^{-1}}R_{\varphi\left(A\right)}^{K}\cdot\pi\left(I_{K}^{K}\right)\right)\pi\left(x\right),\\
 & =\left(\sum_{K\in\Fc}\sum_{\left(A,\overline{\varphi}\right)\in\left[H\times K\right]}I_{\varphi\left(A\right)}^{K}c_{\varphi}c_{\varphi^{-1}}R_{\varphi\left(A\right)}^{K}\cdot\pi\left(I_{K}^{K}x\right)\right),\\
 & =\theta_{H}^{\muFcR}\left(\theta_{\muFcR}^{H}\left(\pi\left(x\right)\right)\right)=\pi\left(x\right)\Gamma\left(H\right).
\end{align*}
\item Let $\psi\colon H\bjarrow H'$ be an isomorphism in $\F$. Item \eqref{enu:Gamma(H)=00003DGamma(H').}
follows from Item \eqref{enu:reinterpret-gamma(H).} and Proposition
\ref{prop:properties-HXK.} \eqref{enu:prod-pullback-right.} via
the identities below
\[
\Gamma\left(H\right)=\sum_{K\in\Fc}\sum_{\left(A,\overline{\varphi}\right)\in\left[K\times H\right]}\pi\left(I_{A}^{K}R_{A}^{K}\right)=\sum_{K\in\Fc}\sum_{\left(A,\overline{\psi\varphi}\right)\in\left[K\times H'\right]}\pi\left(I_{A}^{K}R_{A}^{K}\right)=\Gamma\left(H'\right).
\]
\item Item \eqref{enu:Product-Gamma-H-Gamma-H'.} follows from the identities
below
\begin{align*}
\Gamma\left(K\right)\Gamma\left(H\right) & =\sum_{J\in\Fc}\sum_{\begin{array}{c}
{\scriptstyle \left(A,\overline{\varphi}\right)\in\left[J\times H\right]}\\
{\scriptstyle \left(B,\overline{\psi}\right)\in\left[J\times K\right]}
\end{array}}\pi\left(I_{B}^{J}R_{B}^{J}I_{A}^{J}R_{A}^{J}\right),\\
 & =\sum_{J\in\Fc}\sum_{\begin{array}{c}
{\scriptstyle \left(A,\overline{\varphi}\right)\in\left[J\times H\right]}\\
{\scriptstyle \left(B,\overline{\psi}\right)\in\left[J\times K\right]}
\end{array}}\sum_{\begin{array}{c}
{\scriptstyle x\in\left[B\backslash J/A\right]}\\
{\scriptstyle B^{x}\cap A\in\Fc}
\end{array}}\pi\left(I_{B^{x}\cap A}^{J}R_{B^{x}\cap A}^{J}\right),\\
 & =\sum_{\left(C,\overline{\theta}\right)\in\left[H\times K\right]}\sum_{J\in\Fc}\sum_{\left(D,\overline{\gamma}\right)\in\left[J\times C\right]}\pi\left(I_{D}^{J}R_{D}^{J}\right)=\sum_{\left(C,\overline{\theta}\right)\in\left[H\times K\right]}\Gamma\left(C\right).
\end{align*}
Where, for the first identity, we are using the fact that $\pi$ is
a morphism of $\R$-algebras and Lemma \ref{lem:Many-properties-definition.}
\eqref{enu:property-all-0.}, for the second identity, we are using
Lemma \ref{lem:Many-properties-definition.} \eqref{enu:property-Mackey-formula.}
and definition of $\I$ and, for the third identity, we are using
Proposition \ref{prop:properties-HXK.} \eqref{enu:triple-prod-into-pullback-prod.}.
\end{enumerate}
\end{proof}
\begin{prop}
\label{prop:Action-of-centric-burnside-ring.}Let $p$ be a prime,
let $S$ be a finite $p$-group, let $\F$ be a saturated fusion system
over $S$ (see Notation \ref{nota:p,S,F}), let $\R$ be a commutative
ring with unit (see Notation \ref{nota:Induction-and-restriction-functors.})
and let $\I,\muFcR$ and $\Gamma$ be as in Lemma \ref{lem:Image-into-center.}.
For every $X\in\aOFc$ (see Definition \ref{def:Additive-extension})
denote by $\overline{X}\in\BFcR$ (see Definition \ref{def:centric-Burnside-ring.})
its isomorphism class and define the (non necessarily unit preserving)
$\R$-algebra morphism $\overline{\Gamma}:\BFcR\to\muFcR$ by setting
$\overline{\Gamma}\left(\overline{H}\right):=\Gamma\left(H\right)$
for every $H\in\Fc$ and extending by $\R$-linearity. If $\BFcR$
contains a non-zero divisor then $\overline{\Gamma}$ is injective
and, if $\R$ is $p$-local (see Definition \ref{def:p-local}), then
$\BFcR$ contains a unit (see Proposition \ref{prop:Inverse-of-S.})
and $\overline{\Gamma}\left(1_{\BFcR}\right)=1_{\muFcR}$. Moreover,
if $\R$ is $p$-local, then, for every fusion system $\G$ containing
$\F$ and every $M\in\MackFHR{\G}{\F}\subseteq\MackFcR$ (see Definition
\ref{def:F-centric-Mackey-functor.}), the ring $\BFcR$ acts on $M$
by using the morphisms of Definition \ref{def:theta_S-and-theta^S.}
to define for every $H\in\Fc$
\[
\overline{H}\cdot=\theta_{H}^{M}\theta_{M}^{H}\in\End\left(M\right).
\]
\end{prop}

\begin{proof}
From Items \eqref{enu:Gamma(H)=00003DGamma(H').} and \eqref{enu:Product-Gamma-H-Gamma-H'.}
of Lemma \ref{lem:Image-into-center.} we know that $\overline{\Gamma}$
is a well defined (non necessarily unit preserving) $\R$-algebra
morphism.

Viewing $\FmuFR{\FHH[S]}$ as a subset of $\muFR$ (see Corollary
\ref{cor:Mackey-algebra-inclusion.}) we can define $\varUpsilon\colon\pi\left(I_{S}^{S}\FmuFR{\FHH[S]}I_{S}^{S}\right)\to\End\left(\BFcR\right)$
by setting for every $\pi\left(I_{K}^{S}R_{K}^{S}\right)\in\pi\left(I_{S}^{S}\FmuFR{\FHH[S]}I_{S}^{S}\right)$
and every $H\in\Fc$
\[
\varUpsilon\left(\pi\left(I_{K}^{S}R_{K}^{S}\right)\right)\left(\overline{H}\right):=\sum_{\begin{array}{c}
{\scriptstyle x\in\left[K\backslash S/H\right]}\\
{\scriptstyle K^{x}\cap H\in\Fc}
\end{array}}\overline{K^{x}\cap H},
\]
From Items \eqref{enu:property-I_H^H-is-identity.} and \eqref{enu:Property-conjugation-commutes.}
of Lemma \ref{lem:Many-properties-definition.} and Proposition \ref{prop:Mackey-algebra-basis.}
we know that this is sufficient to define $\varUpsilon$ via $\R$-linearity.
From Lemma \ref{lem:Image-into-center.} \eqref{enu:reinterpret-gamma(H).}
and definition of $\overline{\Gamma}$ we also have that $\overline{\Gamma}\left(\BFcR\right)\subseteq\pi\left(\FmuFR{\FHH[S]}\right)$.
Therefore we can define $\varUpsilon'\colon\overline{\Gamma}\left(\BFcR\right)\to\End\left(\BFcR\right)$
by setting $\varUpsilon'\left(x\right)=\varUpsilon\left(\pi\left(I_{S}^{S}\right)x\pi\left(I_{S}^{S}\right)\right)$
for every $x\in\BFcR$. With this setup we can conclude from Proposition
\ref{prop:properties-HXK.} \eqref{enu:prod-pullback-right.} and
Lemma \ref{lem:Image-into-center.} \eqref{enu:reinterpret-gamma(H).}
that $\varUpsilon'\left(\Gamma\left(\overline{H}\right)\right)\left(\overline{K}\right)=\overline{H\times K}$
for every $H,K\in\Fc$. Assume now that $\BFcR$ admits a non zero
divisor $\overline{\Omega}$. Then, for every $\overline{\Psi},\overline{\Phi}\in\BFcR$,
we have that
\[
\varUpsilon'\left(\Gamma\left(\overline{\Psi}\right)\right)\left(\overline{\Omega}\right)=\varUpsilon'\left(\Gamma\left(\overline{\Phi}\right)\right)\left(\overline{\Omega}\right)\Rightarrow\overline{\Psi\times\Omega}=\overline{\Phi\times\Omega}\Rightarrow\overline{\Psi}=\overline{\Phi}.
\]

This proves that the composition $\varUpsilon'\Gamma$ is injective
and, in particular, that $\Gamma$ is injective.

Assume now that $\R$ is $p$-local. By Proposition \ref{prop:Inverse-of-S.}
we know that $\BFcR$ admits a unit. Let us denote by $1_{\BFcR}=\sum_{K\in\Fc}\lambda_{K}\overline{K}$
this unit. From Lemma \ref{lem:Image-into-center.} \eqref{enu:reinterpret-gamma(H).}
and definition of product in $\BFcR$ we have that
\[
\overline{\Gamma}\left(1_{\BFcR}\right)=\sum_{H,K\in\Fc}\sum_{A\in\left[H\times K\right]}\lambda_{K}\pi\left(I_{A}^{H}R_{A}^{H}\right)=\sum_{H\in\Fc}\pi\left(I_{H}^{H}\right)=1_{\muFcR}.
\]
Finally, for every $M\in\MackFHR{\G}{\F}$, we have that $M\in\MackFcR$.
Therefore, by definition of $\I$, we have that $\I M=0$. In particular
$M$ acquires a $\muFcR$-module structure by setting $\pi\left(y\right)\cdot x=y\cdot x$
for every $y\in\muFR$ and every $x\in M$. This leads us to the equivalence
of $\R$-algebras $\End\left(M\right):=\End_{\muFR}\left(M\right)\cong\End_{\muFcR}\left(M\right)$.
Notice now that there exists a natural map $\Theta\colon Z\left(\muFcR\right)\to\End_{\muFcR}\left(M\right)$
defined by setting $\Theta\left(y\right)\left(x\right)=y\cdot x$
for every $y\in Z\left(\muFcR\right)$ and every $x\in M$. With this
notation we can define $\overline{\Omega}\cdot:=\Theta\left(\overline{\Gamma}\left(\overline{\Omega}\right)\right)\in\End\left(M\right)$
for every $\overline{\Omega}\in\BFcR$. Then, for every $H\in\Fc$
and every $x\in M$, we will have that
\begin{align*}
\overline{H}\cdot x & =\theta_{H}^{\muFcR}\left(\theta_{\muFcR}^{H}\left(1_{\muFcR}\right)\right)\cdot x,\\
 & =\sum_{K\in\Fc}\sum_{\left(A,\overline{\varphi}\right)\in\left[H\times K\right]}\pi\left(I_{\overline{\varphi}\left(A\right)}^{K}c_{\overline{\varphi}}c_{\overline{\varphi^{-1}}}R_{\overline{\varphi}\left(A\right)}^{K}\right)\cdot x,\\
 & =\sum_{K\in\Fc}\sum_{\left(A,\overline{\varphi}\right)\in\left[H\times K\right]}I_{\overline{\varphi}\left(A\right)}^{K}c_{\overline{\varphi}}c_{\overline{\varphi^{-1}}}R_{\overline{\varphi}\left(A\right)}^{K}\cdot x=\theta_{H}^{M}\left(\theta_{M}^{H}\left(x\right)\right).
\end{align*}
Where, in the last identity, we are using the fact that $M\in\MackFHR{\G}{\F}\subseteq\MackFcR$
and, in particular $\sum_{K\in\Fc}I_{K}^{K}\cdot x=x$. This concludes
the proof.
\end{proof}

\section{\label{sec:Relative-projectivity-and-the-Higman's-criterion.}Relative
projectivity and Higman's criterion.}

Let $G$ be a finite group and let $M$ be a Mackey functor over $G$
on $\R$. It is known (see \cite[Section 3]{GuideToMackeyFunctorsWebb})
that there exists a minimal family $\mathcal{X}_{M}$ of subgroups
of $G$ closed under $G$-subconjugacy such that $M$ is a direct
summand of $\bigoplus_{H\in\mathcal{X}_{M}}M\downarrow_{H}^{G}\uparrow_{H}^{G}$.
If $\R$ is a complete local $PID$ then the Krull-Schmidt-Azumaya
theorem (see \cite[Theorem 6.12 (ii)]{MethodsOfRepresentationTheoryCurtisReiner})
allows us to use this fact in order to obtain a decomposition of $M$
of the form $M\cong\bigoplus_{H\in\mathcal{X}_{M}}N^{H}$ where each
$N^{H}$ is a (possibly $0$) direct summand of $M\downarrow_{H}^{G}\uparrow_{H}^{G}$.
From this decomposition and minimality of $\mathcal{X}_{M}$ it can
be deduced that, if $M$ is indecomposable, then $\mathcal{X}_{M}$
is generated by a single element called vertex. This fact is essential
in order to describe the Green correspondence and, during this section,
we will prove that a similar process can be applied to centric Mackey
functors over fusion systems. Moreover we will prove that Higman's
criterion (see \cite[Theorem 2.2]{NagaoHirosiRepresentationsOfFiniteGroups})
can be translated to the context of centric Mackey functors (see Theorem
\ref{thm:Higman's-criterion.}). This will provide us with a link
between the vertex of an indecomposable $M\in\MackFcR$ and certain
ideals of $\End\left(M\right)$. Such link will turn out to be essential
towards proving the Green correspondence for centric Mackey functors.

\subsection{\label{subsec:The-defect-set.}The defect set.}

During this subsection we will extend the notion of relative projectivity
(see \cite[Section 3]{GuideToMackeyFunctorsWebb}) to centric Mackey
functors over a fusion system (see Definition \ref{def:Relative-projectivity.}).
We will also prove that, if $\R$ is $p$-local, the notions of defect
set and vertex (see \cite[ Section 3]{GuideToMackeyFunctorsWebb})
can be translated to the context of centric Mackey functors over fusion
systems (see Definition \ref{def:vertex.}).
\begin{defn}
\label{def:Relative-projectivity.}Let $\G$ be a fusion system containing
$\F$, let $M\in\MackFHR{\G}{\F}$ and let $\mathcal{X}$ be a family
of $\F$-centric subgroups of $S$. With notation as in Definition
\ref{def:theta_S-and-theta^S.} we define
\begin{align*}
M_{\mathcal{X}} & :=\bigoplus_{H\in\mathcal{X}}M_{H}, & \theta_{\mathcal{X}}^{M} & :=\sum_{H\in\mathcal{X}}\theta_{H}^{M}:M_{\mathcal{X}}\to M, & \theta_{M}^{\mathcal{X}} & :=\sum_{H\in\mathcal{X}}\theta_{M}^{H}:M\to M_{\mathcal{X}}.
\end{align*}
If there is no possible confusion regarding $M$ we will write $\theta_{\mathcal{X}}:=\theta_{\mathcal{X}}^{M}$
and $\theta^{\mathcal{X}}:=\theta_{M}^{\mathcal{X}}$. We say that
$M$ is \textbf{projective relative to $\mathcal{X}$ }(or \textbf{$\mathcal{X}$-projective})
if $\theta_{\mathcal{X}}$ is split surjective. If $\mathcal{X}=\left\{ H\right\} $
for some $H\in\Fc$ we simply say that $M$ is \textbf{projective
relative to $H$ }(or $H$\textbf{-projective}).
\end{defn}

There is a key difference between the above definition of relative
projectivity and the one given in the case of Mackey functors over
finite groups (see \cite[Section 3]{GuideToMackeyFunctorsWebb}).
Let $G$ be a finite group and let $M$ be a Mackey functor over $G$.
In this case we have that $M_{G}:=M\downarrow_{G}^{G}\uparrow_{G}^{G}\cong M$
and that $\theta_{G}=\Id_{M}$. In particular $\theta_{G}$ splits
and, therefore, any Mackey functor over $G$ is projective relative
to $G$. This result is however lost in the case of Mackey functors
over fusion systems since, given $N\in\MackFcR$, we do not, in general,
have $N_{S}\cong N$ (unless $\F=\FHH[S]$). We do however have the
following.
\begin{lem}
\label{lem:projective-relative-to-S.}Let $\G$ be a fusion system
containing $\F$, let $\R$ be $p$-local and let $M\in\MackFHR{\G}{\F}$.
Then $M$ is $S$-projective.
\end{lem}

\begin{proof}
Since $\F\subseteq\G$ then all $\G$-centric subgroups of $S$ are
also $\F$-centric. In particular we have that $M\in\MackFcR$. Since
$\R$ is $p$-local then, from Proposition \ref{prop:Inverse-of-S.},
we know that the centric Burnside ring $\BFcR$ contains an inverse
of $\overline{S}$. Then, with notation as in Proposition \ref{prop:Action-of-centric-burnside-ring.}
we have that
\[
\theta_{S}\theta^{S}\overline{S}^{-1}\cdot=\left(\overline{S}\cdot\right)\left(\overline{S}^{-1}\cdot\right)=1_{\BFcR}\cdot=\Id_{M}.
\]
This proves that $\theta_{S}$ is split surjective or, equivalently,
that $M$ is $S$-projective thus concluding the proof.
\end{proof}
This last result tells us that, whenever $\R$ is $p$-local, any
centric Mackey functor is projective relative to some family of $\F$-centric
subgroups of $S$ (namely $\left\{ S\right\} $). We would now like
for this family to be unique under certain minimality conditions and
use this uniqueness to define the defect set. In the case of Mackey
functors over finite groups this uniqueness follows from \cite[Lemma 3.2 and Proposition 3.3]{GuideToMackeyFunctorsWebb}.
In order to translate these results to the context of centric Mackey
functors over fusion systems we first need the following.
\begin{lem}
\label{lem:ifX->Y-and-X-proj.-then-Y-proj.}Let $M\in\MackFcR$, let
$\mathcal{X}$ and $\mathcal{Y}$ be families of objects in $\Fc$,
let $\sigma\colon\mathcal{X}\to\mathcal{Y}$ be a map between sets
and let $\varPhi=\left\{ \overline{\varphi_{H}}:H\to\sigma\left(H\right)\right\} _{H\in\mathcal{X}}$
be a family of morphisms in $\OFc$. There exists a (non necessarily
unique) morphism of $\muFR$-modules $\theta_{\varPhi}:M_{\mathcal{X}}\to M_{\mathcal{Y}}$
such that $\theta_{\mathcal{X}}=\theta_{\mathcal{Y}}\theta_{\varPhi}$.
In particular, if $M$ is $\mathcal{X}$-projective, then it is also
$\mathcal{Y}$-projective.
\end{lem}

\begin{proof}
Because of the direct sum decomposition of $M_{\mathcal{X}}$ and
$M_{\mathcal{Y}}$ given in Definition \ref{def:Relative-projectivity.}
it suffices to prove the claim in the case where $\mathcal{X}:=\left\{ H\right\} $,
$\mathcal{Y}:=\left\{ K\right\} $ and $\varPhi:=\left\{ \overline{\varphi}:H\to K\right\} $
for some $H,K\in\Fc$ and some $\overline{\varphi}\in\Hom_{\OFc}\left(H,K\right)$.

Fix a representative $\varphi$ of $\overline{\varphi}$ and view
it as an isomorphism onto its image. Then, for every $h\in H$ we
have that $\varphi c_{h}=c_{\varphi\left(H\right)}\varphi$ as isomorphisms
from $H$ to $\varphi\left(H\right)$. With this in mind Items \eqref{enu:property-composition-is-nice.}
and \eqref{enu:Property-conjugation-commutes.} of Lemma \ref{lem:Many-properties-definition.}
tell us that, for every $I_{\lui{h}{C}}^{B}c_{c_{h}}R_{C}^{A}\in\FHH$,
we have 
\[
c_{\varphi,B}I_{\lui{h}{C}}^{B}c_{c_{h}}R_{C}^{A}=I_{\lui{\varphi\left(h\right)}{\left(\varphi\left(C\right)\right)}}^{\varphi\left(B\right)}c_{c_{\varphi\left(h\right)}}R_{\varphi\left(C\right)}^{\varphi\left(A\right)}c_{\varphi,A}\in\FHH[\varphi\left(H\right)]c_{\varphi,A}.
\]
Where we are viewing $\varphi$ as an isomorphism between the appropriate
restrictions and we are viewing $\FHH[\varphi\left(H\right)]c_{\varphi,A}$
as a subset of $\muFR$. Because of Proposition \ref{prop:Mackey-algebra-basis.}
this allows us to define the $\muFR$-module morphism $\theta_{\varphi}\colon M_{H}\bjarrow M_{\varphi\left(H\right)}$
that, for every $y\in\muFR$, every $J\in\FHH$ and every $x\in I_{J}^{J}M\downarrow_{\FHH}^{\F}$,
sends $y\otimes_{\FmuFR{\FHH}}x$ to $yc_{\varphi^{-1},\varphi\left(J\right)}\otimes_{\FmuFR{\FHH[\varphi\left(H\right)]}}c_{\varphi,J}x$.
Notice now that $\FHH[\varphi\left(H\right)]\subseteq\FHH[K]$. Because
of Corollary \ref{cor:Mackey-algebra-inclusion.} this inclusion allows
us to define $\theta_{\iota_{\varphi\left(H\right)}^{K}}\colon M_{\varphi\left(H\right)}\twoheadrightarrow M_{K}$
as the natural $\muFR$-module morphism that, for every $y'\in\muFR1_{\FmuFR{\FHH[\varphi\left(H\right)]}}$
and every $x'\in M\downarrow_{\FHH[\varphi\left(H\right)]}^{\F}$,
sends $y'\otimes_{\FmuFR{\FHH[\varphi\left(H\right)]}}x'$ to $y'\otimes_{\FmuFR{\FHH[K]}}x'$.
With this setup we can finally define the $\muFR$-module morphism
$\theta_{\overline{\varphi}}\colon M_{H}\twoheadrightarrow M_{K}$
as $\theta_{\overline{\varphi}}:=\theta_{\iota_{\varphi\left(H\right)}^{K}}\theta_{\varphi}$
and, with $x,y$ and $J$ as above we have that
\begin{align*}
\theta_{H}\left(y\otimes_{\FmuFR{\FHH}}x\right) & =yx=yc_{\varphi^{-1},\varphi\left(J\right)}c_{\varphi,J}x=\theta_{K}\left(\theta_{\overline{\varphi}}\left(y\otimes_{\FmuFR{\FHH}}x\right)\right).
\end{align*}
Where we are viewing $\varphi$ as an isomorphism between the appropriate
restrictions and, for the second identity, we are using Items \eqref{enu:property-I_H^H-is-identity.}
and \eqref{enu:property-composition-is-nice.} of Lemma \ref{lem:Many-properties-definition.}
in order to introduce $c_{\varphi^{-1},\varphi\left(J\right)}c_{\varphi,J}$.
This proves that $\theta_{H}=\theta_{K}\theta_{\overline{\varphi}}$
thus concluding the proof.
\end{proof}
Using Lemma \ref{lem:ifX->Y-and-X-proj.-then-Y-proj.} we can now
translate \cite[Lemma 3.2]{GuideToMackeyFunctorsWebb} to the context
of centric Mackey functors over fusion systems.
\begin{cor}
\label{cor:projective-respect-to-bigger.}Let $M\in\MackFcR$, let
$\mathcal{X}$ and $\mathcal{Y}$ be families of $\F$-centric subgroups
of $S$ and denote by $\mathcal{X}^{\text{max}}\subseteq\mathcal{X}$
any family of maximal elements of $\mathcal{X}$ (under the preorder
$\le_{\F}$ of Notation \ref{nota:p,S,F}) taken up to $\F$-isomorphism.
\begin{enumerate}
\item \label{enu:X<Y implies Y-proj.}If $M$ is $\mathcal{X}$-projective
and $\mathcal{X}\subseteq\mathcal{Y}$ then $M$ is $\mathcal{Y}$-projective.
\item \label{enu:X-proj implies X^max proj.}If $M$ is $\mathcal{X}$-projective
then it is $\mathcal{X}^{\text{max}}$-projective.
\end{enumerate}
\end{cor}

\begin{proof}
From definition of $\mathcal{X}^{\text{max}}$ for every $H\in\mathcal{X}$
exists $J_{H}\in\mathcal{X}^{\text{max}}$ such that $H\le_{\F}J_{H}$
or, equivalently, such that $\Hom_{\OFc}\left(H,J_{H}\right)\not=\emptyset$.
On the other hand, for every $H\in\mathcal{X}$ we can take $K_{H}:=H\in\mathcal{Y}$
and we will have $\text{Id}_{H}\in\Hom_{\OFc}\left(H,K_{H}\right)\not=\emptyset$.
The result now follows from Lemma \ref{lem:ifX->Y-and-X-proj.-then-Y-proj.}.
\end{proof}
Finally we can translate \cite[Proposition 3.3]{GuideToMackeyFunctorsWebb}
to the context of centric Mackey functors over fusion systems.
\begin{prop}
\label{prop:aux-for-defect-set.}Let $M\in\MackFcR$ and let $\mathcal{X}$
and $\mathcal{Y}$ be families of $\F$-centric subgroups of $S$
closed under $\F$-subconjugacy (i.e. $K\in\mathcal{X}$ and $H\le_{\F}K$
imply $H\in\mathcal{X}$ and analogously with $\mathcal{Y}$). If
$M$ is both $\mathcal{X}$-projective and $\mathcal{Y}$-projective
then:
\begin{enumerate}
\item \label{enu:XxY-projective.}$M$ is $\mathcal{X}\times\mathcal{Y}$-projective
where
\[
\mathcal{X\times Y}:=\left\{ A\in\Fc\,|\,\exists H\in\mathcal{X},K\in\mathcal{Y}\text{ and }\overline{\varphi}\colon A\to K\text{ s.t. }\left(A,\overline{\varphi}\right)\in\left[H\times K\right]\right\} .
\]
\item \label{enu:item-XAY-projective.}$M$ is $\mathcal{X}\cap\mathcal{Y}$-projective.
\end{enumerate}
\end{prop}

\begin{proof}
For every $A\in\mathcal{X}\times\mathcal{Y}$, there exist, by definition,
$H\in\mathcal{X}$ and $K\in\mathcal{Y}$ such that $A\le_{\F}H,K$.
Since both $\mathcal{X}$ and $\mathcal{Y}$ are closed under $\F$-subconjugacy
this implies that $A\in\mathcal{X}\cap\mathcal{Y}$. In other words
we have that $\mathcal{X}\times\mathcal{Y}\subseteq\mathcal{X}\cap\mathcal{Y}$.
From Corollary \ref{cor:projective-respect-to-bigger.} \eqref{enu:X<Y implies Y-proj.}
we can now deduce that Item \eqref{enu:item-XAY-projective.} follows
from Item \eqref{enu:XxY-projective.}.

Let's prove Item \eqref{enu:XxY-projective.}. For every $H\in\mathcal{X}$,
every $K\in\mathcal{Y}$ and every $\left(A,\overline{\varphi}\right)\in\left[H\times_{\F}K\right]$
let us fix a representative $\varphi$ of $\overline{\varphi}$ and
view it as an isomorphism onto its image. Using the notation of Lemma
\ref{lem:Mackey-formula-induction-restriction.} we have that
\begin{align*}
M':=\bigoplus_{H\in\mathcal{X},K\in\mathcal{Y}}\bigoplus_{\left(A,\overline{\varphi}\right)\in\left[H\times_{\F}K\right]}M_{\left(A,\overline{\varphi}\right)}\uparrow_{\FHH[\varphi\left(A\right)]}^{\F}\cong & \left(M_{\mathcal{X}}\right)_{\mathcal{Y}}:=\bigoplus_{H\in\mathcal{X},K\in\mathcal{Y}}M\downarrow_{\FHH}^{\F}\uparrow_{\FHH}^{\F}\downarrow_{\FHH[K]}^{\F}\uparrow_{\FHH[K]}^{\F},
\end{align*}
We can now define $\Gamma\colon\left(M_{\mathcal{X}}\right)_{\mathcal{Y}}\bjarrow M'$
to be the inverse of the isomorphism induced from the one described
in Lemma \ref{lem:Mackey-formula-induction-restriction.}. 

We can now define $\varUpsilon\colon M'\to M_{\mathcal{X}\times\mathcal{Y}}$
by setting for every $H\in\mathcal{X}$, every $K\in\mathcal{Y}$,
every $\left(A,\overline{\varphi}\right)\in\left[H\times K\right]$,
every $J\le A$, every $x\in I_{J}^{J}M\downarrow_{\FHH[A]}^{\F}$
and every $y\in\muFR1_{\FmuFR{\FHH[\varphi\left(A\right)]}}$
\[
\varUpsilon\left(y\otimes_{\FmuFR{\FHH[\varphi\left(A\right)]}}x\right):=yc_{\varphi,J}\otimes_{\FmuFR{\FHH[A]}}x.
\]
where we are viewing $\varphi$ as an isomorphism between the appropriate
restrictions and, on the left hand side, we are viewing $x$ as an
element of $M_{\left(A,\overline{\varphi}\right)}$ while, on the
right hand side, we are viewing $x$ as an element of $M\downarrow_{\FHH[A]}^{\F}$.
Notice that, for every $h\in H$, we have that $c_{\varphi\left(H\right)}\varphi=c_{h}\varphi$
as isomorphisms from $H$ to $\varphi\left(H\right)$. With this in
mind Items \eqref{enu:property-composition-is-nice.} and \eqref{enu:Property-conjugation-commutes.}
of Lemma \ref{lem:Many-properties-definition.} and Proposition \ref{prop:Mackey-algebra-basis.}
ensure us that the definition of $\varUpsilon$ does not depend on
the choice of representatives of $y\otimes_{\FmuFR{\FHH[\varphi\left(A\right)]}}x$.
Moreover it is immediate from definition that $\varUpsilon$ commutes
with the action of $\muFR$ and, therefore, it's a $\muFR$-module
morphism.

Finally, since $M$ is both $\mathcal{X}$-projective and $\mathcal{Y}$-projective,
there exist Mackey functor morphisms $u_{\mathcal{X}}\colon M\to M_{\mathcal{X}}$
and $u_{\mathcal{Y}}\colon M\to M_{\mathcal{Y}}$ such that $\theta_{\mathcal{X}}^{M}u_{\mathcal{X}}=\theta_{\mathcal{Y}}^{M}u_{\mathcal{Y}}=\Id_{M}$.
Applying restriction and induction functors to the morphisms $u_{\mathcal{X}}$
and $\theta_{\mathcal{X}}$ we can define 
\begin{align*}
u_{\mathcal{X},\mathcal{Y}} & :=\sum_{K\in\mathcal{Y}}\uparrow_{\FHH[K]}^{\F}\left(\downarrow_{\FHH[K]}^{\F}\left(u_{\mathcal{X}}\right)\right)\colon M_{\mathcal{Y}}\to\left(M_{\mathcal{X}}\right)_{\mathcal{Y}},\\
\theta_{\mathcal{X},\mathcal{Y}} & :=\sum_{K\in\mathcal{Y}}\uparrow_{\FHH[K]}^{\F}\left(\downarrow_{\FHH[K]}^{\F}\left(\theta_{\mathcal{X}}^{M}\right)\right)\colon\left(M_{\mathcal{X}}\right)_{\mathcal{Y}}\to M_{\mathcal{Y}}.
\end{align*}
From functoriality of induction and restriction, we have that $\theta_{\mathcal{X},\mathcal{Y}}u_{\mathcal{X},\mathcal{Y}}=\Id_{M_{\mathcal{Y}}}$.

Let $H\in\mathcal{X}$, let $K,J\in\mathcal{Y}$ such that $K\le J$,
let $\left(A,\overline{\varphi}\right)\in\left[H\times_{\F}K\right]$,
let $\varphi$ be the previously fixed representative of $\overline{\varphi}$
viewed as an isomorphism onto its image, let $\left(C,\overline{\theta}\right)\in\left[\varphi\left(A\right)\times_{\FHH[K]}J\right]$,
let $x\in I_{C}^{C}M$ and let $y\in1_{\FmuFR{\FHH[K]}}\muFR1_{\FmuFR{\FHH}}$.
Using the notation of Corollary \ref{cor:conjugation-in-F-and-in-orbit-F.}
we have that
\begin{align*}
\theta_{\mathcal{Y}}^{M}\left(\theta_{\mathcal{X},\mathcal{Y}}\left(y\otimes_{\FmuFR{\FHH[K]}}I_{\overline{\theta}\left(C\right)}^{J}c_{\overline{\theta\varphi}}\otimes_{\FmuFR{\FHH}}x\right)\right) & =\theta_{\mathcal{Y}}^{M}\left(y\otimes_{\FmuFR{\FHH[K]}}I_{\overline{\theta}\left(C\right)}^{J}c_{\overline{\theta\varphi}}x\right)=yI_{\overline{\theta}\left(C\right)}^{J}c_{\overline{\theta\varphi}}x.
\end{align*}
and that
\begin{align*}
yI_{\overline{\theta}\left(C\right)}^{J}c_{\overline{\theta}\varphi}x & =\theta_{\mathcal{X}\times\mathcal{Y}}^{M}\left(yI_{\overline{\theta}\left(C\right)}^{J}c_{\overline{\theta}\varphi}\otimes_{\FmuFR{\FHH[A]}}x\right),\\
 & =\theta_{\mathcal{X}\times\mathcal{Y}}^{M}\left(\varUpsilon\left(yI_{\overline{\theta}\left(C\right)}^{J}c_{\overline{\theta}}\otimes_{\FmuFR{\FHH[\varphi\left(A\right)]}}x\right)\right),\\
 & =\theta_{\mathcal{X}\times\mathcal{Y}}^{M}\left(\varUpsilon\left(\Gamma\left(y\otimes_{\FmuFR{\FHH[K]}}I_{\overline{\theta}\left(C\right)}^{J}c_{\overline{\theta\varphi}}\otimes_{\FmuFR{\FHH}}x\right)\right)\right).
\end{align*}
Where, in the second identity, we are viewing $x$ as an element of
$M_{\left(A,\overline{\varphi}\right)}$. From Lemma \ref{lem:Mackey-formula-induction-restriction.}
we know that every element in $M\downarrow_{\FHH}^{\F}\uparrow_{\FHH}^{\F}\downarrow_{\FHH[K]}^{\F}\uparrow_{\FHH[K]}^{\F}$
can be written as a finite sum of elements of the form $y\otimes_{\FmuFR{\FHH[K]}}I_{\overline{\theta}\left(C\right)}^{J}c_{\overline{\theta\varphi}}\otimes_{\FmuFR{\FHH}}x$.
Therefore the previous identities prove that $\theta_{\mathcal{X}\times\mathcal{Y}}^{M}\varUpsilon\Gamma=\theta_{\mathcal{Y}}^{M}\theta_{\mathcal{X},\mathcal{Y}}$.
With this in mind we obtain
\[
\theta_{\mathcal{X}\times\mathcal{Y}}^{M}\varUpsilon\Gamma u_{\mathcal{X},\mathcal{Y}}u_{\mathcal{Y}}=\theta_{\mathcal{Y}}^{M}\theta_{\mathcal{X},\mathcal{Y}}u_{\mathcal{X},\mathcal{Y}}u_{\mathcal{Y}}=\theta_{\mathcal{Y}}^{M}u_{\mathcal{Y}}=\Id_{M}.
\]
This proves that $\theta_{\mathcal{X}\times\mathcal{Y}}^{M}$ is split
surjective or, equivalently, that $M$ is $\mathcal{X}\times\mathcal{Y}$-projective
thus concluding the proof.
\end{proof}
We can now finally define the defect set of a centric Mackey functor
over a fusion system.
\begin{cor}
\label{cor:defect-set.}Let $\R$ be $p$-local and let $M\in\MackFcR$.
There exists a unique minimal family of $\F$-centric subgroups of
$S$ that is closed under $\F$-subconjugacy and such that $M$ is
projective relative to it. 
\end{cor}

\begin{proof}
This is an immediate consequence of Lemma \ref{lem:projective-relative-to-S.},
Corollary \ref{cor:projective-respect-to-bigger.} \eqref{enu:X<Y implies Y-proj.}
and Proposition \ref{prop:aux-for-defect-set.} \eqref{enu:item-XAY-projective.}.
\end{proof}
\begin{defn}
\label{def:vertex.}Let $\R$ be $p$-local and let $M\in\MackFcR$.
We call the minimal family of elements in $\Fc$ given in Corollary
\ref{cor:defect-set.} the \textbf{defect set of $M$ }(denoted as
$\mathcal{X}_{M}$). Using the notation of Corollary \ref{cor:projective-respect-to-bigger.}
we call \textbf{defect group} of $M$ any element in $\mathcal{X}_{M}^{\text{max}}$
(for any choice of $\mathcal{X}_{M}^{\text{max}}$. If $\left|\mathcal{X}_{M}^{\text{max}}\right|=1$
we say that $M$ \textbf{admits a vertex} and we call \textbf{vertex
of $M$ }(and denote it by $V_{M}$) any fully $\F$-normalized defect
group of $M$.
\end{defn}

\subsection{\label{subsec:Transfer-maps-and-Higman's criterion.}Transfer maps
and Higman's criterion.}

The main goal of this subsection will be that of translating Higman's
criterion (see \cite[Theorem 2.2]{NagaoHirosiRepresentationsOfFiniteGroups})
to centric Mackey functors over fusion systems (see Theorem \ref{thm:Higman's-criterion.}).
This will allow us to relate the concept of relative projectivity
of an indecomposable Mackey functor $M\in\MackFcR$ to the images
of certain transfer maps (see Definitions \ref{def:Transfer-and-restriction.}
and \ref{def:Image-transfer.}). In order to understand this relation
we need to start by introducing some notation.
\begin{defn}
\label{def:Transfer-and-restriction.}Let $\G$ be a fusion system
containing $\F$, let $M\in\MackFHR{\G}{\F}$, let $H\in\Fc$ and
let $\varphi\colon H\to\varphi\left(H\right)$ be an isomorphism in
$\F$. We define the \textbf{conjugation map from $\FHH$ to $\FHH[\varphi\left(H\right)]$
on $M$} as the $\R$-algebra morphism $\lui{M,\varphi}{\cdot}:\End\left(M\downarrow_{\FHH}^{\F}\right)\to\End\left(M\downarrow_{\FHH[\varphi\left(H\right)]}^{\F}\right)$,
obtained by setting for every $f\in\End\left(M\downarrow_{\FHH}^{\F}\right)$,
every $K\in\FHH[\varphi\left(H\right)]\cap\Fc$ and every $x\in I_{K}^{K}M\downarrow_{\FHH[\varphi\left(H\right)]}^{\F}$
\[
\lui{M,\varphi}{f}\left(x\right):=c_{\varphi,\varphi^{-1}\left(K\right)}\left(f\left(c_{\varphi^{-1},K}\,x\right)\right).
\]
Where we are viewing $\varphi$ as an isomorphism between the appropriate
restrictions and we are viewing $M\downarrow_{\FHH}^{\F}$ and $M\downarrow_{\FHH[\varphi\left(H\right)]}^{\F}$
as subsets of $M$.

We define the \textbf{transfer map from $\FHH$ to $\F$ on $M$}
as the $\R$-module morphism 
\[
\lui{M}{\tr_{\FHH}^{\F}}:\begin{array}{ccc}
\End\left(M\downarrow_{\FHH}^{\F}\right) & \longrightarrow & \End\left(M\right)\\
{\scriptstyle f} & {\scriptstyle \longrightarrow} & {\scriptstyle \theta_{H}^{M}f\uparrow_{\FHH}^{\F}\theta_{M}^{H}}
\end{array}.
\]
where $f\uparrow_{\FHH}^{\F}$ denotes the image of $f$ via the induction
functor $\uparrow_{\FHH}^{\F}$. More precisely, for every $K\in\Fc$,
every $x\in I_{K}^{K}M$ and every $f\in\End\left(M\downarrow_{\FHH}^{\F}\right)$
we have that 
\[
\lui{M}{\tr_{\FHH}^{\F}}\left(f\right)\left(x\right)=\sum_{\left(A,\overline{\varphi}\right)\in\left[H\times_{\F}K\right]}I_{\overline{\varphi}\left(A\right)}^{K}c_{\overline{\varphi}}f\left(c_{\overline{\varphi^{-1}}}R_{\overline{\varphi}\left(A\right)}^{K}x\right).
\]

Finally, given any fusion subsystem $\H\subseteq\F$, we define the
\textbf{restriction map from $\F$ to $\H$ on $M$} as the $\R$-algebra
morphism
\[
\lui{M}{\r_{\H}^{\F}}\colon\begin{array}{ccc}
\End\left(M\right) & \longrightarrow & \End\left(M\downarrow_{\H}^{\F}\right)\\
{\scriptstyle f} & {\scriptstyle \longrightarrow} & {\scriptstyle f\downarrow_{\H}^{\F}}
\end{array}.
\]
where $f\downarrow_{\H}^{\F}$ denotes the image of $f$ via the restriction
functor $\downarrow_{\H}^{\F}$.

Whenever there is no doubt regarding $M$ we will simply write
\begin{align*}
\tr_{\FHH}^{\F} & :=\lui{M}{\tr_{\FHH}^{\F}}, & \r_{\H}^{\F} & :=\lui{M}{\r_{\H}^{\F}}, & \lui{\varphi}{\cdot} & :=\lui{M,\varphi}{\cdot}.
\end{align*}
\end{defn}

Transfer, restriction and conjugation maps satisfy the following properties
which are analogous to those satisfied in the case of Mackey functors
over groups (see \cite[Definition 2.7]{SASAKI198298}). 
\begin{prop}
\label{prop:Properties-of-transfer-restriction-and-conjugation-maps.}Let
$M\in\MackFcR$ then:
\begin{enumerate}
\item \label{enu:Identity.}For every $H\in\Fc$ and $h\in H$ we have that
$\tr_{\FHH}^{\FHH}=\r_{\FHH}^{\FHH}=\lui{c_{h}}{\cdot}=\Id_{\End\left(M\downarrow_{\FHH}^{\F}\right)}$.
\item \label{enu:Composition-restriction.}For all fusion subsystems $\H\subseteq\mathcal{K}\subseteq\F$
we have that $\r_{\H}^{\mathcal{K}}\r_{\mathcal{K}}^{\F}=\r_{\H}^{\F}$.
\item \label{enu:composition-transfer.}For every $H\le K\in\Fc$ we have
that $\tr_{\FHH[K]}^{\F}\tr_{\FHH}^{\FHH[K]}=\tr_{\FHH}^{\F}$.
\item \label{enu:Composition-conjugation.}For all isomorphisms $\varphi,\psi$
in $\Fc$ such that $\varphi\psi$ is defined we have that $\lui{\psi}{\cdot}\lui{\varphi}{\cdot}=\lui{\psi\varphi}{\cdot}$. 
\item \label{enu:Transfer-commutes-with-conjugation.}For every $H\le K\in\Fc$
and every isomorphism $\varphi\colon K\to\varphi\left(K\right)$ in
$\F$ we have that $\lui{\varphi}{\cdot}\tr_{\FHH}^{\FHH[K]}=\tr_{\FHH[\varphi\left(H\right)]}^{\FHH[\varphi\left(K\right)]}\lui{\varphi}{\cdot}$.
\item \label{enu:restriction-commutes-with-conjugation.}For every $H\le K\in\Fc$
and every isomorphism $\varphi\colon K\to\varphi\left(K\right)$ in
$\F$ we have that $\lui{\varphi}{\cdot}\r_{\FHH}^{\FHH[K]}=\r_{\FHH[\varphi\left(H\right)]}^{\FHH[\varphi\left(K\right)]}\lui{\varphi}{\cdot}$.
\item \label{enu:Transfer-eliminates-conjugation.}For every $H\in\Fc$
and every isomorphism $\varphi\colon H\to\varphi\left(H\right)$ in
$\F$ we have that $\tr_{\FHH[\varphi\left(H\right)]}^{\F}\lui{\varphi}{\cdot}=\tr_{\FHH}^{\F}$.
\item \label{enu:Restriction-eliminates-conjugation.}For every $H\in\Fc$
and every isomorphism $\varphi\colon H\to\varphi\left(H\right)$ in
$\F$ we have that $\lui{\varphi}{\cdot}\r_{\FHH[H]}^{\F}=\r_{\FHH[\varphi\left(H\right)]}^{\F}$.
\item \label{enu:Mackey-formula.}For every $K,H\in\Fc$ we have
\[
\r_{\FHH[K]}^{\F}\tr_{\FHH}^{\F}=\sum_{\left(A,\overline{\varphi}\right)\in\left[H\times_{\F}K\right]}\tr_{\FHH[\varphi\left(A\right)]}^{\FHH[K]}\lui{\varphi}{\cdot}\r_{\FHH[A]}^{\FHH[H]}.
\]
Here $\varphi$ is any representative of $\overline{\varphi}$ seen
as an isomorphism onto its image.
\item \label{enu:Green-formula.}For every $H\in\Fc$, every $f\in\End\left(M\right)$
and every $g\in\End\left(M\downarrow_{\FHH}^{\F}\right)$ we have
that
\begin{align*}
f\tr_{\FHH}^{\F}\left(g\right) & =\tr_{\FHH}^{\F}\left(\r_{\FHH}^{\F}\left(f\right)g\right), & \text{ and that} &  & \tr_{\FHH}^{\F}\left(g\right)f & =\tr_{\FHH}^{\F}\left(g\r_{\FHH}^{\F}\left(f\right)\right).
\end{align*}
\item \label{enu:Restriction-and-transfer-makes-burnside.}Let $H\in\Fc$.
Using Notation \ref{nota:Initial-notation.} and the notation of Proposition
\ref{prop:Action-of-centric-burnside-ring.} we have that $\tr_{\FHH}^{\F}\r_{\FHH}^{\F}=\left(\overline{H}\cdot\right)_{*}$.
\end{enumerate}
\end{prop}

\begin{proof}
$\phantom{.}$
\begin{enumerate}
\item Let $K\in\FHH\cap\Fc$, let $x\in I_{K}^{K}M\downarrow_{\FHH}^{\F}$
and let $f\in\End\left(M\downarrow_{\FHH}^{\F}\right)$. By definition
of restriction we have that $\r_{\FHH}^{\F}\left(f\right)\left(x\right)=f\left(x\right)$.
Since $f$ is a $\FmuFR{\FHH}$-module morphism we have that
\begin{align*}
\lui{c_{h}}{f}\left(x\right) & =c_{c_{h}}f\left(c_{c_{h^{-1}}}x\right)=c_{c_{h}}c_{c_{h^{-1}}}f\left(x\right)=f\left(x\right).
\end{align*}
Where we are viewing $c_{h}$ as an isomorhism from $K^{h}$ to $K$.
Finally, from Proposition \ref{prop:properties-HXK.} \eqref{enu:prod-if-F-is-F_S(S).},
we have that $\left[H\times_{\FHH}K\right]=\left\{ \left(K,\overline{\Id_{K}}\right)\right\} $
and, therefore, from Lemma \ref{lem:Many-properties-definition.}
\eqref{enu:property-I_H^H-is-identity.}, we can conclude that
\begin{align*}
\tr_{\FHH}^{\F}\left(f\right)\left(x\right) & =I_{K}^{K}c_{\overline{\Id_{K}}}\left(f\left(c_{\overline{\Id_{K}}}R_{K}^{K}x\right)\right)=f\left(x\right).
\end{align*}
\item Since the restriction functor satisfies $\downarrow_{\H}^{\mathcal{K}}\downarrow_{\mathcal{K}}^{\F}=\downarrow_{\H}^{\F}$,
then Item \eqref{enu:Composition-restriction.} follows.
\item Let $J\in\Fc$, let $x\in I_{J}^{J}M$ and let $f\in\End\left(M\downarrow_{\FHH}^{\F}\right)$.
From Proposition \ref{prop:properties-HXK.} \eqref{enu:prod-pullback-left.}
we have that 
\begin{align*}
\tr_{\FHH}^{\F}\left(f\right)\left(x\right) & =\sum_{\left(A,\overline{\varphi}\right)\in\left[K\times_{\F}J\right]}\sum_{\begin{array}{c}
{\scriptstyle k\in\left[A\backslash K/H\right]}\\
{\scriptstyle A^{k}\cap H\in\Fc}
\end{array}}I_{\overline{\varphi c_{k}}\left(A^{k}\cap H\right)}^{J}c_{\overline{\varphi c_{k}}}\left(f\left(c_{\overline{\left(\varphi c_{k}\right)^{-1}}}R_{\overline{\varphi c_{k}}\left(A^{k}\cap H\right)}^{J}x\right)\right).
\end{align*}
Since $M\downarrow_{\FHH}^{\F}\in\MackFHR{\F}{\FHH}$ we know that
$c_{\overline{\left(\varphi c_{k}\right)^{-1}}}R_{\overline{\varphi c_{k}}\left(A^{k}\cap H\right)}^{J}\cdot x=0$
for every $\left(A,\overline{\varphi}\right)\in\left[K\times_{\F}J\right]$
and every $k\in\left[A\backslash K/H\right]$ such that $A^{k}\cap H\in\FHH^{c}\backslash\left(\FHH^{c}\cap\Fc\right)$.
Thus, we can replace the second sum of the above equation as a sum
over $k\in\left[A\backslash K/H\right]$ such that $A^{k}\cap H\in\FHH^{c}$.
Using Proposition \ref{prop:properties-HXK.} \eqref{enu:prod-if-F-is-F_S(S).}
we can now rewrite. 
\[
\tr_{\FHH}^{\F}\left(f\right)\left(x\right)=\sum_{\left(A,\overline{\varphi}\right)\in\left[K\times_{\F}J\right]}\sum_{\left(B,\overline{\psi}\right)\in\left[H\times_{\FHH[K]}A\right]}I_{\overline{\varphi\psi}\left(B\right)}^{J}c_{\overline{\varphi\psi}}\left(f\left(c_{\overline{\left(\varphi\psi\right)^{-1}}}R_{\overline{\varphi\psi}\left(B\right)}^{J}x\right)\right).
\]
From Corollary \ref{cor:conjugation-in-F-and-in-orbit-F.} we know
that the above is equal to $\tr_{\FHH[K]}^{\F}\left(\tr_{\FHH}^{\FHH[K]}\left(f\right)\right)\left(x\right)$
thus proving Item \eqref{enu:composition-transfer.}.
\item Let $H\in\Fc$, let $\varphi\colon H\to\varphi\left(H\right)$ and
$\psi\colon\varphi\left(H\right)\to\psi\varphi\left(H\right)$ be
isomorphisms in $\F$, let $J\in\FHH[\psi\varphi\left(H\right)]\cap\Fc$,
let $x\in I_{J}^{J}M\downarrow_{\FHH[\psi\varphi\left(H\right)]}^{\F}$
and let $f\in\End\left(M\downarrow_{\FHH}^{\F}\right)$. Item \eqref{enu:Composition-conjugation.}.
follows from Lemma \ref{lem:Many-properties-definition.} \eqref{enu:property-composition-is-nice.}
via the identities below
\begin{align*}
\lui{\psi\varphi}{f}\left(x\right) & =c_{\psi\varphi}\left(f\left(c_{\varphi^{-1}\psi^{-1}}x\right)\right)=c_{\psi}\left(c_{\varphi}f\left(c_{\varphi^{-1}}c_{\psi^{-1}}x\right)\right)=\lui{\psi}{\left(\lui{\varphi}{f}\right)}\left(x\right).
\end{align*}
\item Let $J\in\FHH[K]\cap\Fc$. Viewing $\left[\varphi\left(J\right)\backslash\varphi\left(K\right)/\varphi\left(H\right)\right]$
as a subset of $\varphi\left(K\right)$ we can take $\varphi^{-1}\left(\left[\varphi\left(J\right)\backslash\varphi\left(K\right)/\varphi\left(H\right)\right]\right)=\left[J\backslash K/H\right]$.
Moreover, for every $\varphi\left(A\right)\le\varphi\left(K\right)$
we have that $\varphi\left(A\right)\in\FHH[\varphi\left(K\right)]^{c}$
if and only if $A\in\FHH[K]^{c}$ and, for every $\varphi\left(k\right)\in\varphi\left(K\right)$
we have that $\varphi^{-1}\left(\varphi\left(J\right)^{\varphi\left(k\right)}\cap\varphi\left(H\right)\right)=J^{k}\cap H$.
From Proposition \ref{prop:properties-HXK.} \eqref{enu:prod-if-F-is-F_S(S).}
we can therefore conclude that 
\begin{align*}
\left[H\times_{\FHH[K]}J\right] & =\hspace{-10bp}\bigsqcup_{\begin{array}{c}
{\scriptstyle k\in\left[J\backslash K/H\right]}\\
{\scriptstyle J^{k}\cap H\in\FHH[K]^{c}}
\end{array}}\hspace{-10bp}\left\{ \left(J^{k}\cap H,\overline{\iota c_{k}}\right)\right\} =\hspace{-10bp}\hspace{-10bp}\bigsqcup_{\left(B,\overline{\psi}\right)\in\left[\varphi\left(H\right)\times_{\FHH[\varphi\left(K\right)]}\varphi\left(J\right)\right]}\hspace{-10bp}\left\{ \left(\varphi^{-1}\left(B\right),\overline{\varphi^{-1}\psi\varphi}\right)\right\} .
\end{align*}
Where, for the second identity, we are using that $c_{k}$ and $\varphi^{-1}c_{\varphi\left(k\right)}\varphi$
are equal as automorphisms of $K$. Let $x\in I_{\varphi\left(J\right)}^{\varphi\left(J\right)}M\downarrow_{\FHH[\varphi\left(K\right)]}^{\F}$.
Using the above identity we have that
\begin{align*}
\tr_{\FHH[\varphi\left(H\right)]}^{\FHH[\varphi\left(K\right)]}\left(\lui{\varphi}{f}\right)\left(x\right) & =\sum_{\left(B,\overline{\psi}\right)\in\left[\varphi\left(H\right)\times_{\FHH[\varphi\left(K\right)]}\varphi\left(J\right)\right]}I_{\overline{\psi}\left(B\right)}^{\varphi\left(J\right)}c_{\overline{\psi\varphi}}\left(f\left(c_{\overline{\left(\psi\varphi\right)^{-1}}}R_{\overline{\psi}\left(B\right)}^{\varphi\left(J\right)}x\right)\right),\\
 & =\sum_{\left(C,\overline{\theta}\right)\in\left[H\times_{\FHH[K]}J\right]}c_{\varphi}I_{\overline{\theta}\left(C\right)}^{J}c_{\overline{\theta}}\left(f\left(c_{\overline{\theta}}R_{\overline{\theta}\left(C\right)}^{J}c_{\varphi^{-1}}x\right)\right)=\lui{\varphi}{\left(\tr_{\FHH}^{\FHH[K]}\left(f\right)\right)}\left(x\right).
\end{align*}
Where, for the second identity, we are using Lemma \ref{lem:Many-properties-definition.}
\eqref{enu:property-I_H^H-is-identity.} and \eqref{enu:Property-conjugation-commutes.}
in order to obtain the identities $I_{\psi\left(B\right)}^{\varphi\left(J\right)}=c_{\varphi}I_{\varphi^{-1}\psi\left(B\right)}^{J}c_{\varphi^{-1}}$
and $R_{\psi\left(B\right)}^{\varphi\left(J\right)}=c_{\varphi}R_{\varphi^{-1}\psi\left(B\right)}^{J}c_{\varphi^{-1}}$
for any representative $\psi$ of $\overline{\psi}$. This proves
Item\textbf{ }\eqref{enu:Transfer-commutes-with-conjugation.}.
\item Let $J\in\FHH[\varphi\left(H\right)]\cap\Fc$ and let $x\in I_{J}^{J}M$.
Item \eqref{enu:restriction-commutes-with-conjugation.} follows from
the identities below
\begin{align*}
\lui{\varphi}{\left(\r_{\FHH}^{\FHH[K]}\left(f\right)\right)}\left(x\right) & =c_{\varphi}\left(f\left(c_{\varphi^{-1}}x\right)\right)=\lui{\varphi}{f}\left(x\right)=\r_{\FHH[\varphi\left(H\right)]}^{\FHH[K]}\left(\lui{\varphi}{f}\right)\left(x\right).
\end{align*}
\item Let $K\in\Fc$, let $x\in I_{K}^{K}M$ and let $f\in\End\left(M\downarrow_{\FHH}^{\F}\right)$.
Using Proposition \ref{prop:properties-HXK.} \eqref{enu:prod-iso-left.}
we have that.
\begin{align*}
\tr_{\FHH[\varphi\left(H\right)]}^{\F}\left(\lui{\varphi}{f}\right)\left(x\right) & =\sum_{\left(B,\overline{\psi}\right)\in\left[\varphi\left(H\right)\times_{\F}K\right]}I_{\overline{\theta\varphi}\left(\varphi^{-1}\left(B\right)\right)}^{K}c_{\overline{\theta\varphi}}\left(f\left(c_{\overline{\left(\theta\varphi\right)^{-1}}}R_{\overline{\theta\varphi}\left(\varphi^{-1}\left(B\right)\right)}^{K}x\right)\right),\\
 & =\sum_{\left(C,\overline{\theta}\right)\in\left[H\times_{\F}K\right]}I_{\overline{\theta}\left(C\right)}^{K}c_{\overline{\theta}}\left(f\left(c_{\overline{\theta^{-1}}}R_{\overline{\theta}\left(C\right)}^{K}x\right)\right)=\tr_{\FHH}^{\F}\left(f\right)\left(x\right).
\end{align*}
Where we are viewing $\varphi$ as an isomorphism between the appropriate
restrictions. This proves Item \eqref{enu:Transfer-eliminates-conjugation.}.
\item Let $K\in\FHH[\varphi\left(H\right)]\cap\Fc$, let $x\in I_{K}^{K}M$
and let $f\in\End\left(M\right)$. Since $f$ is a morphism of $\muFR$-modules
we have that
\begin{align*}
\lui{\varphi}{\left(\r_{\FHH}^{\F}\left(f\right)\right)}\left(x\right) & =c_{\varphi}f\left(c_{\varphi^{-1}}x\right)=c_{\varphi}c_{\varphi^{-1}}f\left(x\right)=f\left(x\right).
\end{align*}
Where we are viewing $\varphi$ as an isomorphism between the appropriate
restrictions. This proves Item \eqref{enu:Restriction-eliminates-conjugation.}.
\item Let $J\in\FHH[K]\cap\Fc$, let $x\in I_{J}^{J}M$ and let $f\in\End\left(M\downarrow_{\FHH}^{\F}\right)$.
From Proposition \ref{prop:properties-HXK.} \eqref{enu:prod-pullback-right.}
we have that 
\begin{align*}
\tr_{\FHH}^{\F}\left(f\right)\left(x\right) & =\sum_{\left(A,\overline{\varphi}\right)\in\left[H\times_{\F}K\right]}\sum_{\begin{array}{c}
{\scriptstyle k\in\left[J\backslash K/\varphi\left(A\right)\right]}\\
{\scriptstyle J^{k}\cap\varphi\left(A\right)\in\Fc}
\end{array}}I_{J\cap\lui{k}{\varphi\left(A\right)}}^{J}c_{c_{k}\varphi}\left(f\left(c_{\left(c_{k}\varphi\right)^{-1}}R_{J\cap\lui{k}{\varphi\left(A\right)}}^{J}x\right)\right).
\end{align*}
Where we are fixing a representative $\varphi$ of $\overline{\varphi}$
and viewing it as an isomorphism onto its image. The same arguments
employed to prove Item \eqref{enu:composition-transfer.} allow us
to replace the second sum of the previous equation with a sum over
$\left[\varphi\left(A\right)\times_{\FHH[K]}J\right]$. This leads
us to the identities
\begin{align*}
\tr_{\FHH}^{\F}\left(f\right)\left(x\right) & =\sum_{\left(A,\overline{\varphi}\right)\in\left[H\times_{\F}K\right]}\sum_{\left(B,\overline{\psi}\right)\in\left[J\times_{\FHH[K]}\varphi\left(A\right)\right]}I_{\overline{\psi}\left(B\right)}^{J}c_{\overline{\psi\varphi}}\left(f\left(c_{\overline{\left(\psi\varphi\right)^{-1}}}R_{\overline{\psi}\left(B\right)}^{J}x\right)\right),\\
 & =\sum_{\left(A,\overline{\varphi}\right)\in\left[H\times_{\F}K\right]}\tr_{\FHH[\varphi\left(A\right)]}^{\FHH[K]}\left(\lui{\varphi}{\left(\r_{\FHH[A]}^{\FHH[H]}\left(f\right)\right)}\right)\left(x\right).
\end{align*}
Here we are viewing $M\downarrow_{\FHH[K]}^{\F}$ as a subset of $M$.
With this inclusion in mind we also have that $\tr_{\FHH}^{\F}\left(f\right)\left(x\right)=\r_{\FHH[K]}^{\F}\left(\tr_{\FHH}^{\F}\left(f\right)\right)\left(x\right)$
and, therefore, the above identities prove Item \eqref{enu:Mackey-formula.}.
\item We will prove just the first identity since the second is proved similarly.
Let $K\in\Fc$. Since $f$ is a morphism of $\muFR$-modules then,
for every $y\in M\downarrow_{\FHH}^{\F}\subset M$ and every $\left(A,\overline{\varphi}\right)\in\left[H\times_{\F}K\right]$
we have that 
\[
f\left(I_{\overline{\varphi}\left(A\right)}^{K}c_{\overline{\varphi}}y\right)=I_{\overline{\varphi}\left(A\right)}^{K}c_{\overline{\varphi}}f\left(y\right)=I_{\overline{\varphi}\left(A\right)}^{K}c_{\overline{\varphi}}\r\downarrow_{\FHH}^{\F}\left(f\right)\left(y\right).
\]
Let $x\in I_{K}^{K}M$. Item \eqref{enu:Green-formula.} now follows
from the above via the identities below
\begin{align*}
f\left(\tr_{\FHH}^{\F}\left(g\right)\left(x\right)\right) & =\sum_{\left(A,\overline{\varphi}\right)\in\left[H\times_{\F}K\right]}I_{\overline{\varphi}\left(A\right)}^{K}c_{\overline{\varphi}}\cdot\left(\left(\r\downarrow_{\FHH}^{\F}\left(f\right)g\right)\left(c_{\overline{\varphi^{-1}}}R_{\overline{\varphi}\left(A\right)}^{K}\cdot x\right)\right),\\
 & =\tr_{\FHH}^{\F}\left(\r\downarrow_{\FHH}^{\F}\left(f\right)g\right)\left(x\right).
\end{align*}
\item Let $K\in\Fc$, let $x\in I_{K}^{K}M$ and let $f\in\End\left(M\right)$.
Since $f$ is a $\muFR$-module morphism then, for every $\left(A,\overline{\varphi}\right)\in\left[H\times_{\F}K\right]$,
we have that $f\left(c_{\overline{\varphi^{-1}}}R_{\overline{\varphi}\left(A\right)}^{K}x\right)=c_{\overline{\varphi^{-1}}}R_{\overline{\varphi}\left(A\right)}^{K}f\left(x\right)$.
Item \eqref{enu:Restriction-and-transfer-makes-burnside.} follows
from this and Proposition \ref{prop:Action-of-centric-burnside-ring.}
via the identities below
\begin{align*}
\tr_{\FHH}^{\F}\left(\r_{\FHH}^{\F}\left(f\right)\right)\left(x\right) & =\sum_{\left(A,\overline{\varphi}\right)\in\left[H\times_{\F}K\right]}I_{\overline{\varphi}\left(A\right)}^{K}c_{\overline{\varphi}}c_{\overline{\varphi^{-1}}}R_{\overline{\varphi}\left(A\right)}^{K}\left(f\left(x\right)\right),\\
 & =\theta_{H}^{M}\left(\theta_{M}^{H}\left(f\left(x\right)\right)\right)=\left(\overline{H}\cdot\right)_{*}\left(f\right)\left(x\right).
\end{align*}
\end{enumerate}
\end{proof}
\begin{rem}
Given a fusion $\mathcal{K}$ contained in $\F$ the transfer $\tr_{\mathcal{K}}^{\F}$
is in general not defined. We will however see in Subsection \ref{subsec:Transfer-map-from-trHNF-to-trHF.}
that something similar can be defined when $\mathcal{K}=N_{\F}\left(H\right)$
for some $H\in\Fc$. In this situation we can obtain a result similar
to Proposition \ref{prop:Properties-of-transfer-restriction-and-conjugation-maps.}
\eqref{enu:composition-transfer.} but replacing $\FHH[K]$ with $N_{\F}\left(H\right)$
(see Lemma \ref{lem:transfer-composes-nicely.}).
\end{rem}

\begin{cor}
\label{cor:Conjugation-is-isomorphism.}Let $M\in\MackFcR$, let $H\in\Fc$
and let $\varphi\colon H\to\varphi\left(H\right)$ be an isomorphism
in $\F$ then $\lui{\Id_{H}}{\cdot}=\Id_{\End\left(M\downarrow_{\FHH}^{\Fc}\right)}$
and $\lui{\varphi}{\cdot}$ is an isomorphism.
\end{cor}

\begin{proof}
Let $K\in\FHH\cap\Fc$, let $x\in I_{K}^{K}M$ and let $f\in\End\left(M\downarrow_{\FHH}^{\Fc}\right)$.
From definition of conjugation map and Lemma \ref{lem:Many-properties-definition.}
\eqref{enu:property-I_H^H-is-identity.} we have that 
\[
\lui{\Id_{H}}{f}\left(x\right)=c_{\Id_{H}}f\left(c_{\Id_{H}}x\right)=f\left(x\right).
\]
Thus we have that $\lui{\Id_{H}}{\cdot}=\Id_{\End\left(M\downarrow_{\FHH}^{\Fc}\right)}$.
Using Proposition \ref{prop:Properties-of-transfer-restriction-and-conjugation-maps.}
\eqref{enu:Composition-conjugation.} we can now deduce that
\[
\lui{\varphi}{\cdot}\lui{\varphi^{-1}}{\cdot}=\lui{\varphi^{-1}}{\cdot}\lui{\varphi}{\cdot}=\lui{\Id_{H}}{\cdot}=\Id_{\End\left(M\downarrow_{\FHH}^{\Fc}\right)}.
\]
This proves that $\lui{\varphi}{\cdot}$ has an inverse and, therefore,
is an isomorphism.
\end{proof}
\begin{defn}
\label{def:Image-transfer.}Let $M\in\MackFcR$, let $H\in\Fc$ and
let $\mathcal{X}$ be a family of objects in $\Fc$. We define the
\textbf{transfer image from $H$ to $\F$ on $M$ }and the \textbf{transfer
image from $\mathcal{X}$} \textbf{to $\F$ on $M$ }respectively
as
\begin{align*}
\lui{M}{\Tr_{H}^{\F}} & :=\tr_{\FHH}^{\F}\left(\End\left(M\downarrow_{\FHH}^{\F}\right)\right), & \text{and} &  & \lui{M}{\Tr_{\mathcal{X}}^{\F}} & :=\sum_{H\in\mathcal{X}}\lui{M}{\Tr_{H}^{\F}}.
\end{align*}
If there is no possible confusion we will simply write $\Tr_{H}^{\F}:=\lui{M}{\Tr_{H}^{\F}}$
and $\Tr_{\mathcal{X}}^{\F}:=\lui{M}{\Tr_{\mathcal{X}}^{\F}}$.
\end{defn}

\begin{lem}
\label{lem:Image-transfer-is-ideal.}With the notation of Definition
\ref{def:Image-transfer.}, both $\Tr_{H}^{\F}$ and $\Tr_{\mathcal{X}}^{\F}$
are two sided ideals of $\End\left(M\right)$.
\end{lem}

\begin{proof}
This is an immediate consequence of Proposition \ref{prop:Properties-of-transfer-restriction-and-conjugation-maps.}
\eqref{enu:Green-formula.}.
\end{proof}
We now have the following result reminiscent of Lemma \ref{lem:ifX->Y-and-X-proj.-then-Y-proj.}.
\begin{lem}
\label{lem:Transfer-inclusion.}Let $\mathcal{X}$ and $\mathcal{Y}$
be families of objects in $\Fc$, let $\sigma\colon\mathcal{X}\to\mathcal{Y}$
be a map between sets and let $\varPhi=\left\{ \varphi_{H}:H\to\sigma\left(H\right)\right\} _{H\in\mathcal{X}}$
be a family of morphisms in $\Fc$. Then, we have that $\Tr_{\mathcal{X}}^{\F}\subseteq\Tr_{\mathcal{Y}}^{\F}$
regardless of the associated centric Mackey functor.
\end{lem}

\begin{proof}
From definition of $\Tr_{\mathcal{X}}^{\F}$ and $\Tr_{\mathcal{Y}}^{\F}$
it suffices to prove the statement in the case where $\mathcal{X}:=\left\{ H\right\} $,
$\mathcal{Y}:=\left\{ K\right\} $ and $\Phi:=\left\{ \varphi\colon H\to K\right\} $
for some objects $H,K\in\Fc$ and some morphism $\varphi\in\F$. In
what follows we will view $\varphi$ as an isomorphism onto its image.
From Proposition \ref{prop:Properties-of-transfer-restriction-and-conjugation-maps.}
\eqref{enu:Transfer-eliminates-conjugation.} we have that $\Tr_{H}^{\F}=\tr_{\FHH[\varphi\left(A\right)]}^{\F}\left(\lui{\varphi}{\left(\End\left(M\downarrow_{\FHH}^{\F}\right)\right)}\right)$.
From Corollary \ref{cor:Conjugation-is-isomorphism.} we can conclude
that $\Tr_{H}^{\F}=\Tr_{\varphi\left(H\right)}^{\F}$. Finally, using
Proposition \ref{prop:Properties-of-transfer-restriction-and-conjugation-maps.}
\eqref{enu:composition-transfer.} on the groups $\varphi\left(H\right)\le K$
we can conclude that $\Tr_{H}^{\F}\subseteq\Tr_{K}^{\F}$ just as
we wanted to prove.
\end{proof}
We can now provide the following definition which, as we will later
see (Theorem \ref{thm:Higman's-criterion.}), is closely related to
Definition \ref{def:Relative-projectivity.}.
\begin{defn}
\label{def:relative-projectivity-transfer.}Let $M\in\MackFcR$, let
$f\in\End\left(M\right)$ and let $\mathcal{X}$ be a family of objects
in $\Fc$. We say that $f$ is \textbf{projective relative to $\mathcal{X}$
}(or $\mathcal{X}$\textbf{-projective}) if $f\in\Tr_{\mathcal{X}}^{\F}$.
If $\mathcal{X}=\left\{ H\right\} $ for some $H\in\Fc$ we will simply
say that $f$ is \textbf{projective relative to $H$} (or \textbf{$H$-projective}).
\end{defn}

Let $G$ be a finite group, let $H\le G$ and let $M$ be a Mackey
functor over $G$. Using Equation \eqref{eq:Induction-restriction-groups.}
we can define $\pi_{M}$ to be the natural projection of $M\uparrow_{H}^{G}\downarrow_{H}^{G}$
onto the summand $\left(\lui{1_{G}}{\left(M\downarrow_{H}^{H}\right)}\right)\uparrow_{H}^{H}\cong M$.
By composing it with the natural inclusion, the morphism $\pi_{M}$
can be seen as an endomorphism of $M\downarrow_{H}^{G}\uparrow_{H}^{G}$.
In order to prove Higman's criterion for Mackey functors over finite
groups (see \cite[Theorem 2.2]{NagaoHirosiRepresentationsOfFiniteGroups})
Hirosi and Tsushima use the identity $\tr_{H}^{G}\left(\pi_{M}\right)=\Id_{M\uparrow_{H}^{G}}$
where $\tr_{H}^{G}$ denotes the transfer map for Mackey functors
over finite groups (see \cite[Definition 2.7]{SASAKI198298}). In
order to prove Higman's criterion for centric Mackey functors over
fusion systems (and thus relate Definitions \ref{def:Relative-projectivity.}
and \ref{def:relative-projectivity-transfer.}) we will need a similar
result.
\begin{lem}
\label{lem:Induced-projection-is-identity.}Let $H\in\Fc$, let $M\in\MackFHR{\F}{\FHH}$
and let $\pi_{M}\in\End\left(M\uparrow_{\FHH}^{\F}\downarrow_{\FHH}^{\F}\right)$
be the composition of the projection onto the summand $\left(\lui{\Id_{H}}{\left(M\downarrow_{\FHH}^{\FHH}\right)}\right)\text{\ensuremath{\uparrow}}_{\FHH}^{\FHH}\cong M$
(see Lemma \ref{lem:Mackey-formula-induction-restriction.}) and the
natural inclusion. Then we have that $\tr_{\FHH}^{\F}\left(\pi_{M}\right)=\Id_{M\uparrow_{\FHH}^{\F}}$.
\end{lem}

\begin{proof}
From Definition \ref{def:Restriction-induction-and-conjugation-functors.}
we know that every element in $M\uparrow_{\FHH}^{\F}$ is of the form
$y\otimes x$ for some $y\in\muFR$ and some $x\in M$. Therefore,
since $\tr_{\FHH}^{\F}\left(\pi_{M}\right)$ is a morphism of $\muFR$-modules,
it suffices to prove that $\tr_{\FHH}^{\F}\left(\pi_{M}\right)\left(I_{K}^{K}\otimes x\right)=I_{K}^{K}\otimes x$
for every $K\in\FHH\cap\Fc$ and every $x\in I_{K}^{K}M$. Fix $x$
and $K$ as described. From definition of $\pi_{M}$ we have that
\[
\tr_{\FHH}^{\F}\left(\pi_{M}\right)\left(I_{K}^{K}\otimes x\right)=\sum_{\begin{array}{c}
{\scriptstyle \left(A,\overline{\varphi}\right)\in\left[H\times_{\F}K\right]}\\
{\scriptstyle c_{\overline{\varphi^{-1}}}R_{\overline{\varphi}\left(A\right)}^{K}\in\FmuFR{\FHH}}
\end{array}}I_{\overline{\varphi}\left(A\right)}^{K}c_{\overline{\varphi}}c_{\overline{\varphi^{-1}}}R_{\overline{\varphi}\left(A\right)}^{K}\otimes x.
\]
Since $K\le H$ by assumption, then we have that $c_{\overline{\varphi^{-1}}}R_{\overline{\varphi}\left(A\right)}^{K}\in\FmuFR{\FHH}$
if and only if $\overline{\varphi}\in\Orbitize{\FHH}$. For every
$\left(A,\overline{\varphi}\right)\in\left[H\times_{\F}K\right]$
satisfying $\overline{\varphi}\in\Orbitize{\FHH}$ we can assume without
loss of generality that $A\le K$ and that $\overline{\varphi}=\overline{\iota_{A}^{K}}$
(see Definition \ref{def:HX_FK}). From maximality of the pair $\left(A,\overline{\varphi}\right)$
(see again Definition \ref{def:HX_FK}) the previous description implies
that $A=K$. We can therefore conclude that there exists a unique
$\left(A,\overline{\varphi}\right)\in\left[H\times_{\F}K\right]$
such that $c_{\overline{\varphi^{-1}}}R_{\overline{\varphi}\left(A\right)}^{K}\in\FmuFR{\FHH}$.
Moreover $\left[H\times_{\F}K\right]$ can be taken so that this element
satisfies $c_{\overline{\varphi^{-1}}}R_{\overline{\varphi}\left(A\right)}^{K}=I_{K}^{K}$.
The result now follows from the equation above.
\end{proof}
We are now finally ready to translate Higman's criterion to the context
of centric Mackey functors over fusion systems.
\begin{thm}
\label{thm:Higman's-criterion.}(Higman's criterion) Let $\G$ be
a fusion system containing $\F$, let $M\in\MackFHR{\G}{\F}\subseteq\MackFcR$
(see Definition \ref{def:F-centric-Mackey-functor.}) be an indecomposable
Mackey functor and let $H\in\Fc$. The following are equivalent:
\begin{enumerate}
\item \label{enu:G-centric-summand.}There exists $N\in\MackFHR{\G}{\FHH}$
such that $M$ is a summand of $N\uparrow_{\FHH}^{\F}$ (see Definition
\ref{def:Restriction-induction-and-conjugation-functors.}).
\item \label{enu:summand.}There exists $N\in\MackFHR{\F}{\FHH}$ such that
$M$ is a summand of $N\uparrow_{\FHH}^{\F}$.
\item \label{enu:IdM-projective.}$\Id_{M}$ is $H$-projective (see Definition
\ref{def:relative-projectivity-transfer.}).
\item \label{enu:End(M)=00003DTr_H^F.}\textup{$\End\left(M\right)=\Tr_{H}^{\F}$
(see Definition \ref{def:Image-transfer.}).}
\item \label{enu:theta_H-is-epi-rel-proj.}$\theta_{H}$ (see Definition
\ref{def:theta_S-and-theta^S.}) is an epimorphism and, given $N,L\in\MackFcR$
and Mackey functor morphisms $\varphi\colon N\twoheadrightarrow L$
and $\psi\colon M\to L$ with $\varphi$ surjective, if there exists
a Mackey functor morphism $\gamma\colon M\downarrow_{\FHH}^{\F}\to N\downarrow_{\FHH}^{\F}$
such that $\varphi\downarrow_{\FHH}^{\F}\gamma=\psi\downarrow_{\FHH}^{\F}$
then there exists a Mackey functor morphism $\hat{\gamma}:M\to N$
such that $\varphi\hat{\gamma}=\psi$.
\item \label{enu:theta^H-is-mono-rel-inj.}$\theta^{H}$ (see Definition
\ref{def:theta_S-and-theta^S.}) is a monomorphism and, given $N,L\in\MackFcR$
and Mackey functor morphisms $\varphi\colon L\hookrightarrow N$ and
$\psi\colon L\to M$ with $\varphi$ injective, if there exists a
Mackey functor morphism $\gamma\colon N\downarrow_{\FHH}^{\F}\to M\downarrow_{\FHH}^{\F}$
such that $\gamma\varphi\downarrow_{\FHH}^{\F}=\psi\downarrow_{\FHH}^{\F}$
then there exists a Mackey functor morphism $\hat{\gamma}\colon N\to M$
such that $\hat{\gamma}\varphi=\psi$.
\item \label{enu:theta_H-is-epi-rel-split.}$\theta_{H}$ is an epimorphism
and, given $N\in\MackFcR$ and an epimorphism of Mackey functors $\varphi\colon N\twoheadrightarrow M$,
if $\varphi\downarrow_{\FHH}^{\F}$ splits then $\varphi$ splits.
\item \label{enu:theta^H-is-mono-rel-split.}$\theta^{H}$ is a monomorphism
and, given $N\in\MackFcR$ and a monomorphism of Mackey functors $\varphi\colon M\hookrightarrow N$,
if $\varphi\downarrow_{\FHH}^{\F}$ splits then $\varphi$ splits.
\item \label{enu:projective.}$\theta_{H}$ is split surjective (equivalently
$M$ is $H$-projective see Definition \ref{def:Relative-projectivity.}).
\item \label{enu:injective.}$\theta^{H}$ is split injective.
\item \label{enu:M-summand-M_H.}$M$ is a direct summand of $M_{H}$ (see
Definition \ref{def:theta_S-and-theta^S.}).
\end{enumerate}
\end{thm}

\begin{proof}
The proof is analogous to that of \cite[Theorem 2.2]{NagaoHirosiRepresentationsOfFiniteGroups}
except for some details in the proof of the implications \eqref{enu:summand.}$\Rightarrow$\eqref{enu:IdM-projective.},
\eqref{enu:theta_H-is-epi-rel-split.}$\Rightarrow$\eqref{enu:projective.}
and \eqref{enu:theta^H-is-mono-rel-split.}$\Rightarrow$\eqref{enu:injective.}
for which we will need to use Lemmas \ref{lem:Mackey-formula-induction-restriction.}
and \ref{lem:Induced-projection-is-identity.} in order to replace
analogous results for Mackey functors over finite groups.

\eqref{enu:G-centric-summand.}$\Rightarrow$\eqref{enu:summand.}.

Since $\F\subseteq\G$, then $\FHH\cap\G^{c}\subseteq\FHH\cap\Fc$
and, therefore, $\MackFHR{\G}{\FHH}\subseteq\MackFHR{\F}{\FHH}$.
The implication follows.

\eqref{enu:summand.}$\Rightarrow$\eqref{enu:IdM-projective.}.

Let $N\in\MackFHR{\F}{\FHH}$ such that there exists $L\in\MackFR$
satisfying $N\uparrow_{\FHH}^{\F}=M\oplus L$, Let $\pi_{M}$ be the
endomorphism of $N\uparrow_{\FHH}^{\F}$ given by the natural projection
onto $M$ followed by the natural inclusion and let $\pi_{N}\in\End\left(N\uparrow_{\FHH}^{\F}\downarrow_{\FHH}^{\F}\right)$
be the endomorphism of Lemma \ref{lem:Induced-projection-is-identity.}
satisfying $\lui{N}{\tr_{\FHH}^{\F}}=N\uparrow_{\FHH}^{\F}$. Since
restriction preserves direct sums then we have that $N\uparrow_{\FHH}^{\F}\downarrow_{\FHH}^{\F}=M\downarrow_{\FHH}^{\F}\oplus L\downarrow_{\FHH}^{\F}$
and that the endomorphism $\lui{N}{\r_{\FHH}^{\F}}\left(\pi_{M}\right)$
of $N\uparrow_{\FHH}^{\F}\downarrow_{\FHH}^{\F}$ is the projection
onto $M\downarrow_{\FHH}^{\F}$ followed by the natural inclusion.
We can now define $f\in\End\left(M\downarrow_{\FHH}^{\F}\right)$
by setting for every $x\in M\downarrow_{\FHH}^{\F}$
\[
f\left(x\right):=\lui{N}{\r_{\FHH}^{\F}}\left(\pi_{M}\right)\left(\pi_{N}\left(x\right)\right).
\]
Here we are seeing $M\downarrow_{\FHH}^{\F}$ as a subset of $N\uparrow_{\FHH}^{\F}\downarrow_{\FHH}^{\F}$
in order to apply $\pi_{N}$. With this setup, for every $K\in\Fc$
and every $x\in I_{K}^{K}M\subseteq I_{K}^{K}N$, we have that. 
\[
\lui{M}{\tr_{\FHH}^{\F}}\left(f\right)\left(x\right)=\lui{N}{\tr_{\FHH}^{\F}}\left(\lui{N}{\r_{\FHH}^{\F}}\left(\pi_{M}\right)\pi_{N}\right)\left(x\right)=\pi_{M}\lui{N}{\tr_{\FHH}^{\F}}\left(\pi_{N}\right)\left(x\right)=\pi_{M}\left(x\right)=x.
\]
where the first identity follows from definition, for the second identity
we are using Proposition \ref{prop:Properties-of-transfer-restriction-and-conjugation-maps.}
\eqref{enu:Green-formula.}, for the third we are using Lemma \ref{lem:Induced-projection-is-identity.}
and for the last we are using the fact that $x\in M$ and definition
of $\pi_{M}$. From the above we can conclude that $\lui{M}{\tr_{\FHH}^{\F}}\left(f\right)=\Id_{M}$
which implies that $\Id_{M}$ is $H$-projective thus proving the
implication.

\eqref{enu:IdM-projective.}$\Leftrightarrow$\eqref{enu:End(M)=00003DTr_H^F.}.

By definition we have that $\Id_{M}$ is $H$-projective if and only
if $\Id_{M}\in\Tr_{H}^{\F}$. From Lemma \ref{lem:Image-transfer-is-ideal.}
we know that $\Tr_{H}^{\F}$ is an ideal of $\End\left(M\right)$.
Therefore $\Tr_{H}^{\F}=\End\left(M\right)$ if and only if $\Id_{M}\in\Tr_{H}^{\F}$.
This proves that Items \eqref{enu:IdM-projective.} and \eqref{enu:End(M)=00003DTr_H^F.}
are equivalent.

\eqref{enu:IdM-projective.}$\Rightarrow$\eqref{enu:theta_H-is-epi-rel-proj.}.

If Item \eqref{enu:IdM-projective.} is satisfied then there exists
$f\in\End\left(M\downarrow_{\FHH}^{\F}\right)$ such that $\tr_{\FHH}^{\F}\left(f\right)=\theta_{H}^{M}f\uparrow_{\FHH}^{\F}\theta_{M}^{H}=\Id_{M}$
(see Definition \ref{def:Transfer-and-restriction.}). Therefore $\theta_{M}^{H}$
is a split injective and $\theta_{H}^{M}$ is split surjective. In
particular $\theta_{H}^{M}$ is surjective. Let $N,L,\varphi,\psi$
and $\gamma$ be as in the statement of item \eqref{enu:IdM-projective.}
and define $\hat{\gamma}:=\theta_{H}^{N}\left(\gamma f\right)\uparrow_{\FHH}^{\F}\theta_{M}^{H}$.
Then, for every $x\in N$ and every $y\in\muFR$, we have that 
\[
\varphi\left(\theta_{H}^{N}\left(\gamma\uparrow_{\FHH}^{\F}\left(y\otimes x\right)\right)\right)=y\varphi\downarrow_{\FHH}^{\F}\left(\gamma\left(x\right)\right)=y\psi\downarrow_{\FHH}^{\F}\left(x\right)=\theta_{H}^{M}\left(\psi\downarrow_{\FHH}^{\F}\uparrow_{\FHH}^{\F}\left(y\otimes x\right)\right).
\]
Where, for the first identity, we are using the fact that $\varphi$
is a $\muFR$-module morphism in order to get $\varphi\left(y\gamma\left(x\right)\right)=y\varphi\left(\gamma\left(x\right)\right)=y\varphi\downarrow_{\FHH}^{\F}\left(\gamma\left(x\right)\right)$.
The above equation proves that $\varphi\theta_{H}^{N}\gamma\uparrow_{\FHH}^{\F}=\theta_{H}^{M}\psi\downarrow_{\FHH}^{\F}\uparrow_{\FHH}^{\F}$.
The implication now follows from the identities below
\[
\varphi\hat{\gamma}=\varphi\theta_{H}^{N}\left(\gamma f\right)\uparrow_{\FHH}^{\F}\theta_{M}^{H}=\theta_{H}^{M}\left(\psi\downarrow_{\FHH}^{\F}f\right)\uparrow_{\FHH}^{\F}\theta_{M}^{H}=\tr_{\FHH}^{\F}\left(\r_{\FHH}^{\F}\left(\psi\right)f\right)=\psi\tr_{\FHH}^{\F}\left(f\right)=\psi.
\]
Where, for the third identity, we are using Definition \ref{def:Transfer-and-restriction.}
while, for the fourth identity, we are using Proposition \ref{prop:Properties-of-transfer-restriction-and-conjugation-maps.}
\eqref{enu:Green-formula.}.

\eqref{enu:IdM-projective.}$\Rightarrow$\eqref{enu:theta^H-is-mono-rel-inj.}.

Let $f$ be as in the previous implication. As before we have that
$\theta_{M}^{H}$ is split injective and, in particular, it is injective.
Let $N,L,\varphi,\psi$ and $\gamma$ be as in the statement of Item
\eqref{enu:IdM-projective.} and define $\hat{\gamma}:=\theta_{H}^{M}\left(f\gamma\right)\uparrow_{\FHH}^{\F}\theta_{N}^{H}$.
Then, for every $K\in\Fc$ and every $x\in I_{K}^{K}M$, we have that
\[
\left(\gamma\uparrow_{\FHH}^{\F}\theta_{N}^{H}\varphi\right)\left(x\right)=\sum_{\left(A,\overline{\alpha}\right)\in\left[H\times_{\F}K\right]}I_{\overline{\alpha}\left(A\right)}^{K}c_{\overline{\alpha}}\otimes\gamma\varphi\downarrow_{\FHH}^{\F}\left(c_{\overline{\alpha^{-1}}}R_{\overline{\alpha}\left(A\right)}^{K}x\right)=\left(\psi\downarrow_{\FHH}^{\F}\uparrow_{\FHH}^{\F}\theta_{M}^{H}\right)\left(x\right).
\]
Where, for the second identity, we are using the identity $\gamma\varphi\downarrow_{\FHH}^{\F}=\psi\downarrow_{\FHH}^{\F}$
while, for the first identity, we are using that $\varphi$ is a morphism
of $\muFR$-modules in order to get $c_{\overline{\alpha^{-1}}}R_{\overline{\alpha}\left(A\right)}^{K}\varphi\left(x\right)=\varphi\left(c_{\overline{\alpha^{-1}}}R_{\overline{\alpha}\left(A\right)}^{K}x\right)$
and we are using that $c_{\overline{\alpha^{-1}}}R_{\overline{\alpha}\left(A\right)}^{K}x\in M\downarrow_{\FHH}^{\F}$
in order to write $\varphi\downarrow_{\FHH}^{\F}$ instead of $\varphi$.
The above equation proves that $\gamma\uparrow_{\FHH}^{\F}\theta_{N}^{H}\varphi=\psi\downarrow_{\FHH}^{\F}\uparrow_{\FHH}^{\F}\theta_{M}^{H}$.
The implication now follows from the identities below
\[
\hat{\gamma}\varphi=\theta_{H}^{M}\left(f\gamma\right)\uparrow_{\FHH}^{\F}\theta_{N}^{H}\varphi=\theta_{H}^{M}\left(f\psi\downarrow_{\FHH}^{\F}\right)\uparrow_{\FHH}^{\F}\theta_{M}^{H}=\tr_{\FHH}^{\F}\left(f\r_{\FHH}^{\F}\left(\psi\right)\right)=\tr_{\FHH}^{\F}\left(f\right)\psi=\psi.
\]
Where, for the third identity, we are using Definition \ref{def:Transfer-and-restriction.}
while, for the fourth identity, we are using Proposition \ref{prop:Properties-of-transfer-restriction-and-conjugation-maps.}
\eqref{enu:Green-formula.}.

\eqref{enu:theta_H-is-epi-rel-proj.}$\Rightarrow$\eqref{enu:theta_H-is-epi-rel-split.}.

With the notation of Item \eqref{enu:theta_H-is-epi-rel-proj.} let
$L:=M$ and $\psi:=\Id_{M}$. Since $\varphi\downarrow_{\FHH}^{\F}$
splits then there exists $\gamma\colon M\downarrow_{\FHH}^{\F}\to N\downarrow_{\FHH}^{\F}$
such that $\varphi\downarrow_{\FHH}^{\F}\gamma=\Id_{M\downarrow_{\FHH}^{\F}}=\psi\downarrow_{\FHH}^{\F}$.
Therefore, by hypothesis, there exists $\hat{\gamma}:M\to N$ such
that $\varphi\hat{\gamma}=\psi=\Id_{M}$. In other words $\varphi$
splits.

\eqref{enu:theta^H-is-mono-rel-inj.}$\Rightarrow$\eqref{enu:theta^H-is-mono-rel-split.}.

With notation as in Item \eqref{enu:theta^H-is-mono-rel-inj.} let
$L:=M$, $\psi:=\Id_{M}$ and $\gamma\colon N\to M$ such that $\gamma\varphi\downarrow_{\FHH}^{\F}=\Id_{M\downarrow_{\FHH}^{\F}}=\psi\downarrow_{\FHH}^{\F}$.
Then, by hypothesis, there exists $\hat{\gamma}\colon N\to M$ such
that $\hat{\gamma}\varphi=\Id_{M}$. In other words $\varphi$ splits.

\eqref{enu:theta_H-is-epi-rel-split.}$\Rightarrow$\eqref{enu:projective.}.

Let $f\colon M\downarrow_{\FHH}^{\F}\hookrightarrow M_{H}\downarrow_{\FHH}^{\F}$
be the $\FmuFR{\FHH}$-module morphism given by Lemma \ref{lem:Mackey-formula-induction-restriction.}
and that sends $M\downarrow_{\FHH}^{\F}$ isomorphically into the
summand $\left(\lui{\Id_{H}}{\left(M\downarrow_{\FHH}^{\F}\right)}\right)\text{\ensuremath{\uparrow}}_{\FHH}^{\FHH}$
of $M_{H}\downarrow_{\FHH}^{\F}$. With this setup we have that $\theta_{H}\downarrow_{\FHH}^{\F}f=\Id_{M}\downarrow_{\FHH}^{\F}$.
Item \eqref{enu:projective.} now follows from Item \eqref{enu:theta_H-is-epi-rel-split.}
by taking $N:=M_{H}$ and $\varphi=\theta_{H}$.

\eqref{enu:theta^H-is-mono-rel-split.}$\Rightarrow$\eqref{enu:injective.}.

From Lemma \ref{lem:Mackey-formula-induction-restriction.} we can
take $\pi\colon M_{H}\downarrow_{\FHH}^{\F}\to M\downarrow_{\FHH}^{\F}$
to be the natural projection onto the summand $M\downarrow_{\FHH}^{\F}\cong\left(\lui{\Id_{H}}{\left(M\downarrow_{\FHH}^{\F}\right)}\right)\text{\ensuremath{\uparrow}}_{\FHH}^{\FHH}$.
Dually to the previous implication we have that $\pi\theta^{H}\downarrow_{\FHH}^{\F}=\Id_{M}\downarrow_{\FHH}^{\F}$.
Item \eqref{enu:injective.} now follows from Item \eqref{enu:theta^H-is-mono-rel-split.}
by taking $N:=M_{H}$ and $\varphi=\theta^{H}$.

\eqref{enu:projective.}$\Rightarrow$\eqref{enu:M-summand-M_H.}
and \eqref{enu:injective.}$\Rightarrow$\eqref{enu:M-summand-M_H.}.

The fact that $M$ is a summand of $M_{H}$ is an immediate consequence
of either $\theta_{H}$ being split surjective (Item \eqref{enu:projective.})
or $\theta^{H}$ being split injective (Item \eqref{enu:injective.}).

\eqref{enu:M-summand-M_H.}$\Rightarrow$\eqref{enu:G-centric-summand.}.

From Proposition \ref{prop:centric-Induction-restriction.} we know
that $N:=M\downarrow_{\FHH}^{\F}$ is $\G$-centric and, from Item
\eqref{enu:M-summand-M_H.} we have that $M$ is a summand of $N\uparrow_{\FHH}^{\F}=M_{H}$.
\end{proof}
\begin{rem}
The equivalence \eqref{enu:summand.}$\Leftrightarrow$\eqref{enu:G-centric-summand.}
of Theorem \ref{thm:Higman's-criterion.} can be proven independently
from the rest.
\end{rem}

We conclude this section with the following result which will allow
us to always talk about the vertex of an indecomposable centric Mackey
functor over a fusion system.
\begin{cor}
\label{cor:irreducible-admits-vertex.}Let $\R$ be a complete local
and $p$-local $PID$, let $\G$ be a fusion system containing $\F$
and let $M\in\MackFHR{\G}{\F}$ be an indecomposable Mackey functor.
Then $M$ admits a vertex (see Definition \ref{def:vertex.}). Moreover
$V_{M}\in\F\cap\G^{c}$ and, for any $N\in\MackFcR$ such that $M$
is a summand of $N$, we have that $V_{M}\in\mathcal{X}_{N}$.
\end{cor}

\begin{proof}
By definition of defect set we know that the map $\theta_{\mathcal{X}_{M}}^{M}\colon\bigoplus_{H\in\mathcal{X}_{M}}M_{H}\to M$
is split surjective, in particular $M$ is a summand of $\bigoplus_{H\in\mathcal{X}_{M}}M_{H}$.
Since $\R$ is a complete local $PID$ and $M$ is indecomposable,
then we can apply the Krull-Schmidt-Azumaya theorem (see \cite[Theorem 6.12 (ii)]{MethodsOfRepresentationTheoryCurtisReiner})
in order to deduce that there exists $H\in\mathcal{X}_{M}$ such that
$M$ is a summand of $M_{H}$. Because of Theorem \ref{thm:Higman's-criterion.}
this implies that $M$ is $H$-projective. Since $M$ is $\G$-centric
then $M_{H}=0$ for every $H\in\F\backslash\left(\F\cap\G^{c}\right)$.
Therefore we necessarily have $H\in\F\cap\G^{c}$. Define $\mathcal{X}_{H}:=\left\{ K\in\Fc\,:\,K\le_{\F}H\right\} $.
Since $M$ is $H$-projective we can deduce from Corollary \ref{cor:projective-respect-to-bigger.}
\eqref{enu:X<Y implies Y-proj.} that $M$ is also $\mathcal{X}_{H}$-projective.
From minimality of $\mathcal{X}_{M}$ (see Definition \ref{def:vertex.})
this implies that $\mathcal{X}_{M}\subseteq\mathcal{X}_{H}$. Since
$\mathcal{X}_{M}$ is closed under $\F$-subconjugacy and $H\in\mathcal{X}_{M}$
we also have that $\mathcal{X}_{H}\subseteq\mathcal{X}_{M}$ and,
therefore, $\mathcal{X}_{H}=\mathcal{X}_{M}$. By construction of
$\mathcal{X}_{H}$ this is equivalent to saying that $M$ admits a
vertex (namely any fully $\F$-normalized $K=_{\F}H$).

Let $N$ be as in the statement and let $L\in\MackFcR$ such that
$N=M\oplus L$. Since induction and restriction preserve direct sum
decomposition we have that $N_{\mathcal{X}_{N}}=M_{\mathcal{X}_{N}}\oplus L_{\mathcal{X}_{N}}$.
Immediately from the definition of $\theta_{\mathcal{X}_{N}}^{N}$
we have that $\theta_{\mathcal{X}_{N}}^{N}\left(M_{\mathcal{X}_{N}}\right)\subseteq M$
and that $\theta_{\mathcal{X}_{N}}^{N}\left(L_{\mathcal{X}_{N}}\right)\subseteq L$.
Moreover, the restriction of $\theta_{\mathcal{X}_{N}}^{N}$ to $M_{\mathcal{X}_{N}}$
and $L_{\mathcal{X}_{N}}$ coincide with $\theta_{\mathcal{X}_{N}}^{M}$
and $\theta_{\mathcal{X}_{N}}^{L}$ respectively. In other words we
have that $\theta_{\mathcal{X}_{N}}^{N}=\theta_{\mathcal{X}_{N}}^{M}\pi_{M}+\theta_{\mathcal{X}_{N}}^{L}\pi_{L}$
where $\pi_{M}$ and $\pi_{L}$ denote the natural projections onto
$M_{\mathcal{X}_{N}}$ and $L_{\mathcal{X}_{N}}$ respectively. On
the other hand, from definition of defect set, we know that there
exists a Mackey functor morphism $u\colon N\to N_{\mathcal{X}_{N}}$
such that $\theta_{\mathcal{X}_{N}}^{N}u=\Id_{N}$. Therefore, denoting
by $u_{|M}$ the restriction of $u$ to $M$, we obtain the identity
$\left(\theta_{\mathcal{X}_{N}}^{M}\pi_{M}+\theta_{\mathcal{X}_{N}}^{L}\pi_{L}\right)u_{|M}=\Id_{M}$.
Since $\theta_{\mathcal{X}_{N}}^{L}$ maps to $L$ and $L\cap M=\left\{ 0\right\} $
we can conclude that $\theta_{\mathcal{X}_{N}}^{L}\pi_{L}u_{|M}=0$
and, therefore, $\theta_{\mathcal{X}_{N}}^{M}\pi_{M}u_{|M}=\Id_{M}$.
In particular $\theta_{\mathcal{X}_{N}}^{M}$ is split surjective
or, equivalently, $M$ is $\mathcal{X}_{N}$-projective. From minimality
of the defect set we can then conclude that $V_{M}\in\mathcal{X}_{N}$.
\end{proof}

\section{\label{sec:Green-correspondence.}Green correspondence.}

In this section we will prove the main result of this paper. More
precisely we will prove that the Green holds for centric Mackey functors
over fusion systems (see Theorem \ref{thm:Green-correspondence.}).

We will start in Subsection \ref{subsec:Categorical-version.} by
proving Proposition \ref{prop:Green-correspondence-for-endomorphisms.}
which will give us a list of sufficient conditions to prove a Green
correspondence like result. Subsections \ref{subsec:N-is-a-direct-summand-of-Ninduction-restriction.}
to \ref{subsec:Transfer-map-from-trHNF-to-trHF.} will then be dedicated
to building the tools needed in order to prove that Proposition \ref{prop:Green-correspondence-for-endomorphisms.}
can be applied to endomorphism rings of $\F$-centric Mackey functors.

Finally we will conclude with Subsection \ref{subsec:Proof-of-Green-correspondence.}
where we will use Proposition \ref{prop:Green-correspondence-for-endomorphisms.}
together with Theorem \ref{thm:Higman's-criterion.} in order to translate
Green correspondence to the context of centric Mackey functors over
fusion systems (see Theorem \ref{thm:Green-correspondence.}).

\subsection{\label{subsec:Categorical-version.}Correspondence of endomorphisms.}

The goal of this subsection is that of stating and proving Proposition
\ref{prop:Green-correspondence-for-endomorphisms.}. This result will
become in Subsection \ref{subsec:Proof-of-Green-correspondence.}
one of the key-stones for proving Theorem \ref{thm:Green-correspondence.}.

Let's start with some notation.
\begin{defn}
Let $A$ and $B$ be rings (not necessarily having a unit) and let
$f\colon A\twoheadrightarrow B$ be a surjective ring morphism. We
say that $f$ is a \textbf{near isomorphism }if $A\ker\left(f\right)=\ker\left(f\right)A=0$.
\end{defn}

The following Lemmas will be useful later on and provide examples
of near isomorphisms.
\begin{lem}
\label{lem:Iso-is-near-iso.}Let $A$ and $B$ be rings (not necessarily
having a unit) and let $f\colon A\to B$ be a ring morphism. If $f$
is an isomorphism then it is a near isomorphism and if $f$ is a near
isomorphism and $A$ has a unit then $f$ is an isomorphism.
\end{lem}

\begin{proof}
If $f$ is an isomorphism it is surjective and $\ker\left(f\right)=0$.
In particular $A\ker\left(f\right)=\ker\left(f\right)A=0$ and, therefore,
$f$ is a near isomorphism. Assume now that $f$ is a near isomorphism
and $A$ has a unit. Then, for every $x\in\ker\left(f\right)$, we
have that $x1_{A}=0$ and, therefore, $\ker\left(f\right)=0$. Thus
$f$ is injective. Since $f$ is also surjective by definition of
near isomorphism then it is an isomorphism thus concluding the proof.
\end{proof}
\begin{lem}
\label{lem:Projection-is-near-iso.}Let $A$ be a ring (not necessarily
having a unit) and let $I$ and $J$ be two sided ideals of $A$ such
that $I\subseteq J$ and that $JA,AJ\subseteq I$. Then the natural
surjective ring morphism $f\colon A/I\twoheadrightarrow A/J$ is a
near isomorphism.
\end{lem}

\begin{proof}
For every $C\subseteq A$ denote by $\overline{C}$ the image of $C$
under the natural projection onto $A/I$. Then, by construction, we
have that $\ker\left(f\right)=\overline{J}$. Since $AJ,JA\subseteq I$
we have that $\overline{A}\,\overline{J}=\overline{J}\,\overline{A}=\overline{I}=\overline{0}$
thus concluding the proof.
\end{proof}
\begin{lem}
\label{lem:Factors-of-near-iso-are-near-iso.}Let $A,B$ and $C$
be rings (not necessarily having a unit) and let $f\colon A\to B$
and $g\colon B\to C$ be ring homomorphisms. If $gf$ is a near isomorphism
and $f$ is surjective then both $f$ and $g$ are near isomorphisms.
\end{lem}

\begin{proof}
First of all notice that $\ker\left(f\right)\subseteq\ker\left(gf\right)$.
Since $gf$ is a near isomorphims then we have that $A\ker\left(f\right)\subseteq A\ker\left(gf\right)=0$
and that $\ker\left(f\right)A\subseteq\ker\left(gf\right)A=0$. Since
$f$ is surjective by hypothesis then we can conclude that $f$ is
a near isomorphism.

On the other hand, since $gf$ is a near isomorphism, then it is surjective
and, therefore, $g$ is also surjective. Since $f$ is surjective,
then we have that $\ker\left(g\right)=f\left(\ker\left(gf\right)\right)$
and $B=f\left(A\right)$. Therefore we can conclude that $B\ker\left(g\right)=f\left(A\ker\left(gf\right)\right)=0$
and that $\ker\left(g\right)B=f\left(\ker\left(gf\right)A\right)=0$
thus concluding the proof.
\end{proof}
\begin{lem}
\label{lem:Composition-of-near-iso-and-iso-are-near-iso.}Let $A,B$
and $C$ be rings (not necessarily having a unit), let $f\colon A\twoheadrightarrow B$
be a near isomorphism and let $g\colon B\bjarrow C$ be an isomorphism
then $gf$ is a near isomorphism.
\end{lem}

\begin{proof}
Since both $f$ and $g$ are surjective, then $h:=gf$ is also surjective.
The result follows from applying Lemma \ref{lem:Factors-of-near-iso-are-near-iso.}
to $f=g^{-1}h$.
\end{proof}
The importance of near isomorphisms comes from the following well
known lemma due to Green which we state without proving.
\begin{lem}
\label{lem:Near-isomorphisms-preserve-idempotents.}(\cite[Lemma 4.22]{GreenAxiomaticRepresentation})
Let $A$ and $B$ be $\R$-algebras and let $f\colon A\twoheadrightarrow B$
be a near isomorphism. Denote by $E\left(A\right)$ and $E\left(B\right)$
the sets of idempotents of $A$ and $B$ respectively. Then the following
are satisfied
\begin{enumerate}
\item \label{enu:bijection.}$f$ induces a bijection from $E\left(A\right)$
to $E\left(B\right)$.
\item \label{enu:local-to-local.}Let $x\in E\left(A\right)$ be a local
idempotent. Then $f(x)\in E\left(B\right)$ is also a local idempotent.
\item \label{enu:Conjugate-iff-conjugate.}Let $x,y\in E\left(A\right)$
be idempotents. Then $x$ and $y$ are conjugate in $A$ if and only
if $f\left(x\right)$ and $f\left(y\right)$ are conjugate in $B$.
\end{enumerate}
\end{lem}

With this in mind we can now prove the main result of this subsection.
\begin{prop}
\label{prop:Green-correspondence-for-endomorphisms.}Let $A$ and
$B$ be $\R$-algebras, let $C,J$ be a two sided ideals of $A$,
let $I$ and $K$ be two sided ideals of $C$ and $B$ respectively
($C$ seen as a ring with potentially no unit) and let $f\colon C\to B$
and $g\colon B\to C+J$ be $\R$-linear maps. Assume that the following
are satisfied:
\begin{enumerate}
\item \label{enu:condition-1.}$\left(C\cap J\right)C,C\left(C\cap J\right)\subseteq I\subseteq C\cap J$,
\item \label{enu:condition-2.}$g\left(K\right)\subseteq J$,
\item \label{enu:condition-3.}$f\left(I\right)\subseteq K$,
\item \label{enu:condition-4.}$f$ is surjective.
\item \label{enu:g-preserves-idempotents.}$g$ sends idempotents to idempotents.
\item \label{enu:condition-6.}The $\R$-linear maps $\overline{f}\colon C/I\to B/K$
and $\overline{g}\colon B/K\to\left(C+J\right)/J$ induced by $f$
and $g$ respectively commute with multiplication (i.e. $\overline{f}\left(xy\right)=\overline{f}\left(x\right)\overline{f}\left(y\right)$
and $\overline{g}\left(vw\right)=\overline{f}\left(v\right)\overline{f}\left(w\right)$
for every $x,y\in C/I$ and every $v,w\in B/K$).
\item \label{enu:condition-7.}The natural isomorphism $s\colon C/\left(C\cap J\right)\to\left(C+J\right)/J$
and the natural projection $q\colon C/I\to C/\left(C\cap J\right)$
satisfy $sq=\overline{g}\overline{f}$.
\item \label{enu:Decomposition-idempotent.}For every idempotent $x\in A$
there exists a unique (up to conjugation) decomposition of $x$ as
a finite sum of orthogonal local idempotents.
\end{enumerate}
Let $b\in B$ be a local idempotent such that $b\not\in K$. Then
$g\left(b\right)\in C+J\subseteq A$ and, from Conditions \eqref{enu:g-preserves-idempotents.}
and \eqref{enu:Decomposition-idempotent.}, there exists a unique
$n\in\mathbb{N}$ and a unique (up to conjugation) set of orthogonal
local idempotents $\left\{ a_{0},\dots,a_{n}\right\} \subseteq A$
such that 
\[
g\left(b\right)=\sum_{i=0}^{n}a_{i}.
\]
With this notation there exists exactly one value $j\in\left\{ 0,\dots,n\right\} $
such that $a_{j}\in C\backslash\left(C\cap J\right)$. Moreover, if
we define $a:=a_{j}$, we have that
\begin{align*}
g(b) & \equiv a\mod J, & f\left(a\right) & \equiv b\mod K.
\end{align*}
\end{prop}

\begin{proof}
Since both $C$ and $J$ are two sided ideals of $A$ then $C+J$
is also a two sided ideal of $A$. With notation as in the statement,
since all the $a_{i}$ are pairwise orthogonal, then, for every $i=0,\dots,n$,
we have that $a_{i}=a_{i}g\left(b\right)$ and, since $g\left(b\right)\in C+J$,
we can conclude that $a_{i}\in C+J$. Since $C+J$ is a two sided
ideal of $A$ we can conclude that $a_{i}Aa_{i}\subseteq\left(C+J\right)$.
Since $C+J\subseteq A$ and each $a_{i}$ is an idempotent we obtain
the other inclusion and, therefore, we obtain the identity
\[
a_{i}\left(C+J\right)a_{i}=a_{i}Aa_{i}.
\]
In particular, since $a_{i}$ is a local idempotent of $A$, we have
that $a_{i}\left(C+J\right)a_{i}$ is a local ring and, since $a_{i}\in C+J$,
we can conclude that each $a_{i}$ is a local idempotent of $C+J$
(and not just of $A$). 

Since $b\not\in K$ by hypothesis, then the projection $\overline{b}$
of $b$ onto $B/K$ is non zero. Since, by hypothesis, $b$ is a local
idempotent then we can conclude that $\overline{b}$ is also a local
idempotent (because quotients of local rings are still local rings).
Likewise, for every $i=0,\dots,n$, we have that the projection $\overline{a_{i}}$
of $a_{i}$ onto $\left(C+J\right)/J$ is either $0$ or a local idempotent
of $\left(C+J\right)/J$.

From Lemma \ref{lem:Projection-is-near-iso.} and Condition \eqref{enu:condition-1.}
we know that the natural projection $q$ of Condition \eqref{enu:condition-7.}
is a near isomorphism. From Lemma \ref{lem:Iso-is-near-iso.} we know
that $s$ is also a near isomorphism. From Lemma \ref{lem:Composition-of-near-iso-and-iso-are-near-iso.}
and Condition \eqref{enu:condition-7.} we can conclude that $\overline{g}\overline{f}$
is also a near isomorphism. Finally, from Lemma \ref{lem:Factors-of-near-iso-are-near-iso.}
and Condition \eqref{enu:condition-4.}, we can conclude that $\overline{f}$
and $\overline{g}$ are near isomorphisms. Since $\overline{b}$ is
a local idempotent then we can conclude from Lemma \ref{lem:Near-isomorphisms-preserve-idempotents.}
\eqref{enu:local-to-local.} that $\overline{g}\left(\overline{b}\right)=\sum_{i=0}^{n}\overline{a_{i}}$
is also a local idempotent. Since local idempotents are primitive
we can conclude that there exists exactly one $j\in\left\{ 0,\dots,n\right\} $
such that $\overline{a_{j}}\not=0$. We can assume without loss of
generality that $j=0$ and define $a:=a_{0}$. In other words we have
that $\overline{g}\left(\overline{b}\right)=\overline{a}$ (equivalently
$g\left(b\right)\equiv a\mod\left(J\right)$) while for every $i=1,\dots,n$
we have that $\overline{a_{i}}=0$ (equivalently $a_{i}\in J$). This
proves the first equivalence. Since $a\not\in J$ (because $\overline{a}\not=0$)
then, in order to complete the proof, we just need to prove that $a\in C$
and that the second equivalence is satisfied.

Since both $C$ and $J$ are two sided ideals of $A$ then we can
deduce that $aCa$ and $aJa$ are two sided ideals of $a\left(C+J\right)a$.
Since $a$ is a local idempotent of $C+J$, then, by definition, we
have that $a\left(C+J\right)a$ is a local ring. Notice also that,
from the distributive property of the product, we have that $aCa+aJa=a\left(C+J\right)a$.
From definition of local ring we can conclude that either
\begin{align*}
a\left(C+J\right)a & =aCa\subseteq C, & \text{ or } &  & a\left(C+J\right)a & =aJa\subseteq J.
\end{align*}
Since $a$ is an idempotent and $a\not\in J$ then we can conclude
that the identity on the right in the above equation is not possible.
Therefore the identity on the left must be satisfied and we can conclude
that $a\in C\backslash\left(C\cap J\right)$.

In order to complete the proof we are just left with proving that
$f\left(a\right)$ is equivalent to $b$ modulo $K$. Denote by $\overline{\overline{a}}$
the projection of $a$ on $C/I$. Since $a$ is an idempotent then
$\overline{\overline{a}}$ must also be an idempotent and, from Condition
\eqref{enu:condition-6.} we can deduce that $\overline{f}\left(\overline{\overline{a}}\right)$
is an idempotent. On the other hand, from the first part of the proof,
we know that $\overline{a}=\overline{g}\left(\overline{b}\right)$.
Thus, from Condition \eqref{enu:condition-7.}, we can deduce that
\[
\overline{g}\left(\overline{f}\left(\overline{\overline{a}}\right)\right)=s\left(q\left(\overline{\overline{a}}\right)\right)=\overline{a}=\overline{g}\left(\overline{b}\right).
\]
Since $\overline{g}$ is a near isomorphism (as already proven) then,
from the above identities and Lemma \ref{lem:Near-isomorphisms-preserve-idempotents.}
\eqref{enu:bijection.}, we can conclude that $\overline{f}\left(\overline{\overline{a}}\right)=\overline{b}$.
From Condition \eqref{enu:condition-3.} and definition of $\overline{f}$
this is equivalent to saying that $f\left(a\right)$ is equivalent
to $b$ modulo $K$. This concludes the proof.
\end{proof}
Let's conclude this subsection by giving an example where Proposition
\ref{prop:Green-correspondence-for-endomorphisms.} is used in order
to prove that Green correspondence holds for Green functors (see \cite[Proposition 4.34]{GreenAxiomaticRepresentation}).
\begin{example}
\label{exa:Green.correspondence-Green-functors.}Let $\R$ be a complete
local $PID$, let $G$ be a finite group, let $D,H\le G$ be subgroups
such that $\NG[D]\le H$ and let $M$ be a Green functor over $G$
on $\R$ (see the first definition of \cite[Subsection 1.3]{GreenAxiomaticRepresentation}).
With the notation of Proposition \ref{prop:Green-correspondence-for-endomorphisms.}
we can define 
\begin{align*}
A & :=\End\left(M\downarrow_{H}^{G}\right), & B & :=\tr_{D}^{G}\left(\End\left(M\downarrow_{D}^{G}\right)\right),\\
C & :=\tr_{D}^{H}\left(\End\left(M\downarrow_{D}^{G}\right)\right), & K & :=\sum_{x\in G-H}\tr_{D^{x}\cap D}^{G}\left(\End\left(M\downarrow_{D^{x}\cap D}^{G}\right)\right),\\
I & :=\sum_{x\in G-H}\tr_{D^{x}\cap D}^{H}\left(\End\left(M\downarrow_{D^{x}\cap D}^{G}\right)\right), & J & :=\sum_{x\in G-H}\tr_{D^{x}\cap H}^{H}\left(\End\left(M\downarrow_{D^{x}\cap H}^{G}\right)\right),\\
f & :=\tr_{H}^{G}, & g & :=\r_{H}^{G}.
\end{align*}
With this setup the Green correspondence for Green functors (see \cite[Proposition 4.34]{GreenAxiomaticRepresentation})
follows from Proposition \ref{prop:Green-correspondence-for-endomorphisms.}
and the first remark after \cite[Hypothesis 4.31]{GreenAxiomaticRepresentation}.
\end{example}

\subsection{\label{subsec:N-is-a-direct-summand-of-Ninduction-restriction.}Composing
induction and restriction.}

We have seen in Subsection \ref{subsec:Mackey-functors-over-fusion-systems.}
that, when working with Mackey functors over finite groups, there
exists a way of rewriting the composition of induction and restriction
functors (see Equation \eqref{eq:Induction-restriction-groups.}).
In that same subsection we have proven (see Lemma \ref{lem:Mackey-formula-induction-restriction.})
that a similar result holds for centric Mackey functors over fusion
systems when composing induction functors of the form $\uparrow_{\FHH}^{\F}$
with restriction functors of the form $\downarrow_{\FHH[K]}^{\F}$
for some $H,K\in\Fc$. However, we haven't shown any result regarding
compositions of induction and restriction functors when the fusion
systems $\FHH$ and $\FHH[K]$ of Lemma \ref{lem:Mackey-formula-induction-restriction.}
are replaced with other fusion subsystems of $\F$. The goal of this
subsection will be to do exactly that. More precisely, let $H\in\Fc$
be fully $\F$-normalized, let $M\in\MackFHR{\F}{\FHH}$ and let $N\in\MackFHR{\F}{\NFH}$
(see Example \ref{exa:definition-NF.}). In this subsection we will
study the $\F$-centric Mackey functors of the form $M\uparrow_{\FHH}^{\F}\downarrow_{\NFH}^{\F}$
(see Lemma \ref{lem:Mackey-formula-induction-restriction-to-normalizer.})
and $N\uparrow_{\NFH}^{\F}\downarrow_{\NFH}^{\F}$ (see Lemma \ref{lem:induction-from-normalizer-followed-by-restriction-to-normalizer.}).

Before proceeding let us introduce some notation that will be used
throughout the rest of this document.
\begin{notation}
\label{nota:X-and-Y.}From now on and unless otherwise specified $H$
\textbf{will denote a fully $\F$-normalized, $\F$-centric subgroup
of $S$}, we will denote \textbf{the normalizer of $H$ in $S$ }(i.e.
$\NSH$) simply \textbf{as $\NS$}, we will denote \textbf{the normalizer
fusion system $\NFH$} (see Example \ref{exa:definition-NF.}) simply
\textbf{as $\NF$} and $\mathcal{X}$ and $\mathcal{Y}$ will denote
the following sets
\begin{align*}
\mathcal{Y} & :=\left\{ K\le_{\F}H\,:\,K\le\NS\text{, }K\in\Fc\text{ and }K\not=H\right\} ,\\
\mathcal{X} & :=\left\{ K\lneq H\,:\,K\in\Fc\right\} =\left\{ K\in\mathcal{Y}\,:\,K\le H\right\} .
\end{align*}
\end{notation}

\begin{lem}
\label{lem:in-onf-iff-lifting-in-onf.}Let $\left(A,\overline{\varphi}\right)\in\left[H\times_{\F}\NS\right]$,
fix a representative $\varphi$ of $\overline{\varphi}$, let $K\in\Fc\cap\NF$,
let $\left(B,\overline{\psi}\right)\in\left[\varphi\left(A\right)\times_{\NF}K\right]$
such that $B\in\FHH[\varphi\left(A\right)]\cap\Fc$ and denote by
$\tilde{\varphi}\colon\tilde{\varphi}^{-1}\left(B\right)\to B$ the
morphism $\varphi$ seen as an isomorphism between the given subgroups
(i.e. the unique morphism such that $\varphi\iota_{\tilde{\varphi}^{-1}\left(B\right)}^{A}=\iota_{B}^{\NS}\tilde{\varphi}$).
From the universal properties of products we know that there exist
a unique $\left(B^{\F,\varphi},\overline{\psi^{\F,\varphi}}\right)\in\left[H\times_{\F}K\right]$
and a unique morphism $\overline{\gamma_{\left(B,\overline{\psi}\right)}^{\F,\varphi}}\colon\tilde{\varphi}^{-1}\left(B\right)\to B^{\F,\varphi}$
such that $\overline{\iota_{B^{\F,\varphi}}^{H}}\overline{\gamma_{\left(B,\overline{\psi}\right)}^{\F,\varphi}}=\overline{\iota_{\varphi^{-1}\left(B\right)}^{H}}$
and $\overline{\psi^{\F,\varphi}}\overline{\gamma_{\left(B,\overline{\psi}\right)}^{\F,\varphi}}=\overline{\psi\tilde{\varphi}}$.
With this setup the morphism $\overline{\gamma_{\left(B,\overline{\psi}\right)}^{\F,\varphi}}$
belongs to $\Orbitize{\FHH}$ and the morphism $\overline{\psi^{\F,\varphi}}$
belongs to $\ONF$ if and only if $\overline{\varphi}$ belongs to
$\ONF$.
\end{lem}

\begin{proof}
The fact that $\overline{\gamma_{\left(B,\overline{\psi}\right)}^{\F,\varphi}}$
is a morphism in $\Orbitize{\FHH}$ follows immediately from the identity
$\overline{\iota_{B^{\F,\varphi}}^{H}}\overline{\gamma_{\left(B,\overline{\psi}\right)}^{\F,\varphi}}=\overline{\iota_{\tilde{\varphi}^{-1}\left(B\right)}^{H}}$
and definition of the orbit category (see Definition \ref{def:Orbit-category.}).
Throughout this proof, contrary to Notation \ref{nota:Initial-notation.},
we will write $\overline{\alpha}\in\ONF$ to denote that $\overline{\alpha}$
is a morphism in $\ONF$ instead of an object in $\ONF$.

Assume that $\overline{\varphi}\not\in\ONF$. If $\overline{\tilde{\varphi}}\in\ONF$,
since $\varphi^{-1}\left(B\right)\le H$ then, by definition of $\NF$,
there would exist $\overline{\hat{\varphi}}\in\Hom_{\ONF}\left(H,\NS\right)$
such that $\overline{\hat{\varphi}}\overline{\iota_{\tilde{\varphi}^{-1}\left(B\right)}^{H}}=\overline{\iota_{B}^{\NS}}\overline{\tilde{\varphi}}$.
By definition of $\tilde{\varphi}$ this implies that $\overline{\hat{\varphi}}\overline{\iota_{A}^{H}}\overline{\iota_{\tilde{\varphi}^{-1}\left(B\right)}^{A}}=\overline{\varphi}\overline{\iota_{\tilde{\varphi}^{-1}\left(B\right)}^{A}}$.
From \cite[Theorem 4.9]{IntroductionToFusionSystemsLinckelmann} we
know that $\overline{\iota_{\tilde{\varphi}^{-1}\left(B\right)}^{A}}$
is an epimorphism and, therefore, we can conclude that $\overline{\hat{\varphi}}\overline{\iota_{A}^{H}}=\overline{\varphi}$.
In particular we would have that $\overline{\varphi}\in\ONF$ which
contradicts our assumption. We can therefore deduce that $\overline{\tilde{\varphi}}\not\in\ONF$.
Since $\overline{\psi}\in\ONF$ this implies that $\overline{\psi\tilde{\varphi}}\not\in\ONF$.
On the other hand, since $\Orbitize{\FHH}\subseteq\ONF$, then we
have that $\overline{\gamma_{\left(B,\overline{\psi}\right)}^{\F,\varphi}}\in\ONF$.
Thus, from the identity $\overline{\psi^{\F,\varphi}}\overline{\gamma_{\left(B,\overline{\psi}\right)}^{\F,\varphi}}=\overline{\psi\tilde{\varphi}}$,
we can conclude that $\overline{\psi^{\F,\varphi}}\not\in\ONF$.

Assume now that $\overline{\varphi}\in\ONF$. In this situation we
have that $\overline{\tilde{\varphi}}\in\ONF$ and, therefore, $\overline{\psi\tilde{\varphi}}=\overline{\psi^{\F,\varphi}}\overline{\gamma_{\left(B,\overline{\psi}\right)}^{\F,\varphi}}\in\ONF$.
Since $\tilde{\varphi}^{-1}\left(B\right)\le H$ then, by definition
of $\NF$, there exists a morphism $\overline{\theta}\colon H\to\NS$
in $\ONF$ such that $\overline{\theta}\overline{\iota_{\tilde{\varphi}^{-1}\left(B\right)}^{H}}=\overline{\iota_{K}^{\NS}}\overline{\psi^{\F,\varphi}}\overline{\gamma_{\left(B,\overline{\psi}\right)}^{\F,\varphi}}$.
Since $\overline{\gamma_{\left(B,\overline{\psi}\right)}^{\F,\varphi}}\in\Orbitize{\FHH}$
then there exists $h\in H$ such that $\tilde{\varphi}^{-1}\left(B\right)^{h}\le B^{\F,\varphi}$
and $\overline{\gamma_{\left(B,\overline{\psi}\right)}^{\F,\varphi}}=\overline{\iota_{\tilde{\varphi}^{-1}\left(B\right)^{h}}^{B^{\F,\varphi}}c_{h^{-1}}}$.
Using both these identities and definition of the orbit category we
can deduce that $\overline{\theta}\overline{\iota_{B^{\F,\varphi}}^{H}}\overline{\iota_{\tilde{\varphi}^{-1}\left(B\right)^{h}}^{B^{\F,\varphi}}}=\overline{\iota_{K}^{\NS}}\overline{\psi^{\F,\varphi}}\overline{\iota_{\tilde{\varphi}^{-1}\left(B\right)^{h}}^{B^{\F,\varphi}}}$.
From \cite[Theorem 4.9]{IntroductionToFusionSystemsLinckelmann} we
know that $\overline{\iota_{\tilde{\varphi}^{-1}\left(B\right)^{h}}^{B^{\F,\varphi}}}$
is an epimorphism and, therefore, we can conclude from the previous
identity that $\overline{\theta}\overline{\iota_{B^{\F,\varphi}}^{H}}=\overline{\iota_{K}^{\NS}}\overline{\psi^{\F,\varphi}}\in\ONF$.
In particular $\overline{\psi^{\F,\varphi}}\in\ONF$ thus concluding
the proof.
\end{proof}
Using Lemma \ref{lem:in-onf-iff-lifting-in-onf.} we can now give
the first of the two main results of this section.
\begin{lem}
\label{lem:Mackey-formula-induction-restriction-to-normalizer.}Let
$\R$ be a complete local and $p$-local $PID$, let $\G$ be a fusion
system containing $\F$ and let $M\in\MackFHR{\G}{\FHH}$. Then
\[
M\uparrow_{\FHH}^{\NF}\oplus\bigoplus_{K\in\mathcal{Y}}M^{K}\cong M\uparrow_{\FHH}^{\F}\downarrow_{\NF}^{\F}
\]
where, for every $K\in\mathcal{Y}$, we have that $M^{K}\in\MackFHR{\G}{\NF}$
is $K$-projective. Moreover the isomorphism realizing the above equivalence
can be taken so that the summand $M\uparrow_{\FHH}^{\NF}$ on the
left hand side is mapped isomorphically to the $\FmuFR{\NF}$-submodule
$\FmuFR{\NF}\otimes M$ of $M\uparrow_{\FHH}^{\F}\downarrow_{\NF}^{\F}$.
Here we are using Corollary \ref{cor:Mackey-algebra-inclusion.} in
order to view $\FmuFR{\NF}\otimes M$ as a submodule of $M\uparrow_{\FHH}^{\F}\downarrow_{\NF}^{\F}$. 
\end{lem}

\begin{proof}
In order to simplify notation we define $M^{\NF}:=M\uparrow_{\FHH}^{\F}\downarrow_{\NF}^{\F}$.
From Proposition \ref{prop:centric-Induction-restriction.} we know
that $M^{\NF}$ is $\G$-centric and, therefore, every $M^{K}$ (if
exists) must necessarily be $\G$-centric.

For every $\left(A,\overline{\varphi}\right)\in\left[H\times_{\F}\NS\right]$
fix a representative $\varphi$ of $\overline{\varphi}$ and view
it as an isomorphism onto its image. Since $M^{\NF}\downarrow_{\FHH[\NS]}^{\NF}=M\uparrow_{\FHH}^{\F}\downarrow_{\FHH[\NS]}^{\F}$
we can use Lemma \ref{lem:Mackey-formula-induction-restriction.}
in order to obtain a decomposition of $M^{\NF}\downarrow_{\FHH[\NS]}^{\NF}$
as a direct sum of $\FmuFR{\FHH[\NS]}$-modules. Applying the additive
functor $\uparrow_{\FHH[\NS]}^{\NF}$ to the resulting decomposition
we can conclude that $\left(M^{\NF}\right)_{\NS}=M^{H}\oplus M^{\mathcal{Y}}$
(see Definition \ref{def:theta_S-and-theta^S.}) where, using the
notation of Lemma \ref{lem:Mackey-formula-induction-restriction.},
we have that
\begin{align*}
M^{H} & :=\bigoplus_{\begin{array}{c}
{\scriptstyle {\scriptstyle \left(A,\overline{\varphi}\right)\in\left[H\times_{\F}\NS\right]}}\\
{\scriptstyle \overline{\varphi}\in\ONF}
\end{array}}M_{\left(A,\overline{\varphi}\right)}\uparrow_{\FHH[\varphi\left(A\right)]}^{\NF}, & M^{\mathcal{Y}} & :=\bigoplus_{\begin{array}{c}
{\scriptstyle {\scriptstyle \left(A,\overline{\varphi}\right)\in\left[H\times_{\F}\NS\right]}}\\
{\scriptstyle \overline{\varphi}\not\in\ONF}
\end{array}}M_{\left(A,\overline{\varphi}\right)}\uparrow_{\FHH[\varphi\left(A\right)]}^{\NF}.
\end{align*}
Here we are viewing the right hand sides of the above definitions
as submodules of $\left(M^{\NF}\right)_{\NS}$ via the isomorphism
described in Lemma \ref{lem:Mackey-formula-induction-restriction.}.
From Proposition \ref{prop:Decomposition-centric-Mackey-functor.}
we know that all the elements in $\theta_{\NS}^{M^{\NS}}\left(M_{\left(A,\overline{\varphi}\right)}\uparrow_{\FHH[\varphi\left(A\right)]}^{\NF}\right)$
(see Definition \ref{def:theta_S-and-theta^S.}) can be written as
finite sums of elements of the form $I_{\overline{\psi}\left(B\right)}^{K}c_{\overline{\psi\tilde{\varphi}}}\otimes x$
for some $K\in\NF\cap\G^{c}$, some $\left(B,\overline{\psi}\right)\in\left[\varphi\left(A\right)\times_{\NF}K\right]$
such that $B\in\FHH[\varphi\left(A\right)]\cap\G^{c}$ and some $x\in I_{\varphi^{-1}\left(B\right)}^{\varphi^{-1}\left(B\right)}M$.
Here $\tilde{\varphi}\colon\varphi^{-1}\left(B\right)\to B$ denotes
the morphism $\varphi$ seen as an isomorphism between the given subgroups.
From Lemma \ref{lem:in-onf-iff-lifting-in-onf.} we can now conclude
that, for every $K\in\NF\cap\G^{c}$, the elements of $I_{K}^{K}\theta_{\NS}^{M^{\NS}}\left(M^{H}\right)$
can be written as finite sums of elements of the form $I_{\overline{\theta}\left(C\right)}^{K}c_{\overline{\theta}}\otimes x$
for some $x\in I_{C}^{C}M$ and some $\left(C,\overline{\theta}\right)\in\left[H\times_{\F}K\right]$
such that $\overline{\theta}\in\ONF$. Notice that the tensor product
is over $\FmuFR{\FHH}$ and not $\FmuFR{\FHH[\NS]}$ since $I_{\overline{\theta}\left(C\right)}^{K}c_{\overline{\theta}}\otimes x$
is an element of $M^{\NF}=M\uparrow_{\FHH}^{\F}\downarrow_{\NF}^{\F}$.
Likewise, the elements of $I_{K}^{K}\theta_{\NS}^{M^{\NS}}\left(M^{\mathcal{Y}}\right)$
can be written as finite sums of elements of the form $I_{\overline{\theta}\left(C\right)}^{K}c_{\overline{\theta}}\otimes x$
for some $x\in I_{C}^{C}M$ and some $\left(C,\overline{\theta}\right)\in\left[H\times_{\F}K\right]$
such that $\overline{\theta}\not\in\ONF$. Applying again Proposition
\ref{prop:Decomposition-centric-Mackey-functor.} we can conclude
that $\theta_{\NS}^{M^{\NS}}\left(M^{H}\right)\cap\theta_{\NS}^{M^{\NS}}\left(M^{\mathcal{Y}}\right)=\left\{ 0\right\} $.
On the other hand, since $\R$ is $p$-local, we have from Lemma \ref{lem:projective-relative-to-S.}
that $\theta_{\NS}^{M^{\NS}}$ is split surjective and, in particular,
surjective. Since $\left(M^{\NF}\right)_{\NS}=M^{H}\oplus M^{\mathcal{Y}}$
then, from the previous result, we can conclude that 
\begin{equation}
M^{\NF}=\theta_{\NS}^{M^{\NS}}\left(M^{H}\right)\oplus\theta_{\NS}^{M^{\NS}}\left(M^{\mathcal{Y}}\right).\label{eq:aux-decomposition-M^NF}
\end{equation}

By definition of $\ONF$ (see Example \ref{exa:definition-NF.}) we
have that for every $A\le H$ and every $\overline{\varphi}\colon A\to\NS$
in $\ONF$ there exists a morphism $\overline{\hat{\varphi}}\colon H\to\NS$
in $\ONF$ such that $\overline{\hat{\varphi}}\overline{\iota_{A}^{H}}=\overline{\varphi}$.
We also have that $\ONF\subseteq\OF$. Therefore, for every $\left(A,\overline{\varphi}\right)\in\left[H\times_{\F}\NS\right]$
such that $\overline{\varphi}\in\ONF$, we have from maximality (see
Definition \ref{def:HX_FK}) that $A=H$ and $\overline{\varphi}\in\Aut_{\OF}\left(H\right)$.
From Proposition \ref{prop:properties-HXK.} \eqref{enu:prod-iso-left.}
and the above description of elements in $\theta_{\NS}^{M^{\NS}}\left(M_{\left(A,\overline{\varphi}\right)}\uparrow_{\FHH}^{\NF}\right)$
we can then conclude that, for every $K\in\NF\cap\G^{c}$, the elements
in $I_{K}^{K}\theta_{\NS}^{M^{\NS}}\left(M^{H}\right)$ are finite
sums of elements of the form $I_{\overline{\psi}\left(B\right)}^{K}c_{\overline{\psi}}\otimes x$
for some $\left(B,\overline{\psi}\right)\in\left[H\times_{\NF}K\right]$
and some $x\in I_{B}^{B}M$. From Proposition \ref{prop:Decomposition-centric-Mackey-functor.}
we can then conclude that $\theta_{\NS}^{M^{\NS}}\left(M_{\left(A,\overline{\varphi}\right)}\uparrow_{\FHH}^{\NF}\right)$
is precisely the submodule $\FmuFR{\NF}\otimes M$ of $M\uparrow_{\FHH}^{\F}\downarrow_{\NF}^{\F}$
which is, by definition, isomorphic to $M\uparrow_{\FHH}^{\NF}$.

From Equation \eqref{eq:aux-decomposition-M^NF} and the fact that
$\theta_{\NS}^{M^{\NS}}$ is split surjective, we conclude that the
restriction of $\theta_{\NS}^{M^{\NS}}$ as a map from $M^{\mathcal{Y}}$
to $\theta_{\NS}^{M^{\NS}}\left(M^{\mathcal{Y}}\right)$ is also split
surjective. In particular we have that $\theta_{\NS}^{M^{\NS}}\left(M^{\mathcal{Y}}\right)$
is isomorphic to a summand of $M^{\mathcal{Y}}$. Notice now that,
for every $\left(A,\overline{\varphi}\right)\in\left[H\times_{\F}\NS\right]$,
we have that $\varphi\left(A\right)\le_{\F}H$ and, if $\varphi\left(A\right)=H$,
then we necessarily have that $A=H$ and $\overline{\varphi}\in\Aut_{\OF}\left(H\right)=\Aut_{\ONF}\left(H\right)$.
We can therefore conclude that 
\begin{align*}
M^{\mathcal{Y}} & =\bigoplus_{K\in\mathcal{Y}}M'^{K} & \text{where} &  & M'^{K} & :=\bigoplus_{\begin{array}{c}
{\scriptstyle {\scriptstyle \left(A,\overline{\varphi}\right)\in\left[H\times_{\F}\NS\right]}}\\
{\scriptstyle \varphi\left(A\right)=K}
\end{array}}M_{\left(A,\overline{\varphi}\right)}\uparrow_{\FHH[K]}^{\NF}.
\end{align*}
Since $\R$ is a complete local $PID$ we can now apply the Krull-Schmidt-Azumaya
theorem (see \cite[Theorem 6.12 (ii)]{MethodsOfRepresentationTheoryCurtisReiner})
in order to write $\theta_{\NS}^{M^{\NS}}\left(M^{\mathcal{Y}}\right)=\bigoplus_{K\in\mathcal{Y}}M^{K}$
where each $M^{K}$ is a summand of $M'^{K}$. From Theorem \ref{thm:Higman's-criterion.}
we know that each $M'^{K}$ is $K$-projective. Therefore since each
$M^{K}$ is a summand of $M'^{K}$ we can conclude, again from Theorem
\ref{thm:Higman's-criterion.}, that $M^{K}$ is $K$-projective thus
concluding the proof.
\end{proof}
Using Lemma \ref{lem:Mackey-formula-induction-restriction-to-normalizer.}
we can now obtain the following result with which we conclude this
subsection.
\begin{lem}
\label{lem:induction-from-normalizer-followed-by-restriction-to-normalizer.}Let
$\R$ be a complete local and $p$-local $PID$, let $\G$ be a fusion
system containing $\F$ and let $M\in\MackFHR{\G}{\NF}$ be $H$-projective.
Then, there exists an $\mathcal{Y}$-projective $M'\in\MackFHR{\G}{\NF}$
such that
\[
M\uparrow_{\NF}^{\F}\downarrow_{\NF}^{\F}\cong M\oplus M'.
\]
\end{lem}

\begin{proof}
From Proposition \ref{prop:centric-Induction-restriction.} we know
that if such a direct sum decomposition exists then $M'$ is necessarily
$\G$-centric. From Theorem \ref{thm:Higman's-criterion.} we know
that there exist $N\in\MackFHR{\G}{\FHH}$ and $U\in\MackFHR{\G}{\NF}$
such that $M\oplus U\cong N\uparrow_{\FHH}^{\NF}$. Since induction
and restriction preserve direct sum decomposition then, from Lemma
\ref{lem:Mackey-formula-induction-restriction-to-normalizer.}, we
obtain an isomorphism 
\[
f:M\uparrow_{\NF}^{\F}\downarrow_{\NF}^{\F}\oplus\,U\uparrow_{\NF}^{\F}\downarrow_{\NF}^{\F}\bjarrow N\uparrow_{\FHH}^{\NF}\oplus\bigoplus_{K\in\mathcal{Y}}N^{K}.
\]
Where each $N^{K}$ is $K$-projective. Lemma \ref{lem:Mackey-formula-induction-restriction-to-normalizer.}
also tells us that $f$ sends the sub-module $M\oplus U$ of $M\uparrow_{\NF}^{\F}\downarrow_{\NF}^{\F}\oplus U\uparrow_{\NF}^{\F}\downarrow_{\NF}^{\F}$
isomorphically onto the summand $N\uparrow_{\FHH}^{\NF}$ of the right
hand side. Using this we obtain the following equivalence of $\FmuFR{\NF}$-modules
\begin{align*}
M\oplus\bigoplus_{K\in\mathcal{Y}}N^{K} & \cong\left(M\oplus U\oplus\bigoplus_{K\in\mathcal{Y}}N^{K}\right)/U,\\
 & \cong\left(M\uparrow_{\NF}^{\F}\downarrow_{\NF}^{\F}\oplus U\uparrow_{\NF}^{\F}\downarrow_{\NF}^{\F}\right)/U,\\
 & \cong M\uparrow_{\NF}^{\F}\downarrow_{\NF}^{\F}\oplus\left(U\uparrow_{\NF}^{\F}\downarrow_{\NF}^{\F}/U\right).
\end{align*}
In particular we can conclude that $M\uparrow_{\NF}^{\F}\downarrow_{\NF}^{\F}$
is a summand of $M\oplus\bigoplus_{K\in\mathcal{Y}}N^{K}$. Moreover,
again from the description of $f$, we have that $M\uparrow_{\NF}^{\F}\downarrow_{\NF}^{\F}$
contains the summand $M$. Since $\R$ is complete local and $p$-local
then we can use this and the Krull-Schmidt-Azumaya theorem in order
to conclude that there exists a summand $M'$ of $\bigoplus_{K\in\mathcal{Y}}N^{K}$
(which is necessarily $\mathcal{Y}$-projective from Theorem \ref{thm:Higman's-criterion.})
such that $M\uparrow_{\NF}^{\F}\downarrow_{\NF}^{\F}\cong M\oplus M'$.
This concludes the proof.
\end{proof}

\subsection{\label{subsec:Restriction-map-from-trHF-to-trHNF+trYNF.}Composing
transfer and restriction.}

Let $G$ be a finite group and let $H,K$ and $J$ be subgroups of
$G$ such that $J\le K$. The following decomposition of double cosets
representatives is well known
\begin{equation}
\left[J\backslash G/H\right]=\bigsqcup_{x\in\left[K\backslash G/H\right]}\left[J\backslash K/\left(K\cap\lui{x}{H}\right)\right]x.\label{eq:decomposition-biset-1.}
\end{equation}
where we define 
\[
\left[J\backslash K/\left(K\cap\lui{x}{H}\right)\right]x:=\left\{ yx\in G\,:\,y\in\left[J\backslash K/\left(K\cap\lui{x}{H}\right)\right]\right\} .
\]
Denoting by $\tr_{J}^{G}$ and $\r_{H}^{G}$ the transfer and restriction
maps of the Endomorphism Mackey functor $\End\left(M\right)$ (see
\cite[Definition 2.7]{SASAKI198298}) Equation \eqref{eq:decomposition-biset-1.}
can be used in order to prove that for any Mackey functor $M$ over
$G$
\begin{equation}
\r_{K}^{G}\tr_{H}^{G}=\sum_{x\in\left[K\backslash G/H\right]}\tr_{K\cap\lui{x}{H}}^{K}c_{c_{x}}\r_{K^{x}\cap H}^{H}.\label{eq:r*tr-groups.}
\end{equation}
We know from Proposition \ref{prop:Properties-of-transfer-restriction-and-conjugation-maps.}
\eqref{enu:Mackey-formula.} that a similar result holds in the case
of the transfer and restriction maps of Definition \ref{def:Transfer-and-restriction.}.
However, Proposition \ref{prop:Properties-of-transfer-restriction-and-conjugation-maps.}
\eqref{enu:Mackey-formula.} only involves composition of transfer
and restriction maps of the form $\r_{\FHH[K]}^{\F}\tr_{\FHH[H]}^{\F}$
for some $K,H\in\Fc$ and tells us nothing regarding compositions
of transfer and restriction of the form $\r_{\G}^{\F}\tr_{\FHH[H]}^{\F}$
for other fusion system $\G$ contained in $\F$. Attempting to obtain
a decomposition similar to that of Proposition \ref{prop:Properties-of-transfer-restriction-and-conjugation-maps.}
\eqref{enu:Mackey-formula.} in this situation leads to several complications.
These can be traced back to the lack of a result analogous to Proposition
\ref{prop:properties-HXK.} \eqref{enu:prod-pullback-left.} in the
case where $H$ is replaced with $\G$ and $\left[A\backslash H/J\right]$
is replaced with $\left[A\times_{\G}J\right]$. Some experimentation
leads us to believe that such a result is possible when $K=H$ and
$\G=\NF$ (i.e. a result dual to Theorem \ref{thm:decomposition-direct-product.}),
however we were unable to prove it. Nonetheless we were able to obtain
a result analogous to Equation \eqref{eq:r*tr-groups.} for the composition
$\r_{\NF}^{\F}\tr_{\FHH}^{\F}$ (see Proposition \ref{prop:Conjecture-workaround.})
and this subsection is dedicated to proving it. In order to do so
we first need to develop some tools.
\begin{lem}
\label{lem:transfer-and-restriction.}Let $\R$ be a $p$-local ring,
let $M\in\MackFcR$ and let $\overline{\NS}\in\BGR{\left(\NF\right)^{c}}$
be the isomorphism class of $\NS$. From Proposition \ref{prop:Inverse-of-S.}
we know that $\overline{\NS}$ admits an inverse in $\BGR{\left(\NF\right)^{c}}$.
Then we have that 
\[
\r_{\NF}^{\F}\tr_{\FHH}^{\F}=\sum_{\left(A,\overline{\varphi}\right)\in\left[H\times_{\F}\NS\right]}\left(\overline{\NS}^{-1}\cdot\right)_{*}\tr_{\FHH[\varphi\left(A\right)]}^{\NF}\lui{\varphi}{\cdot}\r_{\FHH[A]}^{\FHH}.
\]
where we are using Notation \ref{nota:Initial-notation.} as well
as the notation of Proposition \ref{prop:Action-of-centric-burnside-ring.}
and Definition \ref{def:Transfer-and-restriction.} and we are viewing
the representative $\varphi$ of $\overline{\varphi}$ as an isomorphism
onto its image. Equivalently, using the same notation, we have that
\[
\left(\overline{\NS}\cdot\right)_{*}\r_{\NF}^{\F}\tr_{\FHH}^{\F}=\sum_{\left(A,\overline{\varphi}\right)\in\left[H\times_{\F}\NS\right]}\tr_{\FHH[\varphi\left(A\right)]}^{\NF}\lui{\varphi}{\cdot}\r_{\FHH[A]}^{\FHH}.
\]
\end{lem}

\begin{proof}
Since the first and second identities of the statement are equivalent
we will just prove the second identity. Let us start by rewriting
\begin{align*}
\left(\overline{\NS}\cdot\right)_{*}\r_{\NF}^{\F}\tr_{\FHH}^{\F} & =\tr_{\FHH[\NS]}^{\NF}\r_{\FHH[\NS]}^{\NF}\r_{\NF}^{\F}\tr_{\FHH}^{\F}=\tr_{\FHH[\NS]}^{\NF}\sum_{\left(A,\overline{\varphi}\right)\in\left[H\times_{\F}\NS\right]}\tr_{\FHH[\varphi\left(A\right)]}^{\FHH[\NS]}\lui{\varphi}{\cdot}\r_{\FHH[A]}^{\FHH}.
\end{align*}
Here we are using Item \eqref{enu:Restriction-and-transfer-makes-burnside.}
of Proposition \ref{prop:Properties-of-transfer-restriction-and-conjugation-maps.}
for the first identity and Items \eqref{enu:Composition-restriction.}
and \eqref{enu:Mackey-formula.} for the second identity. The Lemma
follows after applying Item \eqref{enu:composition-transfer.} to
the identity above.
\end{proof}
\begin{lem}
\label{lem:same-cardinal-for-product.}Let $H\in\Fc$ be such that
$\F=\NFH$ and let $K\in\FHH\cap\Fc$. Then we have that
\begin{align*}
\left[K\times_{\F}S\right] & =\left\{ \left(K,\overline{\varphi}\right)\,|\,\overline{\varphi}\in\Hom_{\OF}\left(K,S\right)\right\} , & \text{and} &  & \Hom_{\OF}\left(K,S\right) & \cong\Hom_{\OF}\left(H,S\right).
\end{align*}
In particular we have the following bijection of finite sets
\begin{align*}
\left[K\times_{\F}S\right] & \cong\left[H\times_{\F}S\right]\cong\Hom_{\OF}\left(K,S\right).
\end{align*}
\end{lem}

\begin{proof}
Since $K\le H$ then, for any subgroup $A\le K$, we have that $HA=H$.
Analogously we also have that $HS=S$. Since $\F=\NFH$ then, by definition
(see Example \ref{exa:definition-NF.}), we can conclude that for
every $A\le K$ and every morphism $\overline{\varphi}\colon A\to S$
in $\OF$ there exists a morphism $\overline{\hat{\varphi}}\colon H\to S$
in $\OF$ such that $\overline{\hat{\varphi}}\overline{\iota_{K}^{H}}\overline{\iota_{A}^{K}}=\overline{\hat{\varphi}}\overline{\iota_{A}^{H}}=\overline{\varphi}$.
From maximality of the pairs $\left(A,\overline{\varphi}\right)\in\left[K\times_{\F}S\right]$
(see Definition \ref{def:HX_FK}) we can conclude that $A=K$ and,
therefore, we have that
\begin{align*}
\left[K\times_{\F}S\right] & =\left\{ \left(K,\overline{\varphi}\right)\,|\,\overline{\varphi}\in\Hom_{\OF}\left(K,S\right)\right\} \cong\Hom_{\OF}\left(K,S\right).
\end{align*}
Thus we are only left with proving that $\Hom_{\OF}\left(K,S\right)\cong\Hom_{\OF}\left(H,S\right)$.

It suffices to prove that the map $\left(\overline{\iota_{K}^{H}}\right)^{*}$
from $\Hom_{\OF}\left(H,S\right)$ to $\Hom_{\OF}\left(K,S\right)$
(see Notation \ref{nota:Initial-notation.}) is bijective. From \cite[Theorem 4.9]{IntroductionToFusionSystemsLinckelmann}
we know that $\overline{\iota_{K}^{H}}$ is surjective. On the other
hand it is well known that the contravariant $\Hom$ functor $\Hom_{\OF}\left(-,S\right)$
is left exact and, in particular, sends surjective morphisms to injective
morphisms. Joining both these facts we can conclude that $\left(\overline{\iota_{K}^{H}}\right)^{*}$
is injective.

On the other hand, as mentioned at the beginning of the proof, for
every morphism $\overline{\varphi}:K\to S$ in $\OF$ there exists
a morphism $\overline{\hat{\varphi}}:H\to S$ in $\OF$ such that
$\overline{\varphi}=\overline{\hat{\varphi}}\,\overline{\iota_{K}^{H}}$.
This proves that $\left(\overline{\iota_{K}^{H}}\right)^{*}$ is also
surjective thus concluding the proof.
\end{proof}
We can now finally obtain the last ingredient needed in order to prove
Proposition \ref{prop:Conjecture-workaround.}.
\begin{lem}
\label{lem:multiple-transfers-averaged-are-single-transfer.}Let $H\in\Fc$
be such that $\F=\NFH$, let $M\in\MackFcR$, let $\R$ be a $p$-local
ring and let $\overline{S}\in\BFcR$ be the isomorphism class of $S$.
From Proposition \ref{prop:Inverse-of-S.} we know that $\overline{S}$
has an inverse in $\BFcR$. With this setup the following equivalent
identities are satisfied 
\begin{align*}
\sum_{\overline{\varphi}\in\Hom_{\OF}\left(H,S\right)}\left(\overline{S}^{-1}\cdot\right)_{*}\tr_{\FHH}^{\F}\lui{\varphi}{\cdot} & =\tr_{\FHH}^{\F}, & \sum_{\overline{\varphi}\in\Hom_{\OF}\left(H,S\right)}\tr_{\FHH}^{\F}\lui{\varphi}{\cdot} & =\left(\overline{S}\cdot\right)_{*}\tr_{\FHH}^{\F}.
\end{align*}
Where we are viewing the representative $\varphi$ of $\overline{\varphi}$
as an isomorphism onto its image and we are dropping the superindex
$M$ in order to keep notation simple.
\end{lem}

\begin{proof}
We will only prove the second identity since both identities are equivalent.
From Proposition \ref{prop:Properties-of-transfer-restriction-and-conjugation-maps.}
\eqref{enu:Restriction-and-transfer-makes-burnside.} we know that
$\left(\overline{S}\cdot\right)_{*}=\tr_{\FHH}^{\F}\r_{\FHH}^{\F}$.
Combining this with Proposition \ref{prop:Properties-of-transfer-restriction-and-conjugation-maps.}
\eqref{enu:Mackey-formula.} we obtain the identity $\left(\overline{S}\cdot\right)_{*}\tr_{\FHH}^{\F}=\tr_{\FHH[S]}^{\F}\sum_{\left(A,\overline{\varphi}\right)\in\left[H\times_{\F}S\right]}\tr_{\FHH}^{\FHH[S]}\lui{\varphi}{\cdot}$.
The result now follows from Proposition \ref{prop:Properties-of-transfer-restriction-and-conjugation-maps.}
\eqref{enu:composition-transfer.} and Lemma \ref{lem:same-cardinal-for-product.}.
\end{proof}
We are now finally able to give a result for centric Mackey functors
over fusion systems analogous to that of Equation \eqref{eq:r*tr-groups.}
in a case not covered by Proposition \ref{prop:Properties-of-transfer-restriction-and-conjugation-maps.}
\eqref{enu:Mackey-formula.}.
\begin{prop}
\label{prop:Conjecture-workaround.}Let $\R$ be a $p$-local ring,
let $M\in\MackFcR$ and for every $\left(A,\overline{\varphi}\right)\in\left[H\times_{\F}\NS\right]$
fix a representative $\varphi$ of $\overline{\varphi}$ seen as an
isomorphism onto its image. From Proposition \ref{prop:Inverse-of-S.}
we know that the $\NF$-conjugacy class $\overline{\NS}\in\BGR{\left(\NF\right)^{c}}$
of $\NS$ has an inverse in $\BGR{\left(\NF\right)^{c}}$ and, using
the notation of Proposition \ref{prop:Action-of-centric-burnside-ring.},
we have that $\overline{\NS}^{-1}\cdot\in\End\left(M\downarrow_{\NF}^{\F}\right)$.
For every $f\in\End\left(M\downarrow_{\FHH}^{\F}\right)$ and every
$K\in\mathcal{Y}$ (see Notation \ref{nota:X-and-Y.}) we can now
define
\[
f_{K}:=\sum_{\begin{array}{c}
{\scriptstyle \left(A,\overline{\varphi}\right)\in\left[H\times_{\F}\NS\right]}\\
{\scriptstyle \varphi\left(A\right)=K}
\end{array}}\left(\r_{\FHH[\varphi\left(A\right)]}^{\NF}\left(\overline{\NS}^{-1}\cdot\right)\right)\lui{\varphi}{\left(\r_{\FHH[A]}^{\FHH}\left(f\right)\right)}\in\End\left(M\downarrow_{\FHH[K]}^{\F}\right).
\]
and the following is satisfied
\[
\r_{\NF}^{\F}\left(\tr_{\FHH}^{\F}\left(f\right)\right)=\tr_{\FHH}^{\NF}\left(f\right)+\sum_{K\in\mathcal{Y}}\tr_{\FHH[K]}^{\NF}\left(f_{K}\right).
\]
Different choices of $\left[H\times_{\F}\NS\right]$ and representative
$\varphi\in\overline{\varphi}$ can lead to different definitions
of each individual $f_{K}$ but the result holds for any such choice.
\end{prop}

\begin{proof}
Applying Proposition \ref{prop:Properties-of-transfer-restriction-and-conjugation-maps.}
\eqref{enu:Green-formula.} to $\tr_{\FHH[K]}^{\NF}\left(f_{K}\right)$
for every $K\in\mathcal{Y}$ we obtain 
\begin{align*}
\sum_{K\in\mathcal{Y}}\tr_{\FHH[K]}^{\NF}\left(f_{K}\right) & =\sum_{K\in\mathcal{Y}}\sum_{\begin{array}{c}
{\scriptstyle \left(A,\overline{\varphi}\right)\in\left[H\times_{\F}\NS\right]}\\
{\scriptstyle \varphi\left(A\right)=K}
\end{array}}\left(\overline{\NS}^{-1}\cdot\right)_{*}\left(\tr_{\FHH[K]}^{\NF}\left(\lui{\varphi}{\left(\r_{\FHH[A]}^{\FHH}\left(f\right)\right)}\right)\right),\\
 & =\sum_{\begin{array}{c}
{\scriptstyle \left(A,\overline{\varphi}\right)\in\left[H\times_{\F}\NS\right]}\\
{\scriptstyle \varphi\left(A\right)\in\mathcal{Y}}
\end{array}}\left(\overline{\NS}^{-1}\cdot\right)_{*}\left(\tr_{\FHH[K]}^{\NF}\left(\lui{\varphi}{\left(\r_{\FHH[A]}^{\FHH}\left(f\right)\right)}\right)\right).
\end{align*}
Subtracting the above identity to the one in the statement and applying
Lemma \ref{lem:transfer-and-restriction.} to $\r_{\NF}^{\F}\left(\tr_{\FHH}^{\F}\left(f\right)\right)$
we obtain that the following identity is equivalent to the one in
the statement
\[
\sum_{\begin{array}{c}
{\scriptstyle \left(A,\overline{\varphi}\right)\in\left[H\times_{\F}\NS\right]}\\
{\scriptstyle \varphi\left(A\right)\not\in\mathcal{Y}}
\end{array}}\left(\overline{\NS}^{-1}\cdot\right)_{*}\,\tr_{\FHH[K]}^{\NF}\,\lui{\varphi}{\cdot}\,\r_{\FHH[A]}^{\FHH}=\tr_{\FHH}^{\NF}.
\]

Because of Lemma \ref{lem:multiple-transfers-averaged-are-single-transfer.}
it now suffices to prove the identity
\[
\left\{ \left(A,\overline{\varphi}\right)\in\left[H\times_{\F}\NS\right]\,|\,\varphi\left(A\right)\not\in\mathcal{Y}\right\} =\left\{ \left(H,\overline{\varphi}\right)\,|\,\overline{\varphi}\in\Hom_{\ONF}\left(H,\NS\right)\right\} .
\]
For every $\left(A,\overline{\varphi}\right)\in\left[H\times_{\F}\NS\right]$
we have that $\varphi\left(A\right)\le_{\F}\NS$ and, therefore, by
definition of $\mathcal{Y}$ (see Notation \ref{nota:X-and-Y.}) we
have that $\varphi\left(A\right)\not\in\mathcal{Y}$ if and only if
$\varphi\left(A\right)=H$. Since $\varphi$ is an isomorphism, $A\le H$
and the groups $A$ and $H$ are finite then the identity $\varphi\left(A\right)=H$
implies that $A=H$ and, therefore, $\varphi\in\Aut_{\F}\left(H\right)=\Aut_{\NF}\left(H\right)$.
Equivalently $\overline{\varphi}\in\Hom_{\ONF}\left(H,\NS\right)$.
This proves one inclusion. On the other hand, for every $\overline{\varphi}\in\Hom_{\ONF}\left(H,\NS\right)$,
we know from the universal properties of product that there exist
a unique $\left(B,\overline{\psi}\right)\in\left[H\times_{\F}\NS\right]$
and a unique $\overline{\gamma}\colon H\to B$ such that $\overline{\psi}\overline{\gamma}=\overline{\varphi}$
and $\overline{\iota_{B}^{H}}\overline{\gamma}=\overline{\iota_{H}^{H}}=\overline{\Id_{H}}$.
From these identities we can conclude that $\left(B,\overline{\psi}\right)=\left(H,\overline{\varphi}\right)$.
This proves the second inclusion thus completing the proof.
\end{proof}

\subsection{\label{subsec:decomposition-of-product-in-aOFc.}Decomposing the
product in $\aOFc$.}

Let $G$ be a finite group and let $H,K$ and $J$ be subgroups of
$G$ such that $J\le K$. The following decomposition of double coset
representatives dual to Equation \eqref{eq:decomposition-biset-1.}
is well known

\begin{align}
\left[H\backslash G/J\right] & =\bigsqcup_{x\in\left[H\backslash G/K\right]}x\left[\left(H^{x}\cap K\right)\backslash K/J\right].\label{eq:decomposition-biset-2.}
\end{align}
where we define 
\[
x\left[\left(H^{x}\cap K\right)\backslash K/J\right]:=\left\{ xy\in G\,:\,y\in\left[\left(H^{x}\cap K\right)\backslash K/J\right]\right\} .
\]
In the case of Mackey functors over finite groups, this can be used
to prove that $\tr_{K}^{G}\tr_{J}^{K}=\tr_{J}^{G}$ where $\tr_{A}^{B}$
denote the transfer maps of the endomorphism Mackey functor $\End\left(M\right)$
for some Mackey functor $M$ over $G$ (see \cite[Definition 2.7]{SASAKI198298}).
Proposition \ref{prop:Properties-of-transfer-restriction-and-conjugation-maps.}
\eqref{enu:composition-transfer.} proves that a similar result holds
for fusion systems. However, in the case of Mackey functors over fusion
systems, given $M\in\MackFcR$ and a fusion system $\mathcal{K}$,
such that $\FHH\subseteq\mathcal{K}\subseteq\F$ the transfer $\tr_{\mathcal{K}}^{\F}\colon\End\left(M\downarrow_{\mathcal{K}}^{\F}\right)\to\End\left(M\right)$
is in general not defined. We show with Definition \ref{def:tr_NF^F.}
and Lemma \ref{lem:transfer-composes-nicely.} that the transfer $\tr_{\mathcal{K}}^{\F}$
can be defined when $\mathcal{K}=\NF$ (see Notation \ref{nota:X-and-Y.})
and that, in this case, we have $\tr_{\NF}^{\F}\tr_{\FHH}^{\NF}=\tr_{\FHH}^{\F}$.
However, in order to prove such a result, we first need to translate
Equation \eqref{eq:decomposition-biset-2.} to the context of fusion
systems. More precisely, we need to prove that, for every $K\in\Fc$,
we can write $\left[H\times_{\F}K\right]$ in terms of sets of the
form $\left[H\times_{\NF}A\right]$ with $A\in\NF\cap\Fc$. This section
is dedicated to proving exactly that (see Theorem \ref{thm:decomposition-direct-product.}).

Let us start by finding what can replace the groups $H^{x}\cap K$
of Equation \eqref{eq:decomposition-biset-2.} in the context of fusion
systems.
\begin{lem}
\label{lem:Normalizer-via-phi.}Let $A,K\in\Fc$ with $A\le\NS$ and
let $\varphi\in\Hom_{\F}\left(A,K\right)$. Using Notation \ref{nota:X-and-Y.}
we define the \textbf{normalizer after $\varphi$ in $\NF$} as
\[
\prescript{\NF}{\varphi}{N}:=\left\{ x\in N_{K}\left(\varphi\left(A\right)\right)\,:\,\varphi^{-1}c_{x}\varphi\in\Aut_{\NF}\left(A\right)\right\} ,
\]
where, on the right hand side, we are viewing $\varphi$ as an isomorphism
onto its image. Then $\prescript{\NF}{\varphi}{N}$ is the unique
maximal subgroup of $N_{K}\left(\varphi\left(A\right)\right)$ such
that 
\[
\Aut_{\prescript{\NF}{\varphi}{N}}\left(\varphi\left(A\right)\right)\le\lui{\varphi}{\Aut_{\NF}\left(A\right)}.
\]
Moreover there exist a fully $\NF$-normalized subgroup $A'\le\NS$,
an isomorphism $\theta\in\Hom_{\NF}\left(A',A\right)$ and a subgroup
$N_{\varphi\theta}^{\NF}$ of $N_{\NS}\left(A'\right)$ containing
$A'$ such that 
\[
\Aut_{\prescript{\NF}{\varphi}{N}}\left(\varphi\left(A\right)\right)=\lui{\varphi\theta}{\Aut_{N_{\varphi\theta}^{\NF}}\left(A'\right)}.
\]
More precisely we can take $\theta$ such that
\[
N_{\varphi\theta}^{\NF}=\left\{ x\in N_{\NS}\left(A'\right)\,:\,c_{x}\in\Aut_{\prescript{\NF}{\varphi}{N}}\left(\varphi\left(A\right)\right)^{\varphi\theta}\right\} .
\]
We call any morphism of the form $\varphi\theta$ with $\theta$ as
before \textbf{$\NF$-top of $\varphi$} and denote it by $\varphi^{\NF}$.
We also call \textbf{normalizer before $\varphi^{\NF}$ in $\NF$}
any group of the form $N_{\varphi^{\NF}}^{\NF}$.
\end{lem}

\begin{proof}
First of all notice that $\varphi^{-1}c_{1_{N_{K}\left(\varphi\left(A\right)\right)}}\varphi=\Id_{A}$,
for any $x\in\prescript{\NF}{\varphi}{N}$ we have that $\varphi^{-1}c_{x^{-1}}\varphi=\left(\varphi^{-1}c_{x}\varphi\right)^{-1}$
and for any other $y\in\prescript{\NF}{\varphi}{N}$ we have that
$\varphi^{-1}c_{xy}\varphi=\left(\varphi^{-1}c_{x}\varphi\right)\left(\varphi^{-1}c_{y}\varphi\right)$.
Since $\Aut_{\NF}\left(A\right)$ is a subgroup of $\Aut\left(A\right)$
the previous equations prove that $1_{N_{K}\left(\varphi\left(A\right)\right)}\in\prescript{\NF}{\varphi}{N}$,
that $x^{-1}\in\prescript{\NF}{\varphi}{N}$ and that $xy\in\prescript{\NF}{\varphi}{N}$
respectively. We can therefore conclude that $\prescript{\NF}{\varphi}{N}$
is indeed a subgroup of $N_{K}\left(\varphi\left(A\right)\right)$.
Moreover, from definition of $\prescript{\NF}{\varphi}{N}$, we have
that $\Aut_{\prescript{\NF}{\varphi}{N}}\left(\varphi\left(A\right)\right)^{\varphi}$
is a subgroup of $\Aut_{\NF}\left(A\right)$. Equivalently, $\Aut_{\prescript{\NF}{\varphi}{N}}\left(\varphi\left(A\right)\right)$
is a subgroup of $\lui{\varphi}{\Aut_{\NF}\left(A\right)}$. Moreover,
for every $x\in N_{K}\left(\varphi\left(A\right)\right)$ such that
$c_{x}\in\lui{\varphi}{\Aut_{\NF}\left(A\right)}$ we have by definition
that $\varphi^{-1}c_{x}\varphi\in\Aut_{\NF}\left(A\right)$ and, therefore,
that $x\in\prescript{\NF}{\varphi}{N}$. This proves that $\prescript{\NF}{\varphi}{N}$
is indeed the unique maximal subgroup of $N_{K}\left(\varphi\left(A\right)\right)$
with the desired properties.

Let's now prove the second half of the statement. Let $A'=_{\NF}A$
be fully $\NF$-normalized and let $\alpha$ be an isomorphism in
$\NF$ from $A'$ to $A$. Since $\prescript{\NF}{\varphi}{N}\le S$
and $S$ is a $p$-group then $\prescript{\NF}{\varphi}{N}$ is also
a $p$-group. It follows that $\Aut_{\prescript{\NF}{\varphi}{N}}\left(\varphi\left(A\right)\right)^{\varphi\alpha}$
is also a $p$-group. From construction of $A'$ and $\alpha$ have
that 
\[
\Aut_{\prescript{\NF}{\varphi}{N}}\left(\varphi\left(A\right)\right)^{\varphi\alpha}\le\Aut_{\NF}\left(A'\right).
\]
From \cite[Proposition 2.5]{StancuEquivalentDefinitions} we know
that $\Aut_{\NS}\left(A'\right)$ is a Sylow $p$-subgroup of $\Aut_{\NF}\left(A'\right)$.
Thus we can apply the second Sylow theorem in order to obtain $\beta\in\Aut_{\NF}\left(A'\right)$
satisfying
\begin{equation}
\Aut_{\prescript{\NF}{\varphi}{N}}\left(\varphi\left(A\right)\right)^{\varphi\alpha\beta}\le\Aut_{\NS}\left(A'\right).\label{eq:aux-N^NF_f.}
\end{equation}
We can now define $\theta:=\alpha\beta$ and let $N_{\varphi\theta}^{\NF}$
be as in the statement. The same arguments used to prove that $\prescript{\NF}{\varphi}{N}$
is a subgroup of $N_{K}\left(\varphi\left(A\right)\right)$ can be
used to prove that $N_{\varphi\theta}^{\NF}$ is a subgroup of $N_{\NS}\left(A'\right)$.
For every $x\in A'$ we have $\varphi\theta c_{x}\left(\varphi\theta\right)^{-1}=c_{\varphi\theta\left(x\right)}\in\Aut_{\prescript{\NF}{\varphi}{N}}\left(\varphi\left(A\right)\right)$
and, therefore, the inclusion $A'\le N_{\varphi\theta}^{\NF}$ follows.
It is also immediate from definition that $\lui{\varphi\theta}{\Aut_{N_{\varphi\theta}^{\NF}}\left(A'\right)}$
is contained in $\Aut_{\prescript{\NF}{\varphi}{N}}\left(\varphi\left(A\right)\right)$.
The other inclusion on the other hand follows from Equation \eqref{eq:aux-N^NF_f.}
and definition of $N_{\varphi\theta}^{\NF}$. This concludes the proof.
\end{proof}
\begin{cor}
\label{cor:varphi(N_=00007Bvarphi=00007D^=00007B=00005CNF=00007D).}With
the notation of Lemma \ref{lem:Normalizer-via-phi.} assume that $\varphi=\varphi^{\NF}$.
If there exists $\hat{\varphi}\in\Hom_{\F}\left(N_{\varphi}^{\NF},S\right)$
such that $\iota_{K}^{S}\varphi=\hat{\varphi}\iota_{A}^{N_{\varphi}^{\NF}}$
then $\hat{\varphi}\left(N_{\varphi}^{\NF}\right)=\prescript{\NF}{\varphi}{N}$.
\end{cor}

\begin{proof}
By definition, we have that $N_{\varphi}^{\NF}\le N_{\NS}\left(A\right)$
and that $\hat{\varphi}\left(A\right)=\varphi\left(A\right)$. Therefore
we can deduce that $\hat{\varphi}\left(N_{\varphi}^{\NF}\right)\le N_{\NS}\left(\varphi\left(A\right)\right)$.
Moreover, from Lemma \ref{lem:Normalizer-via-phi.} we have that
\[
\Aut_{\prescript{\NF}{\varphi}{N}}\left(\varphi\left(A\right)\right)=\lui{\varphi}{\Aut_{N_{\varphi}^{\NF}}\left(A\right)}=\Aut_{\hat{\varphi}\left(N_{\varphi}^{\NF}\right)}\left(\varphi\left(A\right)\right).
\]
From these identities we can conclude that
\[
\prescript{\NF}{\varphi}{N}C_{S}\left(\varphi\left(A\right)\right)=\hat{\varphi}\left(N_{\varphi}^{\NF}\right)C_{S}\left(\varphi\left(A\right)\right).
\]
Recall now that, by hypothesis, we have $A\in\Fc$. Therefore we also
have $\varphi\left(A\right)\in\Fc$. In particular $C_{S}\left(\varphi\left(A\right)\right)\le\varphi\left(A\right)$.
Finally we have from Lemma \ref{lem:Normalizer-via-phi.} that $\varphi\left(A\right)\le\prescript{\NF}{\varphi}{N}$
and that $A\le N_{\varphi}^{\NF}$. Putting all this together we obtain
the following identities from which the result follows.
\begin{align*}
\prescript{\NF}{\varphi}{N} & =\prescript{\NF}{\varphi}{N}C_{S}\left(\varphi\left(A\right)\right)=\hat{\varphi}\left(N_{\varphi}^{\NF}\right)C_{S}\left(\varphi\left(A\right)\right)=\hat{\varphi}\left(N_{\varphi}^{\NF}\right).
\end{align*}
\end{proof}
\begin{cor}
\label{cor:Invariance-^NF_varphiN}With notation as in Lemma \ref{lem:Normalizer-via-phi.},
for every $A'\in\Fc$ and isomorphism $\theta\in\Hom_{\NF}\left(A',A\right)$
we have that $\prescript{\NF}{\varphi}{N}=\prescript{\NF}{\varphi\theta}{N}$.
\end{cor}

\begin{proof}
Since $\theta$ is an isomorphism in $\NF$ then we have that $\Aut_{\NF}\left(A\right)=\lui{\theta}{\Aut_{\NF}\left(A'\right)}$.
With this in mind the result follows from the identities below.
\begin{align*}
\prescript{\NF}{\varphi\theta}{N} & =\left\{ x\in N_{K}\left(\varphi\theta\left(A'\right)\right)\,:\,\left(\varphi\theta\right)^{-1}c_{x}\varphi\theta\in\Aut_{\NF}\left(A'\right)\right\} ,\\
 & =\left\{ x\in N_{K}\left(\varphi\left(A\right)\right)\,:\,\varphi^{-1}c_{x}\varphi\in\lui{\theta}{\left(\Aut_{\NF}\left(A'\right)\right)}\right\} ,\\
 & =\left\{ x\in N_{K}\left(\varphi\left(A\right)\right)\,:\,\varphi^{-1}c_{x}\varphi\in\Aut_{\NF}\left(A\right)\right\} =\prescript{\NF}{\varphi}{N}.
\end{align*}
\end{proof}
\begin{cor}
\label{cor:varphi^NF^NF=00003Dvarphi^NF.}With the notation of Lemma
\ref{lem:Normalizer-via-phi.}:
\begin{enumerate}
\item \label{enu:f^NF^NF-is-f^NF.}We can always take $\left(\varphi^{\NF}\right)^{\NF}=\varphi^{\NF}$.
\item \label{enu:varphi-ch.}If $\varphi=\varphi^{\NF}$ then for every
$h\in H$ we have that $\varphi c_{h}=\left(\varphi c_{h}\right)^{\NF}$
where $c_{h}:A^{h}\to A$ is seen as an isomorphism in $\NF$.
\end{enumerate}
\end{cor}

\begin{proof}
$\phantom{.}$
\begin{enumerate}
\item The result follows from definition of $\left(\varphi^{\NF}\right)^{\NF}$
and the identities below
\[
\lui{\varphi^{\NF}}{\Aut_{N_{\varphi^{\NF}}^{\NF}}\left(A'\right)}=\Aut_{\prescript{\NF}{\varphi}{N}}\left(\varphi\left(A\right)\right)=\Aut_{\prescript{\NF}{\varphi^{\NF}}{N}}\left(\varphi^{\NF}\left(A'\right)\right),
\]
Where we are using Corollary \ref{cor:Invariance-^NF_varphiN} for
the second identity.
\item With the notation of Item \eqref{enu:varphi-ch.} we have that
\[
\Aut_{\prescript{\NF}{\varphi c_{h}}{N}}\left(\varphi c_{h}\left(A^{h}\right)\right)=\Aut_{\prescript{\NF}{\varphi}{N}}\left(\varphi\left(A\right)\right)=\lui{\varphi}{\Aut_{N_{\varphi}^{\NF}}\left(A\right)}=\lui{\varphi c_{h}}{\Aut_{\left(N_{\varphi}^{\NF}\right)^{h}}\left(A^{h}\right)}.
\]
Where, for the first identity, we are using Corollary \ref{cor:Invariance-^NF_varphiN},
while, for the second identity, we are using the fact that $\varphi=\varphi^{\NF}$.
Using the above and the description of $N_{\varphi}^{\NF}$ given
in Lemma \ref{lem:Normalizer-via-phi.} we obtain
\begin{align*}
\left(N_{\varphi}^{\NF}\right)^{h} & =\left\{ x\in N_{\NS}\left(A\right)\,:\,c_{x}\in\Aut_{\prescript{\NF}{\varphi}{N}}\left(\varphi\left(A\right)\right)^{\varphi}\right\} ^{h},\\
 & =\left\{ y\in N_{\NS}\left(A^{h}\right)\,:\,c_{\lui{h}{y}}\in\Aut_{\prescript{\NF}{\varphi}{N}}\left(\varphi c_{h}\left(A^{h}\right)\right)^{\varphi}\right\} ^{h},\\
 & =\left\{ y\in N_{\NS}\left(A^{h}\right)\,:\,c_{y}\in\Aut_{\prescript{\NF}{\varphi}{N}}\left(\varphi c_{h}\left(A^{h}\right)\right)^{\varphi c_{h}}\right\} ^{h},\\
 & =\left\{ y\in N_{\NS}\left(A^{h}\right)\,:\,c_{y}\in\Aut_{\prescript{\NF}{\varphi c_{h}}{N}}\left(\varphi c_{h}\left(A^{h}\right)\right)^{\varphi c_{h}}\right\} ^{h}.
\end{align*}
Where we are using Corollary \ref{cor:Invariance-^NF_varphiN} for
the last identity. The result follows by defining $N_{\varphi c_{h}}^{\NF}:=\left(N_{\varphi}^{\NF}\right)^{h}$.
\end{enumerate}
\end{proof}
\begin{lem}
\label{lem:N_kphi^NF}With the notation of Lemma \ref{lem:Normalizer-via-phi.}
for every $k\in K$ we have that $\lui{k}{\left(\prescript{\NF}{\varphi}{N}\right)}=N_{c_{k}\varphi}^{\NF}$.
Moreover, if $\varphi=\varphi^{\NF}$, we also have that $c_{k}\varphi=\left(c_{k}\varphi\right)^{\NF}$
and that $N_{\varphi}^{\NF}=N_{c_{k}\varphi}^{\NF}$.
\end{lem}

\begin{proof}
First of all notice that
\begin{align*}
\prescript{\NF}{c_{k}\varphi}{N} & =\left\{ x\in N_{K}\left(\lui{k}{\left(\varphi\left(A\right)\right)}\right)\,:\,\left(c_{k}\varphi\right)^{-1}c_{x}c_{k}\varphi\in\Aut_{\NF}\left(A\right)\right\} ,\\
 & =\left\{ x\in\lui{k}{\left(N_{K}\left(\varphi\left(A\right)\right)\right)}\,:\,\varphi^{-1}c_{x^{k}}\varphi\in\Aut_{\NF}\left(A\right)\right\} ,\\
 & =\left\{ \lui{k}{y}\,:\,y\in N_{K}\left(\varphi\left(A\right)\right)\text{ and }\varphi^{-1}c_{y}\varphi\in\Aut_{\NF}\left(A\right)\right\} =\lui{k}{\left(\prescript{\NF}{\varphi}{N}\right)}.
\end{align*}
This proves the first half of the lemma. For the second part we can
use the above and the identity $\varphi=\varphi^{\NF}$ to obtain
the identities below
\[
\lui{c_{k}\varphi}{\Aut_{N_{\varphi}^{\NF}}\left(A\right)}=\lui{c_{k}}{\left(\Aut_{\prescript{\NF}{\varphi}{N}}\left(\varphi\left(A\right)\right)\right)}=\Aut_{\lui{k}{\left(\prescript{\NF}{\varphi}{N}\right)}}\left(c_{k}\varphi\left(A\right)\right)=\Aut_{\prescript{\NF}{c_{k}\varphi}{N}}\left(c_{k}\varphi\left(A\right)\right).
\]
This proves both that $c_{k}\varphi=\left(c_{k}\varphi\right)^{\NF}$
and that $N_{\varphi}^{\NF}=N_{c_{k}\varphi}^{\NF}$.
\end{proof}
Lemma \ref{lem:N_kphi^NF} allows us to introduce the following definition.
\begin{defn}
\label{def:N_varphi^-^NF.}Let $A,K\in\Fc$ with $A\le\NS$, let $\overline{\varphi}\in\Hom_{\F}\left(A,K\right)$
if there exists a representative $\varphi$ of $\overline{\varphi}$
such that $\varphi=\varphi^{\NF}$ then, from Lemma \ref{lem:N_kphi^NF}
this happens for every representative of $\overline{\varphi}$. Then
we write $\overline{\varphi}=\overline{\varphi}^{\NF}$ and we define
the \textbf{normalizer before $\overline{\varphi}$ in $\NF$} as
$N_{\overline{\varphi}}^{\NF}:=N_{\varphi}^{\NF}$ which does not
depend on the choice of representative $\varphi$ of $\overline{\varphi}$
because of Lemma \ref{lem:N_kphi^NF}.
\end{defn}

As we will see in Theorem \ref{thm:decomposition-direct-product.},
for every $H,K\in\Fc$ and every $\left(A,\overline{\varphi}\right)\in\left[H\times K\right]$
such that $\overline{\varphi}=\overline{\varphi}^{\NF}$ the subgroups
$N_{\overline{\varphi}}^{\NF}$ of $H$ play, in the context of fusion
systems, a role analogous to the one that the groups $H^{x}\cap K$
play in Equation \eqref{eq:decomposition-biset-2.}. Let's now look
into what objects play, in the context of fusion systems, a role analogous
to that of the biset representatives $\left[H\backslash G/K\right]$
of Equation \eqref{eq:decomposition-biset-2.}.
\begin{defn}
\label{def:NFXK.}Let $K\in\Fc$. We define an equivalence relation
$\sim$ in $\left[H\times_{\F}K\right]$ by setting $\left(A,\overline{\varphi}\right)\sim\left(B,\overline{\psi}\right)$
if and only if there exists an isomorphism $\overline{\theta}\in\Hom_{\mathcal{O}\left(\NF\right)}\left(A,B\right)$
such that $\overline{\varphi}=\overline{\psi}\overline{\theta}$.
Lemma \ref{lem:Normalizer-via-phi.} ensures us that for each equivalence
class in $\left[H\times K\right]/\sim$ we can choose one representative
$\left(A,\overline{\varphi}\right)$ such that $A$ is fully $\NF$-normalized
and $\overline{\varphi}=\overline{\varphi}^{\NF}$. We define the
\textbf{product of $\NF$ and $K$ in $\F$ }to be any subset $\left[\NF\times K\right]\subseteq\left[H\times_{\F}K\right]$
formed by such representatives.
\end{defn}

We want the elements $\left[\NF\times K\right]$ and $N_{\overline{\varphi}}^{\NF}$
to play, in the context of fusion systems, the same role that the
elements $\left[H\backslash G/K\right]$ and $H^{x}\cap K$ play in
Equation \eqref{eq:decomposition-biset-2.}. In order to do so we
need to be able to define something analogous to the set $x\left[H^{x}\cap K\backslash K/J\right]$
of Equation \eqref{eq:decomposition-biset-2.}. In other words, for
every $\left(A,\overline{\varphi}\right)\in\left[\NF\times K\right]$
we need to be able to lift the morphism $\overline{\varphi}:A\to K$
in a unique way to a morphism $\overline{\hat{\varphi}}:N_{\overline{\varphi}}^{\NF}\to K$.
\begin{prop}
\label{prop:Extension-to-normalizer.}Let $K\in\Fc$, let $\left(A,\overline{\varphi}\right)\in\left[\NF\times K\right]$
and let $\varphi$ be a representative of $\overline{\varphi}$. There
exists a morphism $\hat{\varphi}:N_{\varphi}^{\NF}\to K$ in $\F$
such that $\varphi=\hat{\varphi}\iota_{A}^{N_{\varphi}^{\NF}}$. In
particular, from Corollary \ref{cor:varphi(N_=00007Bvarphi=00007D^=00007B=00005CNF=00007D).},
we have that $\hat{\varphi}\left(N_{\varphi}^{\NF}\right)=\prescript{\NF}{\varphi}{N}\le N_{K}\left(\varphi\left(A\right)\right)$.
Moreover there exists a unique morphism $\overline{\hat{\varphi}}:N_{\overline{\varphi}}^{\NF}\to K$
in $\OFc$ such that $\overline{\hat{\varphi}}\overline{\iota_{A}^{N_{\overline{\varphi}}^{\NF}}}=\overline{\varphi}$
and $\hat{\varphi}$ is necessarily a representative of $\overline{\hat{\varphi}}$.
\end{prop}

\begin{proof}
If the first part of the statement is satisfied then the morphism
$\overline{\hat{\varphi}}$ in $\OFc$ having representative $\hat{\varphi}$
satisfies $\overline{\hat{\varphi}}\overline{\iota_{A}^{N_{\overline{\varphi}}^{\NF}}}=\overline{\varphi}$.
From \cite[Theorem 4.9]{IntroductionToFusionSystemsLinckelmann} we
know that $\overline{\iota_{A}^{N_{\overline{\varphi}}^{\NF}}}$ is
an epimorphism. Therefore, for any morphism $\overline{\psi}\colon N_{\varphi}^{\NF}\to K$
in $\OFc$ satisfying $\overline{\psi}\overline{\iota_{A}^{N_{\varphi}^{\NF}}}=\overline{\varphi}=\overline{\hat{\varphi}}\overline{\iota_{A}^{N_{\varphi}^{\NF}}}$,
we must necessarily have $\overline{\psi}=\overline{\hat{\varphi}}$.
This proves uniqueness of $\overline{\hat{\varphi}}$.

We are now only left with proving that there exists a morphism $\hat{\varphi}$
as in the statement.

We know from definition of $\left[\NF\times K\right]$ that $A\in\Fc$.
Therefore we must also have $\varphi\left(A\right)\in\Fc$ and, in
particular, $\varphi\left(A\right)$ is fully $\F$-centrialized.
From \cite[Proposition 4.4]{StancuEquivalentDefinitions} (see also
\cite[Proposition 2.7]{IntroductionToFusionSystemsLinckelmann}) we
can now deduce that there exists a morphism $\psi\colon N_{\iota_{K}^{S}\varphi}\to S$
(see Definition \ref{def:varphi-normalizer.}) such that $\psi\iota_{A}^{N_{\iota_{K}^{S}\varphi}}=\iota_{K}^{S}\varphi$.
By definition of $N_{\varphi}^{\NF}$ (see Lemma \ref{lem:Normalizer-via-phi.})
we have that $N_{\varphi}^{\NF}\le N_{\iota_{K}^{S}\varphi}$. Therefore
we can apply Corollary \ref{cor:varphi(N_=00007Bvarphi=00007D^=00007B=00005CNF=00007D).}
(taking $\hat{\varphi}:=\psi\iota_{N_{\varphi}^{\NF}}^{N_{\iota_{K}^{S}\varphi}}$)
to deduce that $\psi\left(N_{\varphi}^{\NF}\right)=\prescript{\NF}{\varphi}{N}$.
In particular we have that $\psi\left(N_{\varphi}^{\NF}\right)\le K$.
This allows us to define the morphism $\hat{\varphi}\colon N_{\varphi}^{\NF}\to K$
in $\F$ by setting $\hat{\varphi}\left(x\right):=\psi\left(x\right)$
for every $x\in N_{\varphi}^{\NF}$. Since $\psi\iota_{A}^{N_{\iota_{K}^{S}\varphi}}=\iota_{K}^{S}\varphi$
then we have that $\hat{\varphi}=\varphi\iota_{A}^{N_{\varphi}^{\NF}}$
thus completing the proof.
\end{proof}
Notice that maximality of the pairs $\left(A,\overline{\varphi}\right)\in\left[\NF\times K\right]$
does not imply that the pair $\left(N_{\overline{\varphi}}^{\NF},\overline{\hat{\varphi}}\right)$
given by Proposition \ref{prop:Extension-to-normalizer.} satisfies
$\left(N_{\overline{\varphi}}^{\NF},\overline{\hat{\varphi}}\right)=\left(A,\overline{\varphi}\right)$
since we might have that $N_{\overline{\varphi}}^{\NF}\not\le H$.
However, we have the following corollary.
\begin{cor}
\label{cor:A=00003DHcapN.}Let $K\in\Fc$ and let $\left(A,\overline{\varphi}\right)\in\left[\NF\times K\right]$
then $A=N_{\overline{\varphi}}^{\NF}\cap H$. 
\end{cor}

\begin{proof}
From Lemma \ref{lem:Normalizer-via-phi.} we know that $A\le N_{\overline{\varphi}}^{\NF}$
and from Definitions \ref{def:HX_FK} and \ref{def:NFXK.} we know
that $A\le H$. Therefore we can deduce that $A\le N_{\overline{\varphi}}^{\NF}\cap H$.
From this and using the notation of Proposition \ref{prop:Extension-to-normalizer.}
we obtain the identity $\overline{\hat{\varphi}\iota_{N_{\overline{\varphi}}^{\NF}\cap H}^{N_{\overline{\varphi}}^{\NF}}}\,\overline{\iota_{A}^{N_{\overline{\varphi}}^{\NF}\cap H}}=\overline{\varphi}$.
Since $N_{\varphi}^{\NF}\cap H\le H$ then we can deduce from maximality
of the pair $\left(A,\overline{\varphi}\right)$ that $A=N_{\overline{\varphi}}^{\NF}\cap H$
thus concluding the proof.
\end{proof}
From Corollary \ref{cor:A=00003DHcapN.} we now obtain the following.
\begin{lem}
\label{lem:psi(B)-is-A.}Let $K\in\NF\cap\Fc$ and let $\left(A,\overline{\varphi}\right)\in\left[\NF\times K\right]$.
For every $\left(B,\overline{\psi}\right)\in\left[H\times_{\NF}N_{\overline{\varphi}}^{\NF}\right]$
and every representative $\psi$ of $\overline{\psi}$ we have that
$\psi\left(B\right)=N_{\overline{\varphi}}^{\NF}\cap H=A$.
\end{lem}

\begin{proof}
Because of Corollary \ref{cor:A=00003DHcapN.} we just need to prove
that $\psi\left(B\right)=N_{\overline{\varphi}}^{\NF}\cap H$ for
any representative $\psi$ of $\overline{\psi}$. Since $\psi$ is
a morphism in $\NF$ and, by definition, every morphism on $\NF$
sends subgroups of $H$ to subgroups of $H$, we can conclude that
$\psi\left(B\right)\le H$. Therefore we must have $\psi\left(B\right)\le N_{\overline{\varphi}}^{\NF}\cap H$.
From definition of $\NF$ this implies that there exists $\hat{\psi}\in\Aut_{\NF}\left(H\right)$
such that $\hat{\psi}\left(x\right)=\psi\left(x\right)$ for every
$x\in B$. Define now $\theta\colon C:=\hat{\psi}^{-1}\left(N_{\overline{\varphi}}^{\NF}\cap H\right)\to N_{\overline{\varphi}}^{\NF}$
by setting $\theta\left(x\right):=\hat{\psi}\left(x\right)$ for every
$x\in C$. Then we have that $B\le C$ and $\overline{\theta}\overline{\iota_{B}^{C}}=\overline{\psi}$.
Since $\theta$ is a morphism in $\NF$ and $C\le H$ then, from maximality
of the pair $\left(B,\overline{\psi}\right)$ (see Definition \ref{def:HX_FK}),
we can conclude that $\left(C,\overline{\theta}\right)=\left(B,\overline{\psi}\right)$.
In particular we have that $\psi\left(B\right)=\theta\left(C\right)$
and since $\theta\left(C\right)=\hat{\psi}\left(\hat{\psi}^{-1}\left(N_{\overline{\varphi}}^{\NF}\cap H\right)\right)=N_{\overline{\varphi}}^{\NF}\cap H$
the result follows.
\end{proof}
We have now gathered all the ingredients needed to prove Theorem \ref{thm:decomposition-direct-product.}
with which we conclude this subsection.
\begin{thm}
\label{thm:decomposition-direct-product.}Using Notation \ref{nota:X-and-Y.}
let $K\in\Fc$ (see Definition \ref{def:F-centric.}) and for every
$\left(A,\overline{\varphi}\right)\in\left[\NF\times K\right]$ (see
Example \ref{exa:definition-NF.} and Definition \ref{def:NFXK.})
let $\overline{\hat{\varphi}}$ be as in Proposition \ref{prop:Extension-to-normalizer.}.
Then, for every $\left(A,\overline{\varphi}\right)\in\left[\NF\times K\right]$
we can take $\left[H\times_{\NF}N_{\overline{\varphi}}^{\NF}\right]$
(see Definition \ref{def:HX_FK} and Lemma \ref{lem:N_kphi^NF}) so
that
\[
\left[H\times_{\F}K\right]=\bigsqcup_{\left(A,\overline{\varphi}\right)\in\left[\NF\times K\right]}\overline{\hat{\varphi}}\left[H\times_{\NF}N_{\overline{\varphi}}^{\NF}\right],
\]
where
\[
\overline{\hat{\varphi}}\left[H\times_{\NF}N_{\overline{\varphi}}^{\NF}\right]:=\bigsqcup_{\left(B,\overline{\psi}\right)\in\left[H\times_{\NF}N_{\overline{\varphi}}^{\NF}\right]}\left\{ \left(B,\overline{\hat{\varphi}}\overline{\psi}\right)\right\} .
\]
\end{thm}

\begin{proof}
Let $\left(A,\overline{\varphi}\right)\in\left[\NF\times K\right]$,
let $\left(B,\overline{\psi}\right)\in\left[H\times_{\NF}N_{\overline{\varphi}}^{\NF}\right]$
and let $\psi$ be a representative of $\overline{\psi}$ and let
$\tilde{\psi}:B\to\psi\left(B\right)$ be the morphism $\psi$ seen
as an isomorphism onto its image. From Lemma \ref{lem:psi(B)-is-A.}
we know that $\tilde{\psi}$ is in fact an isomorphism from $B$ to
$A$. From definition of $\NF$ (see Example \ref{exa:definition-NF.})
we can now choose an automorphism $\hat{\psi}\in\Aut_{\NF}\left(H\right)$
satisfying $\hat{\psi}\iota_{B}^{H}=\iota_{A}^{H}\psi$. For every
$D\le H$ we denote by $\hat{\psi}_{D}\colon D\to\hat{\psi}\left(D\right)$
the isomorphism in $\NF$ obtained by restricting $\hat{\psi}$ to
$D$.

From the universal property of products, we know that there exist
a unique $\left(C,\overline{\theta}\right)\in\left[H\times_{\F}K\right]$
and a unique $\overline{\gamma}\colon B\to C$ such that $\overline{\theta}\overline{\gamma}=\overline{\hat{\varphi}}\overline{\psi}$
and $\overline{\iota_{C}^{H}}\overline{\gamma}=\overline{\iota_{B}^{H}}$.
From the second identity we can deduce that $\overline{\gamma}\in\Orbitize{\FHH}$
(see Example \ref{exa:fusion-system.} and Definition \ref{def:Orbit-category.}).
Therefore we can choose $\left[H\times_{\NF}N_{\overline{\varphi}}^{\NF}\right]$
so that $B\le C$ and $\overline{\gamma}=\overline{\iota_{B}^{C}}$.
With this setup the first identity can be rewritten as $\overline{\theta}\overline{\iota_{B}^{C}}=\overline{\hat{\varphi}}\overline{\psi}=\overline{\varphi}\overline{\tilde{\psi}}$.
Using this and the notation introduced at the start we obtain the
identity $\overline{\theta}\overline{\hat{\psi}_{C}^{-1}}\overline{\iota_{A}^{\hat{\psi}\left(C\right)}}=\overline{\varphi}$.
Since $\left(C,\overline{\theta}\right)\in\left[H\times_{\F}K\right]$
we know from Proposition \ref{prop:properties-HXK.} \eqref{enu:prod-iso-left.}
that there exists $h\in H$ satisfying $\left(\hat{\psi}\left(C\right)^{h},\overline{\theta\hat{\psi}_{C}^{-1}c_{h}}\right)\in\left[H\times_{\F}K\right]$.
With this setup we obtain the identities $\overline{\varphi}=\overline{\theta\hat{\psi}_{C}^{-1}c_{h}}\,\overline{c_{h^{-1}}\iota_{A}^{\hat{\psi}\left(C\right)}}$
and $\overline{\iota_{A}^{H}}=\overline{\iota_{\hat{\psi}\left(C\right)^{h}}^{H}}\,\overline{c_{h^{-1}}\iota_{A}^{\hat{\psi}\left(C\right)}}$.
Since $\left(A,\overline{\varphi}\right)\in\left[\NF\times K\right]\subseteq\left[H\times_{\F}K\right]$
we can conclude from the previous identities and the universal properties
of product that $\left(A,\overline{\varphi}\right)=\left(\hat{\psi}\left(C\right)^{h},\overline{\theta\hat{\psi}_{C}^{-1}c_{h}}\right)$
and $\overline{c_{h^{-1}}\iota_{A}^{\hat{\psi}\left(C\right)}}=\overline{\Id_{A}}$.
In particular we have that $A=\hat{\psi}\left(C\right)$ and, from
Lemma \ref{lem:psi(B)-is-A.}, we can conclude that $B=C$. This implies
that $\overline{\iota_{B}^{C}}=\overline{\Id_{B}}$ and this allows
us to deduce from the identity $\overline{\theta}\overline{\iota_{B}^{C}}=\overline{\hat{\varphi}}\overline{\psi}$
that $\left(B,\overline{\hat{\varphi}}\overline{\psi}\right)=\left(C,\overline{\theta}\right)\in\left[H\times_{\F}K\right]$
thus proving that
\begin{equation}
\bigcup_{\left(A,\overline{\varphi}\right)\in\left[\NF\times_{\F}K\right]}\bigcup_{\left(B,\overline{\psi}\right)\in\left[H\times_{\NF}N_{\overline{\varphi}}^{\NF}\right]}\left\{ \left(B,\overline{\hat{\varphi}}\overline{\psi}\right)\right\} \subseteq\left[H\times_{\F}K\right].\label{eq:aux-first-inclusion.}
\end{equation}

Let's now prove the other inclusion. By definition of $\left[\NF\times K\right]$
(see Definition \ref{def:NFXK.}) we know that for every $\left(C,\overline{\theta}\right)\in\left[H\times_{\F}K\right]$
there exists a unique $\left(A,\overline{\varphi}\right)\in\left[\NF\times K\right]$
and an isomorphism $\overline{\gamma}:C\to A$ in $\ONF$ such that
$\overline{\theta}=\overline{\varphi}\overline{\gamma}=\overline{\hat{\varphi}}\overline{\iota_{A}^{N_{\overline{\varphi}}^{\NF}}}\overline{\gamma}$.
From the universal properties of products we know that there exists
a unique $\left(B,\overline{\psi}\right)\in\left[H\times_{\NF}N_{\varphi}^{\NF}\right]$
and a unique map $\overline{\delta}:C\to B$ such that $\overline{\iota_{B}^{H}}\overline{\delta}=\overline{\iota_{C}^{H}}$
and that $\overline{\psi}\,\overline{\delta}=\overline{\iota_{A}^{N_{\overline{\varphi}}^{\NF}}}\overline{\gamma}$.
Joining the second identity with the previous one we obtain the identity
$\overline{\theta}=\overline{\hat{\varphi}}\overline{\psi}\overline{\delta}$.
Since $\left(B,\overline{\hat{\varphi}}\overline{\psi}\right)\in\left[H\times_{\F}K\right]$
as shown in the first part of the proof we can conclude from the universal
properties of products that $\overline{\delta}=\overline{\Id_{B}}$
and $\left(B,\overline{\hat{\varphi}}\overline{\psi}\right)=\left(C,\overline{\theta}\right)$.
This gives us the inclusion dual to that of Equation \eqref{eq:aux-first-inclusion.}
thus leading us to the identity
\[
\left[H\times_{\F}K\right]=\bigcup_{\left(A,\overline{\varphi}\right)\in\left[\NF\times_{\F}K\right]}\bigcup_{\left(B,\overline{\psi}\right)\in\left[H\times_{\NF}N_{\overline{\varphi}}^{\NF}\right]}\left\{ \left(B,\overline{\hat{\varphi}}\overline{\psi}\right)\right\} .
\]

We are now only left with proving that the above unions are disjoint.

Let $\left(A,\overline{\varphi}\right),\left(A',\overline{\varphi'}\right)\in\left[\NF\times K\right]$,
let $\left(B,\overline{\psi}\right)\in\left[H\times_{\NF}N_{\overline{\varphi}}^{\NF}\right]$
and let $\left(B',\overline{\psi'}\right)\in\left[H\times_{\NF}N_{\overline{\varphi'}}^{\NF}\right]$
such that $\left(B',\overline{\hat{\varphi'}}\overline{\psi'}\right)=\left(B,\overline{\hat{\varphi}}\overline{\psi}\right)$.
Fix representatives, $\psi$ and $\psi'$ of $\overline{\psi}$ and
$\overline{\psi'}$ respectively and let $\tilde{\psi}$ and $\tilde{\psi'}$
be as defined at the start of the proof and let $\overline{\hat{\varphi}}$
and $\overline{\hat{\varphi'}}$ be as in Proposition \ref{prop:Extension-to-normalizer.}.
With this setup we have that $\overline{\hat{\varphi}}\overline{\psi}=\overline{\varphi}\overline{\tilde{\psi}}$
and that $\overline{\hat{\varphi'}}\overline{\psi'}=\overline{\varphi'}\overline{\tilde{\psi'}}$.
Thus, from the identity $\overline{\hat{\varphi}}\overline{\psi}=\overline{\hat{\varphi'}}\overline{\psi'}$
we can conclude that $\overline{\varphi}=\overline{\varphi'}\overline{\tilde{\psi'}}\overline{\tilde{\psi}^{-1}}$.
Since $\overline{\tilde{\psi'}}\overline{\tilde{\psi}^{-1}}$ is an
isomorphism in $\ONF$ then, by definition of $\left[\NF\times K\right]$,
we can conclude that $\left(A,\overline{\varphi}\right)=\left(A',\overline{\varphi'}\right)$.

Take now a representative $\varphi$ of $\overline{\varphi}$ and
let $\tilde{\varphi}\colon A\to\varphi\left(A\right)$ be the isomorphism
obtained by viewing $\varphi$ as an isomorphism onto its image. With
this setup we have that $\overline{\varphi}\overline{\tilde{\psi}}=\overline{\iota_{\varphi\left(A\right)}^{K}}\overline{\tilde{\varphi}}\overline{\tilde{\psi}}$
and that $\overline{\varphi}\overline{\tilde{\psi'}}=\overline{\iota_{\varphi\left(A\right)}^{K}}\overline{\tilde{\varphi}}\overline{\tilde{\psi'}}$.
Thus we obtain the identity $\overline{\iota_{\varphi\left(A\right)}^{K}}\overline{\tilde{\varphi}}\overline{\tilde{\psi}}=\overline{\iota_{\varphi\left(A\right)}^{K}}\overline{\tilde{\varphi}}\overline{\tilde{\psi'}}$
and we can deduce that there exists $k\in K$ such that $c_{k}\tilde{\varphi}\tilde{\psi}=\tilde{\varphi}\tilde{\psi'}$
as isomorphisms from $B$ to $\varphi\left(A\right)$. Since $\tilde{\varphi}\tilde{\psi}$
is an isomorphism from $B$ to $A$ then, from the previous identity,
we can conclude that $k\in N_{K}\left(\varphi\left(A\right)\right)$.
Always from the previous identity we obtain the identity $\tilde{\varphi}^{-1}c_{k}\tilde{\varphi}=\tilde{\psi'}\tilde{\psi}^{-1}$.
Since both $\tilde{\psi}$ and $\tilde{\psi'}$ are isomorphism in
$\NF$ we can deduce from Lemma \ref{lem:psi(B)-is-A.} that $\tilde{\psi'}\tilde{\psi}^{-1}\in\Aut_{\NF}\left(A\right)$.
Thus, from Lemma \ref{lem:Normalizer-via-phi.}, we have that $k\in\prescript{\NF}{\varphi}{N}$.
Now let $\hat{\varphi}$ be a representative of $\overline{\hat{\varphi}}$
as in Proposition \ref{prop:Extension-to-normalizer.}. From Corollary
\ref{cor:varphi(N_=00007Bvarphi=00007D^=00007B=00005CNF=00007D).}
we know that there exists a unique $k'\in N_{\overline{\varphi}}^{\NF}$
such that $k=\hat{\varphi}\left(k'\right)$. With this setup we can
deduce that $\tilde{\varphi}^{-1}c_{k}\tilde{\varphi}=\tilde{\varphi}^{-1}\tilde{\varphi}c_{k'}=c_{k'}$.
Thus we can conclude that $c_{k'}\tilde{\psi}=\tilde{\psi'}$ and,
since $k'\in N_{\overline{\varphi}}^{\NF}$ we can conclude by definition
of the orbit category (see Definition \ref{def:Orbit-category.})
that $\overline{\psi}=\overline{\iota_{A}^{K}c_{k'}\tilde{\psi}}=\overline{\iota_{A}^{K}\tilde{\psi'}}=\overline{\psi'}$
thus concluding the proof.
\end{proof}

\subsection{\label{subsec:Transfer-map-from-trHNF-to-trHF.}The transfer from
$\NF$ to $\F$.}

As we already explained at the beginning of Subsection \ref{subsec:decomposition-of-product-in-aOFc.}
given a fusion system $\mathcal{K}\subseteq\F$ and a Mackey functor
$M\in\MackFcR$ the transfer $\tr_{\mathcal{K}}^{\F}\colon\End\left(M\downarrow_{\mathcal{K}}^{\F}\right)\to\End\left(M\right)$
is, in general, not defined. However, just like Equation \eqref{eq:decomposition-biset-2.}
can be used in the case of Mackey functors over finite groups in order
to prove that transfer maps compose nicely, Theorem \ref{thm:decomposition-direct-product.}
can be used in the case of centric Mackey functors over fusion systems
in order to define $\tr_{\NF}^{\F}$ and prove that $\tr_{\NF}^{\F}\tr_{\FHH}^{\NF}=\tr_{\FHH}^{\F}$
(see Examples \ref{exa:fusion-system.} and \ref{exa:definition-NF.}
for notation). More precisely we have the following.
\begin{defn}
\label{def:tr_NF^F.}Let $\R$ be a $p$-local ring and let $M\in\MackFcR$.
From Proposition \ref{prop:Inverse-of-S.} we know that the isomorphism
class $\overline{S}\in\BFcR$ of $S$ has an inverse in $\BFcR$.\textbf{
}Thus, using the notation of Examples \ref{exa:fusion-system.} and
\ref{exa:definition-NF.} and of Proposition \ref{prop:Action-of-centric-burnside-ring.},
we define the \textbf{transfer from $\NF$ to $\F$} as the $\R$-module
morphism
\[
\lui{M}{\tr_{\NF}^{\F}}:=\left(\overline{S}^{-1}\cdot\right)_{*}\sum_{\left(A,\overline{\varphi}\right)\in\left[\NF\times S\right]}\lui{M}{\tr_{\FHH[N_{\overline{\varphi}}^{\NF}]}^{\F}}\lui{M\downarrow_{\NF}^{\F}}{\r_{\FHH[N_{\overline{\varphi}}^{\NF}]}^{\NF}}\colon\End\left(M\downarrow_{\NF}^{\F}\right)\to\End\left(M\right).
\]
From Items \eqref{enu:Transfer-eliminates-conjugation.} and \eqref{enu:Restriction-eliminates-conjugation.}
of Proposition \ref{prop:Properties-of-transfer-restriction-and-conjugation-maps.}
we know that $\lui{M}{\tr_{\NF}^{\F}}$ does not depend on the choice
of the set $\left[\NF\times S\right]$ (see Definition \ref{def:NFXK.}).
Whenever there is no confusion regarding $M$ we will simply write
$\tr_{\NF}^{\F}:=\lui{M}{\tr_{\NF}^{\F}}$.
\end{defn}

\begin{lem}
\label{lem:transfer-composes-nicely.}Let $\R$ be a $p$-local ring.
Regardless of the centric Mackey functor involved we have that $\tr_{\NF}^{\F}\tr_{\FHH}^{\NF}=\tr_{\FHH}^{\F}$.
In particular $\tr_{\NF}^{\F}$ sends $\Tr_{H}^{\NF}$ surjectively
onto $\Tr_{H}^{\F}$ (see Definition \ref{def:Image-transfer.}).
\end{lem}

\begin{proof}
For every $\left(A,\overline{\varphi}\right)\in\left[\NF\times S\right]$
fix a representative $\varphi$ of $\overline{\varphi}$ and a representative
$\hat{\varphi}$ of $\overline{\hat{\varphi}}$ as in Proposition
\ref{prop:Extension-to-normalizer.} and for every $\left(B,\overline{\psi}\right)\in\left[H\times_{\NF}N_{\overline{\varphi}}^{\NF}\right]$
fix a representative $\psi$ of $\overline{\psi}$. Let $\tilde{\varphi},\tilde{\hat{\varphi}}$
and $\tilde{\psi}$ denote the isomorphisms obtained by viewing $\varphi$,
$\hat{\varphi}$ and $\psi$ as isomorphisms onto their images. With
this notation and the notation of Proposition \ref{prop:Action-of-centric-burnside-ring.}
we have that
\begin{align*}
\left(\overline{S}\cdot\right)_{*}\tr_{\NF}^{\F}\tr_{\FHH}^{\NF} & =\sum_{\left(A,\overline{\varphi}\right)\in\left[\NF\times S\right]}\tr_{\FHH[\tilde{\hat{\varphi}}\left(N_{\overline{\varphi}}^{\NF}\right)]}^{\F}\lui{\tilde{\hat{\varphi}}}{\cdot}\r_{\FHH[N_{\overline{\varphi}}^{\NF}]}^{\NF}\tr_{\FHH}^{\NF},\\
 & =\sum_{\left(A,\overline{\varphi}\right)\in\left[\NF\times S\right]}\sum_{\left(B,\overline{\psi}\right)\in\left[H\times_{\NF}N_{\overline{\varphi}}^{\NF}\right]}\tr_{\FHH[\tilde{\hat{\varphi}}\left(N_{\overline{\varphi}}^{\NF}\right)]}^{\F}\lui{\tilde{\hat{\varphi}}}{\cdot}\tr_{\FHH[\tilde{\psi}\left(B\right)]}^{\FHH[N_{\overline{\varphi}}^{\NF}]}\lui{\tilde{\psi}}{\cdot}\r_{\FHH[B]}^{\FHH},\\
 & =\sum_{\left(A,\overline{\varphi}\right)\in\left[\NF\times S\right]}\sum_{\left(B,\overline{\psi}\right)\in\left[H\times_{\NF}N_{\overline{\varphi}}^{\NF}\right]}\tr_{\FHH[\tilde{\varphi}\left(\tilde{\psi}\left(B\right)\right)]}^{\F}\lui{\tilde{\varphi}\tilde{\psi}}{\cdot}\r_{\FHH[B]}^{\FHH}.
\end{align*}
Where we are using Item \eqref{enu:Transfer-eliminates-conjugation.}
of Proposition \ref{prop:Properties-of-transfer-restriction-and-conjugation-maps.}
for the first identity, Item \eqref{enu:Mackey-formula.} for the
second identity and Items \eqref{enu:composition-transfer.} and \eqref{enu:Transfer-commutes-with-conjugation.}
for the third identity. From Theorem \ref{thm:decomposition-direct-product.}
we can now replace the two sums in the last line of the previous equation
with a sum over the pairs $\left(C,\overline{\theta}\right)\in\left[H\times_{\F}S\right]$
and replace the isomorphisms $\tilde{\varphi}\tilde{\psi}$ with the
isomorphisms $\tilde{\theta}$ where $\theta$ is a representative
of $\overline{\theta}$ and $\tilde{\theta}$ is the isomorphism obtained
by viewing $\theta$ as an isomorphism onto its image. With this change
in mind we can apply Items \eqref{enu:composition-transfer.} and
\eqref{enu:Mackey-formula.} of Proposition \ref{prop:Properties-of-transfer-restriction-and-conjugation-maps.}
in order to obtain the identity $\left(\overline{S}\cdot\right)_{*}\tr_{\NF}^{\F}\tr_{\FHH}^{\NF}=\tr_{\FHH[S]}^{\F}\r_{\FHH[S]}^{\F}\tr_{\FHH}^{\F}$.
Applying now Item \eqref{enu:Restriction-and-transfer-makes-burnside.}
we obtain from here the identity $\left(\overline{S}\cdot\right)_{*}\tr_{\NF}^{\F}\tr_{\FHH}^{\NF}=\left(\overline{S}\cdot\right)_{*}\tr_{\FHH}^{\F}$.
Finally, from Propositions \ref{prop:Inverse-of-S.} and \ref{prop:Action-of-centric-burnside-ring.},
we know that $\left(\overline{S}\cdot\right)_{*}$ is invertible and,
therefore, we can deduce from the previous identity that $\tr_{\NF}^{\F}\tr_{\FHH}^{\NF}=\tr_{\FHH}^{\F}$
thus concluding the proof.
\end{proof}
\begin{cor}
\label{cor:tr-maps-Tr^NF-to-Tr^F.}Let $\R$ be a $p$-local ring
and let $M\in\MackFcR$. For every family $\mathfrak{X}$ of elements
in $\FHH\cap\Fc$ we have that $\tr_{\NF}^{\F}\left(\Tr_{\mathfrak{X}}^{\NF}\right)=\Tr_{\mathfrak{X}}^{\F}$.
\end{cor}

\begin{proof}
Because of linearity of $\tr_{\NF}^{\F}$ it suffices to prove the
statement when $\mathfrak{X}=\left\{ K\right\} $ for some $K\in\FHH\cap\Fc$.
From Proposition \ref{prop:Properties-of-transfer-restriction-and-conjugation-maps.}
\eqref{enu:composition-transfer.} we know that $\tr_{\FHH}^{\F}\tr_{\FHH[K]}^{\FHH}=\tr_{\FHH[K]}^{\F}$
and that $\tr_{\FHH}^{\NF}\tr_{\FHH[K]}^{\FHH}=\tr_{\FHH[K]}^{\NF}$.
Thus, from definition of $\Tr_{K}^{\F}$ and $\Tr_{K}^{\NF}$ (see
Definition \ref{def:Image-transfer.}), we have that $\Tr_{K}^{\F}=\tr_{\FHH}^{\F}\left(\Tr_{K}^{\FHH}\right)$
and that $\Tr_{K}^{\NF}=\tr_{\FHH}^{\NF}\left(\Tr_{K}^{\FHH}\right)$.
Since $\Tr_{K}^{\NF}\subseteq\End\left(M\downarrow_{\NF}^{\F}\right)$
the result now follows from the above and Lemma \ref{lem:transfer-composes-nicely.}
after applying $\tr_{\NF}^{\F}$ to $\Tr_{K}^{\NF}$.
\end{proof}
Corollary \ref{cor:tr-maps-Tr^NF-to-Tr^F.} allows us to give the
following definition.
\begin{defn}
\label{def:Quotient-tr}Let $\R$ be a $p$-local ring, let $M\in\MackFcR$,
let $\mathcal{X}$ be as in Notation \ref{nota:X-and-Y.} and let
$E_{\mathcal{X}}^{\NF}:=\End\left(M\downarrow_{\NF}^{\F}\right)/\Tr_{\mathcal{X}}^{\NF}$
and $E_{\mathcal{X}}^{\F}:=\End\left(M\right)/\Tr_{\mathcal{X}}^{\F}$.
We define the \textbf{quotient transfer from $\NF$ to $\F$ }as the
$\R$-module morphism $\overline{\tr_{\NF}^{\F}}:E_{\mathcal{X}}^{\NF}\to E_{\mathcal{X}}^{\F}$
obtained by setting $\overline{\tr_{\NF}^{\F}}\left(\overline{f}\right):=\overline{\tr_{\NF}^{\F}\left(f\right)}$
for every $f\in\End\left(M\downarrow_{\NF}^{\F}\right)$. Here we
are using the overline ($\overline{\cdot}$) to denote the projections
onto the appropriate quotients. Corollary \ref{cor:tr-maps-Tr^NF-to-Tr^F.}
assures us that $\overline{\tr_{\NF}^{\F}}\left(\overline{f}\right)$
does not depend of the chosen endomorphism $f$ projecting to $\overline{f}$.
\end{defn}

An important property of the $\R$-module morphism $\overline{\tr_{\NF}^{\F}}$
of Definition \ref{def:Quotient-tr} is the following.
\begin{lem}
\label{lem:quotient-is-multiplicative.}With notation as in Definition
\ref{def:Quotient-tr} view $\overline{\tr_{\NF}^{\F}}$ as a morphism
from $\overline{\Tr_{H}^{\NF}}:=\Tr_{H}^{\NF}/\Tr_{\mathcal{X}}^{\NF}$
to $\overline{\Tr_{H}^{\F}}:=\Tr_{H}^{\F}/\Tr_{\mathcal{X}}^{\F}$.
Then $\overline{\tr_{\NF}^{\F}}$ is surjective and commutes with
multiplication (i.e. $\overline{\tr_{\NF}^{\F}}\left(\overline{fg}\right)=\overline{\tr_{\NF}^{\F}}\left(\overline{f}\right)\overline{\tr_{\NF}^{\F}}\left(\overline{g}\right)$
for every $\overline{f},\overline{g}\in\overline{\Tr_{H}^{\NF}}$).
In particular, if $\Tr_{H}^{\NF}$ has a multiplicative unit, then
$\overline{\tr_{\NF}^{\F}}$ is a surjective $\R$-algebra morphism.
\end{lem}

\begin{proof}
During this proof we will use the overline symbol ($\overline{\cdot}$)
in order to represent the projection of an endomorphism on the appropriate
quotient ring. From Corollary \ref{cor:tr-maps-Tr^NF-to-Tr^F.} we
know that the map $\overline{\tr_{\NF}^{\F}}$ viewed as in the statement
is surjective. If $\Tr_{H}^{\NF}$ has a multiplicative unit ($1_{\Tr_{H}^{\NF}}$)
and commutes with multiplication then, from surjectivenes, we necessarily
have that $\overline{\tr_{\NF}^{\F}}\left(1_{\Tr_{H}^{\NF}}\right)=1_{\Tr_{H}^{\F}}$.
Thus we only need to prove that $\overline{\tr_{\NF}^{\F}}$ commutes
with multiplication.

Let $\mathcal{K}\in\left\{ \NF,\F\right\} $, let $\left(A,\overline{\varphi}\right)\in\left[H\times_{\mathcal{K}}H\right]$
and let $\varphi$ be a representative of $\overline{\varphi}$. Since
$A\le H$, $\varphi$ is injective and both $A$ and $H$ are finite
groups then we have that $\varphi\left(A\right)\lneq H$ unless $A=H$.
From definition of $\mathcal{X}$ (see Notation \ref{nota:X-and-Y.}),
this is equivalent to saying that $\varphi\left(A\right)\in\mathcal{X}$
unless $\varphi\in\Aut_{\NF}\left(H\right)=\Aut_{\F}\left(H\right)$.
On the other hand, from maximality of the pairs in $\left[H\times_{\mathcal{K}}H\right]$
(see Definition \ref{def:HX_FK}), we have that $\left(H,\overline{\varphi}\right)\in\left[H\times_{\mathcal{K}}H\right]$
for every $\overline{\varphi}\in\Aut_{\ONF}\left(H\right)=\Aut_{\OF}\left(H\right)$.
Let $M\in\MackFcR$ and let $f\in\End\left(M\downarrow_{\FHH}^{\F}\right)$.
From the above discussion and Proposition \ref{prop:Properties-of-transfer-restriction-and-conjugation-maps.}
\eqref{enu:Mackey-formula.} we can conclude that
\begin{align*}
\r_{\FHH}^{\mathcal{K}}\left(\tr_{\FHH}^{\mathcal{K}}\left(f\right)\right) & \in\lui{\mathcal{F},H}{f}+\Tr_{\mathcal{X}}^{\FHH}. & \text{where} &  & \lui{\mathcal{F},H}{f} & :=\sum_{\overline{\varphi}\in\Aut_{\OF}\left(H\right)}\lui{\varphi}{f}.
\end{align*}
From Proposition \ref{prop:Properties-of-transfer-restriction-and-conjugation-maps.}
\eqref{enu:composition-transfer.} we also have that $\tr_{\FHH}^{\mathcal{K}}\left(\Tr_{\mathcal{X}}^{\FHH}\right)=\Tr_{\mathcal{X}}^{\mathcal{K}}$.
Using the above and Proposition \ref{prop:Properties-of-transfer-restriction-and-conjugation-maps.}
\eqref{enu:Green-formula.} we can conclude that
\[
\overline{\tr_{\FHH}^{\mathcal{K}}\left(f\right)\tr_{\FHH}^{\mathcal{K}}\left(g\right)}=\overline{\tr_{\FHH}^{\mathcal{K}}\left(f\r_{\FHH}^{\mathcal{K}}\left(\tr_{\FHH}^{\mathcal{K}}\left(g\right)\right)\right)}=\overline{\tr_{\FHH}^{\mathcal{K}}\left(f\,\lui{\F,H}{g}\right)}
\]
for every $f,g\in\End\left(M\downarrow_{\FHH}^{\F}\right)$. On the
other hand we know from Lemma \ref{lem:transfer-composes-nicely.}
that $\overline{\tr_{\NF}^{\F}}\left(\overline{\tr_{\FHH}^{\NF}\left(\alpha\right)}\right)=\overline{\tr_{\FHH}^{\F}}\left(\overline{\alpha}\right)$
for $\alpha\in\left\{ f,g,f\,\lui{\F,H}{g}\right\} $ and, therefore,
we can conclude that
\[
\overline{\tr_{\NF}^{\F}}\left(\overline{\tr_{\FHH}^{\NF}\left(f\right)}\right)\overline{\tr_{\NF}^{\F}}\left(\overline{\tr_{\FHH}^{\NF}\left(g\right)}\right)=\overline{\tr_{\NF}^{\F}}\left(\overline{\tr_{\FHH}^{\NF}\left(f\lui{\F,H}{g}\right)}\right)=\overline{\tr_{\NF}^{\F}}\left(\overline{\tr_{\FHH}^{\NF}\left(f\right)}\overline{\tr_{\FHH}^{\NF}\left(g\right)}\right).
\]
Since all elements in $\overline{\Tr_{H}^{\NF}}$ are, by definition,
of the form $\overline{\tr_{\FHH}^{\NF}\left(f\right)}$ for some
$f\in\End\left(M\downarrow_{\FHH}^{\F}\right)$ the result follows.
\end{proof}

\subsection{\label{subsec:Proof-of-Green-correspondence.}Green correspondence
for centric Mackey functors.}

In this subsection we will follow similar ideas as those Sasaki uses
in \cite[Proposition 3.1]{SASAKI198298} in order to prove that the
Green correspondence holds for centric Mackey functors over fusion
systems (Theorem \ref{thm:Green-correspondence.}). To do so we will
need to replace some results valid for Mackey functors over finite
groups with the analogue results developed in Subsections \ref{subsec:N-is-a-direct-summand-of-Ninduction-restriction.}
to \ref{subsec:Transfer-map-from-trHNF-to-trHF.}. First however we
need to prove that Proposition \ref{prop:Green-correspondence-for-endomorphisms.}
can be applied to centric Mackey functors over fusion systems just
like it can be applied to Green functors (see Example \ref{exa:Green.correspondence-Green-functors.}).
\begin{lem}
\label{lem:Green-correspondent-below.}Let $\R$ be a complete local
and $p$-local $PID$ and let $M\in\MackFcR$ be indecomposable with
vertex $H$ (see Corollary \ref{cor:irreducible-admits-vertex.}).
Using Notation \ref{nota:X-and-Y.} we can define
\begin{align*}
A & :=\End\left(M\downarrow_{\NF}^{\F}\right), & B & :=\Tr_{H}^{\F}=\End\left(M\right),\\
C & :=\Tr_{H}^{\NF}, & K & :=\Tr_{\mathcal{X}}^{\F},\\
I & :=\Tr_{\mathcal{X}}^{\NF}, & J & :=\Tr_{\mathcal{Y}}^{\NF},\\
f & :=\tr_{\NF}^{\F}, & g & :=\r_{\NF}^{\F}.
\end{align*}
Here we are using Theorem \ref{thm:Higman's-criterion.} and the fact
that $M$ has vertex $H$ to define $B$ and we are viewing $f$ as
a morphism from $\Tr_{H}^{\NF}$ to $\Tr_{H}^{\F}$ (see Lemma \ref{lem:transfer-composes-nicely.})
and $\r_{\NF}^{\F}$ as a morphism from $\Tr_{H}^{\F}$ to $\Tr_{H}^{\NF}+\Tr_{\mathcal{Y}}^{\NF}$
(see Proposition \ref{prop:Conjecture-workaround.}). With these definitions
the conditions needed to apply Proposition \ref{prop:Green-correspondence-for-endomorphisms.}
are met. Moreover $\Id_{M}$ is a local idempotent of $\End\left(M\right)$
satisfying $\r_{\NF}^{\F}\left(\Id_{M}\right)=\Id_{M\downarrow_{\NF}^{\F}}$.
Therefore, applying Proposition \ref{prop:Green-correspondence-for-endomorphisms.},
we have that there exists a unique way (up to conjugation) of writing
\[
\Id_{M\downarrow_{\NF}^{\F}}=\sum_{i=0}^{n}\varepsilon_{i}
\]
where each $\varepsilon_{i}$ is a local idempotent in $\Tr_{H}^{\NF}$
and they are all mutually orthogonal. Moreover there exists a unique
$j\in\left\{ 0,\dots,n\right\} $ such that $\varepsilon_{j}\in\Tr_{H}^{\NF}-\Tr_{\mathcal{Y}}^{\NF}$
and, defining $\left(\Id_{M}\right)_{\NF}:=\varepsilon_{j}$ the following
hold
\begin{align*}
\tr_{\NF}^{\F}\left(\left(\Id_{M}\right)_{\NF}\right) & \equiv\Id_{M}\mod\Tr_{\mathcal{X}}^{\F}, & \r_{\NF}^{\F}\left(\Id_{M}\right) & \equiv\left(\Id_{M}\right)_{\NF}\mod\Tr_{\mathcal{Y}}^{\F}.
\end{align*}
\end{lem}

\begin{proof}
Since $\R$ is a complete local $PID$ and $M$ is indecomposable
we can apply \cite[Proposition 6.10 (ii)]{MethodsOfRepresentationTheoryCurtisReiner}
to deduce that $\End\left(M\right)$ is a local ring. In particular
$\Id_{M}$ is a local idempotent of $\End\left(M\right)$. The identity
$\r_{\NF}^{\F}\left(\Id_{M}\right)=\Id_{M\downarrow_{\NF}^{\F}}$
follows immediately from Definition \ref{def:Transfer-and-restriction.}.

Therefore, if we prove the first part of the statement, the second
part follows.

Let us start by proving that $A,B,C,K,I,J,f$ and $g$ are as defined
in Proposition \ref{prop:Green-correspondence-for-endomorphisms.}.
First of all notice that $\End\left(M\downarrow_{\NF}^{\F}\right)$
and $\End\left(M\right)$ are both $\R$-algebras. From Lemma \ref{lem:Image-transfer-is-ideal.}
we also know that $\Tr_{\mathcal{X}}^{\F}$ is a two sided ideal of
$\End\left(M\right)$ and that $\Tr_{H}^{\NF},\Tr_{\mathcal{X}}^{\NF}$
and $\Tr_{\mathcal{Y}}^{\NF}$ are two sided ideals of $\End\left(M\downarrow_{\NF}^{\F}\right)$.
By definition of $\mathcal{X}$ we know that for every $K\in\mathcal{X}$
then $\Hom_{\F}\left(K,H\right)\not=\emptyset$. Therefore, from Lemma
\ref{lem:Transfer-inclusion.} we can conclude that $\Tr_{\mathcal{X}}^{\NF}\subseteq\Tr_{H}^{\NF}$
and since $\Tr_{\mathcal{X}}^{\NF}$ is a two sided ideal of $\End\left(M\downarrow_{\NF}^{\F}\right)$
then we can view $\Tr_{\mathcal{X}}^{\NF}$ as a two sided ideal of
$\Tr_{H}^{\NF}$ (seen as a ring with potentially no unit). As mentioned
in the statement we can use Lemma \ref{lem:transfer-composes-nicely.}
to view $\tr_{\NF}^{\F}$ as a morphism of $\R$-modules from $\Tr_{H}^{\NF}$
to $\Tr_{H}^{\F}$. Finally, writing $E_{K}:=\End\left(M\downarrow_{\FHH[K]}^{\F}\right)$
for any $K\in\Fc$, we have from Proposition \ref{prop:Conjecture-workaround.}
that
\[
\r_{\NF}^{\F}\left(\Tr_{H}^{\F}\right)=\r_{\NF}^{\F}\left(\tr_{\FHH}^{\F}\left(E_{H}\right)\right)\subseteq\tr_{\FHH}^{\NF}\left(E_{H}\right)+\sum_{K\in\mathcal{Y}}\tr_{\FHH[K]}^{\NF}\left(E_{K}\right)=\Tr_{H}^{\NF}+\Tr_{\mathcal{Y}}^{\NF}.
\]
Thus, we can view $\r_{\NF}^{\F}$ as an $\R$-module morphism from
$\Tr_{H}^{\F}$ to $\Tr_{H}^{\NF}+\Tr_{\mathcal{Y}}^{\NF}$.

With this setup we just need to prove that the Conditions \eqref{enu:condition-1.}-\eqref{enu:Decomposition-idempotent.}
of Proposition \ref{prop:Green-correspondence-for-endomorphisms.}
are met for our choices of $A,B,C,K,I,J,f$ and $g$.
\begin{enumerate}
\item \emph{For Condition \eqref{enu:condition-1.} we need to check that
the following inclusions are satisfied}\textbf{\emph{
\begin{align*}
\left(\Tr_{H}^{\NF}\cap\Tr_{\mathcal{Y}}^{\NF}\right)\Tr_{H}^{\NF} & \subseteq\Tr_{\mathcal{X}}^{\NF}, & \Tr_{H}^{\NF}\left(\Tr_{H}^{\NF}\cap\Tr_{\mathcal{Y}}^{\NF}\right) & \subseteq\Tr_{\mathcal{X}}^{\NF},\\
\Tr_{\mathcal{X}}^{\NF} & \subseteq\Tr_{H}^{\NF}\cap\Tr_{\mathcal{Y}}^{\NF}.
\end{align*}
}}From definition of $\mathcal{X}$ and $\mathcal{Y}$ (see Notation
\ref{nota:X-and-Y.}) we know that $\mathcal{X}\subseteq\mathcal{Y}$
and that every element in $\mathcal{X}$ is a subgroup of $H$. Therefore,
from Lemma \ref{lem:Transfer-inclusion.} we can conclude that $\Tr_{\mathcal{X}}^{\NF}\subseteq\Tr_{H}^{\NF}$
and $\Tr_{\mathcal{X}}^{\NF}\subseteq\Tr_{\mathcal{Y}}^{\NF}$. This
proves the bottom inclusion. Let's now prove the top right inclusion
(the top left inclusion follows similarly). Let $f\in\End\left(M\downarrow_{\FHH}^{\F}\right)$
and for every $K\in\mathcal{Y}$ let $g_{K}\in\End\left(M\downarrow_{\FHH[K]}^{\F}\right)$
such that $\sum_{K\in\mathcal{Y}}\tr_{\FHH[K]}^{\NF}\left(g_{K}\right)\in\Tr_{H}^{\NF}\cap\Tr_{\mathcal{Y}}^{\NF}$.
Then, from Items \eqref{enu:Mackey-formula.} and \eqref{enu:Green-formula.}
of Proposition \ref{prop:Properties-of-transfer-restriction-and-conjugation-maps.}
we have that
\begin{align*}
\tr_{\FHH}^{\NF}\left(f\right)\left(\sum_{K\in\mathcal{Y}}\tr_{\FHH[K]}^{\NF}\left(g_{K}\right)\right) & =\sum_{K\in\mathcal{Y}}\tr_{\FHH}^{\NF}\left(f\r_{\FHH}^{\NF}\left(\tr_{\FHH[K]}^{\NF}\left(g_{K}\right)\right)\right),\\
 & =\sum_{K\in\mathcal{Y}}\sum_{\left(A,\overline{\varphi}\right)\in\left[K\times_{\NF}H\right]}\hspace{-10bp}\hspace{-10bp}\tr_{\FHH}^{\NF}\left(f\tr_{\FHH[\varphi\left(A\right)]}^{\FHH}\left(\lui{\varphi}{\left(\r_{\FHH[A]}^{\FHH[K]}\left(g_{K}\right)\right)}\right)\right).
\end{align*}
Where $\varphi$ is a representative of $\overline{\varphi}$ seen
as an isomorphism onto its image. Fix now $K\in\mathcal{Y}$ and $\left(A,\overline{\varphi}\right)\in\left[K\times_{\NF}H\right]$
and take $\varphi$ as before. If $A\not\in\Fc$ then, since $M$
is $\F$-centric, we have that $M\downarrow_{\FHH[A]}^{\F}=0$ and,
in particular, $\r_{\FHH[A]}^{\FHH[K]}\left(g_{K}\right)=0$. We can
therefore assume without loss of generality that $A\in\Fc$. Since
$A\in\FHH[K]\cap\Fc$ and $K\in\mathcal{Y}$ then we can conclude
from definition of $\mathcal{Y}$ (see Notation \ref{nota:X-and-Y.})
that $A\in\mathcal{Y}$. From this we can conclude that $\varphi\left(A\right)\in\mathcal{Y}$
and since $\varphi\left(A\right)\le H$ we can conclude that $\varphi\left(A\right)\in\mathcal{X}$.
With this in mind, applying Items \eqref{enu:composition-transfer.}
and \eqref{enu:Green-formula.} of Proposition \ref{prop:Properties-of-transfer-restriction-and-conjugation-maps.},
we have that
\[
\tr_{\FHH}^{\NF}\left(f\tr_{\FHH[\varphi\left(A\right)]}^{\FHH}\hspace{-2bp}\left(\lui{\varphi}{\left(\r_{\FHH[A]}^{\FHH[K]}\left(g_{K}\right)\right)}\right)\right)\hspace{-2bp}=\hspace{-2bp}\tr_{\FHH[\varphi\left(A\right)]}^{\F}\left(\r_{\FHH[\varphi\left(A\right)]}^{\FHH}\left(f\right)\lui{\varphi}{\left(\r_{\FHH[A]}^{\FHH[K]}\left(g_{K}\right)\right)}\right)\hspace{-2bp}\in\hspace{-2bp}\Tr_{\varphi\left(A\right)}^{\NF}\hspace{-2bp}\subseteq\Tr_{\mathcal{X}}^{\NF}.
\]
Therefore $\Tr_{H}^{\NF}\hspace{-2bp}\left(\Tr_{H}^{\NF}\hspace{-2bp}\cap\Tr_{\mathcal{Y}}^{\NF}\right)\hspace{-2bp}\subseteq\hspace{-2bp}\Tr_{\mathcal{X}}^{\NF}$
thus proving that Condition \eqref{enu:condition-1.} is satisfied.
\item \emph{For Condition \eqref{enu:condition-2.} we need to check that}\textbf{\emph{
$\r_{\NF}^{\F}\left(\Tr_{\mathcal{X}}^{\F}\right)\subseteq\Tr_{\mathcal{Y}}^{\NF}$}}\emph{.}
For every $K\in\mathcal{X}$ we have that $K\lneq H$ and, therefore,
from Proposition \ref{prop:Properties-of-transfer-restriction-and-conjugation-maps.}
\eqref{enu:composition-transfer.}, we have that $\Tr_{K}^{\F}=\tr_{\FHH}^{\F}\left(\Tr_{K}^{\FHH}\right)$.
From Proposition \ref{prop:Conjecture-workaround.} we can then deduce
that $\r_{\NF}^{\F}\left(\Tr_{K}^{\F}\right)\subseteq\tr_{\FHH}^{\NF}\left(\Tr_{K}^{\FHH}\right)+\Tr_{\mathcal{Y}}^{\NF}$.
Applying Proposition \ref{prop:Properties-of-transfer-restriction-and-conjugation-maps.}
\eqref{enu:composition-transfer.} once again we obtain that $\tr_{\FHH}^{\NF}\left(\Tr_{K}^{\FHH}\right)=\Tr_{K}^{\NF}$
and, since $K\in\mathcal{X}\subseteq\mathcal{Y}$, we can deduce that
$\Tr_{K}^{\NF}\subseteq\Tr_{\mathcal{Y}}^{\NF}$. Thus we can conclude
that $\r_{\NF}^{\F}\left(\Tr_{K}^{\F}\right)\subseteq\Tr_{\mathcal{Y}}^{\NF}$.
Since this works for every $K\in\mathcal{X}$ the result follows.
\item \emph{For Condition \eqref{enu:condition-3.} we need to check that
$\tr_{\NF}^{\F}\left(\Tr_{\mathcal{X}}^{\NF}\right)\subseteq\Tr_{\mathcal{X}}^{\F}$.}
This follows from Corollary \ref{cor:tr-maps-Tr^NF-to-Tr^F.}.
\item \emph{For Condition \eqref{enu:condition-4.} we need to check that
$\tr_{\NF}^{\F}$, seen as a morphism from $\Tr_{H}^{\NF}$ to $\Tr_{H}^{\F}$,
is surjective.} This is given by Lemma \ref{lem:transfer-composes-nicely.}.
\item \emph{For Condition \eqref{enu:g-preserves-idempotents.} we need
to check that $\r_{\NF}^{\F}$ sends idempotents to idempotents.}
This follows immediately from definition of \textbf{$\r_{\NF}^{\F}$}.
\item \emph{For Condition \eqref{enu:condition-6.} we need to check that
the $\R$-linear maps $\overline{\tr_{\NF}^{\F}}:\Tr_{H}^{\NF}/\Tr_{\mathcal{X}}^{\NF}\to\Tr_{H}^{\F}/\Tr_{\mathcal{X}}^{\F}$
and $\overline{\r_{\NF}^{\F}}:\Tr_{H}^{\F}/\Tr_{\mathcal{X}}^{\F}\to\left(\Tr_{H}^{\NF}+\Tr_{\mathcal{Y}}^{\NF}\right)/\Tr_{\mathcal{Y}}^{\NF}$
commute with multiplication.} From Lemma \ref{lem:quotient-is-multiplicative.}
we know that $\overline{\tr_{\NF}^{\F}}$ commutes with multiplication.
On the other hand it is immediate from definition that $\r_{\NF}^{\F}$
commutes with multiplication and, therefore, so does $\overline{\r_{\NF}^{\F}}$.
\item \emph{For Condition \eqref{enu:condition-7.} we need to check that
the natural isomorphism 
\[
s:\Tr_{H}^{\NF}/\left(\Tr_{H}^{\NF}\cap\Tr_{\mathcal{Y}}^{\NF}\right)\bjarrow\left(\Tr_{H}^{\NF}+\Tr_{\mathcal{Y}}^{\NF}\right)/\Tr_{\mathcal{Y}}^{\NF}
\]
and the natural projection 
\[
q:\Tr_{H}^{\NF}/\Tr_{\mathcal{X}}^{\NF}\twoheadrightarrow\Tr_{H}^{\NF}/\left(\Tr_{H}^{\NF}\cap\Tr_{\mathcal{Y}}^{\NF}\right)
\]
satisfy $sq=\overline{\r_{\NF}^{\F}}\,\overline{\tr_{\NF}^{\F}}$.}
Abusing a bit of notation we denote with an overline ($\overline{\cdot}$)
the projection of an endomorphism on the appropriate quotient. With
this notation, for every $f\in\End\left(M\downarrow_{\FHH}^{\F}\right)$,
we have that
\[
\overline{\r_{\NF}^{\F}}\left(\overline{\tr_{\NF}^{\F}}\left(\overline{\tr_{\FHH}^{\NF}\left(f\right)}\right)\right)=\overline{\r_{\NF}^{\F}\left(\tr_{\FHH}^{\F}\left(f\right)\right)}=\overline{\tr_{\FHH}^{\NF}\left(f\right)}=s\left(q\left(\overline{\tr_{\FHH}^{\NF}\left(f\right)}\right)\right),
\]
Where we are using Lemma \ref{lem:transfer-composes-nicely.} for
the first identity, we are using Proposition \ref{prop:Conjecture-workaround.}
and the fact that $\overline{\Tr_{\mathcal{Y}}^{\NF}}=\overline{0}$
for the second identity and we are using the definitions of $q$ and
$s$ for the third identity. Since every element in $\Tr_{H}^{\NF}/\Tr_{\mathcal{X}}^{\NF}$
is of the form $\overline{\tr_{\FHH}^{\NF}\left(f\right)}$ for some
$f\in\End\left(M\downarrow_{\FHH}^{\F}\right)$ the result follows.
\item \emph{For Condition \eqref{enu:Decomposition-idempotent.} we need
to prove that for every idempotent $f\in\End\left(M\downarrow_{\NF}^{\F}\right)$
there exists a unique (up to conjugation) decomposition of $f$ as
a finite sum of orthogonal local idempotents.} From Proposition \ref{prop:Mackey-algebra-basis.}
we know that the $\R$-algebra $\FmuFR{\NF}$ is finitely generated
as an $\R$-module. Therefore we can apply the Krull-Schmidt-Azumaya
theorem (see \cite[Theorem 6.12 (ii)]{MethodsOfRepresentationTheoryCurtisReiner})
together with \cite[Proposition 6.10 (ii)]{MethodsOfRepresentationTheoryCurtisReiner}
to conclude that Condition \eqref{enu:Decomposition-idempotent.}
is satisfied.
\end{enumerate}
Since all conditions are verified we can conclude the proof.
\end{proof}
\begin{cor}
\label{cor:Green-correspondent-below.}Let $\R$ be a complete local
and $p$-local $PID$ and let $M\in\MackFcR$ be indecomposable with
vertex $H$ (see Corollary \ref{cor:irreducible-admits-vertex.}).
There exists a unique (up to isomorphism) decomposition of $M\downarrow_{\NF}^{\F}$
as a direct sum of indecomposable $\F$-centric Mackey functors
\[
M\downarrow_{\NF}^{\F}=\bigoplus_{i=0}^{n}M_{i}.
\]
With this notation, there exists exactly one $j\in\left\{ 0,\dots,n\right\} $
such that $M_{j}$ has vertex $H$ while, for every other $i\in\left\{ 0,\dots,n\right\} -\left\{ j\right\} $,
we have that $M_{i}$ has vertex in $\mathcal{Y}$. We call $M_{j}$
the \textbf{Green correspondent} of $M$ and denote it by $M_{\NF}$.
\end{cor}

\begin{proof}
Applying Lemma \ref{lem:Green-correspondent-below.} we know that
there exists a unique (up to conjugation) decomposition of $\Id_{M\downarrow_{\NF}^{\F}}$
of the form
\[
\Id_{M\downarrow_{\NF}^{\F}}=\left(\Id_{M}\right)_{\NF}+\sum_{i=1}^{n}\varepsilon_{i}
\]
where the $\varepsilon_{i}$ and $\left(\Id_{M}\right)_{\NF}$ are
mutually orthogonal local idempotents satisfying $\left(\Id_{M}\right)_{\NF}\in\Tr_{H}^{\F}-\Tr_{\mathcal{Y}}^{\F}$
and $\varepsilon_{i}\in\Tr_{\mathcal{Y}}^{\F}$. From this decomposition
and \cite[Proposition 6.10 (ii)]{MethodsOfRepresentationTheoryCurtisReiner}
we can deduce that there exists a unique (up to isomorphism) decomposition
of $M\downarrow_{\NF}^{\F}$ as a direct sum of indecomposable Mackey
functors and it is given by
\[
M\downarrow_{\NF}^{\F}=\left(\Id_{M}\right)_{\NF}\left(M\downarrow_{\NF}^{\F}\right)\oplus\bigoplus_{i=1}^{n}\varepsilon_{i}\left(M\downarrow_{\NF}^{\F}\right).
\]
Since $\left(\Id_{M}\right)_{\NF}\in\Tr_{H}^{\F}-\Tr_{\mathcal{Y}}^{\F}$
we can conclude from Theorem \ref{thm:Higman's-criterion.} that $\left(\Id_{M}\right)_{\NF}\left(M\downarrow_{\NF}^{\F}\right)$
has vertex $H$. On the other hand, since $\varepsilon_{i}\in\Tr_{\mathcal{Y}}^{\F}$
and $\Tr_{\mathcal{Y}}^{\F}$ is an ideal (see Lemma \ref{lem:Image-transfer-is-ideal.})
we have that $\varepsilon_{i}\in\varepsilon_{i}\Tr_{\mathcal{Y}}^{\F}\varepsilon_{i}=\varepsilon_{i}\End\left(M\right)\varepsilon_{i}$.
Since $\varepsilon_{i}$ is a local idempotent then we can conclude
that $\varepsilon_{i}\Tr_{\mathcal{Y}}^{\F}\varepsilon_{i}$ is a
local ring and, therefore, there exists $K\in\mathcal{Y}$ such that
$\varepsilon_{i}\Tr_{K}^{\F}\varepsilon_{i}$. In particular $\varepsilon_{i}$
is $K$-projective and, from Theorem \ref{thm:Higman's-criterion.},
we can conclude that $\varepsilon_{i}\left(M\downarrow_{\NF}^{\F}\right)$
is also $K$-projective. Since $K\in\mathcal{Y}$ we can conclude
from minimality of the defect set that $\varepsilon_{i}\left(M\downarrow_{\NF}^{\F}\right)$
has vertex in $\mathcal{Y}$. The result follows by setting $M_{\NF}:=\left(\Id_{M}\right)_{\NF}\left(M\downarrow_{\NF}^{\F}\right)$.
\end{proof}
Corollary \ref{cor:Green-correspondent-below.} gives us the first
half of the Green correspondence. Let's now get the other half.
\begin{lem}
\label{lem:Green-correspondent-above.}Let $\R$ be a complete local
and $p$-local $PID$, let $N\in\MackFHR{\F}{\NF}$ be indecomposable
with vertex $H$. Using the notation of Corollary \ref{cor:Green-correspondent-below.}
there exists an indecomposable $M\in\MackFcR$ with vertex $H$ such
that $M_{\NF}\cong N$. Moreover $M$ is a summand of $N\uparrow_{\NF}^{\F}\cong M_{\NF}\uparrow_{\NF}^{\F}$.
\end{lem}

\begin{proof}
From Proposition \ref{prop:Mackey-algebra-basis.} we know that $\muFR$
is finitely generated as an $\R$-module. Therefore we can apply the
Krull-Schmidt-Azumaya theorem (see \cite[Theorem 6.12 (ii)]{MethodsOfRepresentationTheoryCurtisReiner})
in order to write $N\uparrow_{\NF}^{\F}\cong\bigoplus_{i=0}^{n}M_{i}$.
Where each $M_{i}\in\MackFcR$ (see Proposition \ref{prop:centric-Induction-restriction.})
is indecomposable. From Lemma \ref{lem:induction-from-normalizer-followed-by-restriction-to-normalizer.}
we know that $N$ is a summand of $N\uparrow_{\NF}^{\F}\downarrow_{\NF}^{\F}$.
Since the restriction functor is additive we can now use the fact
that $N$ is indecomposable and uniqueness of the Krull-Schmidt-Azumaya
theorem (now applied to $N\uparrow_{\NF}^{\F}\downarrow_{\NF}^{\F}$)
in order to chose $j\in\left\{ 0,\dots,n\right\} $ such that $N$
is a summand of $M_{j}\downarrow_{\NF}^{\F}$. To simplify notation
let us define $M:=M_{j}$. We are now only left with proving that
$M$ has vertex $H$. Since $N$ is $H$-projective and $M$ is a
summand of $N\uparrow_{\NF}^{\F}$ then we can deduce from Theorem
\ref{thm:Higman's-criterion.} that $M$ is $H$-projective. From
minimality of the defect set (see Corollary \ref{cor:defect-set.})
we can now conclude that the vertex $V_{M}$ of $M$ satisfies $V_{M}\le_{\F}H$
(see Notation \ref{nota:p,S,F}). Assume that $V_{M}\lneq_{\F}H$.
From Corollary \ref{cor:projective-respect-to-bigger.} \eqref{enu:X-proj implies X^max proj.}
we can deduce that there exists $K\lneq H$ such that $M$ is $K$-projective.
From Theorem \ref{thm:Higman's-criterion.} we can then deduce that
there exists $P\in\MackFHR{\F}{\FHH[K]}$ such that $M$ is a summand
of $P'\uparrow_{\NF}^{\F}$ where $P':=P\uparrow_{\FHH[K]}^{\NF}$.
Since $N$ is a summand of $M\downarrow_{\NF}^{\F}$ we can deduce
that $N$ is a summand of $P'\uparrow_{\NF}^{\F}\downarrow_{\NF}^{\F}$.
From Lemma \ref{lem:induction-from-normalizer-followed-by-restriction-to-normalizer.}
we can now deduce that there exists an $\mathcal{Y}$-projective $Q\in\MackFHR{\F}{\NF}$
such that $P'\uparrow_{\NF}^{\F}\downarrow_{\NF}^{\F}=P'\oplus Q=P\uparrow_{\FHH[K]}^{\NF}\oplus\,Q$.
Since $K\lneq H$ by hypothesis then we can conclude that $K\in\mathcal{X}\subseteq\mathcal{Y}$
and, therefore, that $P\uparrow_{\FHH[K]}^{\NF}\oplus\,Q$ is $\mathcal{Y}$-projective.
Since $N$ is an indecomposable summand of $P\uparrow_{\FHH[K]}^{\NF}\oplus\,Q$
we can then conclude from Corollary \ref{cor:irreducible-admits-vertex.}
that the vertex of $N$ lies in $\mathcal{Y}$. Since $H\not\in\mathcal{Y}$
this contradicts the hypothesis that $N$ has vertex $H$. Therefore
we cannot have $V_{M}\lneq_{\F}H$. Since we have proven that $V_{M}\le_{\F}H$
we can therefore conclude that $V_{M}=_{\F}H$. Since $H$ is fully
$\F$-normalized (see Notation \ref{nota:X-and-Y.}) then we can conclude
that $M$ has vertex $H$. We can therefore apply Corollary \ref{cor:Green-correspondent-below.}
to $M$ in order to conclude that there is a unique (up to isomorphism)
indecomposable summand $M_{\NF}$ of $M\downarrow_{\NF}^{\F}$ with
vertex $H$. Since $N$ is a summand of $M\downarrow_{\NF}^{\F}$
and has vertex $H$ the result follows.
\end{proof}
\begin{lem}
\label{lem:Green-correspondence-relation.}Let $\R$ be a complete
local and $p$-local $PID$, let $M\in\MackFcR$ be indecomposable
with vertex $H$ and let $M_{\NF}$ be as in Corollary \ref{cor:Green-correspondent-below.}.
Since $\muFR$ is finitely generated as an $\R$-module (see Proposition
\ref{prop:Mackey-algebra-basis.}) we can apply the Krull-Schmidt-Azumaya
theorem (see \cite[Theorem 6.12 (ii)]{MethodsOfRepresentationTheoryCurtisReiner})
together with Lemma \ref{lem:Green-correspondent-above.} in order
to write
\[
M_{\NF}\uparrow_{\NF}^{\F}\cong M\oplus\bigoplus_{i=1}^{n}M_{i}.
\]
Where each $M_{i}$ is indecomposable. With this notation we have
that each $M_{i}$ is $\F$-centric and has vertex $\F$-conjugate
to an element in $\mathcal{X}$ (see Notation \ref{nota:X-and-Y.}).
\end{lem}

\begin{proof}
From Proposition \ref{prop:centric-Induction-restriction.} we know
that $M_{i}\in\MackFcR$ for every $i=1,\dots,n$. From Corollary
\ref{cor:Green-correspondent-below.} we know $M_{\NF}$ has vertex
$H$. Therefore, from Theorem \ref{thm:Higman's-criterion.}, we know
that there exists $P\in\MackFHR{\F}{\FHH}$ such that $M_{\NF}$ is
a summand of $P\uparrow_{\FHH}^{\NF}$. Since induction preserves
direct sum decomposition then we can conclude that each $M_{i}$ is
a summand of $P\uparrow_{\FHH}^{\F}$. From Theorem \ref{thm:Higman's-criterion.}
this implies that each $M_{i}$ is $H$-projective. Assume now that
there exists $j\in\left\{ 1,\dots,n\right\} $ such that $M_{j}$
has vertex $H$. Since restriction preserves direct sum decomposition
we can conclude, using Corollary \ref{cor:Green-correspondent-below.}
that, $M_{\NF}\oplus\left(M_{j}\right)_{\NF}$ is a summand of $M_{\NF}\uparrow_{\NF}^{\F}\downarrow_{\NF}^{\F}$.
However, from Lemma \ref{lem:induction-from-normalizer-followed-by-restriction-to-normalizer.},
we know that there exists an $\mathcal{Y}$-projective $Q\in\MackFHR{\F}{\NF}$
such that $M_{\NF}\uparrow_{\NF}^{\F}\downarrow_{\NF}^{\F}\cong M_{\NF}\oplus Q$.
From uniqueness of the Krull-Schmidt-Axumaya theorem we can then conclude
that $\left(M_{j}\right)_{\NF}$ is a summand of $Q$. Thus, from
Corollary \ref{cor:irreducible-admits-vertex.}, we can conclude that
$\left(M_{j}\right)_{\NF}$ has vertex in $\mathcal{Y}$. This contradicts
Corollary \ref{cor:Green-correspondent-below.}. Thus we can conclude
that none of the $M_{i}$ has vertex in $H$. Since they are all $H$-projective
then we can conclude from minimality of the defect set and Corollary
\ref{cor:projective-respect-to-bigger.} that they are all $\mathcal{X}$-projective.
From Corollary \ref{cor:irreducible-admits-vertex.} this implies
that each $M_{i}$ has vertex $\F$-conjugate to an element in $\mathcal{X}$
thus concluding the proof.
\end{proof}
Putting the previous results together we can prove that the Green
correspondence holds for centric Mackey functors over fusion systems.
\begin{thm}
\label{thm:Green-correspondence.}(Green correspondence) Let $\R$
be a complete local and $p$-local $PID$ (see Definition \ref{def:p-local}),
let $M\in\MackFcR$ (see Definition \ref{def:F-centric-Mackey-functor.})
be indecomposable with vertex $H$ (see Definition \ref{def:vertex.}
and Notation \ref{nota:X-and-Y.}) and let $N\in\MackFHR{\F}{\NF}$
(see Example \ref{exa:definition-NF.}) be indecomposable with vertex
$H$. There exist unique (up to isomorphism) decompositions of $M\downarrow_{\NF}^{\F}$
and $N\uparrow_{\NF}^{\F}$ (see Definition \ref{def:Restriction-induction-and-conjugation-functors.})
into direct sums of indecomposable Mackey functors. Moreover, writing
these decompositions as
\begin{align*}
M\downarrow_{\NF}^{\F} & :=\bigoplus_{i=0}^{n}M_{i}, & N\uparrow_{\NF}^{\F} & :=\bigoplus_{j=0}^{m}N_{j},
\end{align*}
there exist unique $i\in\left\{ 0,\dots,n\right\} $ and $j\in\left\{ 0,\dots,m\right\} $
such that both $M_{i}$ and $N_{j}$ have vertex $H$. We call these
summands the \textbf{Green correspondents} of $M$ and $N$ and denote
them as $M_{\NF}$ and $N^{\NF}$ respectively. Every indecomposable
summand of $M\downarrow_{\NF}^{\F}$ other than $M_{\NF}$ has vertex
in $\mathcal{Y}$ (see Notation \ref{nota:X-and-Y.}) while every
indecomposable summand of $N\uparrow_{\NF}^{\F}$ other than $N^{\NF}$
has vertex $\F$-conjugate to an element in $\mathcal{X}$ (see Notation
\ref{nota:X-and-Y.}). Finally we have that $\left(M_{\NF}\right)^{\F}\cong M$
and that $\left(N^{\F}\right)_{\NF}\cong N$.
\end{thm}

\begin{proof}
From Lemma \ref{lem:Green-correspondent-above.} we know that there
exists an indecomposable $P\in\MackFcR$ with vertex $H$ such that
$N=P_{\NF}$. It follows from Lemma \ref{lem:Green-correspondence-relation.}
that there exists a unique (up to isomorphism) decomposition of $N\uparrow_{\NF}^{\F}$
as the one in the statement and that $P\cong N^{\NF}$. In particular
$\left(N^{\F}\right)_{\NF}\cong N$.

From Corollary \ref{cor:Green-correspondent-below.} we know that
there exists a unique (up to isomorphism) decomposition of $M\downarrow_{\NF}^{\F}$
as the one in the statement. From Lemma \ref{lem:Green-correspondence-relation.}
and the first part of the statement we have that $\left(M_{\NF}\right)^{\F}\cong M$
which concludes the proof.
\end{proof}
Before concluding this paper let us see an example where Theorem \ref{thm:Green-correspondence.}
can be applied.
\begin{example}
\label{exa:minimal-F-centric.}Let $\R$ be a complete local and $p$-local
$PID$ and let $\F$ be a fusion system. For example, we can take
$\R=\Zp$ and, using the notation of Example \ref{exa:fusion-system.},
we can take the fusion system $\F_{1}:=\FAB{D_{8}}{GL_{2}\left(3\right)}$,
or the Ruiz-Viruel exotic fusion system $\F_{2}$ on $7_{+}^{1+2}$
having two $\F$-orbits of elementary abelian subgroups of rank 2
the first of which has 6 elements while the second has 2 elements
(see \cite[Theorem 1.1]{RuizViruelClassificationOverExtraspecial.}).

Choose now $H\in\Fc$ fully $\F$-normalized and minimal under the
preorder $\le_{\F}$ (see Notation \ref{nota:p,S,F}). For $\F_{1}$
we can take $H_{1}$ to be any one of the two characteristic elementary
abelian subgroups of rank 2 of $D_{8}$. For $\F_{2}$ we can take
$H_{2}$ to be one of the two elementary abelian subgroups of rank
two whose $\F_{2}$-orbit contains only 2 elements (make sure to take
one that is fully $\F_{2}$-normalized).

In order to visualize this example it might help to notice the following
identities 
\begin{align*}
N_{\F_{1}}\left(H_{1}\right) & =\FAB{D_{8}}{S_{4}}, & N_{\F_{2}}\left(H_{2}\right) & =\FAB{7_{+}^{1+2}}{L_{3}\left(7\right).3}.
\end{align*}
The first one follows after a straightforward calculation while the
second one follows from \cite[Theorem 1.1]{RuizViruelClassificationOverExtraspecial.}
and \cite[Section 4]{SubgroupFamiliesControllingpLocalFiniteGroups}.

Let $\I$ be as in Proposition \ref{prop:Decomposition-centric-Mackey-functor.}
and for every $x\in\muFR$ denote by $\overline{x}\in\muFR/\I$ its
image via the natural projection. From Proposition \ref{prop:Mackey-algebra-basis.}
we know that $\muFR$ is finitely generated as an $\R$-module. As
a consequence $\muFR/\I$ is also finitely generated as an $\R$-module.
Therefore we can apply the Krull-Schmidt-Azumaya theorem (see \cite[Theorem 6.12 (ii)]{MethodsOfRepresentationTheoryCurtisReiner})
together with \cite[Proposition 6.10 (ii)]{MethodsOfRepresentationTheoryCurtisReiner}
in order to conclude that, for every $H\in\Fc$, there exists a unique
(up to conjugation) decomposition of $\overline{I_{H}^{H}}$ in $\muFR/\I$
as a sum of orthogonal local idempotents. Let $\overline{I_{H}^{H}}=\sum_{i=0}^{n}\overline{x_{i}}$
be such decomposition. Define now $\overline{x}:=\overline{x_{0}}$.
For example, for $\F_{1}$ we have that $\Aut_{\F_{1}}\left(H_{1}\right)\cong S_{3}$
and, therefore, we can take $\varphi\in\Aut_{\F_{1}}\left(H_{1}\right)$
to be one of the two elements of order 3 and $\frac{2}{3}\overline{I_{H_{1}}^{H_{1}}}-\frac{1}{3}\overline{c_{\varphi}}-\frac{1}{3}\overline{c_{\varphi^{2}}}$
is a local idempotent in the decomposition of $\overline{I_{H_{1}}^{H_{1}}}$.

Since $\overline{I_{H}^{H}xI_{H}^{H}}=\overline{x}$ by construction
then, from Proposition \ref{prop:Mackey-algebra-basis.}, we know
that
\[
\overline{x}=\sum_{j=0}^{m}\lambda_{j}\overline{I_{\varphi_{j}\left(A_{j}\right)}^{H}c_{\varphi_{j}}R_{A_{j}}^{H}}.
\]
for some $\lambda_{j}\in\R$, some $A_{j}\le H$ and some isomorphisms
$\varphi_{j}\colon A_{j}\to\varphi_{j}\left(A_{j}\right)$ in $\F$
such that $\varphi_{j}\left(A_{j}\right)\le H$. Since $H$ is minimal
$\F$-centric then, by definition of $\I$, we can conclude that $\overline{I_{\varphi_{j}\left(A_{j}\right)}^{H}c_{\varphi_{j}}R_{A_{j}}^{H}}=\overline{0}$
unless $A_{j}=H$. In this situation we necessarily have that $\varphi_{j}\in\Aut_{\F}\left(H\right)=\Aut_{\NF}\left(H\right)$.
Viewing $\FmuFR{\NF}$ as a subset of $\muFR$ (see Corollary \ref{cor:Mackey-algebra-inclusion.}),
we have in particular that $\overline{x}\in\overline{\FmuFR{\NF}}$.
Define now the two sided ideal $\J$ of $\FmuFR{\NF}$ as $\J:=\I\cap\FmuFR{\NF}$.
We know that $\FmuFR{\NF}/\J\cong\left(\FmuFR{\NF}+\I\right)/\I$
and, therefore, we can view $\FmuFR{\NF}/\J$ as a subset of $\muFR\I$
and $\overline{x}$ as an idempotent in $\FmuFR{\NF}/\J$. Since $\overline{x}$
is a primitive idempotent of $\muFR/\I$ (recall that every local
idempotent is primitive), it is also a primitive idempotent of $\FmuFR{\NF}/\J$.
In particular we have that $M:=\left(\muFR/\I\right)\overline{x}$
and $N:=\left(\FmuFR{\NF}/\J\right)\overline{x}$ are indecomposable
as left $\muFR/\I$ and $\FmuFR{\NF}/\J$-modules respectively. In
particular they are indecomposable as $\muFR$ and $\FmuFR{\NF}$
modules respectively (i.e. as Mackey functors over $\F$ and $\NF$
respectively). From definition of $\I$ and $\J$ we can also conclude
that $M\in\MackFcR$ and $N\in\MackFHR{\F}{\NF}$.

From Lemma \ref{lem:Many-properties-definition.}, Proposition \ref{prop:Mackey-algebra-basis.}
and \cite[Proposition 4.4]{IntroductionToFusionSystemsLinckelmann}
we know that $\I$ is spanned as an $\R$-module by elements of the
form $I_{\varphi\left(C\right)}^{B}c_{\varphi}R_{C}^{A}$ with $C\in\FHH[A]\backslash\left(\FHH[A]\cap\Fc\right)$.
In particular $R_{C}^{A}\in\FmuFR{\FHH[A]}\cap\I$ and we can write
any element in $\I I_{H}^{H}$ (resp. $\J I_{H}^{H}$) as a finite
sum of elements of the form $bc$ with $b\in\muFR$ (resp. $\FmuFR{\NF}$)
and $c\in\FmuFR{\FHH}\cap\I$. Therefore, for every $y\otimes_{\FmuFR{\FHH}}\overline{x}\in M_{H}$
(resp. $y\otimes_{\FmuFR{\FHH}}\overline{x}\in N_{H}$) such that
$y\in\I$ (resp. $y\in\J$) we have that $y\otimes_{\FmuFR{\FHH}}\overline{x}=0$.
This allows us to define the morphisms of Mackey functors $u_{H}^{M}:M\to M_{H}$
and $u_{H}^{N}:N\to N_{H}$ by setting $u_{H}^{M}\left(\overline{a}\right)=aI_{H}^{H}\otimes\overline{x}$
and $u_{H}^{N}\left(\overline{b}\right)=bI_{H}^{H}\otimes\overline{x}$
for any representative $a\in\muFR$ of $\overline{a}\in M$ and any
representative $b\in\FmuFR{\NF}$ of $\overline{b}\in N$. Since $\overline{x}$
is an idempotent and $\overline{I_{H}^{H}}\overline{x}=\overline{x}$
by construction then, with this notation, we have that $\overline{aI_{H}^{H}}\overline{x}=\overline{a}$
and that $\overline{bI_{H}^{H}}\overline{x}=\overline{b}$. In other
words we have that $\theta_{H}^{M}u_{H}^{M}=\Id_{M}$ and that $\theta_{H}^{N}u_{H}^{N}=\Id_{N}$.
Equivalently both $M$ and $N$ are $H$-projective. Since $H$ is
minimal $\F$-centric and fully $\F$-normalized we can conclude from
minimality of the defect set and Corollary \ref{cor:irreducible-admits-vertex.}
that $H$ is in fact the vertex of both $M$ and $N$. We now have
by construction that $M\cong N\uparrow_{\NF}^{\F}$ which proves that
$M\cong N^{\NF}$. From Theorem \ref{thm:Green-correspondence.} we
can then conclude that $N=M_{\NF}$ and, therefore, that there exists
an $\mathcal{Y}$-projective $N'\in\MackFHR{\F}{\NF}$ such that $M\downarrow_{\NF}^{\F}\cong N\oplus N'$.
In the case of $\F_{1}$, since $H_{1}$ is characteristic and $H_{1}$
is minimal $\F$-centric, then we have that $\mathcal{Y}=\emptyset$
and, therefore, $N'=0$ and $M\downarrow_{\NF}^{\F}\cong N$. On the
other hand, in the case of $\F_{2}$, we have that $\mathcal{Y}=\left\{ K\right\} $
where $K$ is the only other group in the $\F_{2}$ orbit of $H_{2}$.
Thus $N'$ is $K$-projective. Since $H_{2}$ is minimal $\F$-centric
then so is $K$ and since $N'$ is $\F$-centric then $N'_{J}=0$
for every $J\lneq K$ and, therefore, we can conclude from Theorem
\ref{thm:Higman's-criterion.} that $N'$ has vertex $K$.
\end{example}

\bibliographystyle{unsrt}
\bibliography{bibliography}

\begin{thebibliography}{10}

\bibitem{BurnsideRingFusionSystemsDiazLibman}
Antonio D{\'{i}}az and Assaf Libman.
\newblock The {B}urnside ring of fusion systems.
\newblock {\em Advances in Mathematics}, 222(6):1943--1963, December 2009.

\bibitem{TwoClassificationWebb}
Peter Webb.
\newblock Two classifications of simple {M}ackey functors with applications to
  group cohomology and the decomposition of classifying spaces.
\newblock {\em Journal of Pure and Applied Algebra}, 88(1):265--304, 1993.

\bibitem{FrobeniusCategoriesPuig}
Llu{\'{i}}s Puig.
\newblock {F}robenius categories.
\newblock {\em Journal of Algebra}, 303(1):309--357, 2006.

\bibitem{GREENTransferTheoremForModularRepresentations}
J.~A. Green.
\newblock A transfer theorem for modular representations.
\newblock {\em Journal of Algebra}, 1(1):73--84, 1964.

\bibitem{GreenAxiomaticRepresentation}
J.~A. Green.
\newblock Axiomatic representation theory for finite groups.
\newblock {\em Journal of Pure and Applied Algebra}, 1(1):41--77, 1971.

\bibitem{SASAKI198298}
Hiroki Sasaki.
\newblock {G}reen correspondence and transfer theorems of {W}ielandt type for
  {G}-functors.
\newblock {\em Journal of Algebra}, 79(1):98--120, 1982.

\bibitem{SubgroupFamiliesControllingpLocalFiniteGroups}
Carles Broto, Nat{\`a}lia Castellana, Jesper Grodal, Ran Levi, and Bob Oliver.
\newblock {Subgroup Families Controlling p-Local Finite Groups}.
\newblock {\em Proceedings of the London Mathematical Society}, 91(2):325--354,
  09 2005.

\bibitem{StructureMackeyFunctors}
Jacques Th{\'{e}}venaz and Peter Webb.
\newblock The structure of {M}ackey functors.
\newblock {\em Transactions of the American Mathematical Society},
  347(6):1865--1961, 1995.

\bibitem{IntroductionToFusionSystemsLinckelmann}
Markus Linckelmann.
\newblock Introduction to fusion systems.
\newblock In {\em In Group representation theory, EPFL}, pages 79--113. Press,
  2007.

\bibitem{JackowskiMcClureHomotopyDecompositionViaAbelianSubgroups}
Stefan Jackowski and James McClure.
\newblock Homotopy decomposition of classifying spaces via elementary abelian
  subgroups.
\newblock {\em Topology}, 31(1):113--132, 1992.

\bibitem{fuserdFusionBouc}
Serge Bouc.
\newblock Fused {M}ackey functors.
\newblock {\em Geom. Dedicata}, 176:225--240, 2015.

\bibitem{MackeyFunctorsAndBisets}
I.~Hambleton, L.~R. Taylor, and E.~B. Williams.
\newblock Mackey functors and bisets.
\newblock {\em Geometriae Dedicata}, 148(1):157--174, jan 2010.

\bibitem{bouc-serge-biset-functors}
Serge Bouc.
\newblock {\em Biset functors for finite groups}, volume 1990.
\newblock 01 2010.

\bibitem{GuideToMackeyFunctorsWebb}
Peter Webb.
\newblock A guide to {M}ackey functors.
\newblock {\em Handbook of Algebra}, 2, 12 2000.

\bibitem{ReehTransferCharacteristicIdempotentFS}
Sune~Precht Reeh.
\newblock Transfer and characteristic idempotents for saturated fusion systems.
\newblock {\em Advances in Mathematics}, 289:161--211, 2016.

\bibitem{ThevenazWebb1989SimpleMackeyFunctors}
Jacques Th{\'{e}}venaz and Peter Webb.
\newblock Simple {M}ackey functors.
\newblock In {\em roc. of 2nd International Group Theory Conference,
  Bressanone}, 1990.

\bibitem{MethodsOfRepresentationTheoryCurtisReiner}
C.W. Curtis and I.~Reiner.
\newblock {\em Methods of Representation Theory: Vol.: 1. : With Applications
  to Finite Groups and Orders}.
\newblock Pure and Applied Mathematics - Wiley. John Wiley \& Sons, 1981.

\bibitem{NagaoHirosiRepresentationsOfFiniteGroups}
Hirosi Nagao and Yukio Tsushima.
\newblock {\em Representations of finite groups}.
\newblock Academic Press, London, 1989.

\bibitem{StancuEquivalentDefinitions}
Radu Stancu.
\newblock Equivalent definitions of fusion systems.
\newblock Preprint, 2003.

\bibitem{RuizViruelClassificationOverExtraspecial.}
Albert Ruiz and Antonio Viruel.
\newblock The classification of $p$-local finite groups over the extraspecial
  group of order $p^3$ and exponent $p$.
\newblock {\em Mathematische Zeitschrift}, 248:45--65, 09 2004.

\end{thebibliography}

\end{document}